\documentclass[11pt,letterpaper,twoside,reqno,nosumlimits]{amsart}

\usepackage[usenames,dvipsnames]{xcolor}
\usepackage{fancyhdr}
\usepackage{amsmath,amsfonts,amsbsy,amsgen,amscd,amssymb,amsthm}
\usepackage{url}

\usepackage{mathtools}

\usepackage[font=small,margin=0.25in,labelfont={bf},labelsep={colon}]{caption}

\usepackage{subcaption}

\usepackage{tikz}
\usepackage{microtype}
\usepackage{enumitem}

\definecolor{dark-gray}{gray}{0.3}
\definecolor{dkgray}{rgb}{.4,.4,.4}
\definecolor{dkblue}{rgb}{0,0,.5}
\definecolor{medblue}{rgb}{0,0,.75}
\definecolor{rust}{rgb}{0.5,0.1,0.1}

\usepackage[colorlinks=true]{hyperref}

\hypersetup{urlcolor=rust}
\hypersetup{citecolor=dkblue}
\hypersetup{linkcolor=dkblue}

\usepackage{setspace}

\usepackage{graphicx}
\usepackage{booktabs,longtable,tabu} 
\usepackage{multirow} 

\usepackage{float}

\usepackage[full]{textcomp}

\usepackage[scaled=.98,sups]{XCharter}
\usepackage[scaled=1.04,varqu,varl]{inconsolata}
\usepackage[type1]{cabin}
\usepackage[charter,vvarbb,scaled=1.07]{newtxmath}
\usepackage[cal=boondoxo]{mathalfa}
\usepackage[T1]{fontenc}

\usepackage{bm} 

\graphicspath{{figures/}}

\newtheorem{theorem}{Theorem}[section]
\newtheorem{lemma}[theorem]{Lemma}

\newtheorem{proposition}[theorem]{Proposition}

\newtheorem{corollary}[theorem]{Corollary}

\theoremstyle{definition}

\newtheorem{assumption}[theorem]{Assumption}

\numberwithin{equation}{section}
\numberwithin{figure}{section}
\numberwithin{table}{section}

\floatstyle{plaintop}
\newfloat{recipe}{thp}{lor}
\floatname{recipe}{Recipe}
\numberwithin{recipe}{section}

\providecommand{\mathbold}[1]{\bm{#1}}  

\renewcommand{\phi}{\varphi}

\renewcommand{\mid}{\mathrel{\mathop{:}}}

\newcommand{\iunit}{\mathrm{i}}

\providecommand{\mathbbm}{\mathbb} 

\newcommand{\R}{\mathbbm{R}}

\newcommand{\sgn}[1]{\operatorname{sgn}{#1}}

\newcommand{\diff}[1]{\mathrm{d}{#1}}
\newcommand{\idiff}[1]{\, \diff{#1}}

\newcommand{\mtx}[1]{\mathbold{#1}}

\newcommand{\triplenorm}[1]{{\left\vert\kern-0.25ex\left\vert\kern-0.25ex\left\vert #1
            \right\vert\kern-0.25ex\right\vert\kern-0.25ex\right\vert}}

\allowdisplaybreaks

\usepackage[numbers]{natbib}
\usepackage{diagbox}

\usepackage{soul}

\newcommand{\om}{\omega}

\mathtoolsset{showonlyrefs}

\newcounter{mylabelcounter}

\makeatletter
\newcommand{\labelText}[2]{%
#1\refstepcounter{mylabelcounter}%
\immediate\write\@auxout{%
  \string\newlabel{#2}{{1}{\thepage}{{\unexpanded{#1}}}{mylabelcounter.\number\value{mylabelcounter}}{}}%
}%
}

\begin{document}

\title[Self-similar blowups with singular profiles of the gCLM model]{Self-similar finite-time blowups with singular profiles of the generalized Constantin--Lax--Majda model: theoretical and numerical investigations}
\author[D. Huang, J. Tong, X. Wang]{De Huang$^1$, Jiajun Tong$^2$, and Xiuyuan Wang$^3$}
\thanks{$^1$ School of Mathematical Sciences, Peking University. Email: dhuang@math.pku.edu.cn}
\thanks{$^2$ Beijing International Center for Mathematical Research, Peking University. Email: tongj@bicmr.pku.edu.cn}
\thanks{$^3$ School of Mathematical Sciences, Peking University. Email: wangxiuyuan@stu.pku.edu.cn}

\begin{abstract}
    For the generalized Constantin--Lax--Majda model with a parameter $a$, we investigate novel scenarios of self-similar finite-time blowups that arise from a new setting where the smooth initial data satisfy a certain derivative degeneracy condition. In this setting, our numerical study reveals distinct self-similar blowup behaviors depending on the sign of $a$. For $a>0$, we observe one-scale self-similar blowups with regular profiles that have not been found in previous studies. In contrast, for $a\le 0$, we discover a novel two-scale self-similar blowup scenario where the outer profile converges to a singular function at the blowup time while the inner profile remains regular on a much smaller scale. Correspondingly, an $a$-parameterized family of singular self-similar profiles with explicit expressions are constructed for $a<0$ and shown to match nicely with the limiting profiles obtained in numerical simulation. In particular, for the specific case of $a=0$, we rigorously prove the convergence of the outer profile to an explicit singular function in self-similar coordinates. Furthermore, we demonstrate the two-scale nature of the blowup in this scenario by showing that the local inner profile behavior around the singularity point of the outer profile is governed by a traveling wave on a smaller scale. To support this observation, we rigorously establish the existence of such traveling wave solutions via a fixed-point method.
\end{abstract}

\maketitle

\vspace{-4mm}
\setcounter{tocdepth}{1}
\tableofcontents

\vspace{-6mm}
\section{Introduction}
We consider the 1D generalized Constantin--Lax--Majda (gCLM) equation
\begin{equation}\label{eqt:gCLM}
    \omega_t+au\omega_x=u_x\omega,\quad u_x=\mtx{H}(\omega),
\end{equation}
for $x \in \mathbb{R}$, where $\mtx{H}(\cdot)$ denotes the Hilbert transform on the real line. This equation is a 1D model for the vorticity formulation of the 3D incompressible Euler equations, proposed to study the competitive relation between advection and vortex stretching. In
particular, $\omega$ models the vorticity, and the nonlinear terms $u\omega_x$ and $u_x\omega$ model the advection term and
the vortex stretching term, respectively. The 3D Biot--Savart law that recovers the velocity from the
vorticity is modeled by $u_x = \mtx{H}(\omega)$, which has the same scaling as the original Biot--Savart law.

The fundamental question on the global regularity of the 3D Euler equations with smooth initial data of finite energy remains one of the most challenging open problems in fluid dynamics. It is widely believed that the vortex stretching effect has the potential to induce an infinite growth of the vorticity in finite time. The first construction of a (stable) self-similar finite-time blowup for the 3D incompressible Euler equations on $\mathbb{R}^3$ was established by Elgindi \cite{elgindi2021finite} in the axisymmetric setting
from $C^{1,\alpha}$ initial velocity without swirl for sufficiently small $\alpha$ (with stability of the blowup discussed in \cite{elgindi2021stable}). Under some change of variables, the $C^{1,\alpha}$ regularity of the solution gives rise to a small coefficient $\alpha$ in the advection term that weakens its effect, which is similar to the role of the parameter $a$ in \eqref{eqt:gCLM}. Later, Cordoba et al. \cite{cordoba2025finite} also constructed a finite-time blowup of the axisymmetric incompressible Euler equations from $C^{1,\alpha}$ initial velocity without swirl for small $\alpha$. Instead of relying on self-similarity, they constructed the finite-time blowup in a novel way by manipulating an infinite cascade of vorticity. More recently, several breakthroughs have further advanced this direction by extending singularity formation in the axisymmetric setting for $C^{1,\alpha}$ initial velocity without swirl up to the threshold case $\alpha=1/3-$ \cite{shkoller2026incompressible,chen2026asymptotically,chen2026asymptotically2,shao2026self}. Note that for $\alpha\geq 1/3$, global regularity has been established for axisymmetric $C^{1,\alpha}$ initial velocity without swirl \cite{danchin2007axisymmetric,shao2026global}. 

As for smooth initial data, finite-time singularity formation has also been constructed for the axisymmetric incompressible Euler equations in a different scenario. Chen and Hou \cite{chen2022stable,chen2025singularity} employed rigorous computer-assisted proof to show for the first time that the 3D vortex stretching can actually dominate the advection and lead to an asymptotically self-similar finite-time singularity from smooth initial data with swirl in the presence of a solid boundary.

However, whether a finite-time singularity can happen for the 3D incompressible Euler equations from smooth initial data in the free space $\mathbb{R}^3$ still remains open.
Nevertheless, Hou and Huang \cite{hou2022potential} presented numerical evidence of a potential self-similar finite-time blowup of the 3D Euler equations from smooth initial data. In particular, the self-similar solution they discovered exhibits a two-scale feature in the sense that the outer profile becomes more and more singular while the local inner profile remains smooth on a much smaller spatial scale. The mechanism of this self-similar blowup has not been fully understood. Therefore, it is still worthwhile to explore more sophisticated mechanisms that lead to novel self-similar finite-time blowups from smooth initial data for simplified models of the 3D Euler equations such as the gCLM model.

The original version of \eqref{eqt:gCLM} with $a=0$ was proposed by Constantin, Lax and Majda \cite{constantin1985simple} to demonstrate that a nonlocal vortex stretching term can lead to finite-time blowup in the absence of advection. Later, De Gregorio \cite{de1990one} included an advection term $u\omega_x$ back in the equation (known as the De Gregorio model) and conjectured the occurrence of a finite-time singularity nonetheless. As a generalization, Okamoto, Sakajo and Wunsch \cite{okamoto2008generalization} introduced the real parameter $a$ to modify the effect of advection in the competition against vortex stretching. Hence, equation \eqref{eqt:gCLM} is also referred to as the Okamoto--Sakajo--Wunsch model.

Motivated by the long-standing problem on finite-time blowup of the 3D incompressible Euler equations, finite-time singularity formation of the gCLM model for a wide range of $a$ has been studied extensively in the literature. In view of the scaling property of equation \eqref{eqt:gCLM}, we are particularly interested in finite-time blowups that take the self-similar form
\begin{equation}\label{eqt:exact_selfsimilar_ansatz}
    \omega(x,t)=(T-t)^{\lambda}\cdot\bar\Omega\left(\frac{x}{(T-t)^{\gamma}}\right),
\end{equation}
where $\bar\Omega$ is referred to as the self-similar profile, and $\lambda,\gamma$ are the scaling factors. Plugging the ansatz into \eqref{eqt:gCLM}, assuming that $\bar\Omega$ is sufficiently regular, and taking $t \to T^-$ yields that the only possible non-zero value for $\lambda$ is $-1$, which corresponds to the quadratic nonlinearity in vortex stretching. The value of $\gamma$ determines the spatial feature of the blowup: A positive $\gamma$ corresponds to a spatially focusing blowup at $x = 0$, while a negative $\gamma$ corresponds to an expanding blowup. The existence of self-similar finite-time blowups of \eqref{eqt:gCLM} has been established for a wide range of the parameter $a$ in the literature \cite{castro2009self,elgindi2020effects,lushnikov2021collapse,chen2020singularity,chen2021finite,huang2023self,huang2024self,wang2025discovery}. Note that an exact self-similar solution as in \eqref{eqt:exact_selfsimilar_ansatz} means that the initial condition $\omega(x,0)$ is already a rescaling of the self-similar profile $\bar\Omega$, which is not favorable if $\bar\Omega$ is not smooth or does not decay sufficiently fast in the far field. In order to construct singularity formation for more general initial data (especially for smooth ones), we need to study self-similar finite-time blowups in an asymptotic sense that
\begin{equation}\label{eqt:asymptotic_selfsimilar}
    \omega(x,t)=(T-t)^{-1}\cdot\Omega\left( \frac{x}{(T-t)^{\gamma}},t\right) ,
\end{equation}
where $\Omega(X,t)\to \bar \Omega(X)$ as $t\to T^-$ for some exact self-similar profile $\bar \Omega$. More rigorously, we say that $\omega$ develops an asymptotically self-similar blowup (with $\lambda=-1$) at some finite time $T\in(0,+\infty)$, if
\[
    \lim_{t\to T^-}(T-t)\cdot \omega\big((T-t)^{\gamma}X,t\big)=\bar\Omega(X)
\]
for some $\gamma\in \mathbb{R} $ and some (limiting) profile function $\bar\Omega$.
Naturally, an asymptotically self-similar blowup requires not only the existence of a profile $\bar\Omega$ but also its stability in some sense. One common way to accomplish this is by applying a time-dependent change of variables that transforms the gCLM equation \eqref{eqt:gCLM} into a dynamic equation of the profile $\Omega(X,t)$ and then studying the evolution and potential convergence of $\Omega(X,t)$. This approach is usually referred to as the dynamic rescaling method. Stability of self-similar profiles of the gCLM model has been established through this approach in a few special cases; see \cite{elgindi2021stable,chen2021finite,chen2020singularity}. We refer to \cite{elgindi2020effects} and \cite{huang2024self} for a more comprehensive review on related results.

It is worth noticing that stable self-similar blowup profiles of the gCLM model that have been obtained in previous research so far are interiorly smooth or at worst H\"older continuous. This naturally raises the question whether the gCLM equation can also develop self-similar finite-time blowups with singular profiles (not in $L^\infty$) from smooth initial data. We remark that earlier works typically assume that the initial solution $\omega_0=\omega(\cdot,0)$ is odd symmetric (which helps to locate the blowup at the symmetry point $x=0$) and has a non-zero derivative at the expected blowup point, i.e. $\omega'_0(0)\ne 0$. We will refer to such initial data as \textit{non-degenerate} initial data in the sequel. In contrast, we call an initial datum \textit{degenerate} if $\omega$ has zero first derivative at the eventual blowup point. \textit{The main purpose of this paper is to find out whether and how self-similar blowups can develop from degenerate initial data and in what ways they are different from those in the non-degenerate case}.

Our numerical investigation in this problem reveals a whole new picture of more complicated blowup behavior of the gCLM model \eqref{eqt:gCLM}. For $a>0$, we observe that smooth degenerate initial data could lead to asymptotically self-similar finite-time blowups with undiscovered regular profiles that share the same type of derivative degeneracy as the initial data. As far as we know, such degenerate self-similar profiles of the gCLM model have not been reported in the literature before. More interestingly, for $a\le0$, we find that \textit{smooth degenerate initial data can lead to asymptotically self-similar finite-time blowups with singular profiles} in the sense that the profile $\Omega(X,t)$ in \eqref{eqt:asymptotic_selfsimilar} not only converges in some sense but also becomes infinitely large at a certain point as $t\to T^-$ (therefore $\bar\Omega$ not in $L^\infty$).
This new self-similar blowup phenomenon is remarkably different from all the previously studied self-similar blowups with non-degenerate initial data and therefore broadens our understanding of finite-time singularity formation for models of the Euler equations. In Figure \ref{fig:a_neg_1_profiles}, we compare the formation of a regular self-similar profile (Figure \ref{fig:a_neg_1_profiles}(c)) in the development of a blowup to that of a singular self-similar profile (Figure \ref{fig:a_neg_1_profiles}(d)).
Note that an asymptotically self-similar blowup of the form \eqref{eqt:asymptotic_selfsimilar} with a singular limiting profile implies that, even though the overall blowup of $\om$ is at the scale of $(T-t)^{-1}$, the maximum of $|\om|$ actually blows up much faster than $(T-t)^{-1}$, which is not typical in view of the quadratic nonlinearity of the equation. Such non-typical growth of $\om$ in the gCLM model has been rigorously constructed in \cite{huang2025multiscale} for $a=0$ and is now observed for $a<0$.

At the theoretical level, we construct a family of singular self-similar profiles $\bar\Omega_a$ with explicit expressions parameterized by $a$ for $a\leq0$, each of which gives an exact self-similar solution of the form \eqref{eqt:exact_selfsimilar_ansatz} (though in this case $\om$ can only be understood as a weak solution). For the special case $a=0$, we can additionally prove the convergence of the profile function $\Omega(X,t)$ to the self-similar profile $\bar\Omega_a$. As for $a<0$, our numerical simulation further shows that, with a particular setting of the initial data (that will be made clearer in Section \ref{sec:blowup_scenarios}), the profile $\Omega(X,t)$ in \eqref{eqt:asymptotic_selfsimilar} will converge to $\bar\Omega_a$ as $t\to T^-$. With the explicit expressions of $\bar \Omega_a$ in hand, it is possible to follow the methods in \cite{elgindi2021stable,chen2021finite,chen2020singularity} to theoretically prove the stability of $\bar \Omega_a$ and therefore rigorously establish asymptotically self-similar finite-time blowups with singular profiles for all $a<0$. We leave this to future work.

In addition to singular self-similar profiles, another interesting feature of our new blowup scenario for $a<0$ is the breakdown of the strictly outgoing property. More precisely, for the limiting self-similar profile, the induced flow field in the dynamic rescaling equation (i.e. the term $c_l X+aU$ in \eqref{eqt:dynamic_rescaling} below) is not strictly positive for every $X>0$. Our numerical simulation also indicates that the singularity point of the limiting profile occurs precisely at a stagnation point of the flow field. This is also true for the explicitly constructed singular self-similar profiles $\bar\Omega_a$ for $a<0$. We will return to the discussion of this phenomenon later in Section \ref{sec:a_le_0} with more details based on our numerical observations. Notably, the absence of the strictly outgoing property for focusing self-similar blowups (i.e. $\gamma>0$) only occurs in the case $a<0$ (where $c_l X$ and $aU$ have opposite signs). This heuristically explains why self-similar blowups with singular profiles can appear for $a<0$ but not for $a>0$. In contrast, all previously found stable, focusing (asymptotically) one-scale self-similar blowups of the gCLM model satisfy the strictly outgoing property and have bounded self-similar profiles (see e.g. \cite{elgindi2020effects,huang2024self}). This comparison suggests a close connection between the breakdown of strictly outgoing property and the emergence of singular self-similar profiles. We remark that, as discussed in \cite{elgindi2026dynamics,constantin2026putativeselfsimilarityincompressible3d}, the strictly outgoing property has been found to play an important role in the formation and stability of putative self-similar finite-time blowups for the 3D incompressible Euler equations.

\begin{figure}[h]
    \centering
    \begin{subfigure}{0.45\textwidth}
        \centering
        \includegraphics[width=\textwidth]{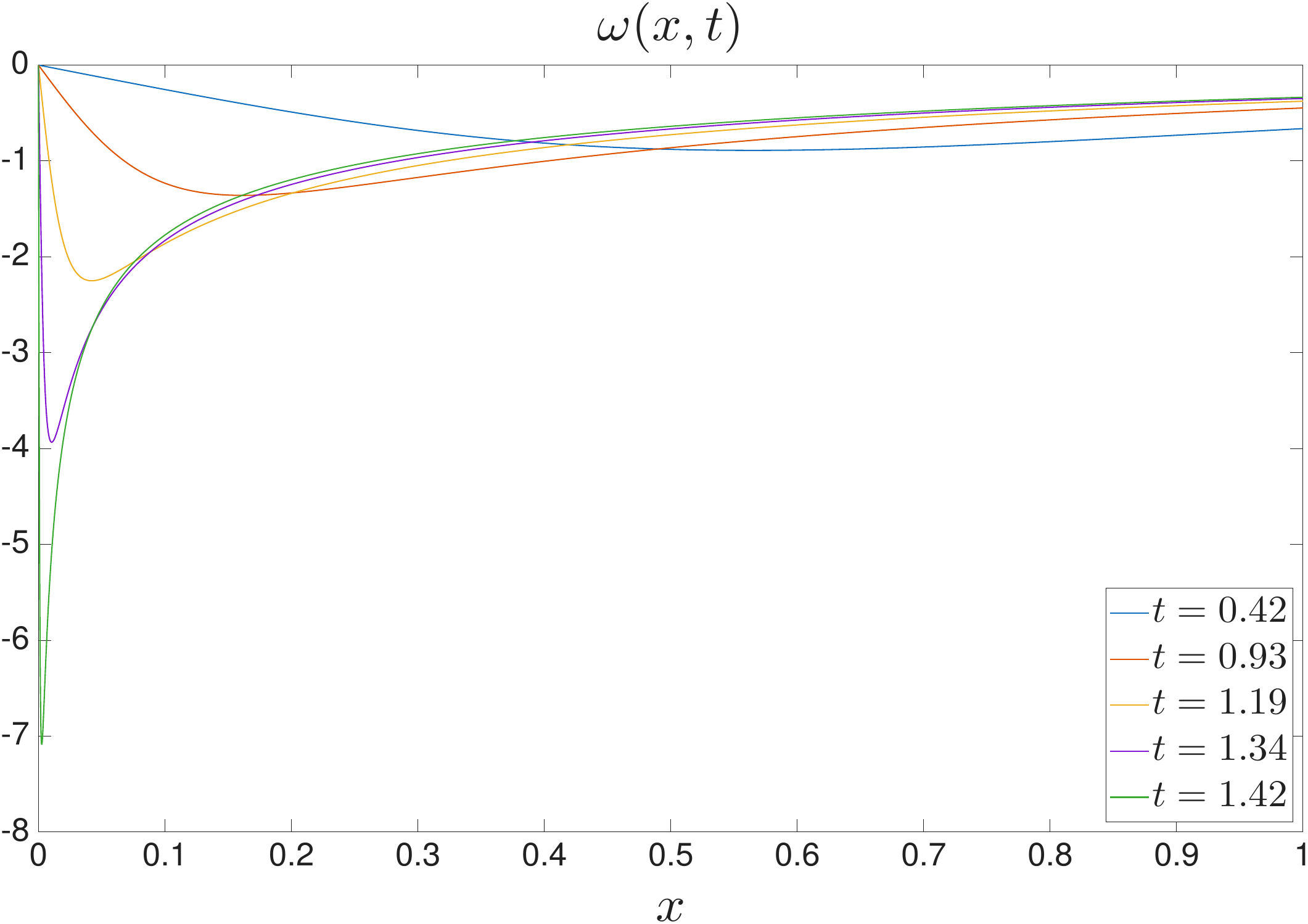}
        \caption{{Evolution of $\omega$ from non-degenerate initial data}}
        \label{fig:a_neg_1_nondeg_phys}
    \end{subfigure}
    \qquad 
    \begin{subfigure}{0.45\textwidth}
        \centering
        \includegraphics[width=\textwidth]{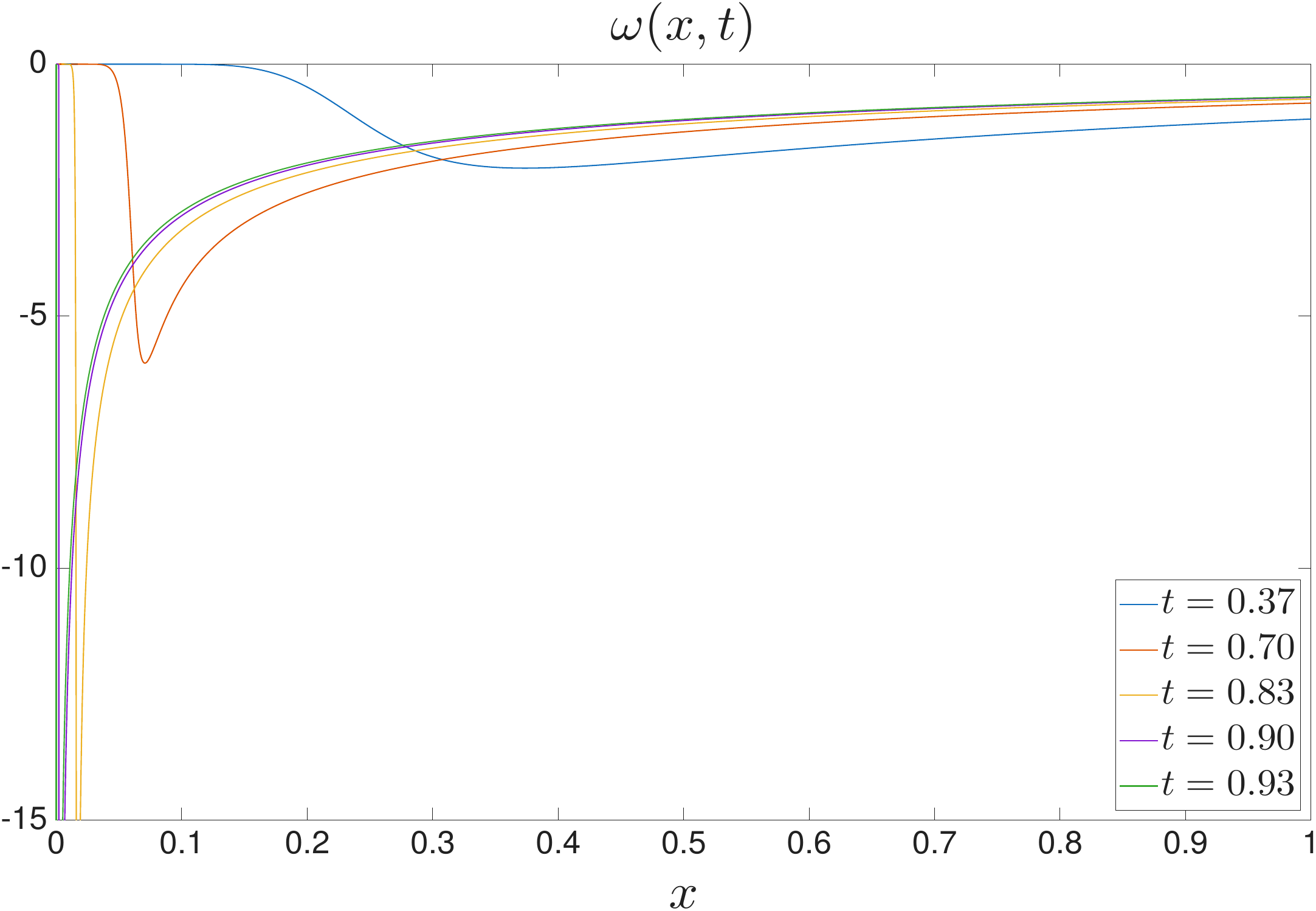}
        \caption{{Evolution of $\omega$ from degenerate initial data}}
        \label{fig:a_neg_1_deg_phys}
    \end{subfigure}
    \begin{subfigure}{0.45\textwidth}
        \centering
        \includegraphics[width=\textwidth]{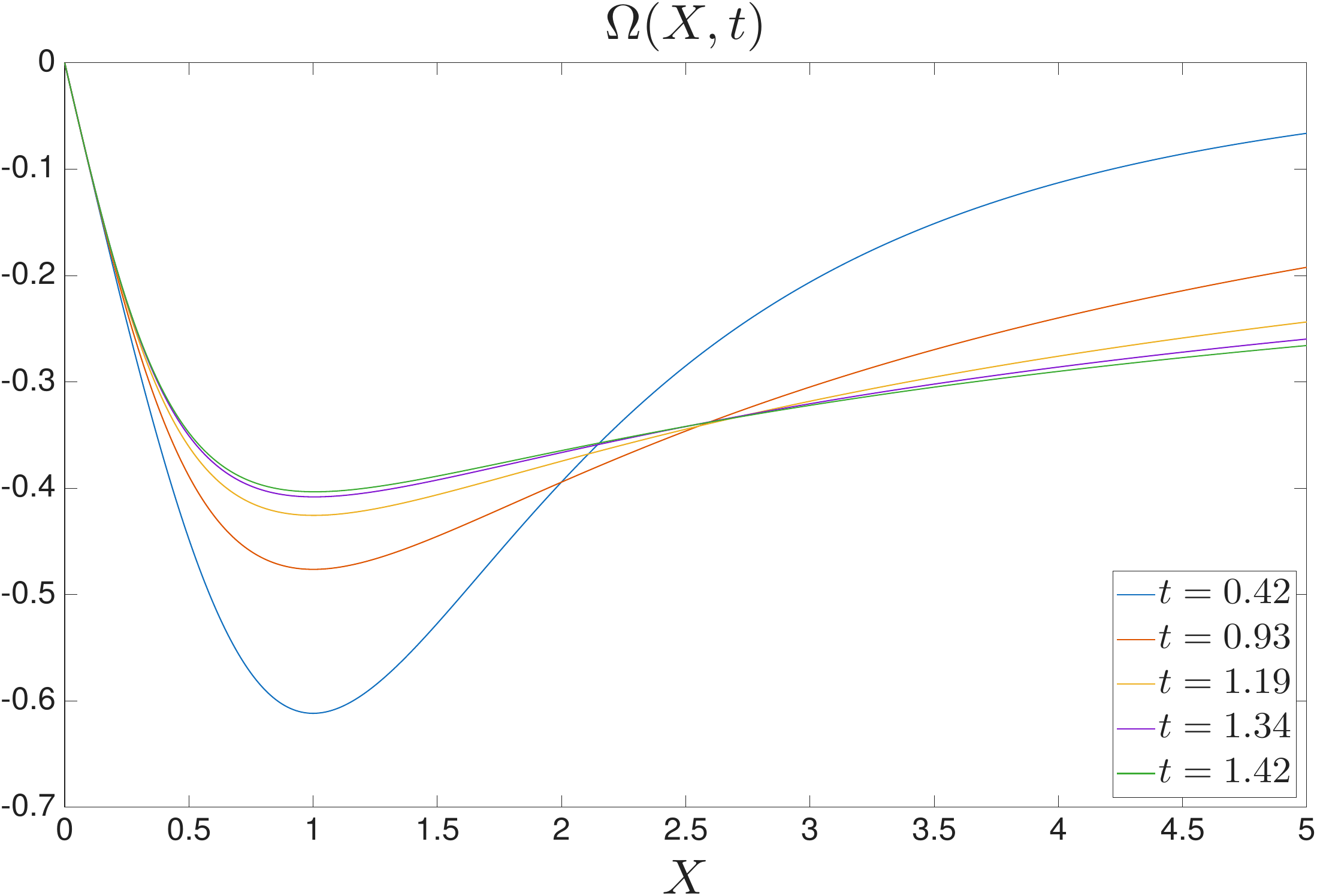}
        \caption{{Profile evolution from non-degenerate initial data}}
        \label{fig:a_neg_1_nondeg}
    \end{subfigure}
    \qquad 
    \begin{subfigure}{0.45\textwidth}
        \centering
        \includegraphics[width=\textwidth]{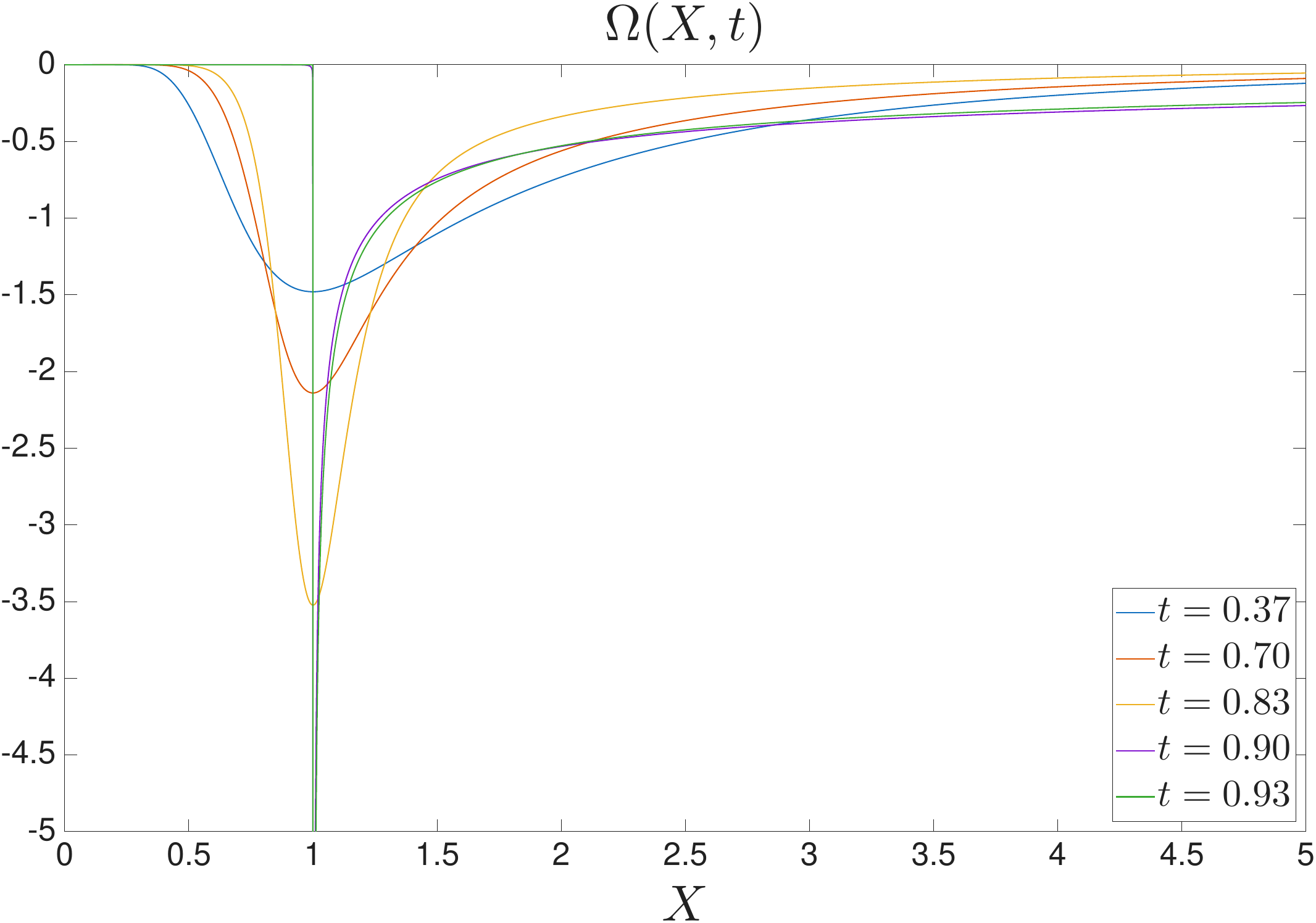}
        \caption{{Profile evolution from degenerate initial data}}
        \label{fig:a_neg_1_deg}
    \end{subfigure}
    \caption{First row: Evolution of $\omega$ (in the original spatial coordinate) for $a=-1$ (a) with non-degenerate odd-symmetric initial data and (b) with degenerate odd-symmetric initial data, respectively. Second row: Corresponding evolution of the profile (in the dynamically rescaled spatial coordinate) (c) with non-degenerate odd-symmetric initial data and (d) with degenerate odd-symmetric initial data, respectively. A regular profile arises from non-degenerate initial data, while a singular profile with singularities at $X=\pm 1$ arises from degenerate initial data. Here $X=x/(T-t)^\gamma$ is the dynamically rescaled spatial variable. Only the parts for $x\geq 0$ (or $X\geq 0$) are plotted due to the odd symmetry.}
    \label{fig:a_neg_1_profiles}
\end{figure}

Moreover, for $a\leq0$, if we zoom in around the singularity point of the (outer) profile $\Omega(X,t)$, we observe another inner (smaller) scale of the solution with an inner profile. That is, the solution has a two-scale behavior. Let the larger scale be denoted by $(T-t)^\gamma$, and let the trajectory of the maximum point of $|\om|$ be given by $x=r(t)(T-t)^\gamma$ for some function $r(t)$.
Zooming in around that point,
we observe that the inner profile, after being properly rescaled in magnitude, remains regular on a secondary spatial scale denoted by $(T-t)^{\hat \gamma}$ with some $\hat \gamma >\gamma$.
In other words, for $x=r(t)(T-t)^\gamma+O((T-t)^{\hat \gamma})$ the solution behaves like
\begin{equation}\label{eqt:2scale-ansatz}
    \omega(x,t)=(T-t)^{\hat\lambda}\left( \widehat \Omega\left( \frac{x-r(t)(T-t)^{\gamma}}{(T-t)^{\hat\gamma}}\right)+o(1)\right)
\end{equation}
for some regular inner profile $\widehat \Omega$ and some $\hat \lambda <-1$. The observation $\hat \lambda<-1$ indicates that the maximum of $|\om|$ grows faster than $(T-t)^{-1}$, i.e. $\lim_{t\to T^-}(T-t)\cdot\|\omega(\cdot,t)\|_{L^\infty} = +\infty$, which is sometimes referred to as a Type II blowup in the literature. By a formal asymptotic analysis, we show that the inner profile (if exists) must be an exact traveling wave profile of the gCLM model \eqref{eqt:gCLM}. To support our observation and analysis of this two-scale blowup phenomenon, we also prove the existence of such traveling wave profiles for all $a\le1$ via a fixed-point approach. We remark that for the original CLM model ($a=0$), the two-scale self-similar blowup \eqref{eqt:2scale-ansatz} has been rigorously established in \cite{huang2025multiscale}.
However, how the outer profile $\Omega$ behaves as $t\to T^-$ was not discussed in \cite{huang2025multiscale} but will be rigorously addressed in this paper.
Moreover, this two-scale self-similarity setting \eqref{eqt:2scale-ansatz} was also adopted in \cite{hou2022potential} for understanding the numerically observed potential two-scale self-similar blowup of the axisymmetric incompressible Euler equations in $\mathbb{R}^3$. 
In fact, locally self-similar blowups with two-scale features as in \eqref{eqt:2scale-ansatz} have been studied and established as Type II blowups in other PDE systems such as the Keller--Segel system \cite{collot2023collapsing} and the semilinear heat equation \cite{del2020type,del2020new}. 
We hope that our numerical simulation of self-similar blowups with singular profiles and a two-scale feature will in the future turn into a rigorous proof and in the end help us construct finite-time blowups of the incompressible Euler equations from smooth initial data in $\mathbb{R}^3$.

The main contributions of this paper are the following.
\begin{enumerate}
    \item Theoretical aspect:
          \begin{itemize}
              \item For the case $a=0$ with degenerate initial data, we prove the convergence of the (outer) profile to an exact singular function under certain assumptions and normalization conditions, thereby establishing an asymptotically self-similar blowup in the form of \eqref{eqt:asymptotic_selfsimilar} with a singular profile in this case. (Theorem \ref{thm:convergence_a_0} and Corollary \ref{cor:convergence_a_0_halfline})
              \item For $a<0$, we construct a family of exact self-similar solutions in the form of \eqref{eqt:exact_selfsimilar_ansatz} with explicit singular profiles $\bar \Omega_a$ parameterized by $a$. (Theorem \ref{thm:explicit_singular_solution})
              \item We prove the existence of traveling wave solutions for the gCLM model for all $a\le1$ using a fixed-point method. (Theorem \ref{thm:traveling_wave_existence})
          \end{itemize}
    \item Numerical aspect:
          \begin{itemize}
              \item For $a>0$, we show that degenerate initial data can lead to self-similar finite-time blowups with regular profiles that have the same type of degeneracy.
              \item For $a<0$, we show that degenerate initial data can lead to self-similar finite-time blowups with singular profiles. In a particular setting, we show that the profile converges to $\bar \Omega_a$.
              \item We provide numerical evidence of the two-scale behavior of self-similar blowups for $a<0$ with degenerate initial data, and we show that the inner profile on a smaller scale converges to a traveling wave profile.
          \end{itemize}
\end{enumerate}

It is worth mentioning that the gCLM model on the torus $\mathbb{T}$ has also been widely considered in parallel studies \cite{de1996partial,okamoto2008generalization,jia2019gregorio,lei2020constantin,chen2021finite,lushnikov2021collapse,chen2020singularity,chen2025regularity,chen2021slightly}. It is therefore natural to ask whether the novel two-scale blowup scenarios discovered here can also occur in the periodic setting. Heuristically speaking, if a self-similar finite-time blowup on $\R$ is stable and focusing (i.e. $\gamma >0$ in \eqref{eqt:asymptotic_selfsimilar}), then it can also happen in the periodic setting, since the effect of periodicity is asymptotically negligible in the focusing blowup region of size $(T-t)^\gamma$. Indeed, all the two-scale self-similar blowups for $a\leq0$ discovered in this work are numerically observed to be robust and satisfy $\hat\gamma>\gamma>0$. This suggests that analogous two-scale blowups should also happen on the torus $\mathbb{T}$ in the case $a\leq 0$ for initial data under a similar degeneracy condition.

The paper is organized as follows. In Section \ref{sec:blowup_scenarios}, we introduce the dynamic rescaling framework, review previous results on self-similar blowups with non-degenerate initial data, and present new blowup scenarios involving singular profiles arising from degenerate initial data. We also discuss the multiscale structure of these solutions and show that the inner profile converges to a traveling wave. Section \ref{sec:a_eq_0} focuses on the case $a=0$ and proves that with odd degenerate initial data, the profile converges to an exact singular solution under certain assumptions. Section \ref{sec:a_gt_0} details the numerical computations and convergence analysis for one-scale blowups from degenerate initial data when $a>0$. Section \ref{sec:a_le_0} presents numerical evidence that, for $a<0$, degenerate initial data can lead to finite-time blowups with singular self-similar profiles. Section \ref{sec:numerical_multiscale} investigates the associated two-scale behavior for $a<0$, including a comparison between numerically extracted inner profiles and traveling wave profiles, as well as a scaling analysis of the two-scale blowup. In Section \ref{sec:existence_traveling_wave}, we prove the existence of traveling wave solutions for the gCLM model using a fixed-point method and derive regularity and tail estimates for these solutions. Finally, technical lemmas and details of the numerical implementation are provided in the appendices.

\section{Dynamic rescaling formulation and self-similar blowups of the gCLM model}{\label{sec:blowup_scenarios}}
We first review the dynamic rescaling method and its application to self-similar finite-time blowups of the gCLM model \eqref{eqt:gCLM}. Based on this method (and its variation), self-similar blowups with non-degenerate odd symmetric regular profiles have been found for a wide range of the parameter $a$. We then introduce novel self-similar blowup scenarios of the gCLM model arising from degenerate initial data (to be clarified later). Most interestingly, when $a\le 0$, this new class of self-similar blowups are characterized by singular self-similar profiles and locally two-scale behavior. 

Before we proceed, we clarify how to recover $u$ from $\omega$ in the gCLM equation \eqref{eqt:gCLM}. The first step is to determine $u_x$ from $\omega$ by the Hilbert transform:
\begin{equation}\label{eqt:hilbert_transform}
    u_x(x,t)=\mtx{H}(\omega)(x,t)=\frac{1}{\pi} \mathrm{P.V.}\int_{\mathbb{R}}\frac{\omega(y,t)}{x-y}\idiff{y}.
\end{equation}
Then we recover $u$ from $u_x$ by integrating.
One way to do this is to formally write
\[
    u(x,t)=-(-\Delta)^{-1/2}\omega(x,t)=\frac{1}{\pi}\int_{\mathbb{R}}\omega(y,t)\ln|x-y|\idiff{y},
\]
which, however, may not be well defined if $\omega$ does not decay fast enough at infinity.
An alternative way is to fix a point $x_0$ and set $u(x_0,t)=0$ for all time $t$, which gives
\[
    u(x,t)=\int_{x_0}^x \mtx{H}(\omega)(y,t)\idiff{y}=\frac{1}{\pi}\int_{\mathbb{R}} \omega(y,t)\ln\left| \frac{x-y}{x_0-y} \right|\idiff{y}.
\]
In particular, when $\omega$ is odd, a natural choice (as commonly used in previous works) is to set $x_0=0$, thus making $u$ also an odd function. We remark that the choice of $x_0$ does not essentially affect the behavior of the solution; that is, solutions with the same initial data but different choices of $x_0$ are equivalent under a time-dependent translation.
We will rigorously argue this in Appendix \ref{app:translation_invariance}.
Therefore, for simplicity, we fix $x_0=0$ as default, which yields
\begin{equation}\label{eqt:biot_savart_law}
    u(x,t)=\frac{1}{\pi}\int_{\mathbb{R}} \omega(y,t)\ln\left| \frac{x-y}{y} \right|\idiff{y}.
\end{equation}

\subsection{Self-similar finite-time blowup via the dynamic rescaling method}
One important way to study and construct self-similar finite-time blowups is by considering the dynamic rescaling formulation. The idea of dynamic rescaling was introduced in \cite{landman1988rate,mclaughlin1986focusing} to study self-similar blowups of the nonlinear Schr\"odinger equations. A combination of the dynamic rescaling method and the method of computer-assisted proof has been successfully applied to proving self-similar finite-time blowups with smooth initial data for the De Gregorio model \cite{chen2021finite} (\eqref{eqt:gCLM} with $a=1$), the 1D Hou--Luo model \cite{chen2022asymptotically}, and the 3D axisymmetric incompressible Euler equations with boundary \cite{chen2022stable,chen2025singularity}. 

We will explain in detail how to apply the dynamic rescaling method to studying self-similar finite-time blowups of the gCLM model \eqref{eqt:gCLM}. To be specific, we consider the following time-dependent change of variables
\begin{equation}\label{eqt:change_of_variables}
    \omega(x,t)=C_\omega^{-1}\Omega(C_lx,\tau(t)),
\end{equation}
where
\begin{equation}\label{eqt:c_to_C}
C_\omega(\tau)=\exp\int_0^\tau c_\omega(s)\idiff{s},\quad C_l(\tau)=\exp\int_0^\tau c_l(s)\idiff{s}
\end{equation}
are cumulative scaling factors representing the magnitude blowup of $\omega$ and the spatial scale of the blowup, respectively. The relation between the dynamic rescaling time $\tau$ and the physical time $t$ will be determined soon below. This change of variables rescales the solution with the hope that the profile function $\Omega$ will have a more stable behavior if the dynamic rescaling factors $c_\omega(\tau)$ and $c_l(\tau)$ are chosen properly, e.g. via some normalization conditions. Through this change of variables, studying the self-similar finite-time blowup of $\omega(x,t)$ boils down to studying the evolution of the profile function $\Omega(X,\tau)$. In fact, substituting \eqref{eqt:change_of_variables} into the gCLM equation \eqref{eqt:gCLM} yields
\[C_\omega^{-1}\Omega_\tau\tau_t-c_\omega C_\omega^{-1}\Omega\tau_t+c_lC_\omega^{-1} X\Omega_X\tau_t+aC_\omega^{-2}U\Omega_X=C_\omega^{-2}U_X\Omega,\]
where $X=C_lx, U_X=\mtx{H}(\Omega)$, and in view of \eqref{eqt:biot_savart_law},
\begin{equation}\label{eqt:Omega_U}
    U(X,\tau)=\frac{1}{\pi}\int_{\mathbb{R}} \Omega(Y,\tau)\ln\left| \frac{X-Y}{Y} \right|\idiff{Y}.
\end{equation}
Simplifying the above equation gives
\[
    \left( \Omega_\tau+c_lX\Omega_X-c_\omega\Omega \right)\tau_t=C_\omega^{-1}\left( U_X\Omega-aU\Omega_X \right).
\]
To balance the profile equation, it remains to define the dynamic rescaling time $\tau$ by
\begin{equation}\label{eqt:time_rescaling}
\tau'(t)=C_\omega(\tau)^{-1}=\exp\int_0^\tau -c_\omega(s) \idiff{s},\quad \tau(0)=0,
\end{equation}
which implicitly determines the relation between physical time $t$ and the dynamic rescaling time $\tau$. With this definition of $\tau$, we arrive at the famous dynamic rescaling equation of the gCLM model
\begin{equation}\label{eqt:dynamic_rescaling}
    \Omega_\tau+(c_lX+aU)\Omega_X=(c_\omega+U_X)\Omega,\quad U_X=\mtx{H}(\Omega).
\end{equation}
We remark that \eqref{eqt:dynamic_rescaling} is completely equivalent to \eqref{eqt:gCLM} under the change of variables given by \eqref{eqt:change_of_variables}, \eqref{eqt:c_to_C}, and \eqref{eqt:time_rescaling} with the same initial condition $\Omega(X,0)=\omega(x,0)$. We will use $\Omega_0$ and $\omega_0$ to denote initial data in what follows.

We are also particularly interested in steady state solutions to \eqref{eqt:dynamic_rescaling} that satisfy the self-similar profile equation
\begin{equation}\label{eqt:steady_profile}
    (\bar c_lX+a\bar U)\bar \Omega_X=(\bar c_\omega+\bar U_X)\bar \Omega,\quad \bar U_X=\mtx{H}(\bar \Omega),
\end{equation}
where $\bar U$ is related to $\bar \Omega$ as in \eqref{eqt:Omega_U}.
One can easily verify that every non-trivial solution  $(\bar \Omega,\bar c_l,\bar c_\omega)$ to \eqref{eqt:steady_profile} with $\bar c_\omega<0$ corresponds to an exact self-similar finite-time blowup solution to the gCLM model in the form of \eqref{eqt:exact_selfsimilar_ansatz} with $\lambda=-1$ and $\gamma=-\bar c_l/\bar c_\omega$.
Note that if $(\bar \Omega(X), \bar c_l,\bar  c_\omega)$ is a solution to \eqref{eqt:steady_profile}, then
\begin{equation}\label{eqt:scaling_invariance}
    \left(\alpha\bar \Omega(\beta X), \alpha \bar c_l, \alpha \bar c_\omega\right)
\end{equation} is also a solution for any $\alpha\in\mathbb{R},\beta>0$, which indicates that this equation admits two degrees of freedom. Hence, in order to uniquely identify the profile up to rescaling, we need to impose two normalization conditions to fix these two degrees of freedom.

Moreover, if under some normalization conditions the solution tuple $(\Omega,c_l,c_\omega)$ converges to a steady state $(\bar \Omega,\bar c_l,\bar c_\omega)$ of \eqref{eqt:dynamic_rescaling} as $\tau\to+\infty$, then the solution $\omega(x,t)$ would blow up in an asymptotically self-similar manner given by \eqref{eqt:asymptotic_selfsimilar}. We demonstrate this more precisely by the following proposition.

\begin{proposition}\label{prop:convergence_implies_onescale_blowup}
    Assume the solution $(\Omega,c_l,c_\omega)$ to \eqref{eqt:dynamic_rescaling} converges to a non-trivial steady state $(\bar \Omega,\bar c_l,\bar c_\omega)$ (a solution of \eqref{eqt:steady_profile}) with $\bar c_\omega<0$ as $\tau\to+\infty$. In particular, assume the convergence of $\Omega$ is locally uniform, the limiting profile $\bar \Omega$ is continuous, and the scaling factors $c_l,c_\omega$ converge sufficiently fast such that
    \[
        \int_0^{+\infty}\left|c_l(\tau)-\bar c_l\right|\idiff \tau <+\infty,\quad \int_0^{+\infty}\left|c_\omega(\tau)-\bar c_\omega\right|\idiff \tau <+\infty.
    \]
    Then the solution $\omega(x,t)$ develops an asymptotically self-similar blowup in the sense that
    \[
        \lim_{t\to T^-}(T-t)\cdot \omega\big((T-t)^{\gamma}X,t\big)=\frac{1}{|\bar c_\omega|}\bar \Omega(C_0 X)
    \]
    for some constant $C_0>0$, where $\gamma=-\bar c_l/\bar c_\omega$ and $T<+\infty$ is the finite blowup time given by
    \[
        T=\int_0^{+\infty}\left(\exp\int_0^\tau c_\omega(s) \idiff{s}\right)\idiff \tau.
    \]
    In particular, the blowup is focusing and happens at the origin if $\gamma>0$; the blowup is expanding and happens on the whole real line $\mathbb{R}$ if $\gamma<0$.
\end{proposition}

\begin{proof}
    By the assumption on the convergence of $c_\omega(\tau)$, we have 
    \[C_\omega(\tau)=\exp\left(\bar c_\omega \tau + \int_0^\tau (c_\omega(s) - \bar c_\omega)\idiff{s} \right)\lesssim \exp(\bar c_\omega \tau).\]
    As $\bar c_\omega<0$, it follows that the physical time $t$ as a function of dynamic rescaling $\tau$ given by \eqref{eqt:time_rescaling} has a finite limit
    \[
        T:= \lim_{\tau\to +\infty}t(\tau) = \int_0^{+\infty} t_\tau(\tau)\idiff \tau=\int_0^{+\infty} \tau_t^{-1}(\tau)\idiff \tau=\int_0^{+\infty}C_\omega(\tau)\idiff\tau<+\infty.
    \]
    Then by L'H\^opital's rule we can compute that
    \[
        \lim_{t\to T^-}\frac{C_\omega(\tau)}{T-t}=-\lim_{t\to T^-}\frac{\left( C_\omega \right)_\tau}{t_\tau}=-\lim_{t\to T^-} c_\omega(\tau) = -\lim_{\tau\to +\infty} c_\omega(\tau) =-\bar c_\omega>0,
    \]
    i.e. $C_\omega(\tau)\sim -\bar c_\omega (T-t)$ as $t\to T^-$. Moreover, by the assumption on the convergence of $c_l,c_\omega$, we have
    \begin{align*}
            \lim_{t\to T^-}C_l\cdot(T-t)^\gamma
             & =\lim_{t\to T^-}C_l \cdot \left( \frac{C_\omega}{-\bar c_\omega} \right)^\gamma\\
             &=(-\bar c_\omega)^{-\gamma}\lim_{t\to T^-}C_l\cdot C_\omega^{\gamma}\\
             & =(-\bar c_\omega)^{-\gamma}\exp\int_0^{+\infty}( c_l(\tau)+\gamma c_\omega(\tau) )\idiff{\tau}                                                      \\
             & = (-\bar c_\omega)^{-\gamma}\exp\int_0^{+\infty}( (c_l(\tau)-\bar c_l)+\gamma (c_\omega(\tau)-\bar c_\omega) )\idiff{\tau}                          \\
             & =:C_0>0.
    \end{align*}
    Here $\gamma = -\bar c_l/\bar c_\omega$.
    Recall, in view of the change of the variables in \eqref{eqt:change_of_variables} and \eqref{eqt:c_to_C}, we have
    \[
        (T-t)\cdot\omega\big((T-t)^\gamma X,t\big)= \frac{T-t}{C_\omega}\cdot \Omega( C_l\cdot(T-t)^\gamma X, \tau(t) ).
    \]
    Therefore
    \[
            \lim_{t\to T^-}(T-t)\cdot\omega\big((T-t)^\gamma X,t\big)
            = \lim_{t\to T^-}\frac{T-t}{C_\omega}\cdot \Omega( C_l\cdot(T-t)^\gamma X , \tau(t) )=-\frac{1}{\bar c_\omega}\bar \Omega(C_0 X),
    \]
    since the convergence of $\Omega$ is locally uniform.
\end{proof}

As a remark, the assumptions on the convergence of $\Omega$ and on the regularity of $\bar \Omega$ in Proposition \ref{prop:convergence_implies_onescale_blowup} can be weakened (e.g. only in $L^2$ sense), and a similar conclusion (also in a weaker sense) can be derived with some extra technical effort. We state the proposition in the presented form only for demonstration purposes.

One can see from Proposition \ref{prop:convergence_implies_onescale_blowup}  that the magnitude blowup of $\omega$ is always at the scale of $(T-t)^{-1}$ regardless of the particular values of $\bar c_\omega,\bar c_l$ (as long as $\bar c_\omega <0$). What matters is the ratio $\gamma=-\bar c_l/\bar c_\omega$ that determines the spatial shrinking rate of the solution, rather than the individual values of $\bar c_\omega,\bar c_l$. Moreover, the asymptotically self-similar blowup behavior of $\om$ is governed by a rescaling of the limiting profile $\bar \Omega$, regardless of the specific form of the initial data. Nevertheless, the finite blowup time $T$ still relies on the initial condition.

Guided by Proposition \ref{prop:convergence_implies_onescale_blowup}, the problem of finding asymptotically self-similar blowups can be decomposed into two steps:
\begin{enumerate}
    \item Finding a steady state of the dynamic rescaling equation \eqref{eqt:dynamic_rescaling}.
    \item Proving the stability of the steady state.
\end{enumerate}
In general, it is not easy to theoretically construct exact self-similar profiles, let alone prove their stability. A more promising way to investigate the existence and stability of self-similar profiles is to solve the dynamic rescaling equation \eqref{eqt:dynamic_rescaling} numerically, and the key point is to appropriately impose normalization conditions to ensure the convergence of the profile $\Omega$. Recall that the profile equation \eqref{eqt:steady_profile} admits two degrees of freedom due to its scaling property \eqref{eqt:scaling_invariance}. Hence, we need to impose two normalization conditions to fix these two degrees of freedom, which in general also determine how the two variables $c_l(\tau),c_\omega(\tau)$ depend on the profile $\Omega(X,\tau)$ dynamically.

In previous studies on the blowup of the gCLM model, the initial condition $\omega_0$ is usually chosen to be an odd function with negative derivative at the origin, i.e. $\omega_0'(0)<0$. In view of \eqref{eqt:biot_savart_law}, $u_0$ is also an odd function. This odd-symmetry of $\omega$ is then preserved by the gCLM equation \eqref{eqt:gCLM} for all $t$ up to the blowup time, and so is the sign of $\omega_x(0,\tau)$. Therefore, one common normalization condition is to fix the slope of $\Omega$ at the blowup point, i.e. $\Omega_X(0,\tau)=\Omega_X(0,0)$ for all $\tau>0$, which easily leads to
\begin{equation}\label{eqt:1st_normalization}
    c_l(\tau) = c_\omega(\tau) + (1-a)U_X(0,\tau).
\end{equation}
This settles one degree of freedom. Nevertheless, we still need to determine another normalization to settle the second degree of freedom. While the first normalization condition \eqref{eqt:1st_normalization} is commonly used in previous studies, the second normalization condition may differ from case to case, since the stability of the dynamic rescaling equation may heavily depend on the concrete choice of the normalization conditions. 

\subsection{A short survey of existing results}
In this subsection, we review some existing results on self-similar finite-time blowups of the gCLM model \eqref{eqt:gCLM}. Some of these self-similar blowups were constructed using the dynamic rescaling method, for which we will also mention their dynamical normalization conditions. We again emphasize that in all results below, the self-similar profile $\bar \Omega$ is odd symmetric and has a non-zero derivative (or is H\"older continuous) at $X=0$ (non-degenerate in our terminology). 

\begin{itemize}
    \item For the case $a=0$, the gCLM model \eqref{eqt:gCLM} reduces to the original CLM model \cite{constantin1985simple}, and a self-similar blowup solution was first found in \cite{elgindi2020effects} with
          \[
              \bar \Omega_0(X)=-\frac{4X}{1+4X^2},\quad \mtx{H}(\bar \Omega_0)(X)=\frac{2}{1+4X^2},\quad \gamma=1.
          \]
          Here the subscript ``$0$'' corresponds to $a=0$, not the initial condition.
          The stability of this profile was proved in \cite{elgindi2021stable} and \cite{chen2021finite} along with the existence of asymptotic blowup for small $|a|$. Especially, the proof in \cite{chen2021finite} is based on the dynamic rescaling method with one of the normalization conditions chosen as \eqref{eqt:1st_normalization} and the other one as
          \[
              c_\omega(\tau) = 1 - U_X(0,\tau).
          \]
    \item For the case $a=1/2$, an exact self-similar blowup solution was first found in \cite{chen2020singularity} and \cite{lushnikov2021collapse} independently, with the self-similar profile and the scaling factor given by
          \[
              \bar \Omega_{1/2}(X)=-\frac{\sqrt{3/2}\, X}{(3/8+X^2)^2},\quad \mtx{H}(\bar \Omega_{1/2})(X)=\frac{3/8-X^2}{(3/8+X^2)^2},\quad \gamma=\frac{1}{3}.
          \]
          In \cite{chen2020singularity}, the stability of this profile and the existence of asymptotically self-similar blowups for $a$ close to $1/2$ were also proved via the dynamic rescaling method, and the normalization conditions are chosen to be \eqref{eqt:1st_normalization} plus the second one as 
          \[c_\omega(\tau) = \frac{5}{3} - U_X(0,\tau).\]
    \item For the case $a=1$, the gCLM model \eqref{eqt:gCLM} reduces to the De Gregorio model \cite{de1990one}, and \eqref{eqt:1st_normalization} directly leads to $\gamma=-1$. The existence and stability of a compactly supported self-similar profile $\bar\Omega_1$ was obtained in \cite{chen2021finite} using a combination of the dynamic rescaling method and a computer-assisted proof. In addition to \eqref{eqt:1st_normalization}, the second normalization condition was imposed as
          \[
              c_l(\tau)=-\frac{U(L,\tau)}{L},
          \]
          where $x=\pm L$ are the two boundary points of the compact support of $\bar\Omega_1$. More recently, \cite{huang2023self} constructed infinitely many compactly supported self-similar profiles (all with $\gamma=-1$) for the De Gregorio model as eigenfunctions of a linear compact operator, where the leading eigenfunction coincides with $\bar \Omega_1$ up to rescaling.
    \item Recently, a general existence result of interiorly smooth self-similar profiles for all $a\le1$ was established in \cite{huang2024self} using a unified fixed-point argument. To be specific, for each $a$, an exact self-similar profile $\bar\Omega_a$ is constructed as the fixed point of a certain nonlinear map. The profiles above for $a=0,1/2,1$ are special cases in this family. They also provided theoretical estimates for a critical value $a_c$ that divides focusing self-similar blowups and expanding ones, the existence of which was first numerically observed in \cite{lushnikov2021collapse}.
    \item Remarkably, a universal singular self-similar profile for all values of $a$ was constructed in \cite{castro2009self} in a closed form
          \[
              \bar \Omega(X)=-\boldsymbol{1}_{\{-1<X<1\}}\frac{X}{\sqrt{1-X^2}},\quad \mtx H(\bar \Omega)(X)=\begin{cases}
                  1-\frac{X}{\sqrt{X^2-1}}, & X>1,        \\
                  1,                        & X\in(-1,1), \\
                  1+\frac{X}{\sqrt{X^2-1}}, & X<-1,
              \end{cases}
          \]
          with an $a$-dependent scaling factor $\gamma = \gamma_a = -a$.
          Note that this singular profile $\bar \Omega$ is distinct from those numerically observed and explicitly constructed in this paper. However, no evidence has been found that an asymptotically self-similar blowup of the form \eqref{eqt:asymptotic_selfsimilar} with this singular profile $\bar\Omega$ can emerge from smooth initial data. 
    \item Besides, a family of $C^\alpha$ self-similar blowup profiles was found in \cite{elgindi2020effects} for $\alpha \in \{1/n\mid n\in \mathbb N\}$ and $|a|<\epsilon_0/\alpha$ with some small uniform constant $\epsilon_0$. For $a=0$, each of these profiles can be explicitly written as
          \[
              \bar \Omega^{(\alpha)}(X)=-\frac{\sin\left(\frac{\alpha\pi}{2}\right)\mathrm{sgn}(X)|X|^{\alpha}}{1+2\cos\left(\frac{\alpha\pi}{2}\right)|X|^\alpha+|X|^{2\alpha}},
              \quad
              \mtx H(\bar \Omega^{(\alpha)})(X)=\frac{1+\cos\left(\frac{\alpha\pi}{2}\right)|X|^{\alpha}}{1+2\cos\left(\frac{\alpha\pi}{2}\right)|X|^\alpha+|X|^{2\alpha}}
          \]
          with a scaling factor $\gamma = \gamma_\alpha = 1/\alpha$.
          For general values of $a$, $C^\alpha$ self-similar profiles were constructed by power series expansion with respect to $a$.
          Later this result was improved in \cite{zheng2023exactly} by relaxing the restriction $\alpha \in \{1/n\mid n\in \mathbb N\}$, but still requiring that $a\alpha\ll 1$.
    \item It is worth mentioning that modern techniques built on neural network architectures and training schemes have recently been developed to discover new self-similar blowups for the 3D Euler equations and related models \cite{wang2023asymptotic,wang2025discovery}. In particular, they numerically discovered a stable self-similar profile plus multiple unstable self-similar profiles for the C\'ordoba--C\'ordoba--Fontelos model that corresponds to the case $a=-1$. The stable one numerically coincides (up to rescaling) with that obtained in \cite{huang2024self} by a fixed-point method, though the uniqueness of which has not been established. 
\end{itemize}

\subsection{Novel self-similar finite-time blowups with degenerate initial data}\label{sec:novel_blowups}
Note that all stable self-similar profiles constructed so far are odd functions of $X$ and have a non-zero derivative at the origin. Correspondingly, the initial data must also have a non-zero derivative at the origin, since the sign of the derivative at the origin is preserved by the equation. Such initial data will be referred to as initial data in the sequel. 
It is then natural to ask whether self-similar finite-time blowups can develop from more degenerate initial data and how they differ from those arising from non-degenerate data. As stated in the introduction, the main purpose of this paper is to investigate this problem. To be specific, we investigate novel scenarios where the solution may develop a finite-time singularity at the origin while the initial condition $\omega_0$ satisfies $\omega_0'(0)=0$. In what follows, we call a solution $\om$ or a profile $\Omega$ \textit{degenerate} if its first derivative in space is zero at the origin. Specifically, we consider two kinds of degenerate initial data:
\begin{assumption}[Odd-symmetric degenerate case]\label{ass:odd_sym}
    The initial solution $\omega_0$ (or the initial profile $\Omega_0$) satisfies
    \begin{itemize}
        \item Odd symmetry: $\omega_0(-x)=-\omega_0(x)$;
        \item Degeneracy: $\omega_0'(0)=0$;
        \item Sign property: $\omega_0(x)\le 0$ for all $x>0$.
    \end{itemize}
\end{assumption}

\begin{assumption}[Half-line degenerate case]\label{ass:half_line}
    The initial solution $\omega_0$ (or the initial profile $\Omega_0$) satisfies
    \begin{itemize}
        \item Supported inside the positive half-line $\mathbb{R} _+$: $\omega_0(x)=0$ for all $x\le 0$;
        \item Degeneracy: $\omega_0'(0)=0$;
        \item Sign property: $\omega_0(x)\le 0$ for all $x>0$.
    \end{itemize}
\end{assumption}
In addition to the above assumptions, we also assume that the initial data $\omega_0$ satisfies suitable regularity and decay conditions (for instance, $\omega_0\in H^1(\mathbb{R})$), so that the corresponding Hilbert transform \eqref{eqt:hilbert_transform} and its antiderivative \eqref{eqt:biot_savart_law} can be well defined.

Finite-time singularity formation from degenerate initial data has been studied for the special case $a=0$ in a recent work \cite{huang2025multiscale}, where they proved that degenerate initial data can lead to self-similar finite-time blowups with a two-scale feature. In this paper we perform a systematic numerical study on finite-time blowups with degenerate initial data for a larger range of $a$. Our numerical simulation reveals novel self-similar blowup scenarios that are not only essentially different from those in the non-degenerate case but also have qualitatively distinct behaviors for $a>0$ and for $a\le 0$.

\subsubsection{Regular self-similar profiles for $a>0$} In the case $a>0$, our numerical simulation indicates that odd-symmetric degenerate initial data (Assumption \ref{ass:odd_sym}) can lead to ordinary self-similar finite-time blowups but with new regular profiles. These blowups exhibit behavior similar to that observed in the previous non-degenerate setting as in \cite{elgindi2020effects,chen2020singularity,lushnikov2021collapse,chen2021finite,huang2024self}, in the sense that the limiting self-similar profiles are regular and smooth in the interior of their support. Nevertheless, the novelty lies in the fact that these new profiles satisfy the same derivative degeneracy property as the initial data, a feature that has not been discovered and studied in the previous non-degenerate setting.

The behavior of the scaling factor $\gamma$ in this scenario also differs from that in the non-degenerate case. Firstly, for each fixed $a>0$, the value of $\gamma$ not only depends on whether the corresponding profile is non-degenerate or degenerate but also on the vanishing order of the (limit) profile at the origin. Here the vanishing order (which is preserved by the profile equation \eqref{eqt:dynamic_rescaling}) of a profile $\Omega(X)$ is defined as the smallest integer $k$ such that $\partial_X^k\Omega(0)\neq 0$. Secondly, the critical value $a_c$ as the transition point of the sign of $\gamma$ also depends on the vanishing order of the profile. As reported in \cite{lushnikov2021collapse}, in the non-degenerate case, a critical value $a_c\approx 0.689$ was numerically found to separate focusing self-similar blowups ($\gamma>0$) from expanding ones ($\gamma<0$). However, in the degenerate case, for example with vanishing order $k=3$, we have identified two critical values $a_{c,1}\approx0.233$ and $a_{c,2}\approx0.751$ such that $\gamma>0$ for (all numerically investigated) $a\in(0,a_{c,1})\cup(a_{c,2},+\infty)$ while $\gamma < 0$ for (all numerically investigated) $a\in(a_{c,1},a_{c,2})$. Moreover, the limiting self-similar profile starts to be compactly supported when $a$ increases past $a_{c,1}$ and then transitions from a stable one to an unstable one when $a$ increases through $a_{c,2}$. Detailed numerical results supporting these observations are presented in Section \ref{sec:a_gt_0}.

\subsubsection{Singular self-similar profiles for $a\leq 0$}
Things become more interesting when $a\le 0$. In contrast to the non-degenerate setting, our numerical simulations suggest that for $a\le 0$, degenerate initial data satisfying Assumption \ref{ass:odd_sym} or Assumption \ref{ass:half_line} lead to self-similar blowups with singular profiles. As we can see in Figure \ref{fig:a_neg_1_profiles}, profile evolution observed in the degenerate setting is qualitatively quite distinct from that in the non-degenerate setting. Specifically, the magnitude of the profile $\Omega(X,\tau)$ becomes unbounded at a certain point in $[0,+\infty)$ as $\tau\rightarrow +\infty$, while elsewhere the profile converges to some limiting function. In other words, $\Omega(X,\tau)$ converges to a singular function as $\tau\rightarrow +\infty$. The location of the singularity point in the profile can be predetermined by imposing a normalization condition that will be explained later. 

More precisely, we observe the following phenomenon for $a<0$: as the solution evolves, the profile develops a sharp front near $X=\pm 1$ in the odd-symmetry case under Assumption \ref{ass:odd_sym} (or near $X=1$ in the half-line case under Assumption \ref{ass:half_line}) under our normalization conditions. The part of the profile in $(-1,1)$ in the odd-symmetry case (or in $(0,1)$ in the half-line case) flattens in time and asymptotically vanishes, while the part beyond eventually converges to a spatially decaying tail extending to the far field with an algebraic decay rate. Notably, the profile behaviors in the odd-symmetry case and in the half-line case are qualitatively quite similar, only differing primarily in their limiting tail decay rates as $X\to+\infty$. Corresponding numerical results and analysis for $a<0$ are presented in Section \ref{sec:a_le_0}.

In addition to numerical evidence, we can in fact rigorously prove in the special case $a=0$ that the profile solution $\Omega(X,\tau)$ to the dynamic rescaling equation \eqref{eqt:dynamic_rescaling} converges to a singular function as $\tau\rightarrow +\infty$. The following theorem describes such convergence in the odd-symmetry case.

\begin{theorem}\label{thm:convergence_a_0}
    Let $\Omega(X,\tau)$ be a solution to the dynamic rescaling equation \eqref{eqt:dynamic_rescaling} for $a=0$ with smooth initial data $\Omega_0$ satisfying Assumption \ref{ass:odd_sym} (the odd-symmetry case). Moreover, assume $\Omega_0(X)<0$ for $X>0$ and
    \[\limsup_{X\to 0}\left|\frac{X\Omega_0'(X)}{\Omega_0(X)}\right|<+\infty.\]
    Without loss of generality, consider the normalization conditions
    \[\mtx{H}(\Omega)(0,\tau)\equiv \mtx{H}(\Omega_0)(0)=2,\quad \mtx{H}(\Omega)_{XX}(0,\tau)\equiv \mtx{H}(\Omega_0)_{XX}(0)=4,\]
    which yield constant scaling factors 
    \[c_l(\tau) \equiv  \bar c_l = \frac{1}{2},\quad c_\omega(\tau) \equiv \bar c_\omega=-1.\]
    Then, as $\tau\to +\infty$, the solution $\Omega(X,\tau)$ converges to a singular limiting profile $\bar \Omega$ given by 
    \[\bar\Omega(X)=-\pi( \delta(X-1)-\delta(X+1) ),\quad \mtx{H}(\bar\Omega)(X)=\frac{2}{1-X^2},\]
    in the sense that
    \[\lim_{\tau\to+\infty} \Omega(X,\tau) = \bar \Omega(X),\quad \lim_{\tau\to+\infty} \mtx{H}(\Omega)(X,\tau) = \mtx{H}(\bar \Omega)(X),\quad \text{for $X\neq \pm1$},\]
    and for any test function $\phi\in C_c^\infty(\mathbb{R})$,
    \[\lim_{\tau\to+\infty} \int_{-\infty}^{+\infty}\Omega(X,\tau)\phi(X)\idiff X=-\pi(\phi(1)-\phi(-1)).\]
\end{theorem}

The proof of Theorem \ref{thm:convergence_a_0} is deferred to Section \ref{sec:a_eq_0}. As a remark, the degeneracy condition in Theorem \ref{thm:convergence_a_0}, i.e. $\limsup_{X\to0}|{X\Omega_0'(X)}/{\Omega_0(X)}|<+\infty$, is satisfied by any initial data with a finite vanishing order at the origin. Also note that, due to the sign property and the degeneracy property of the initial data $\Omega_0$, the normalization conditions in Theorem \ref{thm:convergence_a_0} are possible since 
\[\mtx{H}(\Omega_0)(0) = -\frac{1}{\pi}\int_{\R}\frac{\Omega_0(X)}{X}\idiff X,\quad \mtx{H}(\Omega_0)_{XX}(0) = -\frac{2}{\pi}\int_{\R}\frac{\Omega_0(X)}{X^3}\idiff X.\]

A similar result can be proved if the odd symmetry is replaced by the half-line condition, which is in fact a consequence of the above theorem. 

\begin{corollary}\label{cor:convergence_a_0_halfline}
    Let $\Omega(X,\tau)$ be a solution to the dynamic rescaling equation \eqref{eqt:dynamic_rescaling} for $a=0$ with initial data $\Omega_0$ satisfying Assumption \ref{ass:half_line} (the half-line case). Moreover, assume $\Omega_0(X)<0$ for $X>0$ and
    \[\limsup_{X\to 0^+}\left|\frac{X\Omega_0'(X)}{\Omega_0(X)}\right|<+\infty.\]
    Without loss of generality, consider the normalization conditions
    \[\mtx{H}(\Omega)(0,\tau)\equiv \mtx{H}(\Omega_0)(0)=2,\quad \mtx{H}(\Omega)_X(0,\tau)\equiv \mtx{H}(\Omega_0)_X(0)=2,\]
    which yield constant scaling factors
    \[c_l(\tau) \equiv  \bar c_l = 1,\quad c_\omega(\tau) \equiv \bar c_\omega=-1.\]
    Then, as $\tau\to +\infty$, the solution $\Omega(X,\tau)$ converges to a singular limiting profile $\bar \Omega$ given by 
    \[\bar\Omega(X)=-2\pi \delta(X-1),\quad \mtx{H}(\bar\Omega)(X)=\frac{2}{1-X},\]
    in the sense that
    \[\lim_{\tau\to+\infty} \Omega(X,\tau) = \bar \Omega(X),\quad \lim_{\tau\to+\infty} \mtx{H}(\Omega)(X,\tau) = \mtx{H}(\bar \Omega)(X),\quad \text{for $X\neq 1$},\]
    and for any test function $\phi\in C_c^\infty(\mathbb{R})$,
    \[\lim_{\tau\to+\infty} \int_{-\infty}^{+\infty}\Omega(X,\tau)\phi(X)\idiff X = -2\pi\phi(1).\]
\end{corollary}

The proof of Corollary \ref{cor:convergence_a_0_halfline} is also deferred to Section \ref{sec:a_eq_0}.

As for $a<0$, we are unfortunately not yet able to prove the convergence of the profile $\Omega(X,\tau)$. However, we have found a family of explicit singular self-similar profiles in the half-line case (Assumption \ref{ass:half_line}) for $a<0$.

\begin{theorem}\label{thm:explicit_singular_solution}
    For any $a<0$, the self-similar profile equation \eqref{eqt:steady_profile} admits an exact solution $(\bar \Omega_a,\bar c_{l,a},\bar c_{\omega,a})$ given by
    \[
        \bar \Omega_a(X)=-\sin(\pi \mu)\frac{\boldsymbol{1}_{\{X>1\}}}{|1-X|^\mu},\quad \bar c_{l,a}=1-a,\quad \bar c_{\omega,a}=-1,
    \]
    where $\mu=1/(1-a)$.

\end{theorem}

The proof follows from direct verification and is provided in Appendix \ref{app:explicit_singular_solution}. In Section \ref{sec:a_le_0}, we will provide strong evidence that, for each $a<0$, solutions to \eqref{eqt:dynamic_rescaling} in the half-line degenerate case (Assumption \ref{ass:half_line}) converge to the exact steady state given in Theorem \ref{thm:explicit_singular_solution} under appropriate normalization conditions. Besides, it is also an interesting problem whether one can construct exact singular self-similar profiles in the odd-symmetry case for all $a<0$, which we leave for future studies.

\subsection{Two-scale feature of self-similar blowups with singular profiles}\label{sec:2scale_analysis}
The extra singularity formation in the profile $\Omega(X,\tau)$ in self-similar coordinates indicates that the solution may exhibit a two-scale behavior. In order to understand this two-scale feature, we introduce a secondary, smaller scale to resolve the inner profile. Recall that the outer scale, namely the ordinary self-similar scaling, is captured by $(T-t)^{\gamma}$. We also assume the trajectory of the maximum point of $|\omega|$ is given by $x(t) = r(t) (T-t)^\gamma$ for some function $r(t)$ such that the limit $\lim_{t\to T^-}r(t)$ exists. If we zoom into the neighborhood of $x(t)$ on a smaller spatial scale $(T-t)^{\hat\gamma}$ for some $\hat\gamma > \gamma$, we may observe a finer behavior of the solution in the form
\begin{equation}\label{eqt:multiscale_blowup}
    \omega(x,t)=(T-t)^{\hat \lambda}\left(\widehat \Omega\left( \frac{x-r(t)(T-t)^{\gamma}}{(T-t)^{\hat\gamma}}\right) +o(1)\right),
\end{equation}
where $\widehat \Omega$ denotes the inner profile, as opposed to the outer profile $\Omega(X,\tau)$. We would also expect $\hat\lambda<-1$, since the inner profile is obtained by zooming in around the singularity of the outer profile. The validity of this two-scale ansatz depends on whether a secondary scale $(T-t)^{\hat\gamma}$ actually exists and captures the characteristic width of the shrinking peak of the outer profile. Fortunately, our numerical simulation has identified such a secondary scale for some $\hat\gamma > \gamma$ and $\hat\lambda<-1$ when the outer profile forms a singularity for $a<0$. Moreover, we find that the inner profile $\widehat\Omega$ remains smooth for all time and converges as $t\to T^-$ (or as $\tau \to +\infty$).

As mentioned in the introduction, two-scale self-similar finite-time blowups of the form \eqref{eqt:multiscale_blowup} with degenerate initial data have been rigorously established in \cite{huang2025multiscale} for the special case $a=0$, where the inner profile $\widehat\Omega$ is proved to be a traveling wave profile of the physical equation \eqref{eqt:gCLM} with $a=0$. Below we argue that for a general $a<0$, if the two-scale ansatz \eqref{eqt:multiscale_blowup} is in fact valid, then the inner profile $\widehat\Omega$ must also be a traveling wave profile of the original physical equation \eqref{eqt:gCLM}.

Since we have assumed that the limit $\lim_{t\to T^-}r(t)$ exists, we can simply replace $r(t)$ by a constant $r$ (i.e. the limit) by putting the resulting error in the residual term in \eqref{eqt:multiscale_blowup}. We also omit the $o(1)$ term in \eqref{eqt:multiscale_blowup}. The advection speed $u$ given by the Biot--Savart law \eqref{eqt:biot_savart_law} can be decomposed into two parts as 
    \begin{equation}\label{eqt:decomposition_u}
        u(x,t)=(T-t)^{\hat\lambda+\hat\gamma}\left(\widehat U(\widehat X)-\widehat U(c(t))\right)+(T-t)^{\gamma-1} U(X),
    \end{equation}
    where
    \[
        \widehat X=\frac{x-r(T-t)^{\gamma}}{(T-t)^{\hat\gamma}},\quad X=\frac{x}{(T-t)^{\gamma}},\quad\widehat U_{\widehat X}=\mtx{H}(\widehat\Omega),\quad U_X=\mtx{H}(\Omega),
    \]
    and $\widehat U$ and $U$ are defined in terms of $\widehat\Omega$ and $\Omega$ respectively via \eqref{eqt:biot_savart_law}, i.e. $\widehat U(0)=0$ and $U(0)=0$. Hence the condition $u(0,t)\equiv 0$ yields
    \[c(t)=-r(T-t)^{\gamma-\hat\gamma}.\]
    Substituting the two-scale ansatz \eqref{eqt:multiscale_blowup} and \eqref{eqt:decomposition_u} into the gCLM equation \eqref{eqt:gCLM} gives
        \begin{align*}
             & -\hat\lambda(T-t)^{\hat \lambda-1}\widehat \Omega+(T-t)^{\hat \lambda-1}\left( \hat\gamma \widehat X+r\gamma(T-t)^{\gamma-\hat\gamma} \right)\widehat \Omega_{\widehat X}+a(T-t)^{2\hat\lambda}\left(\widehat U-\widehat U(c(t))\right)\widehat \Omega_{\widehat X} \\&+a(T-t)^{\hat\lambda+\gamma-\hat\gamma-1}U(X) \widehat \Omega_{\widehat X} =(T-t)^{2\hat\lambda}\widehat U_{\widehat X}\widehat \Omega+(T-t)^{\hat\lambda-1}U_X(X)\widehat \Omega,
        \end{align*}
    which can be rearranged as
    \begin{equation}\label{eqt:two_scale_equation}
        \begin{split}
             & (T-t)^{-1}\left( -\big(\hat\lambda+U_X(X)\big)\widehat \Omega+ \hat\gamma \widehat X\widehat \Omega_{\widehat X} \right) \\&=(T-t)^{\hat\lambda}\left( \widehat U_{\widehat X}\widehat \Omega -a\left(\widehat U-\widehat U(c(t))+(T-t)^{\gamma-\hat\gamma-\hat\lambda-1}\left(\frac{r\gamma}{a}+U(X)\right)\right)\widehat \Omega_{\widehat X} \right).
        \end{split}
    \end{equation}
    By the two-scale assumption \eqref{eqt:multiscale_blowup}, we should have $\hat\lambda < -1$. Therefore the right-hand side of \eqref{eqt:two_scale_equation} should balance by itself, i.e. the leading-order equation of \eqref{eqt:two_scale_equation} is
    \begin{equation}\label{eqt:two_scale_equation_leading}
        a\left(\widehat U-\widehat U(c(t))+(T-t)^{\gamma-\hat\gamma-\hat\lambda-1}\left(\frac{r\gamma}{a}+U(r+(T-t)^{\hat\gamma-\gamma}\widehat X)\right)\right)\widehat \Omega_{\widehat X}=\widehat U_{\widehat X}\widehat \Omega ,
    \end{equation}
    where we have used $X=r+(T-t)^{\hat\gamma-\gamma}\widehat X$. Furthermore, our numerical fitting in Section \ref{sec:numerical_multiscale} suggests that $\gamma-\hat \gamma-\hat\lambda-1<0$. Therefore, for \eqref{eqt:two_scale_equation_leading} to balance, a necessary condition is that the term $r\gamma/a+U(r+(T-t)^{\hat\gamma-\gamma}\widehat X)$ must converge to zero as $t\to T^-$. Note that $\hat\gamma>\gamma$ due to the two-scale assumption, which implies that $U(r+(T-t)^{\hat\gamma-\gamma}\widehat X)$ converges to $U(r)$ as $t\to T^-$. Hence we must have $U(r)=-r\gamma/a$. Then, we may expect that the limit
    \[
        \lim_{t\to T^-} a\left( -\widehat U(c(t))+(T-t)^{\gamma-\hat\gamma-\hat\lambda-1}\left(\frac{r\gamma}{a}+U(r+(T-t)^{\hat\gamma-\gamma}\widehat X)\right) \right)
    \]
    exists, and we denote it by $\hat r$. If this is true, taking the limit $t\to T^-$ in \eqref{eqt:two_scale_equation_leading} then yields the steady-state equation for $\widehat\Omega$ as
    \begin{equation}\label{eqt:leading_order}
        (a\widehat U+\hat r)\widehat \Omega_{\widehat X}=\widehat U_{\widehat X} \widehat \Omega,
    \end{equation}
    which in fact describes the profile of a traveling wave solution to the original gCLM model \eqref{eqt:gCLM}.

To support the two-scale ansatz \eqref{eqt:multiscale_blowup} and our asymptotic analysis above, we perform the following two additional studies. Firstly, we rigorously prove in Section \ref{sec:existence_traveling_wave} the existence of a traveling wave solution to \eqref{eqt:gCLM} (whose profile is a solution to \eqref{eqt:leading_order}) via a fixed-point method for all $a\le1$. This fixed-point method also leads to a numerical fixed-point iteration that generates a traveling wave profile up to high accuracy (see Appendix \ref{apx:iterative_method}). Secondly, for $a<0$, we perform in Section \ref{sec:numerical_multiscale} a comparison between the inner profile $\widehat \Omega$ obtained by numerically solving the dynamic rescaling equation \eqref{eqt:dynamic_rescaling} and the traveling wave profile obtained by the fixed-point iteration, showing that they match well with each other. Therefore our theoretical and numerical results together strongly support the validity of the two-scale ansatz \eqref{eqt:multiscale_blowup} and our asymptotic analysis in this subsection. 

\subsection{Traveling wave solutions}\label{sec:sec2_traveling_wave}
As argued above, the inner profile $\widehat\Omega$ in a two-scale self-similar blowup \eqref{eqt:multiscale_blowup}, if it exists, must be a traveling wave profile of the gCLM model \eqref{eqt:gCLM}. This in turn leads to the question of whether such traveling wave solutions exist. To be specific, we look for traveling wave solutions to \eqref{eqt:gCLM} of the form
\begin{equation}\label{eqt:traveling_wave_ansatz}
    \omega(x,t)=\bar\omega(x-rt),\quad u(x,t) = \bar u(x-rt),\quad u(rt,t)=0,
\end{equation}
where $\bar\omega$ is the traveling wave profile (and also the initial solution), the absolute value $|r|$ represents the traveling speed, and the sign of $r$ indicates the traveling direction. Note that here we set $u(rt,t)=0$ only for the convenience of studying the existence of such a traveling wave profile $\bar\omega$. If instead we choose to set $u(0,t)=0$ as in \eqref{eqt:biot_savart_law}, then by Lemma \ref{lem:translation_invariance} there exists some function $c(t)$ such that with the same initial data $\omega(x,0) = \bar \omega(x)$ the solution will evolve like $\omega(x,t)= \bar \omega(x-rt-c(t))$ for all $t\geq 0$, which is still a traveling wave solution with a time-dependent translating speed.

Substituting \eqref{eqt:traveling_wave_ansatz} into the gCLM model \eqref{eqt:gCLM} yields
\[
    (a\bar u-r) \bar\omega_x=\bar u_x\bar\omega,
\]
where $\bar u_x=\mtx{H}(\bar\omega)$ and $\bar u(0) = 0$. One may immediately notice that this equation is identical to \eqref{eqt:leading_order} up to change of notations and possible sign flipping. Using a similar fixed-point method as in \cite{huang2024self}, we can actually prove that traveling wave solutions exist for any $a\in(-\infty,1]$, with the profile $\bar\omega$ being non-zero on $\mathbb{R}$ for $a\leq 0$ but compactly supported for $a\in(0,1]$.

\begin{theorem}\label{thm:traveling_wave_existence}
    For any $a\in(-\infty,1]$, the generalized Constantin--Lax--Majda model admits a traveling wave solution of the form \eqref{eqt:traveling_wave_ansatz} (up to rescaling), where $\bar\omega$ is a non-negative even function such that $\bar\omega(x)$ is in $C(\mathbb{R})$ and non-increasing on $[0,+\infty)$, $\bar\omega(\sqrt s)$ is convex in $s$ for $s>0$, $\bar\omega(0)>0$, $\bar\omega''(0)<0$, and
    \[
        r = r_a=-\frac{2\bar\omega(0)}{\pi\bar\omega_{xx}(0)}\int_0^{+\infty}\frac{\bar\omega(0)-\bar\omega(x)}{x^2}\idiff{x}>0.
    \]
    Moreover, the following holds depending on the sign of $a$:
    \begin{itemize}
        \item $a\in(0,1]$: There exists some $L_a>0$ such that $\bar\omega$ is compactly supported on $[-L_a,L_a]$, strictly positive on $(-L_a,L_a)$, and smooth in the interior of $(-L_a,L_a)$, and $\bar \omega\in C(\mathbb{R})\cap C^1([-L_a,L_a])$. There exists some finite number $C_a>0$ such that 
        \[
        \lim_{x\to L_a^-}\frac{\bar\omega(x)}{(L_a-x)^{1/a}}=C_a.
        \]
        In particular, for $a\in(0,1)$, $\bar \omega\in C^1(\mathbb{R})$.
        \item $a=0$: $\bar\omega(x)=1/(1+x^2)$.
        \item $a<0$: $\bar\omega(x)$ is strictly positive and smooth on $\mathbb{R}$, and $\bar\omega'\in H^k(\mathbb{R})$ for any $k\ge0$. There exist some finite number $\bar C_a,\tilde C_a>0$ such that, for all $|x|\ge 1 $,
        \[\bar C_a |x|^{-1/(1+|a|)}\le \bar\omega(x)\le \tilde C_a |x|^{-1/(1+|a|)}.\]
    \end{itemize}
\end{theorem}

A detailed proof of Theorem \ref{thm:traveling_wave_existence} via a fixed-point approach, which consists of a series of intermediate lemmas, will be elaborated in Section \ref{sec:existence_traveling_wave}. As aforementioned, we also make use of this fixed-point framework to numerically compute the traveling wave profiles. The numerical fixed-point iteration method is detailed in Appendix \ref{apx:iterative_method}, where we include a figure of numerically computed traveling wave profiles for different values of $a$ (Figure \ref{fig:travelingwaves}).

\section{Convergence to singular profiles of the dynamic rescaling equation for $a=0$}\label{sec:a_eq_0}

When $a=0$, the gCLM model \eqref{eqt:gCLM} reduces to the original Constantin--Lax--Majda (CLM) model
\begin{equation}\label{eqt:CLM}
    \omega_t=u_x\omega,\quad u_x=\mtx H(\omega).
\end{equation}
Correspondingly, the dynamic rescaling equation \eqref{eqt:dynamic_rescaling} simplifies to
\begin{equation}\label{eqt:dynamic_rescaling_a_0}
    \Omega_\tau+c_l X\Omega_X=c_\omega\Omega+U_X\Omega.
\end{equation}
It has been rigorously proved in \cite{huang2025multiscale} that odd degenerate initial data can lead to multi-scale self-similar finite-time blowups for the CLM model \eqref{eqt:CLM}. However, the convergence of the outer profile has not been established yet. In this section, we supplement their result by providing proofs for Theorem \ref{thm:convergence_a_0} and Corollary \ref{cor:convergence_a_0_halfline}, establishing that the outer profile converges to a singular function under appropriate conditions.

\begin{proof}[Proof of Theorem \ref{thm:convergence_a_0}]
    For notational simplicity, we will use $x,t$ for $X,\tau$ and write $F=\Omega,G=\mtx H(\Omega)$ in this proof. We also denote $F_0=F(\cdot,0),G_0=G(\cdot,0)$. By the odd symmetry and degeneracy of the initial data $\Omega_0$, we have
    \[
        F_0(0)=0,\quad F_0'(0)=0,\quad F_0''(0)=0,\quad \lim_{x\to0}\frac{|F_0(x)|}{|x|^3} <+\infty.  
    \]
    Note that the odd symmetry and degeneracy of the initial datum are preserved by the dynamic rescaling equation \eqref{eqt:dynamic_rescaling_a_0}, i.e.
    \[
        F(x,t)=-F(-x,t),\quad F_x(0,t)=F_{xx}(0,t)=0,\quad \text{for all $t\geq 0$}.
    \]
    By applying the Hilbert transform to the whole equation \eqref{eqt:dynamic_rescaling_a_0}, we obtain a system of $(F,G)$ as
    \begin{align*}
        F_t+c_l xF_x & =c_\omega F+F G,                    \\
        G_t+c_l xG_x & =c_\omega G+\frac{1}{2}( G^2-F^2 ).
    \end{align*}
    We have used the fact that $\mtx H(F\mtx H F)=(\mtx H(F)^2-F^2)/2$. Moreover, it is straightforward to check that the prescribed normalization conditions
    \[
        G(0,t)\equiv G(0,0)=2,\quad G_{xx}(0,t)\equiv G_{xx}(0,0)=4,
    \]
    determine the scaling factors $c_l,c_\omega$ as
    \[
        c_l(t) \equiv \frac{1}{2},\quad c_\omega(t)\equiv -1.
    \]
    We thus obtain the dynamic system
    \begin{align*}
            &F_t+\frac{1}{2}xF_x =- F+F G,                    \\
            &G_t+\frac{1}{2}xG_x =- G+\frac{1}{2}( G^2-F^2 ), \\
            &G(0,t)=2,\quad G_{xx}(0,t)=4.
    \end{align*}
    Note that the relation $G=\mtx H(F)$ is preserved by this system. If we define $h=F+\mathrm{i}G$, then this complex-valued function solves
    \[
        h_t+\frac{1}{2}xh_x=-h-\frac{\iunit}{2} h^2,\quad h(0,t)=2\iunit,\quad h_{xx}(0,t)=4\iunit.
    \]
    Now consider the characteristic trajectories
    \[
        X(x,t)=x \mathrm{e}^{t/2},\quad x\in \mathbb{R}.
    \]
    Denote $h_0=h(\cdot,0)$. Then, along a characteristic $X(x,t)$, $h(X(x,t),t)$ solves
    \[
        \frac{D h}{D t}=-h-\frac{\mathrm{i}}{2} h^2,
    \]
    which yields
    \[
        h( X(x,t),t )=\frac{h(x,0)}{\mathrm{e}^t+\frac{\iunit}{2} (\mathrm{e}^t-1)h(x,0)}=\frac{h_0(x)}{\mathrm{e}^t+\frac{\mathrm{i}}{2}(\mathrm{e}^t-1)h_0(x)}.
    \]
    By reversing the characteristics, we obtain
    \[
        h(x,t)=\frac{h_0(x\mathrm{e}^{-t/2})}{\mathrm{e}^t+\frac{\mathrm{i}}{2}( \mathrm{e}^t-1 )h_0(x\mathrm{e}^{-t/2})},
    \]
    that is,
        \begin{align*}
            F(x,t) & =\frac{4\mathrm{e}^t F_0(x\mathrm{e}^{-t/2})}{(\mathrm{e}^t-1)^2 F_0(x\mathrm{e}^{-t/2})^2+\left( 2\mathrm{e}^t-(\mathrm{e}^t-1) G_0(x\mathrm{e}^{-t/2})\right)^2},                                                                                     \\
            G(x,t) & =\frac{4\mathrm{e}^t G_0(x\mathrm{e}^{-t/2})-2(\mathrm{e}^t-1)\left( F_0(x\mathrm{e}^{-t/2})^2+ G_0(x\mathrm{e}^{-t/2})^2\right)}{(\mathrm{e}^t-1)^2 F_0(x\mathrm{e}^{-t/2})^2+\left( 2\mathrm{e}^t-(\mathrm{e}^t-1) G_0(x\mathrm{e}^{-t/2})\right)^2}. \\
        \end{align*}
    Note that we have assumed $F_0(x)<0$ for all $x>0$, and thus $F(x,t)^2+G(x,t)^2<+\infty$ for any finite $t\ge 0$. Besides, recall that we have
    \[
        \lim_{x\to0}\frac{|F_0(x)|}{|x|^3}<+\infty,\quad G_0(0)=2,\quad \lim_{x\to 0}\frac{G_0(x)-G_0(0)}{x^2} = \frac{1}{2}(G_0)_{xx}(0)=2.
    \]
    Under these assumptions, we may further assume that 
    \[|G_0(x) - (2 + 2x^2)|\leq Ax^4,\quad |x|\leq 1,\]
    for some finite constant $A>0$. Then, it can be easily verified that, for any compact set $S\subset \R\setminus\{\pm1\}$, and for any sufficiently large $t$ satisfying
    \[x^2\mathrm{e}^{-t}\leq \min\left\{1\,,\,\frac{|x^2-1|}{2(1+Ax^2)}\right\},\quad \forall x\in S,\]
    there exists some constant $C>0$ (only depending on $F_0,G_0$) such that 
    \begin{equation}\label{eqt:F_G_convergence_rate}
    |F(x,t)| \leq C\mathrm{e}^{-t/2}\cdot \frac{|x|^3}{|1-x^2|^2},\quad \left|G(x,t)-\frac{2}{1-x^2}\right| \leq C\mathrm{e}^{-t}\cdot\frac{(1+x^2)^2x^2}{|1-x^2|^3},\quad x\in S.
    \end{equation}
    In particular, for each $x\ne \pm1$, we have
    \[
        \lim_{t\to +\infty}F(x,t)=0=\bar \Omega(x),\quad \lim_{t\to+\infty}G(x,t)=\frac{2}{1-x^2}=\mtx{H}(\bar \Omega)(x).
    \]
    In fact, one can see from the calculation above that the convergence of $F(x,t)$ roots in the fact that $\mathrm{e}^t F_0(x\mathrm{e}^{-t/2})$ converges to $0$ as $t\to+\infty$, which is true only when the vanishing order of $F_0(x)$ at $x=0$ is strictly higher than $2$. This reveals the essential difference between non-degenerate data and degenerate data in their blowup behaviors. 

    Next, we show that $F(x,t)$ converges to $\bar \Omega(x) = -\pi(\delta(x-1)-\delta(x+1))$ in the weak sense. By odd symmetry we only need to consider $x\geq 0$. For any compactly supported test function $\phi\in C_c^\infty(\mathbb{R}_+)$, we consider the integral decomposition 
    \[ \int_0^{+\infty} F(x,t)\phi(x)\idiff x = \int_{|x-1|\leq\epsilon} F(x,t)\phi(x)\idiff x + \int_{|x-1|>\epsilon,\,x\geq 0} F(x,t)\phi(x)\idiff x=: I_\epsilon(t) + J_\epsilon(t),\]
    for some sufficiently small but fixed $\epsilon\in(0,1)$. We can use \eqref{eqt:F_G_convergence_rate} to obtain
    \[ |J_\epsilon(t)| \leq C\frac{\mathrm{e}^{-t/2}}{\epsilon^2}.\] 
    The constant $C$ may vary from line to line but is always independent of $t$ and $\epsilon$. Let us write 
    \[G_0(x) = 2 + 2x^2q(x),\]
    where $q$ is a smooth even function such that $q(0)=1$, $|q(x)-1|\leq C x^2$ and $|q'(x)|\leq C|x|$ for $|x|\leq 2$. We then rewrite $I_\epsilon(t)$ as
    \begin{align*}
     I_\epsilon(t) &= \int_{1-\epsilon}^{1+\epsilon}\frac{4\mathrm{e}^t F_0(x\mathrm{e}^{-t/2})\phi(x)}{(\mathrm{e}^t-1)^2 F_0(x\mathrm{e}^{-t/2})^2+\left( 2\mathrm{e}^t-(\mathrm{e}^t-1) G_0(x\mathrm{e}^{-t/2})\right)^2}\idiff x \\
     &= \int_{1-\epsilon}^{1+\epsilon}\frac{4\mathrm{e}^t F_0(x\mathrm{e}^{-t/2})\phi(x)}{(\mathrm{e}^t-1)^2 F_0(x\mathrm{e}^{-t/2})^2+4\left( 1-x^2(1-\mathrm{e}^{-t})q(x\mathrm{e}^{-t/2})\right)^2}\idiff x\\
     &= \frac{1}{(1-\mathrm{e}^{-t})^2}\cdot \int_{1-\epsilon}^{1+\epsilon}\frac{-4\phi(x)}{1 +16\left( \frac{1-x^2(1-\mathrm{e}^{-t})q(x\mathrm{e}^{-t/2})}{2(1-\mathrm{e}^{-t})\eta(x,t)}\right)^2}\cdot \frac{1}{\eta(x,t)}\idiff x,
    \end{align*}
    where
    \[\eta(x,t):= -\mathrm{e}^t F_0(x\mathrm{e}^{-t/2}).\]
    Note that $\eta(x,t)>0$ for all $x>0$ and $t\geq 0$. Moreover, by the degeneracy assumption that $\lim_{x\to 0}|F_0(x)|/|x|^3<+\infty$, we have 
    \[\eta(x,t)< C\mathrm{e}^{-t/2},\quad x\in[1-\epsilon,1+\epsilon].\]
    Next, define
    \[Y(x,t) := \frac{x^2q(x\mathrm{e}^{-t/2})-(1-\mathrm{e}^{-t})^{-1}}{2\eta(x,t)},\quad x\in[1-\epsilon, 1+\epsilon].\]
    Note that by the properties of $q$, we have 
    \begin{equation}\label{eqt:Y_boundary}
    Y(1+\epsilon,t) \geq \left(c_0\epsilon  - C\mathrm{e}^{-t}\right)\mathrm{e}^{t/2},\quad Y(1-\epsilon,t)\leq -\left(c_0\epsilon - C\mathrm{e}^{-t}\right)\mathrm{e}^{t/2}, 
    \end{equation}
    for some constant $c_0>0$. In particular, when $t$ is sufficiently large, $Y(1+\epsilon,t)>0$ and $Y(1-\epsilon,t)<0$. We can also compute that
    \begin{align*}
    Y_x(x,t) &= \frac{xq(x\mathrm{e}^{-t/2})}{\eta(x,t)} + \frac{x^2\mathrm{e}^{-t/2}q'(x\mathrm{e}^{-t/2})}{2\eta(x,t)} - \frac{x^2q(x\mathrm{e}^{-t/2})-(1-\mathrm{e}^{-t})^{-1}}{2\eta(x,t)}\cdot \frac{\eta_x(x,t)}{\eta(x,t)}\\
    &= \frac{1}{\eta(x,t)}\left(1 + (x-1)q(x\mathrm{e}^{-t/2}) + (q(x\mathrm{e}^{-t/2}) -1) + \frac{x^2\mathrm{e}^{-t/2}q'(x\mathrm{e}^{-t/2})}{2} \right.\\
    &\qquad\qquad\qquad - \left.\frac{x^2q(x\mathrm{e}^{-t/2})-(1-\mathrm{e}^{-t})^{-1}}{2x}\cdot \frac{x\eta_x(x,t)}{\eta(x,t)}\right).
    \end{align*}
    Note that by the assumption 
    \[\limsup_{x\to 0}\left|\frac{xF_0'(x)}{F_0(x)}\right|<+\infty,\]
    we have 
    \[\left|\frac{x\eta_x(x,t)}{\eta(x,t)}\right| = \left| \frac{x\mathrm{e}^{-t/2}F_0'(x\mathrm{e}^{-t/2})}{F_0(x\mathrm{e}^{-t/2})}\right| \leq C,\quad x\in[1-\epsilon,1+\epsilon].\]
    It follows that, for all $x\in [1-\epsilon, 1+\epsilon]$ and for sufficiently large $t$,
    \begin{equation}\label{eqt:Y_x}
        \left|\eta(x,t)Y_x(x,t) - 1\right| \leq C (\epsilon+\mathrm{e}^{-t}).
    \end{equation}
    In particular, there exists some $t_0>0$ and some $\epsilon_0\in(0,1)$ such that, for $t>t_0$ and for $x\in [1-\epsilon, 1+\epsilon]\subset [1-\epsilon_0,1+\epsilon_0]$,
    \[Y_x(x,t) > \frac{1}{2\eta(x,t)} >0.\]
    This means for any $t\geq t_0$, $Y(x,t)$ as a function of $x\in [1-\epsilon, 1+\epsilon]$ admits an inverse function $x = X(y,t)$ for $y\in [Y(1-\epsilon,t), Y(1+\epsilon,t)]$. Using $x = X(y,t)$ as a change of variable, we may continue the calculation of $I_\epsilon(t)$ as
    \begin{align*}
    I_\epsilon(t) &= \frac{1}{(1-\mathrm{e}^{-t})^2}\cdot \int_{1-\epsilon}^{1+\epsilon}\frac{-4\phi(x)}{1 +16\left( \frac{1-x^2(1-\mathrm{e}^{-t})q(x\mathrm{e}^{-t/2})}{2(1-\mathrm{e}^{-t})\eta(x,t)}\right)^2}\cdot \frac{1}{\eta(x,t)}\idiff x\\
    &= \frac{1}{(1-\mathrm{e}^{-t})^2}\cdot \int_{Y(1-\epsilon,t)}^{Y(1+\epsilon,t)}\frac{-4\phi(X(y,t))}{1 +16y^2}\cdot \frac{1}{\eta(X(y,t),t)Y_x(X(y,t),t)}\idiff y.
    \end{align*}
    Note that 
    \[\pi = \int_{\R}\frac{4}{1+16y^2}\idiff y = \int_{Y(1-\epsilon,t)}^{Y(1+\epsilon,t)}\frac{4}{1+16y^2}\idiff y + \int_{\R\backslash[Y(1-\epsilon,t),Y(1+\epsilon,t)]}\frac{4}{1+16y^2}\idiff y.\]
    Then, using \eqref{eqt:Y_boundary}, \eqref{eqt:Y_x}, and the smoothness of $\phi$, we can derive that
    \begin{align*}
    |I_\epsilon(t) - (-\pi\phi(1))| & \leq \left|\int_{Y(1-\epsilon,t)}^{Y(1+\epsilon,t)}\left( \frac{(1-\mathrm{e}^{-t})^{-2}}{\eta(X(y,t),t)Y_x(X(y,t),t)}\phi(X(y,t)) -\phi(1)\right)\frac{4}{1 +16y^2}\idiff y\right|\\
    &\quad\, + \left|\int_{\R\backslash[Y(1-\epsilon,t),Y(1+\epsilon,t)]}\frac{4\phi(1)}{1+16y^2}\idiff y\right|\\
    &\leq C\left(\epsilon + \mathrm{e}^{-t} + \frac{\mathrm{e}^{-t/2}}{c_0\epsilon - C\mathrm{e}^{-t}}\right).
    \end{align*}
    Finally, we obtain
    \begin{align*}
        \left|\int_0^{+\infty} F(x,t)\phi(x)\idiff x - (-\pi\phi(1))\right| &\leq |I_\epsilon(t) - (-\pi\phi(1))| + |J_\epsilon(t)| \\
        &\leq C\left(\epsilon + \mathrm{e}^{-t} + \frac{\mathrm{e}^{-t/2}}{c_0\epsilon - C\mathrm{e}^{-t}} + \frac{\mathrm{e}^{-t/2}}{\epsilon^2}\right).
    \end{align*}
    In particular, for $t$ sufficiently large such that $\mathrm{e}^{-t}\ll \epsilon^6$, we have 
    \[\left|\int_0^{+\infty} F(x,t)\phi(x)\idiff x - (-\pi\phi(1))\right|\leq C\epsilon.\]
    Since $\epsilon \in(0,1)$ can be arbitrarily small, it follows that 
    \[\lim_{t\to +\infty} \int_0^{+\infty} F(x,t)\phi(x)\idiff x = -\pi\phi(1).\]
    This completes the proof.
\end{proof}

\begin{proof}[Proof of Corollary \ref{cor:convergence_a_0_halfline}]
    This corollary follows directly from Theorem \ref{thm:convergence_a_0} via a change of variable. Let $\Omega(X,\tau)$ be a solution to \eqref{eqt:dynamic_rescaling} for $a=0$ that satisfies all assumptions and normalization conditions of Corollary \ref{cor:convergence_a_0_halfline}. Specifically, $\Omega$ solves the dynamic rescaling equation
    \[\Omega_\tau+c_lX\Omega_X=(c_\omega+\mtx{H}(\Omega))\Omega\]
    with $c_l = 1$ and $c_\om=-1$. Define 
    \[\tilde \Omega(X,\tau) := \sgn(X)\cdot \Omega(X^2,\tau).\]
    It is not hard to see that, if the initial condition $\Omega_0(x) = \Omega(x,0)$ satisfies Assumption \ref{ass:half_line} (half-line case), then the new initial condition $\tilde \Omega_0(x) = \tilde \Omega(x,0)$ satisfies Assumption \ref{ass:odd_sym} (odd-symmetry case). It also holds that 
    \[\limsup_{X\to 0}\left|\frac{X\tilde \Omega_0'(X)}{\tilde \Omega_0(X)}\right| = 2\limsup_{X\to 0}\left|\frac{X^2\Omega_0'(X^2)}{\Omega_0(X^2)}\right| = 2\limsup_{X\to 0^+}\left|\frac{X\Omega_0'(X)}{\Omega_0(X)}\right|<+\infty.\]
    Moreover, one can easily use the formula of the Hilbert transform to verify that
    \[\mtx{H}(\tilde \Omega)(X,\tau) = \mtx{H}(\Omega)(X^2,\tau). \]
    As an immediate consequence, $\tilde \Omega(X,\tau)$ solves the corresponding dynamic rescaling equation
    \[\tilde \Omega_\tau+\tilde c_lX\tilde \Omega_X=(\tilde c_\omega+\mtx{H}(\tilde\Omega))\tilde \Omega\]
    with $\tilde c_l = c_l/2 = 1/2$ and $\tilde c_\om = c_\om =-1$, and it satisfies the normalization conditions
    \begin{align*}
        &\mtx{H}(\tilde\Omega)(0,\tau)\equiv \mtx{H}(\tilde\Omega_0)(0) = \mtx{H}(\Omega_0)(0) =2,\\
        &\mtx{H}(\tilde\Omega)_{XX}(0,\tau)\equiv \mtx{H}(\tilde \Omega_0)_{XX}(0)=2\mtx{H}(\Omega_0)_{X}(0) =4.
    \end{align*}
    That is, the solution $\tilde \Omega$ satisfies all assumptions and normalization conditions of Theorem \ref{thm:convergence_a_0}. Therefore, the convergence result of Corollary \ref{cor:convergence_a_0_halfline} readily follows from the convergence result of Theorem \ref{thm:convergence_a_0}.
    \end{proof}

\section{Numerical results with degenerate initial data for $a>0$}\label{sec:a_gt_0}

In this section, we present numerical results obtained from dynamic rescaling simulations of the gCLM model \eqref{eqt:gCLM} with degenerate initial data in the case $a>0$. This scenario of blowup has been briefly discussed in Section \ref{sec:novel_blowups}. Recall that the dynamic rescaling equation is given by
\begin{equation}\label{eqt:dynamic_rescaling_a_gt_0}
    \Omega_\tau+(c_l X+ a U)\Omega_X=(c_\omega+U_X)\Omega,\quad U_X=\mtx H(\Omega),\quad U(0)=0.
\end{equation}
As discussed in Section \ref{sec:blowup_scenarios}, \eqref{eqt:dynamic_rescaling_a_gt_0} is equivalent to the original gCLM equation \eqref{eqt:gCLM} via time-dependent scaling transformations. Furthermore, the convergence of the profile solution would imply that the original gCLM equation develops a finite-time singularity in a self-similar manner. Therefore, we simulate the dynamic rescaling equation \eqref{eqt:dynamic_rescaling_a_gt_0} to investigate potential self-similar blowup of the gCLM model with degenerate initial data.

It is straightforward to verify that the dynamic rescaling equation \eqref{eqt:dynamic_rescaling_a_gt_0} preserves the vanishing order of the solution at the origin. Numerically, however, directly evolving \eqref{eqt:dynamic_rescaling_a_gt_0} with degenerate initial data is delicate, since round-off and discretization errors tend to destroy the high-order degeneracy near $X=0$. To deal with this issue, we factor out the prescribed vanishing order by introducing
\[
    f(X,\tau)=\frac{\Omega(X,\tau)}{X^k},
\]
where $k$ is the vanishing order of the initial data $\Omega_0$ at the origin, and we instead evolve $f$ via
\begin{equation}\label{eqt:dynamic_rescaling_F_a_gt_0}
    f_\tau+(c_l X+aU)f_X=\left(c_\omega+U_X-k\left(c_l+a\frac{U}{X}\right)\right)f,
\end{equation}
where $U_X=\mtx H(X^k f)$. To uniquely determine the scaling factors $c_l$ and $c_\omega$, we must impose two normalization conditions. Since the model preserves the degeneracy of the initial data, a natural normalization condition is to fix the value of $f$ at the origin, i.e.
\[
    f(0,\tau)=\lim_{X\to 0}\frac{\Omega(X,\tau)}{X^k}\equiv -1.
\]
For the second normalization condition, we fix the location of the minimum of $\Omega$ at $X=1$ by imposing
\[
    \Omega_X(1,\tau)\equiv0.
\]

In this section, we focus on odd-symmetric smooth initial data (Assumption \ref{ass:odd_sym}). In fact, under odd symmetry, a finite vanishing order $k$ of $\Omega$ must be an odd natural number and is preserved for all time by the equation. Throughout this section, we take the initial condition of $f$ as
\[
    f_0(X)=-\frac{3}{3+kX^{k+3}},
\]
so that $\Omega_0(X)=X^k f_0(X)=-3X^k/(3+kX^{k+3})$. This choice satisfies the normalization conditions at $\tau=0$, namely $f_0(0)=-1$ and $\Omega_0'(1)=0$. We defer the implementation details of the dynamic rescaling simulations to Appendix \ref{app:numerics}.

In the remainder of this section, we first demonstrate the convergence of the dynamic rescaling simulations using the case $a=0.5$ and $k=3$ as an illustrative example. We then present numerical results for various values of $a>0$ with a fixed vanishing order $k=3$, followed by results for different vanishing orders $k\ge 3$ with a fixed $a=0.5$. Finally, we provide a comprehensive table summarizing the scaling factors $c_l, c_\omega$ and the spatial shrinking rate $\gamma=-c_l/c_\omega$ for different combinations of $a>0$ and $k\ge 3$.

\subsection{Convergence of profiles}
Below we present illustrative numerical results demonstrating convergence of the dynamic rescaling method for the case $a=0.5$ with vanishing order $k=3$. Figure \ref{fig:a0.5_k3_convergence} shows the evolution of the rescaled profile $\Omega(X,\tau)$ and $f(X,\tau)$. The condition $f(0)=-1$ is preserved throughout the computation, confirming that the prescribed degeneracy at the origin is maintained numerically. We declare convergence when the residual of \eqref{eqt:dynamic_rescaling_a_gt_0} falls below $1.5\times 10^{-8}$, and denote the resulting steady state by $\bar\Omega$. As illustrated in Figure \ref{fig:a0.5_k3_convergence}, the limiting profile $\bar\Omega$ is regular and smooth in the interior of its support and vanishes at the origin with the same order as the initial data. In particular, this profile differs from previously studied self-similar blowup profiles, which typically exhibit linear vanishing at $X=0$.

\begin{figure}[htbp]
    \begin{subfigure}{0.49\textwidth}
        \includegraphics[width=\textwidth]{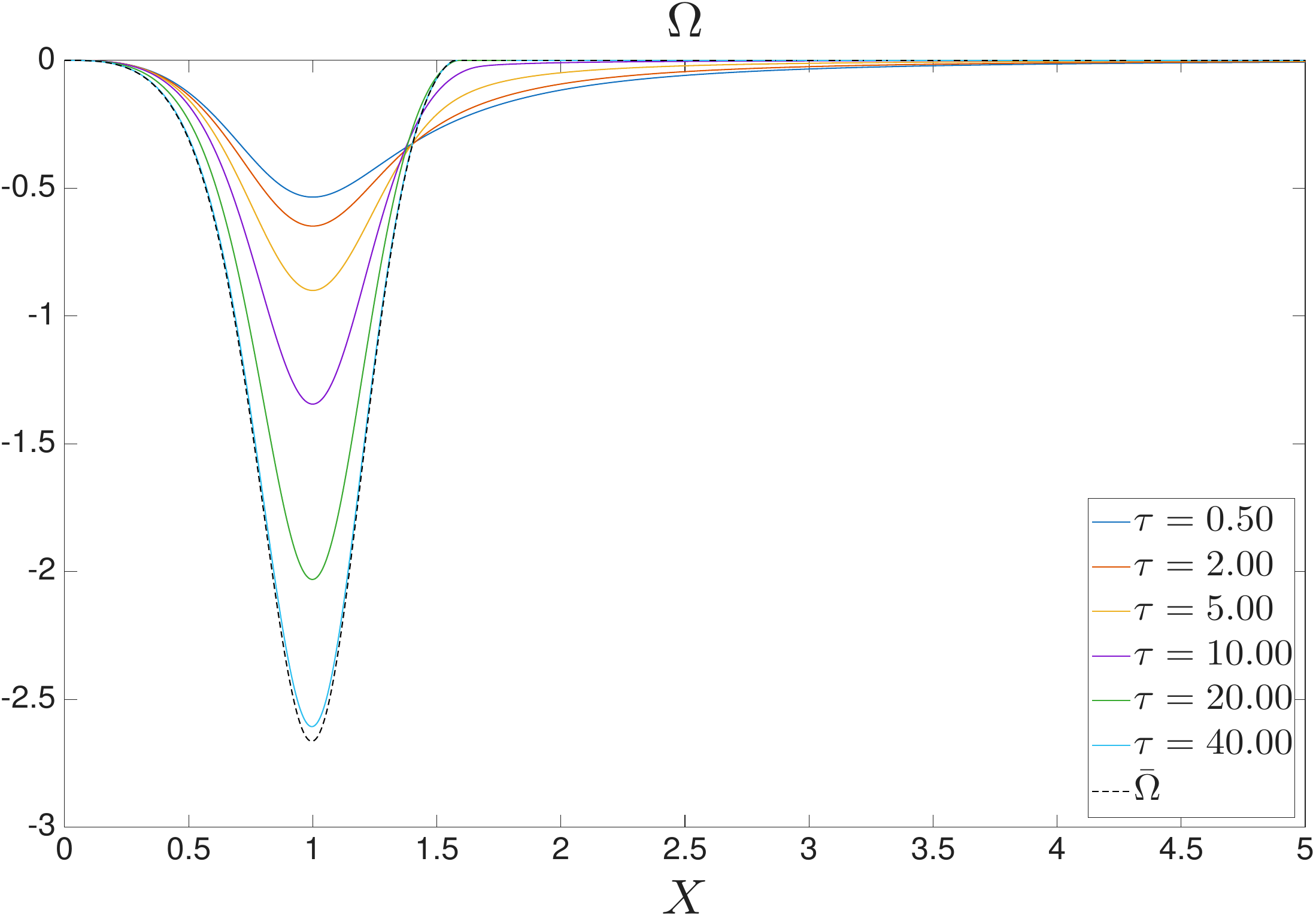}
        \caption{Evolution of $\Omega$}
        \label{fig:a0.5_k3_evolution}
    \end{subfigure}
    \begin{subfigure}{0.49\textwidth}
        \includegraphics[width=\textwidth]{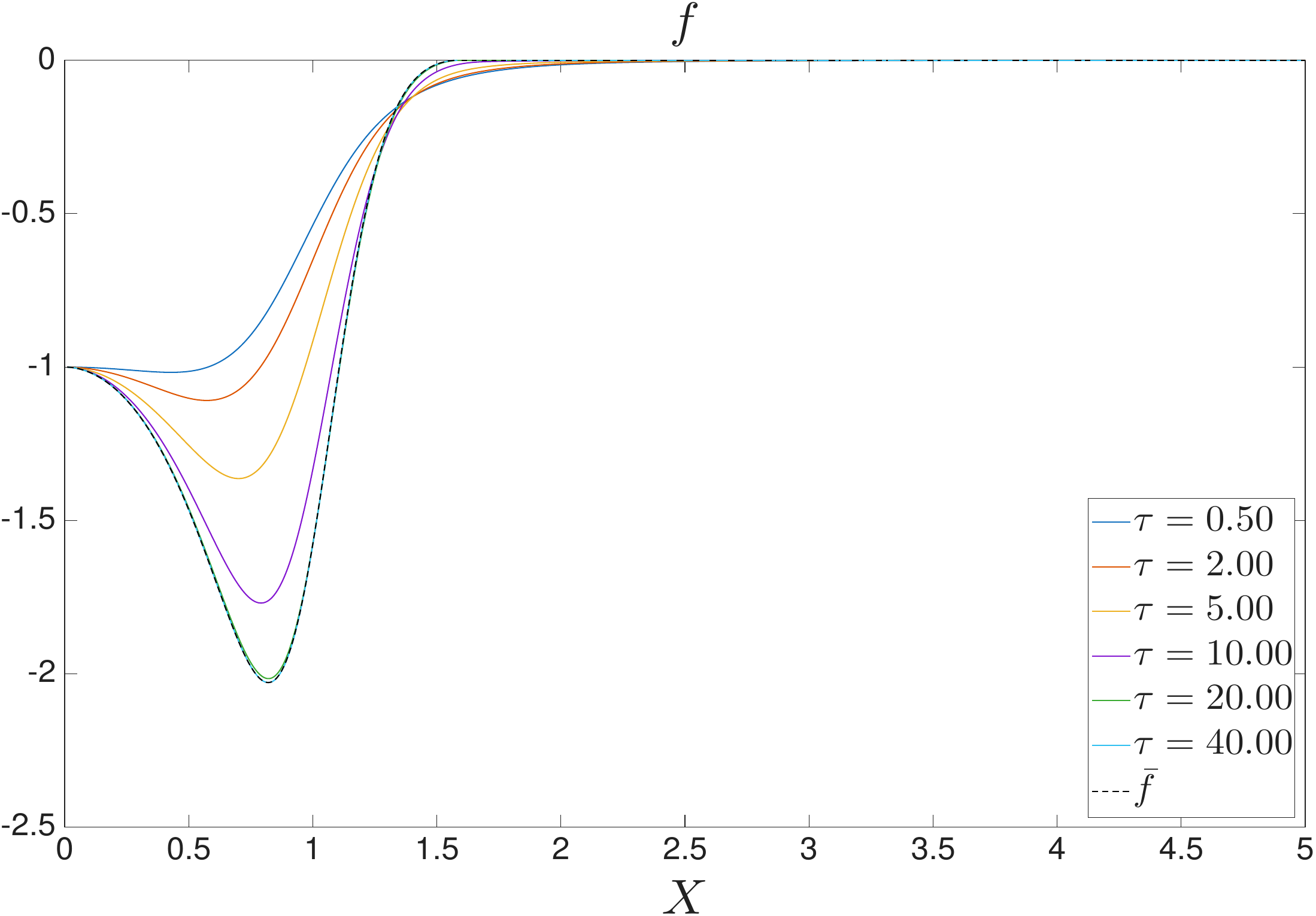}
        \caption{Evolution of $f$}
        \label{fig:a0.5_k3_evolution_f}
    \end{subfigure}
    \caption{Convergence of the profile for $a=0.5$ and $k=3$. (a) Evolution of the rescaled profile $\Omega(X,\tau)$ at different times. The dashed line plots the final steady state $\bar \Omega$, of which the residual has dropped below $1.5\times 10^{-8}$. (b) Evolution of $f(X,\tau)$, related to $\Omega(X,\tau)$ by $f(X,\tau)=\Omega(X,\tau)/X^k$. The preservation of $f(0)=-1$ throughout the evolution demonstrates that the degeneracy of the solution at the origin is maintained numerically.}
    \label{fig:a0.5_k3_convergence}
\end{figure}

The convergence of the scaling factors $c_l$, $c_\omega$, and the spatial shrinking rate $\gamma=-c_l/c_\omega$ is shown in Figure \ref{fig:a0.5_k3_convergence_scaling_factors}. These quantities stabilize as $\Omega(X,\tau)$ approaches the steady state, supporting the formation of a one-scale self-similar blowup in this setting.

\begin{figure}[htbp]
    \begin{subfigure}{0.32\textwidth}
        \includegraphics[width=\textwidth]{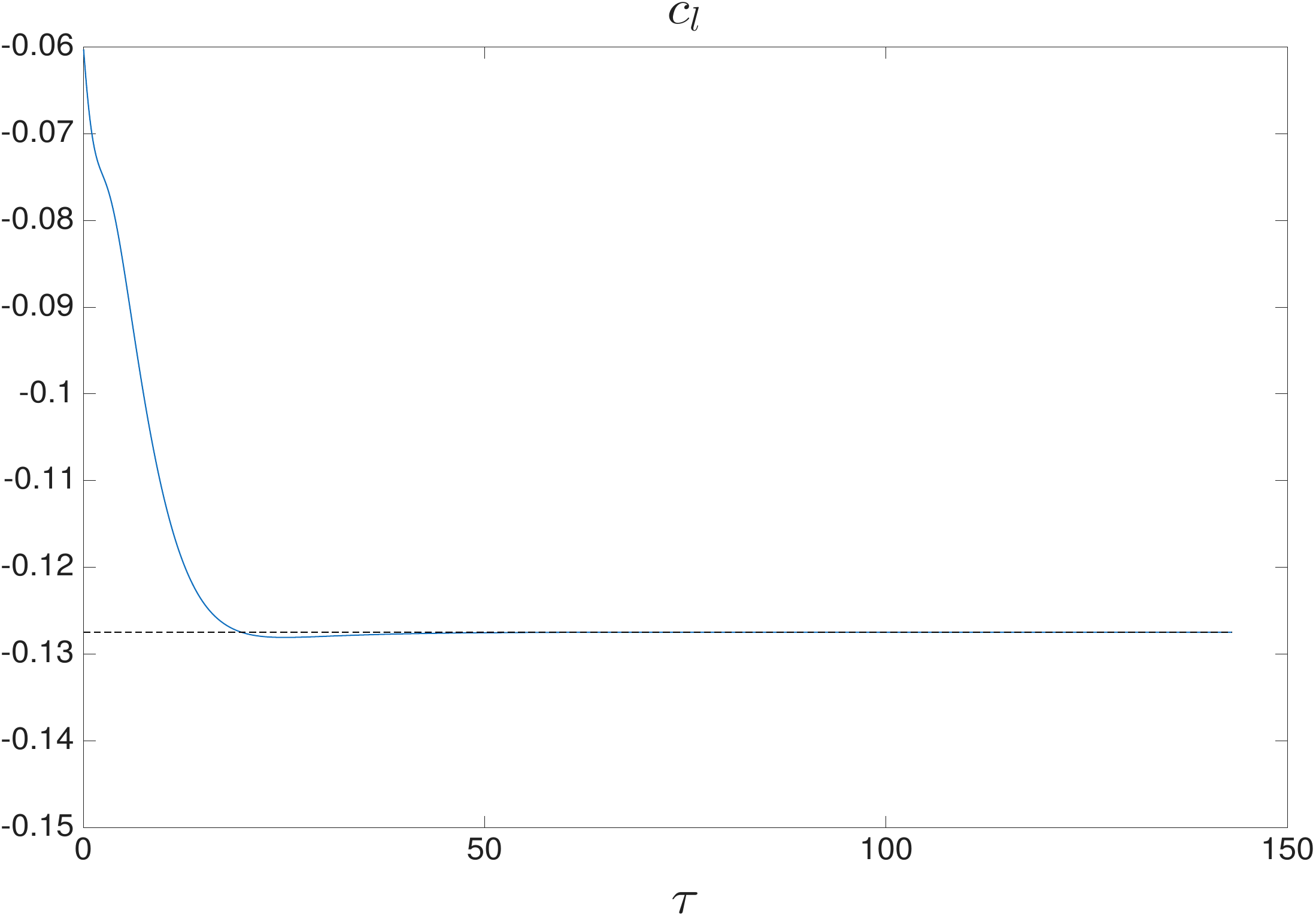}
        \caption{Scaling factor $c_l$}
    \end{subfigure}
    \begin{subfigure}{0.32\textwidth}
        \includegraphics[width=\textwidth]{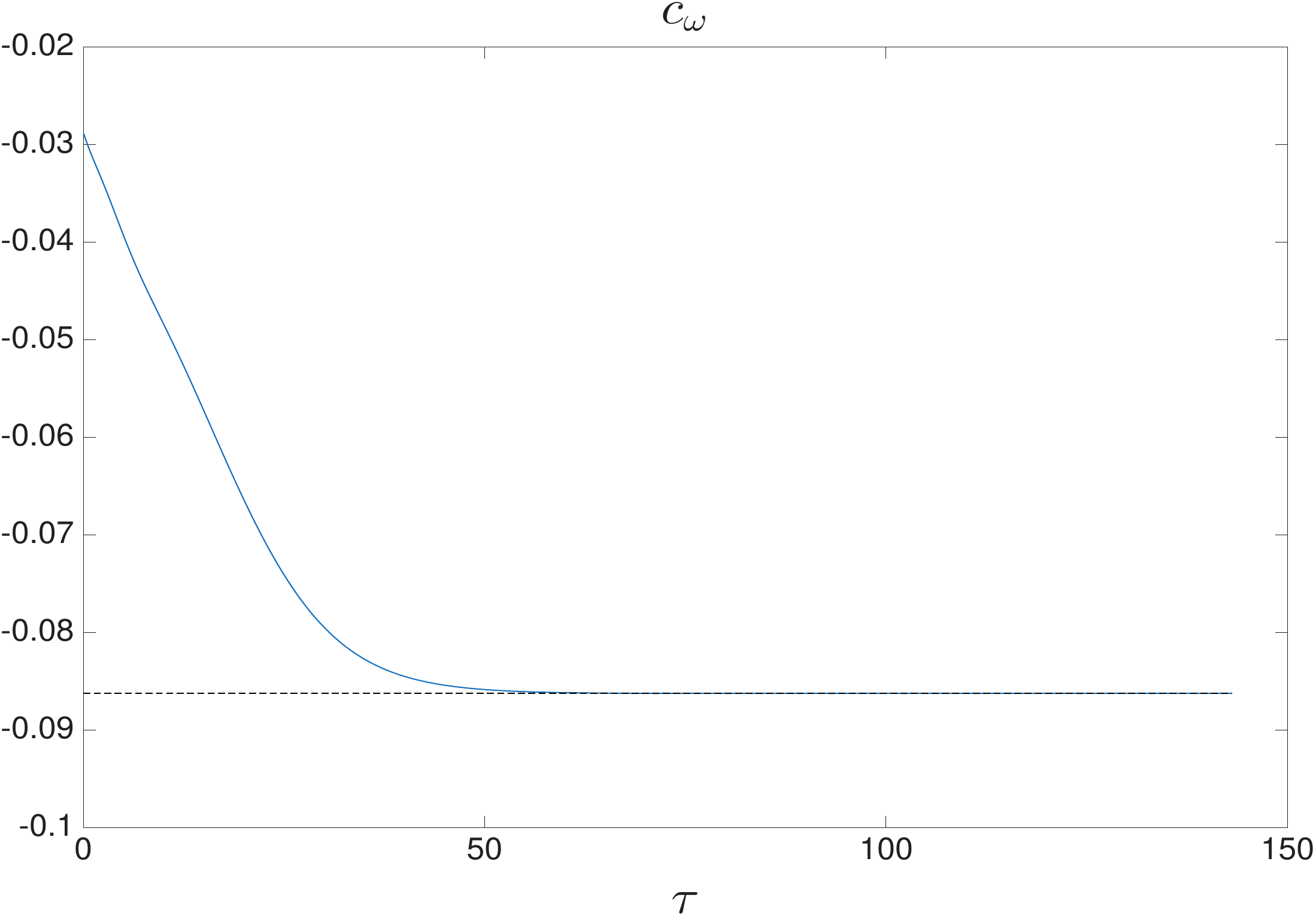}
        \caption{Scaling factor $c_\omega$}
    \end{subfigure}
    \begin{subfigure}{0.32\textwidth}
        \includegraphics[width=\textwidth]{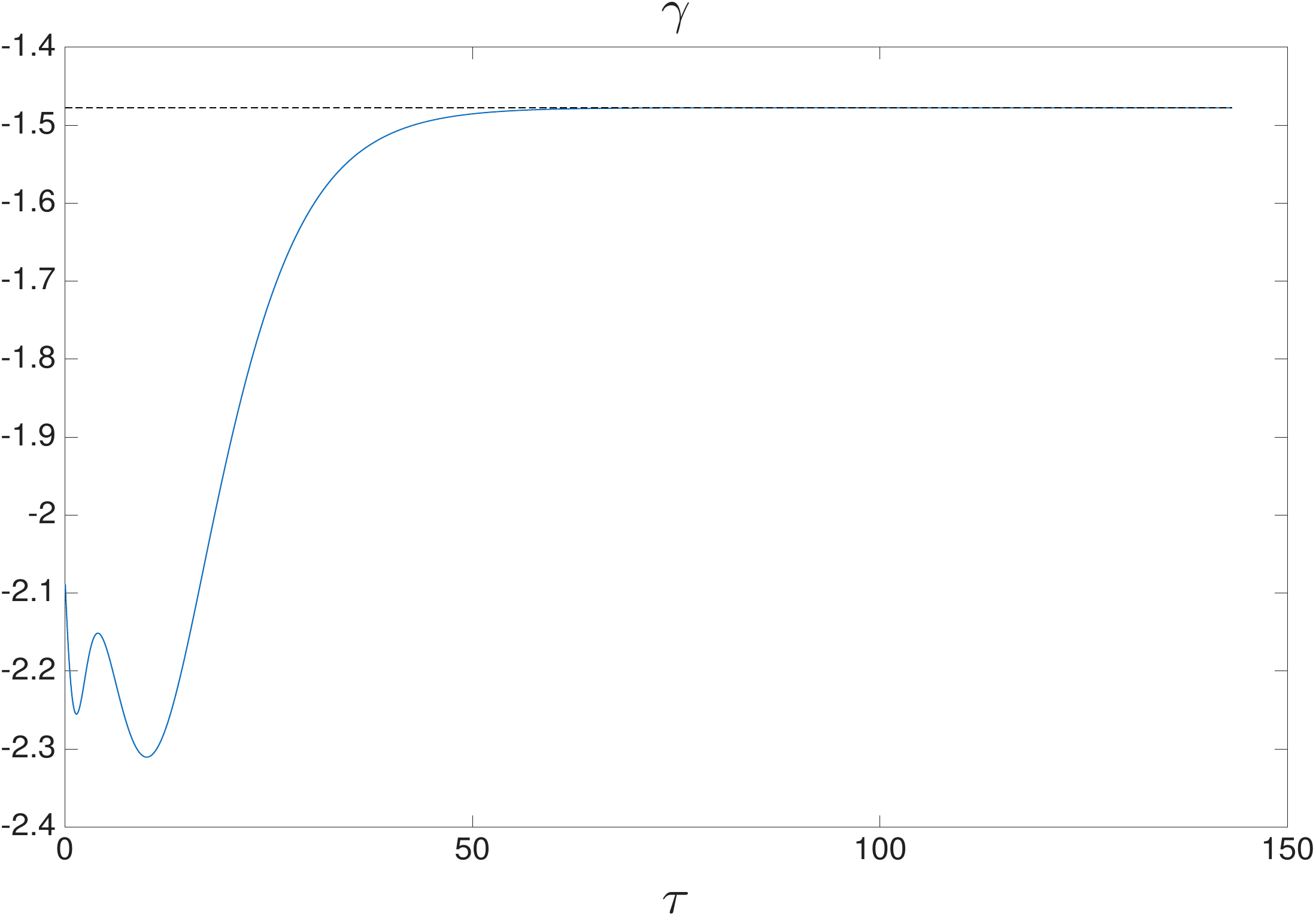}
        \caption{Spatial shrinking rate $\gamma=-c_l/c_\omega$}
    \end{subfigure}
    \caption{Convergence of the scaling factors for $a=0.5$ and $k=3$. The stabilization of $c_l$, $c_\omega$, and $\gamma$ as $\tau$ increases supports convergence towards a one-scale self-similar blowup.}
    \label{fig:a0.5_k3_convergence_scaling_factors}
\end{figure}

\subsection{Numerical results for different $a>0$ with a fixed vanishing order $k=3$}
Here we present the numerical results of the dynamic rescaling simulations for different $a>0$ with a fixed vanishing order $k=3$ for the initial data. The scaling factors of the steady state for $0.1\le a\le1.5$ are listed in Table \ref{tab:one_scale_scaling_factors_fixed_a}, and the corresponding self-similar profiles are plotted in Figure \ref{fig:a_positive_fix_k}. Note that due to the scaling invariance property \eqref{eqt:scaling_invariance}, the precise values of $\bar c_l$ and $\bar c_\omega$ depend on the normalization conditions and are thus not intrinsic. The physically relevant quantity is the ratio $\bar\gamma=-\bar c_l/\bar c_\omega$, which determines the spatial shrinking rate of the solution.

\begin{table}[htbp]
    \begin{tabular}{|c|c|c|c|c|c|c|c|c|c|c|c|}
        \hline
        {$a$}           & $ 0.1     $ & $ 0.2     $  & $ 0.232931     $             & $ 0.232932     $            & $ 0.3     $  & $ 0.4     $ & $ 0.5     $ \\ \hline
        $\bar c_l$      & $ 1.0888  $ & $ 0.0691  $  & $ 6.4688\times 10^{-7}    $  & $ -9.7142\times10^{-7}    $ & $ -0.0717 $  & $ -0.1135 $ & $ -0.1288 $ \\ \hline
        $\bar c_\omega$ & $ -3.8668 $ & $ -0.8357 $  & $ -0.6055    $               & $ -0.6055    $              & $ -0.3479 $  & $ -0.1718 $ & $ -0.0872 $ \\ \hline
        $\bar\gamma$    & $ 0.2816  $ & $ 0.0827  $  & $ 1.0683\times 10^{-6}     $ & $-1.6043\times 10^{-6}   $  & $ -0.2061 $  & $ -0.6604 $ & $ -1.4771 $ \\ \hline\hline
        {$a$}           & $ 0.6     $ & $ 0.7      $ & $0.751405$                   & $0.751406$                  & $ 0.8$       & $ 0.9     $ & $ 1       $ \\ \hline
        $\bar c_l$      & $ -0.1345 $ & $ -0.1364  $ & $-0.1364$                    & $-0.1364$                   & $ -0.1365 $  & $ -0.1360 $ & $ -0.1352 $ \\ \hline
        $\bar c_\omega$ & $ -0.0399 $ & $ -0.0108  $ & $-5.8090\times10^{-8}$       & $1.3080\times 10^{-7}$      & $ 0.0084 $   & $ 0.0217  $ & $ 0.0314  $ \\ \hline
        $\bar \gamma$   & $ -3.3734 $ & $ -12.6465 $ & $-2.3481\times10^6$          & $1.0429\times{10^6}$        & $ 16.2599$   & $ 6.2550  $ & $ 4.2996  $ \\ \hline\hline
        {$a$}           & $ 1.1     $ & $ 1.2     $  & $ 1.3     $                  & $ 1.4      $                & $ 1.5     $  &             &             \\ \hline
        $\bar c_l$      & $ -0.1345 $ & $ -0.1338 $  & $ -0.1334 $                  & $ -0.1332 $                 & $ -0.1331  $ &             &             \\ \hline
        $\bar c_\omega$ & $ 0.0388  $ & $ 0.0446  $  & $ 0.0492  $                  & $ 0.0531  $                 & $ 0.0565   $ &             &             \\ \hline
        $\bar \gamma$   & $ 3.4658  $ & $ 3.0034  $  & $ 2.7094  $                  & $ 2.5058  $                 & $ 2.3563   $ &             &             \\ \hline
    \end{tabular}
    \caption{Scaling factors of the steady state for $0.1\le a\le 1.5$ with fixed vanishing order $k=3$. We observe two critical values: $a_{c,1}\in(0.232931,0.232932)$, where $c_l$ changes sign from positive to negative, and $a_{c,2}\in(0.751405,0.751406)$, where $c_\omega$ changes sign from negative to positive.}
    \label{tab:one_scale_scaling_factors_fixed_a}
\end{table}

\begin{figure}[htbp]
    \includegraphics[width=0.49\textwidth]{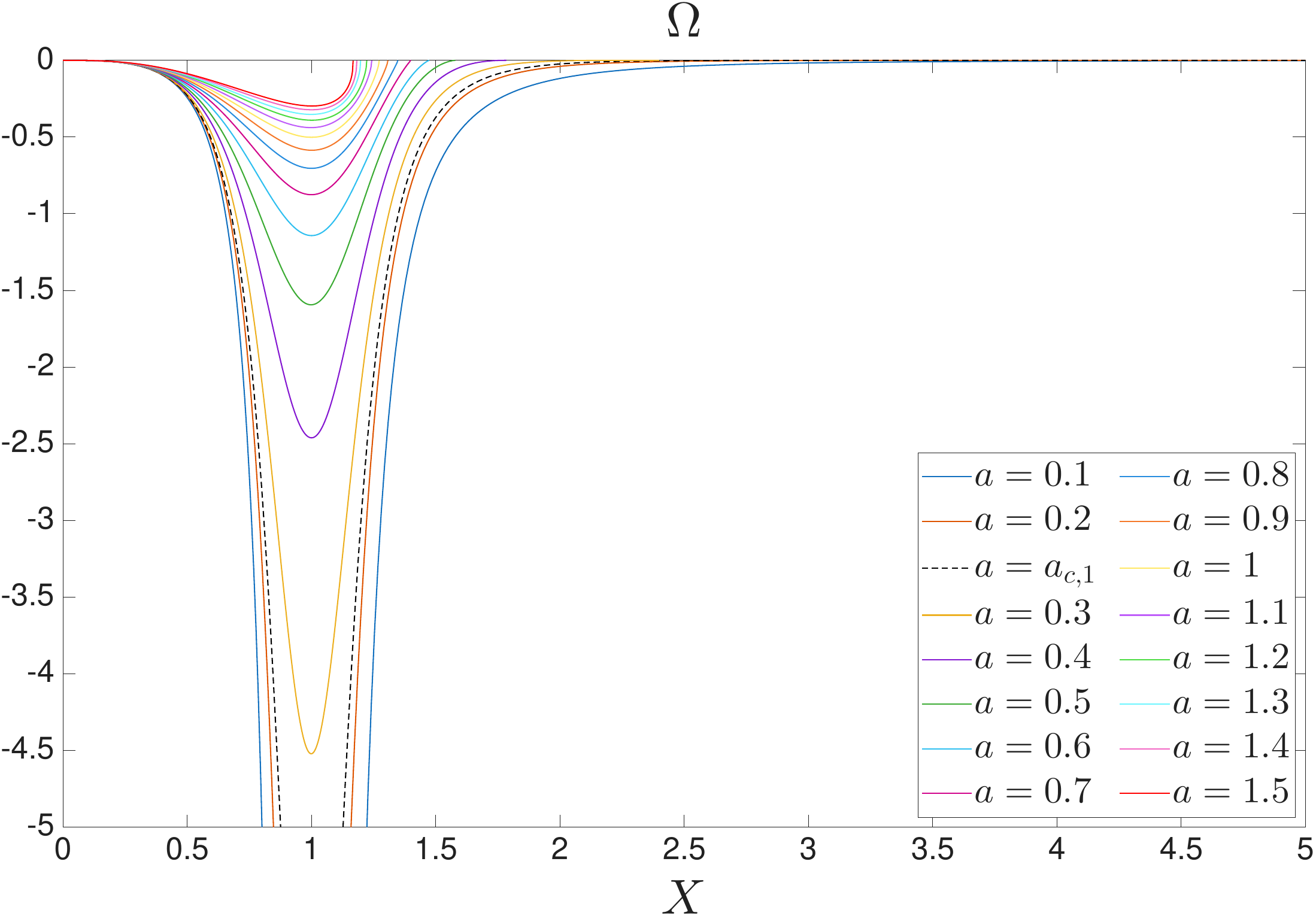}
    \caption{Steady states of the profile $\bar\Omega(X)$ for $0.1\le a\le 1.5$ with fixed vanishing order $k=3$. The dashed curve corresponds to the critical value $a=a_{c,1}$ at which the profile transitions from being smooth with full-line support to being compactly supported.}
    \label{fig:a_positive_fix_k}
\end{figure}

Lushnikov et al. \cite{lushnikov2021collapse} performed direct numerical simulations of the gCLM model \eqref{eqt:gCLM} and found evidence of self-similar finite-time blowup arising from smooth initial data for a wide range of the parameter $a$. They observed a critical value $a_c\approx 0.6891$ such that for $a<a_c$, the self-similar profile is supported on the whole real line and the blowup is focusing (i.e., $\gamma>0$), while for $a>a_c$, the self-similar profile is compactly supported and the blowup is expanding (i.e., $\gamma<0$). 

Our numerical results for degenerate initial data suggest a similar phenomenon, but with different critical values. As shown in Table \ref{tab:one_scale_scaling_factors_fixed_a}, there exists a critical value $a_{c,1}\in(0.232931,0.232932)$ across which $\bar c_l$ changes sign from positive to negative. This corresponds to a transition of the self-similar profile from being smooth and supported on the whole real line to being compactly supported. This critical value $a_{c,1}$ differs from the $a_c$ reported in \cite{lushnikov2021collapse}; we conjecture this discrepancy arises from the different vanishing order of the initial data at the origin.

Moreover, we identify a second critical value $a_{c,2}\in(0.751405,0.751406)$ at which $c_\omega$ changes sign from negative to positive. This indicates a transition from blowup to non-blowup in the original gCLM equation \eqref{eqt:gCLM}, since Proposition \ref{prop:convergence_implies_onescale_blowup} requires $\bar c_\omega<0$ for a one-scale asymptotically self-similar blowup. Nevertheless, even when $a>a_{c,2}$, one can still obtain an exact self-similar blowup solution of \eqref{eqt:gCLM} from the steady state $\bar\Omega$ computed via the dynamic rescaling equation \eqref{eqt:dynamic_rescaling_a_gt_0}. Define
\[
    \Omega^{\mathrm{flip}}(X,\tau):=-\Omega(X,-\tau),\quad c_l^{\mathrm{flip}}=-c_l(-\tau),\quad c_\omega^{\mathrm{flip}}=-c_\omega(-\tau),
\]
and correspondingly
\[
\bar\Omega^{\mathrm{flip}}(X):=-\bar\Omega(X),\quad \bar c_l^{\mathrm{flip}}=-\bar c_l,\quad \bar c_\omega^{\mathrm{flip}}=-\bar c_\omega.
\]
Then, by the scaling invariance \eqref{eqt:scaling_invariance}, the triple $(\bar\Omega^{\mathrm{flip}},\bar c_l^{\mathrm{flip}},\bar c_\omega^{\mathrm{flip}})$ also solves the steady profile equation \eqref{eqt:steady_profile}. Hence $\bar\Omega^{\mathrm{flip}}$ is an exact self-similar blowup profile for \eqref{eqt:gCLM}. On the other hand, $\Omega^{\mathrm{flip}}$ satisfies the dynamic rescaling equation with scaling factors $c_l^{\mathrm{flip}}$ and $c_\omega^{\mathrm{flip}}$:
\begin{equation}
    \Omega^{\mathrm{flip}}_\tau+\bigl(c_l^{\mathrm{flip}} X+aU^{\mathrm{flip}}\bigr)\Omega^{\mathrm{flip}}_X
    =\bigl(c_\omega^{\mathrm{flip}}+U^{\mathrm{flip}}_X\bigr)\Omega^{\mathrm{flip}},
\end{equation}
where $U^{\mathrm{flip}}(X,\tau)=-U(X,-\tau)$. The linearization of this flipped evolution about $\bar\Omega^{\mathrm{flip}}$ has the opposite sign compared to the linearization of \eqref{eqt:dynamic_rescaling} about $\bar\Omega$. Therefore, the observed convergence of the dynamic rescaling equation toward $\bar\Omega$ (i.e. stability of $\bar\Omega$) implies that $\bar\Omega^{\mathrm{flip}}$ is unstable. Consequently, as $a$ increases across $a_{c,2}$, the blowup transitions from stable to unstable, while the self-similar profile remains compactly supported. Below we plot the profiles $\bar\Omega$ for $a<a_{c,2}$ and the flipped profiles $\bar\Omega^{\mathrm{flip}}$ for $a>a_{c,2}$ in Figure \ref{fig:a_positive_fix_k_seperated} to better illustrate the transition across $a_{c,2}$.

\begin{figure}[htbp]
    \begin{subfigure}{0.49\textwidth}
        \includegraphics[width=\textwidth]{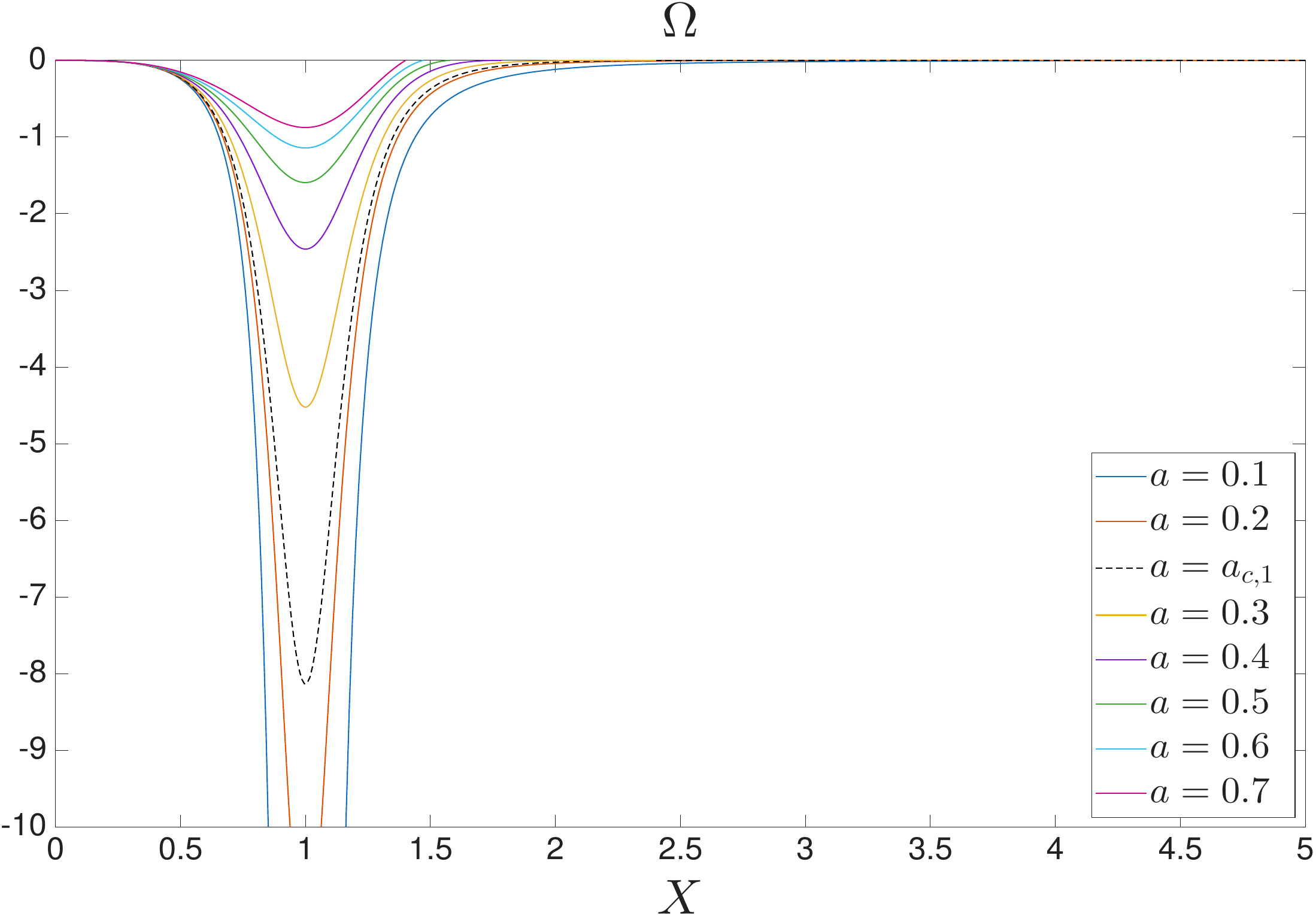}
        \caption{ $a<a_{c,2}$}
        \label{fig:a_positive_fix_k_0.1to0.7}
    \end{subfigure}
    \begin{subfigure}{0.49\textwidth}
        \includegraphics[width=\textwidth]{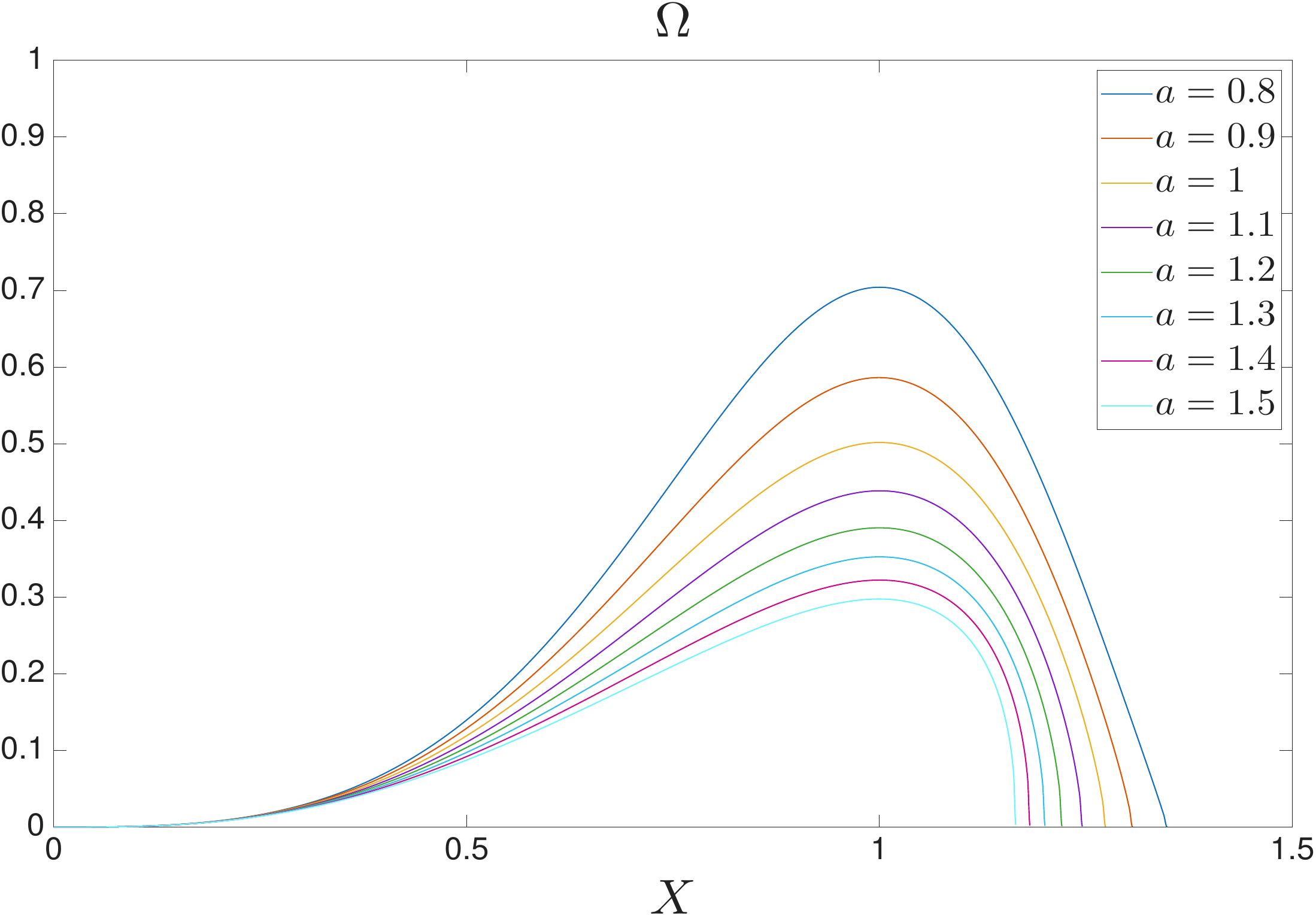}
        \caption{$a>a_{c,2}$}
        \label{fig:a_positive_fix_k_0.8to1.5}
    \end{subfigure}
    \caption{One-scale self-similar blowup profiles for fixed vanishing order $k=3$: (a) $a<a_{c,2}$ and (b) $a>a_{c,2}$. In (b), we plot the flipped profiles $\bar\Omega^{\mathrm{flip}}(X):=-\bar\Omega(X)$ since $\bar c_\omega>0$ in these cases; these $\bar\Omega^{\mathrm{flip}}$ correspond to unstable self-similar blowups.}
    \label{fig:a_positive_fix_k_seperated}
\end{figure}

\subsection{Numerical results for different vanishing orders $k$ with fixed $a=0.5$}
We next present dynamic rescaling results for initial data with different vanishing orders $k$ at the origin, while keeping $a=0.5$ fixed. The steady-state scaling factors for $k=3,5,7,9,11$ are reported in Table \ref{tab:one_scale_scaling_factors_fixed_k}, and the corresponding self-similar profiles are shown in Figure \ref{fig:a_positive_fix_a}.

\begin{table}[htbp]
    \begin{tabular}{|c|c|c|c|c|c|c|c|}
        \hline
        {$k$}           & $ 3       $ & $ 5       $ & $ 7        $ & $ 9          $           & $ 11       $            \\ \hline
        $\bar c_l$      & $ -0.1288 $ & $ -0.1052 $ & $ -0.0497  $ & $ -0.0168    $           & $ -0.0044    $          \\ \hline
        $\bar c_\omega$ & $ -0.0872 $ & $ -0.0193 $ & $ -0.0016  $ & $ 6.7927\times 10^{-4} $ & $ 3.7201\times 10^{-4}$ \\ \hline
        $\bar \gamma$   & $ -1.4771 $ & $ -5.4586 $ & $ -30.5045 $ & $ 24.7677    $           & $ 11.9515  $            \\ \hline
    \end{tabular}
    \caption{Scaling factors for profiles with vanishing orders $k=3,5,7,9,11$ and fixed $a=0.5$.}
    \label{tab:one_scale_scaling_factors_fixed_k}
\end{table}

\begin{figure}[htbp]
    \includegraphics[width=0.49\textwidth]{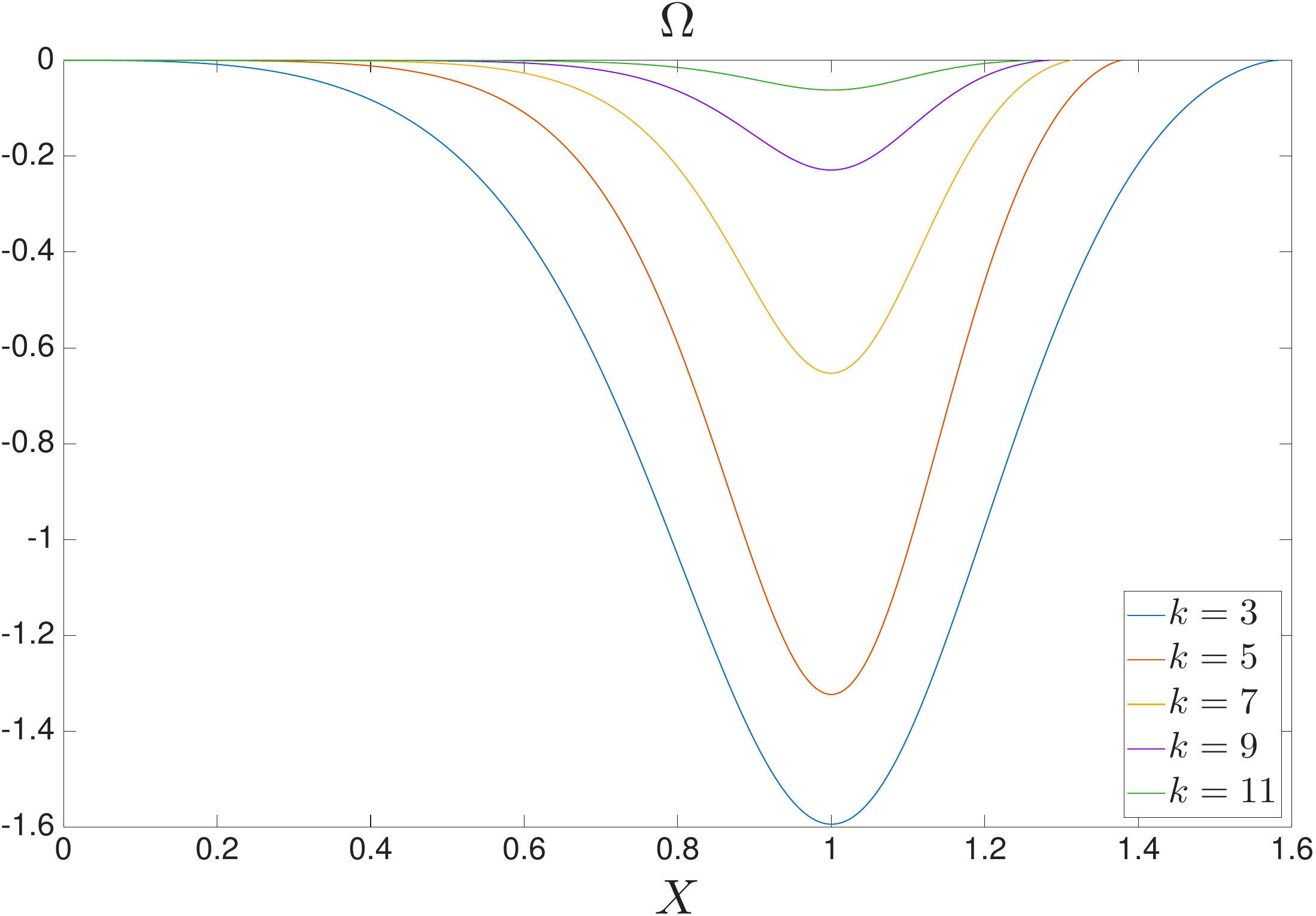}
    \caption{One-scale self-similar profiles $\Omega(X)$ for $k=3,5,7,9,11$ with fixed $a=0.5$.}
    \label{fig:a_positive_fix_a}
\end{figure}

We observe that, for a fixed $a$, there can be multiple blowup profiles distinguished by their vanishing orders at the origin, and the associated scaling factors may differ substantially. As shown in Table \ref{tab:one_scale_scaling_factors_fixed_k}, $c_\omega$ changes sign from negative to positive between $k=7$ and $k=9$. This transition suggests that, for $a=0.5$, initial data with sufficiently high vanishing order (in particular, $k\ge 9$) may not exhibit the stable self-similar blowup behavior observed for smaller values of $k$.

\subsection{Numerical results for various values of $a>0$ and $k\geq 3$}

Table \ref{tab:one_scale_scaling_factors} summarizes the measured shrinking rates $\bar\gamma$ for $0.1\le a\le 1$ and vanishing orders $k\in\{3,5,7,9,11\}$. These data highlight a strong dependence of $\bar\gamma$ on both the advection parameter $a$ and the vanishing order $k$. For each fixed $k$, as $a$ varies from $0.1$ to $1$, $\bar\gamma$ may change sign multiple times, indicating transitions between focusing ($\bar\gamma>0$) and expanding ($\bar\gamma<0$) blowups. Likewise, for each fixed $a$, varying $k$ can also induce sign changes in $\bar\gamma$, showing that the vanishing order at the origin can substantially affect the observed self-similar blowup behavior.

\begin{table}[htbp]
    \begin{tabular}{|c|c|c|c|c|c|c|c|}
        \hline
        \diagbox{$k$}{$\bar\gamma$}{$a$} & $ 0.1      $ & $ 0.2      $ & $ 0.3      $ & $ 0.4         $ & $ 0.5       $  \\ \hline
        $3$                              & $0.2816$     & $0.0827  $   & $  -0.2061 $ & $  -0.6604  $   & $  -1.4771   $ \\ \hline
        $5$                              & $ 0.1201 $   & $ -0.2000 $  & $  -0.7406 $ & $ -1.8523 $     & $ -5.4586$     \\ \hline
        $7$                              & $ 0.0112 $   & $ -0.4125 $  & $ -1.2239 $  & $   -3.4137 $   & $ -30.5045 $   \\ \hline
        $9$                              & $ -0.0678 $  & $  -0.5798$  & $-1.6672  $  & $  -5.5676   $  & $ 24.7677 $    \\ \hline
        $11$                             & $ -0.1292 $  & $ -0.7152$   & $ -2.0756$   & $ -8.7223  $    & $ 11.9515  $   \\ \hline\hline
        \diagbox{$k$}{$\bar\gamma$}{$a$} & $ 0.6     $  & $ 0.7      $ & $ 0.8      $ & $ 0.9         $ & $ 1       $    \\ \hline
        $3$                              & $-3.3734 $   & $ -12.6465 $ & $  16.2599 $ & $  6.2550  $    & $ 4.2996    $  \\ \hline
        $5$                              & $ 61.5417 $  & $ 6.9537 $   & $ 4.3061  $  & $ 3.3675 $      & $ 2.8870$      \\ \hline
        $7$                              & $ 8.3875 $   & $ 4.5758 $   & $ 3.4829 $   & $   2.9737 $    & $  2.6862$     \\ \hline
        $9$                              & $6.5143  $   & $4.5096  $   & $ 3.9087 $   & $  3.7528   $   & $ 3.8290 $     \\ \hline
        $11$                             & $11.4908 $   & $ 15.2982$   & $ 222.7743$  & $ -13.4241  $   & $  -6.2561 $   \\ \hline
    \end{tabular}
    \caption{Scaling factors $\bar\gamma$ for profiles with $0.1\le a\le 1$ and vanishing orders $k=3,5,7,9,11$.}
    \label{tab:one_scale_scaling_factors}
\end{table}

\section{Numerical results with degenerate initial data for $a<0$}\label{sec:a_le_0}
In this section, we present the numerical results of the dynamic rescaling simulations for the gCLM model \eqref{eqt:gCLM} with degenerate initial data when $a< 0$. The purpose of this investigation has been introduced in Section \ref{sec:novel_blowups}.

Recall that $\omega_0$ denotes the initial data of the gCLM model \eqref{eqt:gCLM}. In the following subsections, we consider the initial data in both the odd-symmetry case (Assumption \ref{ass:odd_sym}) and the half-line case (Assumption \ref{ass:half_line}) with the following subdivisions:
\begin{itemize}
    \item Case \labelText{1.1}{cas:initialdata_case1.1}: $\omega_0$ is odd symmetric (i.e. $\omega_0(-x)=-\omega_0(x)$), $\omega_0'(0)=0$, and $\omega_0(x)<0$ for $x>0$;
    \item Case \labelText{1.2}{cas:initialdata_case1.2}: $\omega_0$ is odd symmetric, $\omega_0(x)=0$ on $[-1,1]$, and $\omega_0(x)<0$ for $x>1$;
    \item Case \labelText{2.1}{cas:initialdata_case2.1}: $\omega_0(x)=0$ for $x\le 0$, $\omega_0'(0)=0$, and $\omega_0(x)<0$ for $x>0$;
    \item Case \labelText{2.2}{cas:initialdata_case2.2}: $\omega_0(x)=0$ for $x\le 1$, and $\omega_0(x)<0$ for $x>1$.
\end{itemize}

In all four cases, the initial condition $\omega_0$ is degenerate at the origin and $C^\infty$-smooth on the entire real line. Note that in the dynamic rescaling formulation, the initial condition $\omega_0$ for the gCLM model \eqref{eqt:gCLM} corresponds directly to the initial profile $\Omega_0=\Omega(\cdot,0)$ for the dynamic rescaling equation \eqref{eqt:dynamic_rescaling}. The specific initial data chosen for these four cases are plotted in Figure \ref{fig:initial_data_a_negative}. For Case \nameref{cas:initialdata_case1.1}, we select an initial datum with infinite vanishing order at the origin. This choice is motivated by our numerical observation that, under our normalization conditions, all derivatives at the origin satisfy $\partial_X^j\Omega(0,\tau)\to 0$ as $\tau\to+\infty$. This suggests that the self-similar profile possesses infinite vanishing order at the origin, regardless of the specific vanishing order of the initial data. We will discuss this phenomenon further in Section \ref{sec:odd_initial_data}.

\begin{figure}[htbp]
    \begin{subfigure}{0.49\textwidth}
        \includegraphics[width=\textwidth]{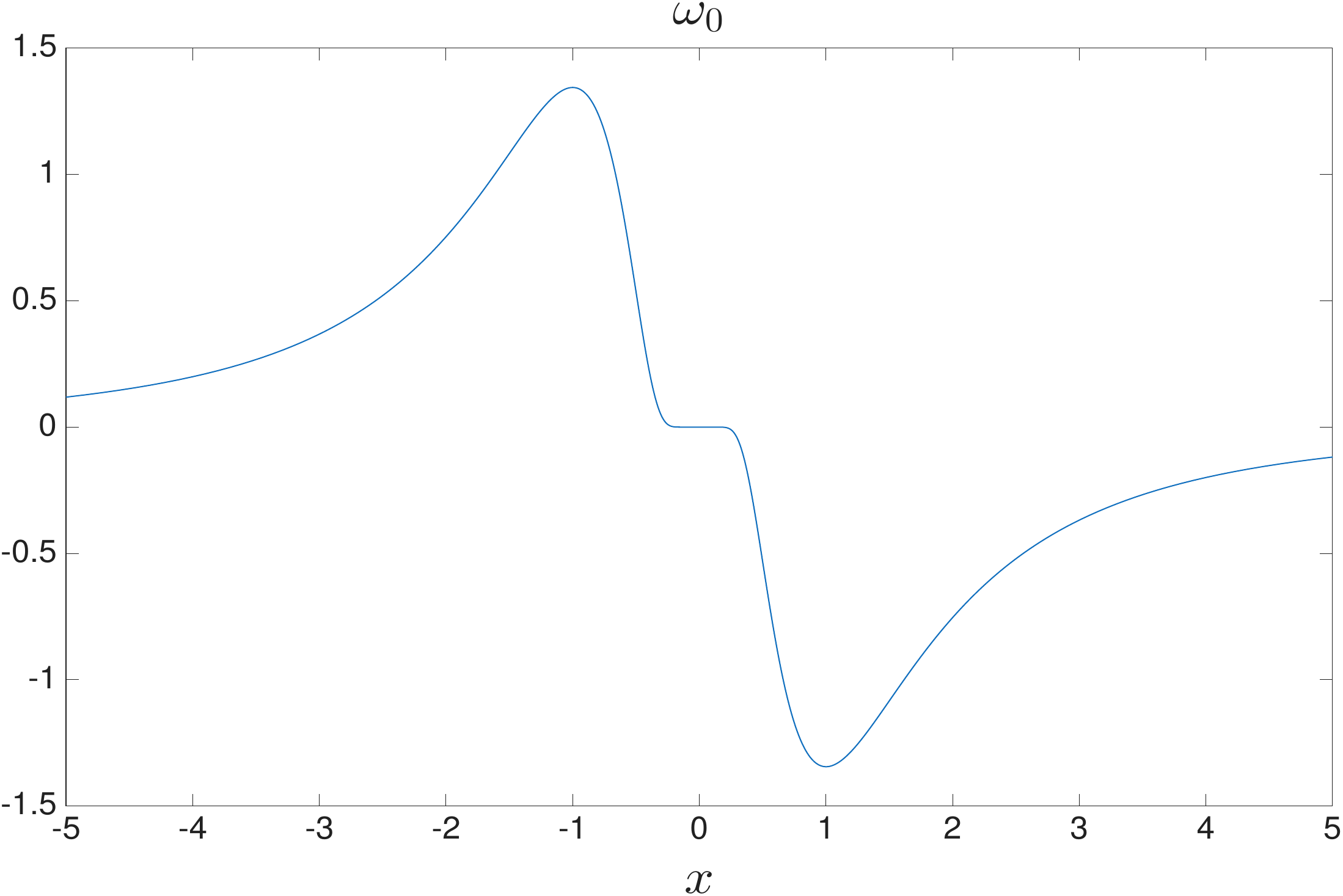}
        \caption{Case \nameref{cas:initialdata_case1.1}: $\omega_0=-(x/3)^{-3}\mathrm{e}^{-\frac{1}{\left|x/3\right|}}$ }
    \end{subfigure}
    \begin{subfigure}{0.49\textwidth}
        \includegraphics[width=\textwidth]{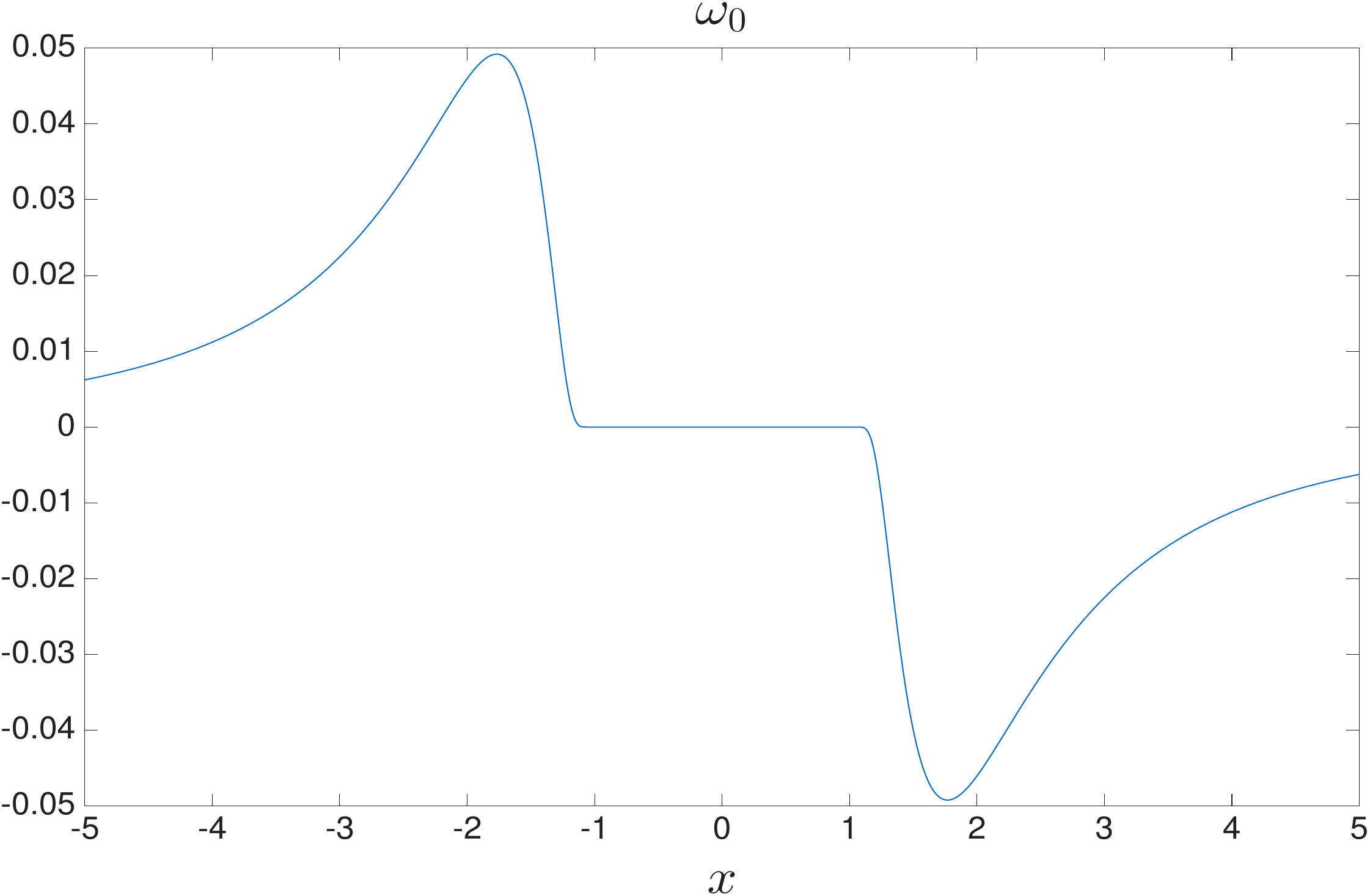}
        \caption{Case \nameref{cas:initialdata_case1.2}: $\omega_0=-x^{-3}\mathrm{e}^{-\frac{1}{|x|-1}}\boldsymbol{1}_{\{|x|>1\}}$ }
    \end{subfigure}
    \begin{subfigure}{0.49\textwidth}
        \includegraphics[width=\textwidth]{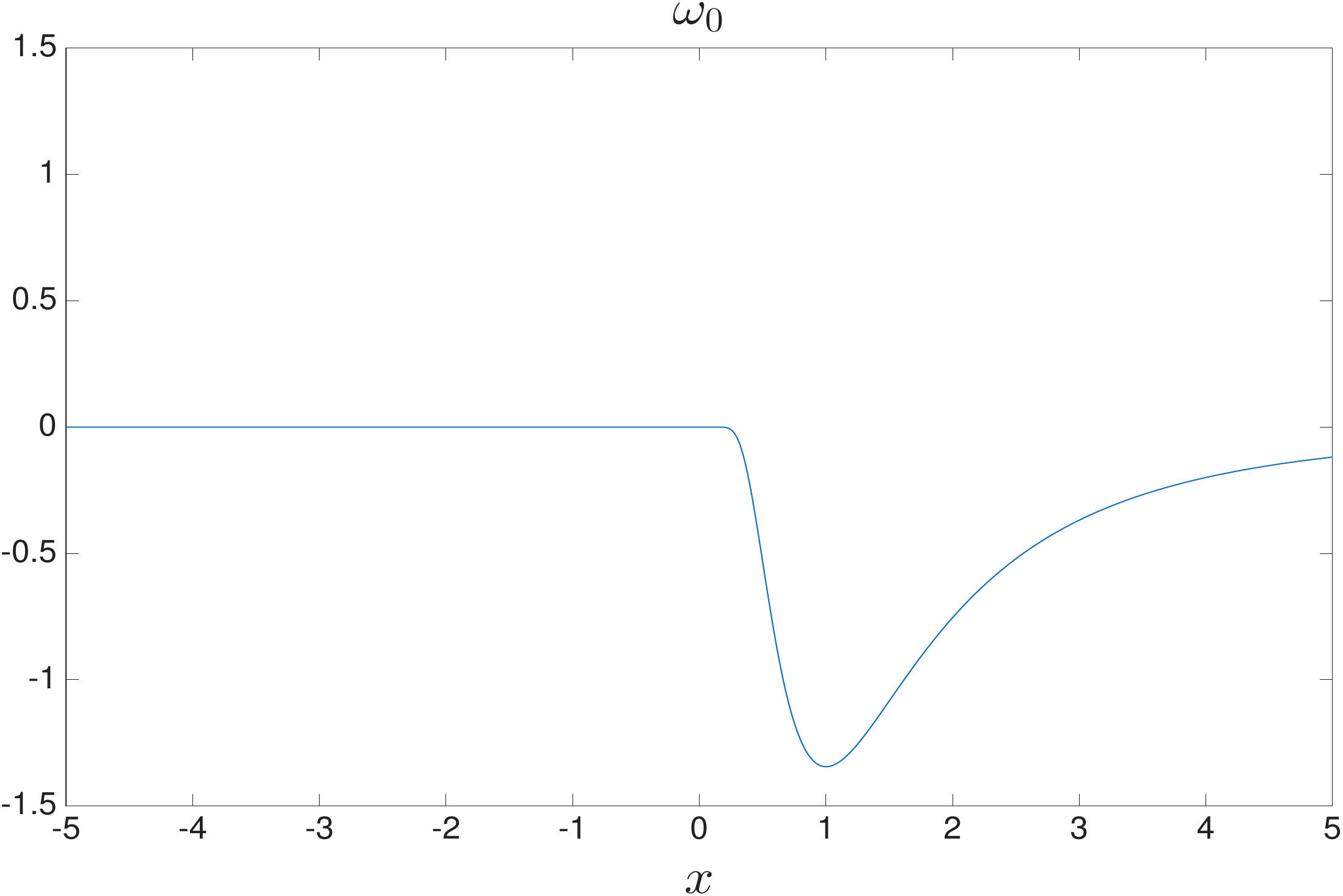}
        \caption{Case \nameref{cas:initialdata_case2.1}: $\omega_0=-(x/3)^{-3}\mathrm{e}^{-\frac{1}{|x/3|}}\boldsymbol{1}_{\{x>0\}}$ }
    \end{subfigure}
    \begin{subfigure}{0.49\textwidth}        \includegraphics[width=\textwidth]{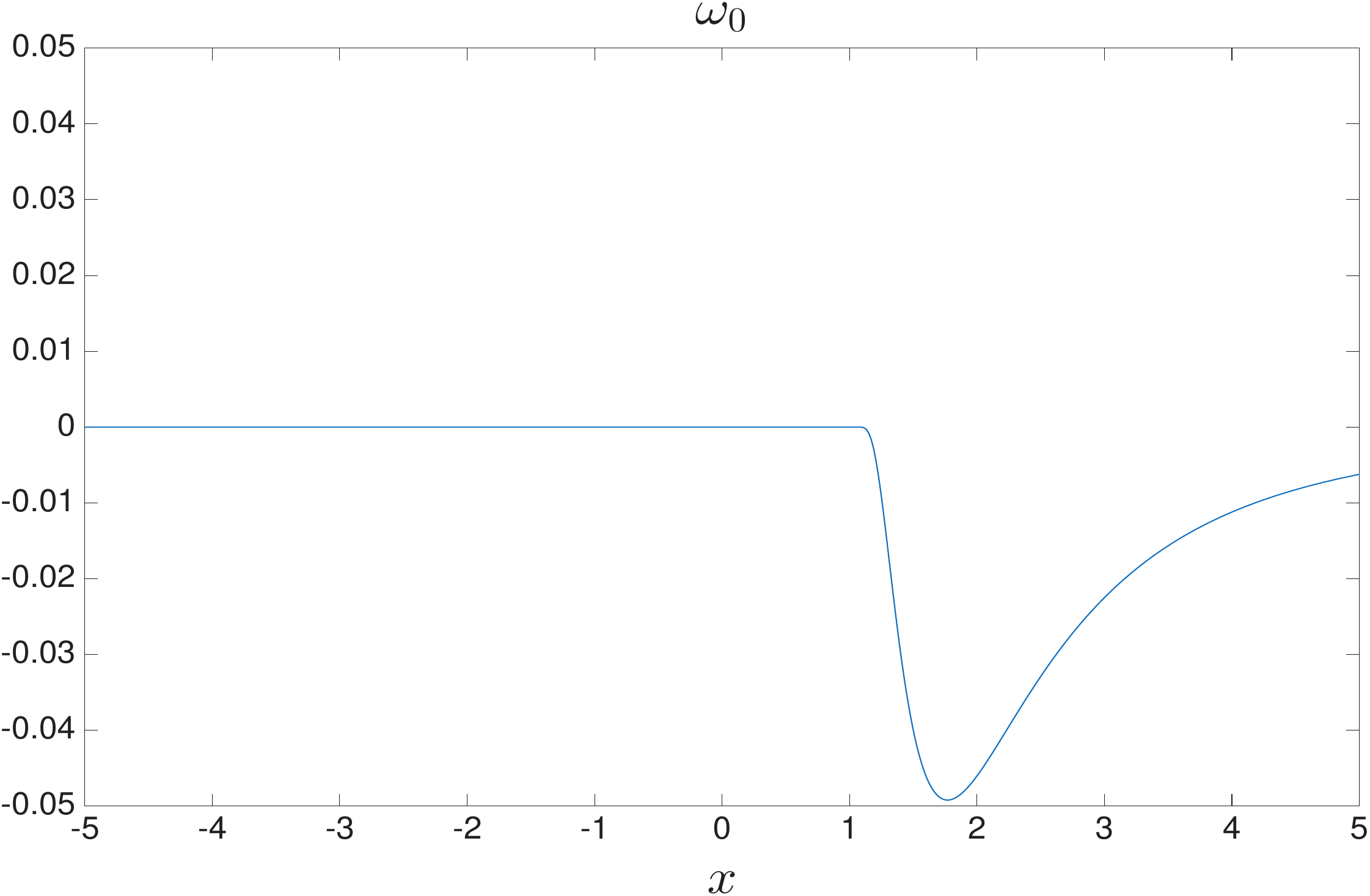}
        \caption{Case \nameref{cas:initialdata_case2.2}: $\omega_0=-x^{-3}\mathrm{e}^{-\frac{1}{x-1}}\boldsymbol{1}_{\{x>1\}}$ }
    \end{subfigure}
    \caption{Initial data for $a<0$. All the initial data are $C^\infty$-smooth, nonpositive when $x>0$, and have infinite vanishing order at the origin.}
    \label{fig:initial_data_a_negative}
\end{figure}

\subsection{Numerical results for different $a<0$ with odd symmetric degenerate initial data}\label{sec:odd_initial_data}

Here we present numerical results from dynamic rescaling simulations for $a<0$ with odd-symmetric degenerate initial data (under Assumption \ref{ass:odd_sym}), corresponding to Cases \nameref{cas:initialdata_case1.1} and \nameref{cas:initialdata_case1.2}. The normalization conditions used to close the dynamic rescaling equation \eqref{eqt:dynamic_rescaling} differ between these two cases. When $a<0$, our simulations indicate that the rescaled profile $\Omega$ evolving from degenerate data develops \emph{infinite} vanishing order at the origin (see further discussion below). Consequently, the normalization condition $\lim_{X\to0}-\Omega(X,\tau)/X^k\equiv 1$ (which fixes a prescribed finite vanishing order) is not suitable in this case. We therefore adopt normalization conditions that do not depend on the local Taylor behavior at $X=0$.

In particular, for Case \nameref{cas:initialdata_case1.1}, we choose to fix the value of the Hilbert transform at the origin, i.e. $\mtx H(\Omega)(0,\tau)\equiv \mtx H(\Omega_0)(0)$. Together with pinning the location of the minimum at $X=1$, we impose
\begin{equation}\label{eqt:normalization_condition_case1.1}
    \partial_X\Omega(1,\tau)\equiv 0,\quad \mtx H(\Omega)(0,\tau)\equiv \mtx H(\Omega_0)(0).
\end{equation}
For Case \nameref{cas:initialdata_case1.2}, since the initial data vanish on $[-1,1]$, the condition $\partial_X\Omega(1,\tau)\equiv 0$ is not suitable for pinning the location of the minimum. Instead, we fix the boundary of the support of $\Omega$ at $X=1$ by imposing
\[
    c_l(\tau)+aU(1,\tau)\equiv 0,
\]
which makes the convection speed in the dynamic rescaling equation \eqref{eqt:dynamic_rescaling} vanish at $X=1$. As a result, the support boundary remains stationary at $X=1$ throughout the evolution. Therefore, for Case \nameref{cas:initialdata_case1.2}, we impose
\begin{equation}\label{eqt:normalization_condition_case1.2}
    c_l(\tau)+aU(1,\tau)\equiv 0,\quad \mtx H(\Omega)(0,\tau)\equiv \mtx H(\Omega_0)(0).
\end{equation}

Figures \ref{fig:a=-1_odd_evolution} and \ref{fig:a=-1_odd_compact_evolution} display the evolution of the rescaled profile $\Omega(X,\tau)$ and the corresponding advection speed $c_l(\tau)X+aU(X,\tau)$ for $a=-1$ in Case \nameref{cas:initialdata_case1.1} and Case \nameref{cas:initialdata_case1.2}, respectively. In Case \nameref{cas:initialdata_case1.1}, as $\tau$ increases, $\Omega$ develops an increasingly sharp front near $X=1$ and appears to approach a singular limiting profile there as $\tau\to+\infty$. At the same time, the profile on $[0,1]$ flattens progressively and tends to vanish. This behavior is related to the evolution of the advection speed $c_lX+aU$. As we can see in Figure \ref{fig:a=-1_odd_evolution}(b), $c_lX+aU$ remains positive for $X>0$, decreases to $0$ at $X=1$ as $\tau$ grows, and develops a cusp-like shape near $X=1$. The positivity of $c_lX+aU$ on $(0,1)$ indicates that the degeneracy of $\Omega$ near $X=0$ is transported toward $X=1$, thus suppressing $\Omega$ on $[0,1)$ and leading to the observed flattening. At the same time, the emergence of a stagnation point at $X=1$ breaks the strictly outgoing property. This prevents the flattening from going beyond $X=1$, thereby allowing the profile $\Omega$ to develop a singularity at that point. For $X>1$, the velocity becomes positive again, transporting $\Omega$ away from the front and generating the observed tail.

\begin{figure}[htbp]
    \begin{subfigure}{0.49\textwidth}
        \includegraphics[width=\textwidth]{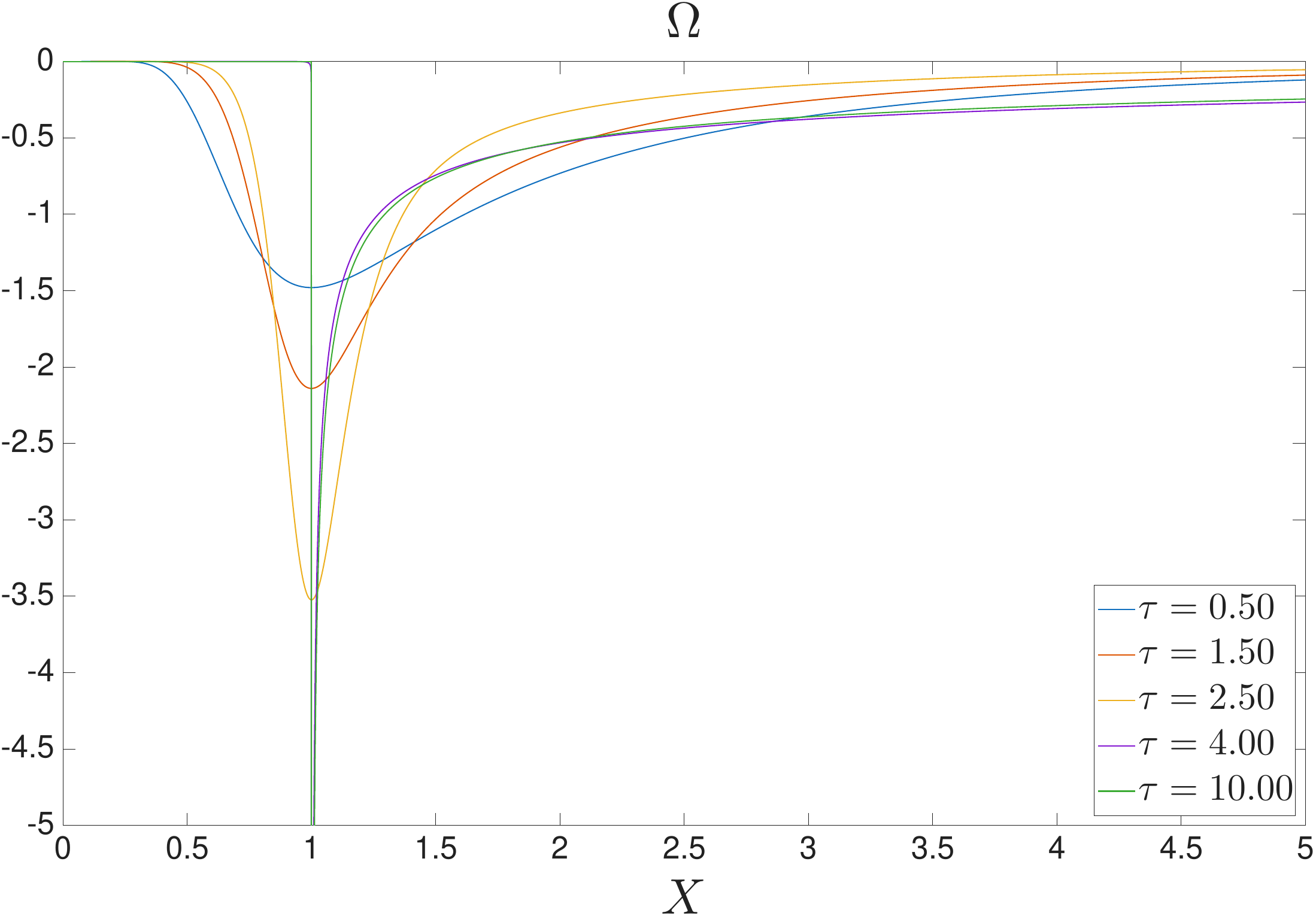}
        \caption{Evolution of the profile $\Omega(X,\tau)$}
    \end{subfigure}
    \begin{subfigure}{0.49\textwidth}
        \includegraphics[width=\textwidth]{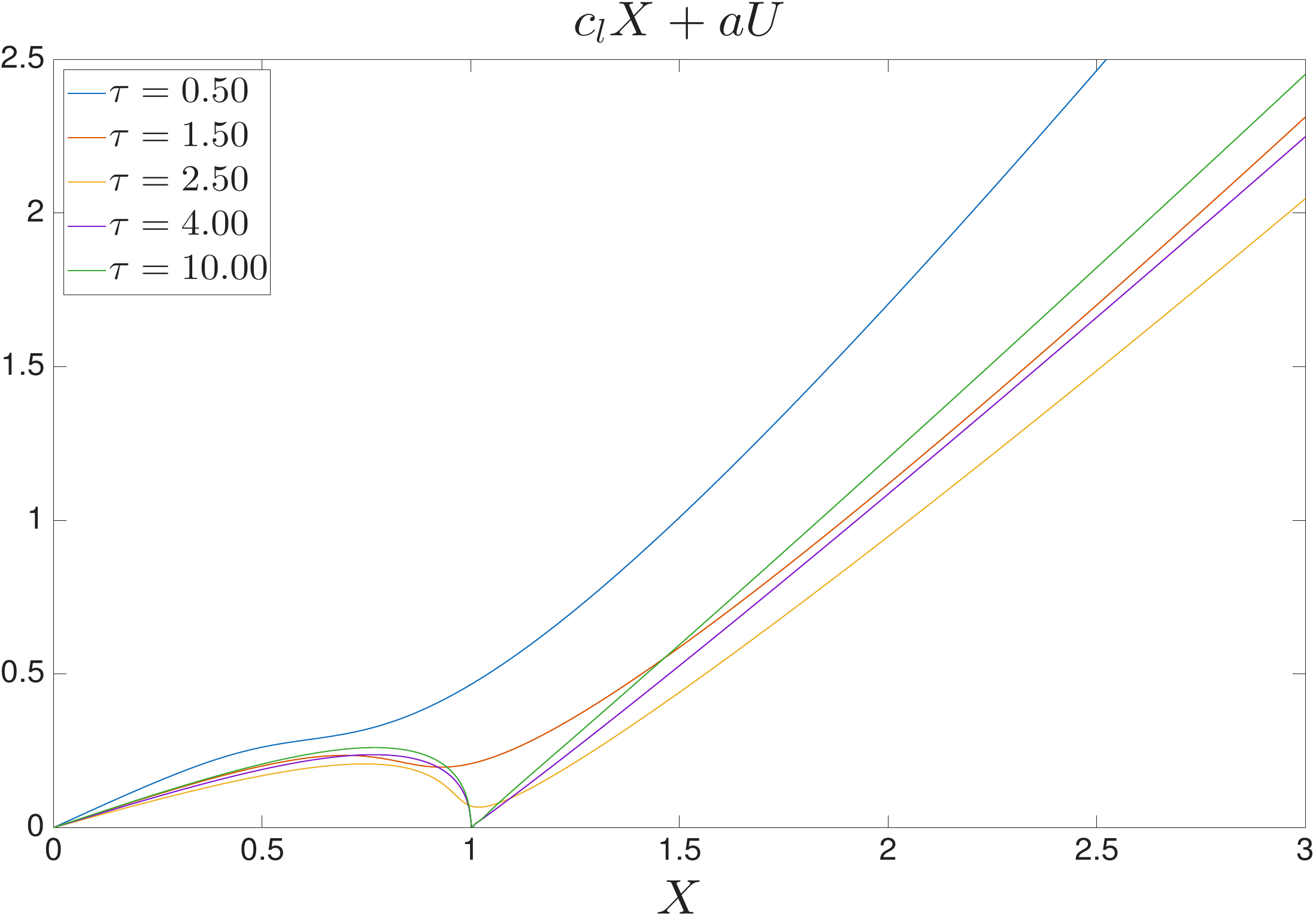}
        \caption{Evolution of the advection speed $c_l(\tau)X+aU(X,\tau)$}
    \end{subfigure}
    \caption{Dynamic rescaling simulation for $a=-1$ with odd-symmetric degenerate initial data supported on the whole real line (Case \nameref{cas:initialdata_case1.1}). We impose the normalization conditions $\partial_X\Omega(1,\tau)\equiv 0$ (pinning the minimum at $X=1$) and $\mtx H(\Omega)(0,\tau)\equiv \mtx H(\Omega_0)(0)$.}
    \label{fig:a=-1_odd_evolution}
\end{figure}

In Case \nameref{cas:initialdata_case1.2}, in contrast, the profile stays identically zero on $[0,1]$ for all $\tau$. This is a direct consequence of the normalization $c_l(\tau)+aU(1,\tau)\equiv 0$, which pins the transition point at $X=1$ and prevents transport across it. Meanwhile, $\Omega$ develops a sharp front to the right of $X=1$, and this front travels toward the transition point as $\tau$ increases. Such behavior is also related to the structure of the advection flow $c_lX+aU$ displayed in Figure \ref{fig:a=-1_odd_compact_evolution}(b). In addition to the origin, $c_lX+aU$ has two other zeros on the half-line $X>0$ before convergence, one at the fixed transition point $X=1$ and the other located slightly to the right. Thus, the strictly outgoing property fails. The front of $\Omega$ consistently lies between these two zeros between which $c_lX+aU<0$, and is therefore advected leftward. As $\tau\to+\infty$, the rightmost zero approaches $X=1$, causing the front to concentrate at the transition point and producing the observed cusp-like behavior of $c_lX+aU$ near $X=1$. Similar to Case \nameref{cas:initialdata_case1.1}, $c_lX+aU$ is positive beyond its rightmost zero, thus transporting the profile rightward and generating the observed tail. Numerically, the limiting profiles of $\Omega$ and $c_lX+aU$ in Cases \nameref{cas:initialdata_case1.1} and \nameref{cas:initialdata_case1.2} appear to coincide, suggesting the convergence to the same steady state with these two different choices of initial data and normalization conditions.

\begin{figure}[htbp]
    \begin{subfigure}{0.49\textwidth}
        \includegraphics[width=\textwidth]{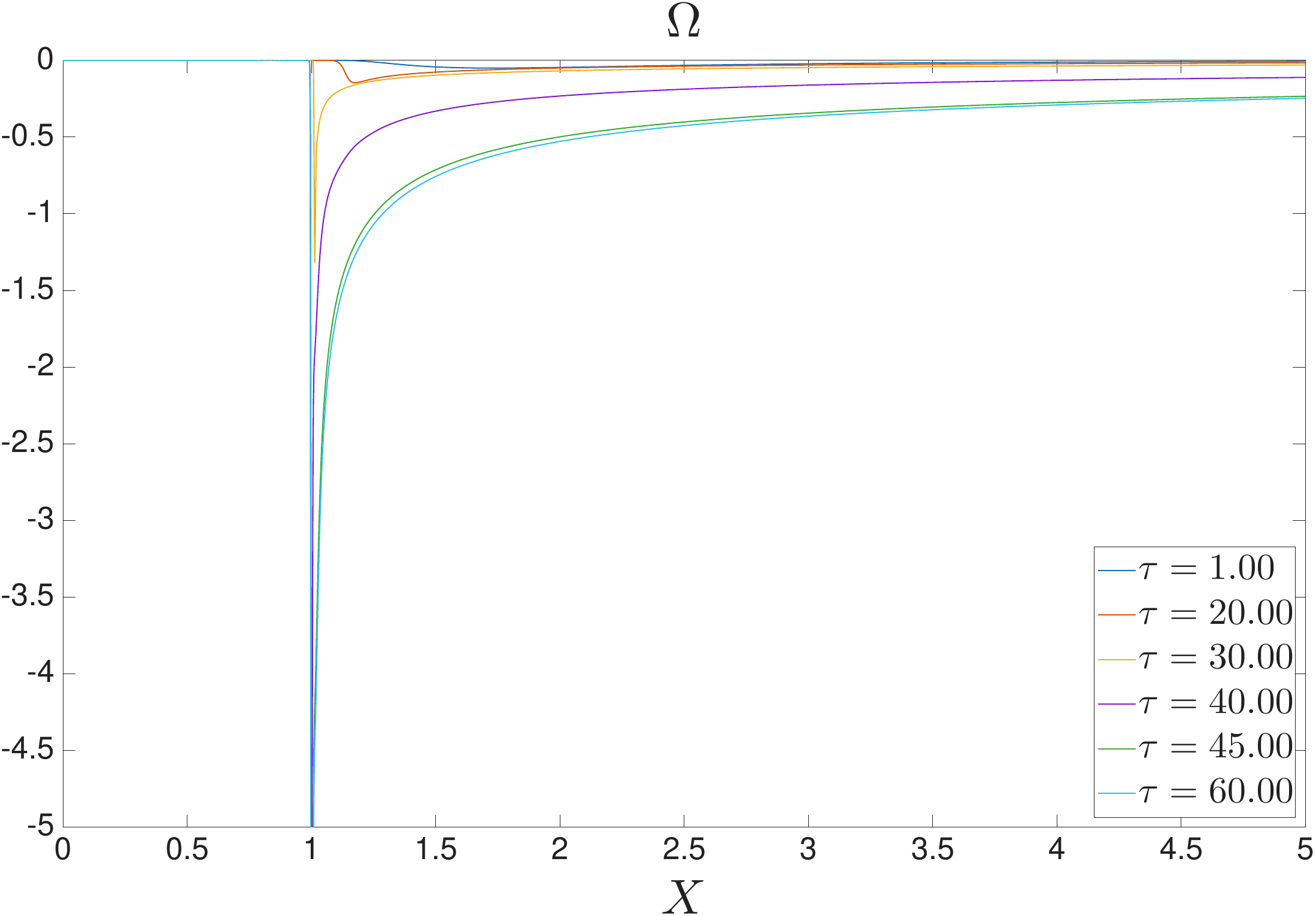}
        \caption{Evolution of the profile $\Omega(X,\tau)$}
    \end{subfigure}
    \begin{subfigure}{0.49\textwidth}
        \includegraphics[width=\textwidth]{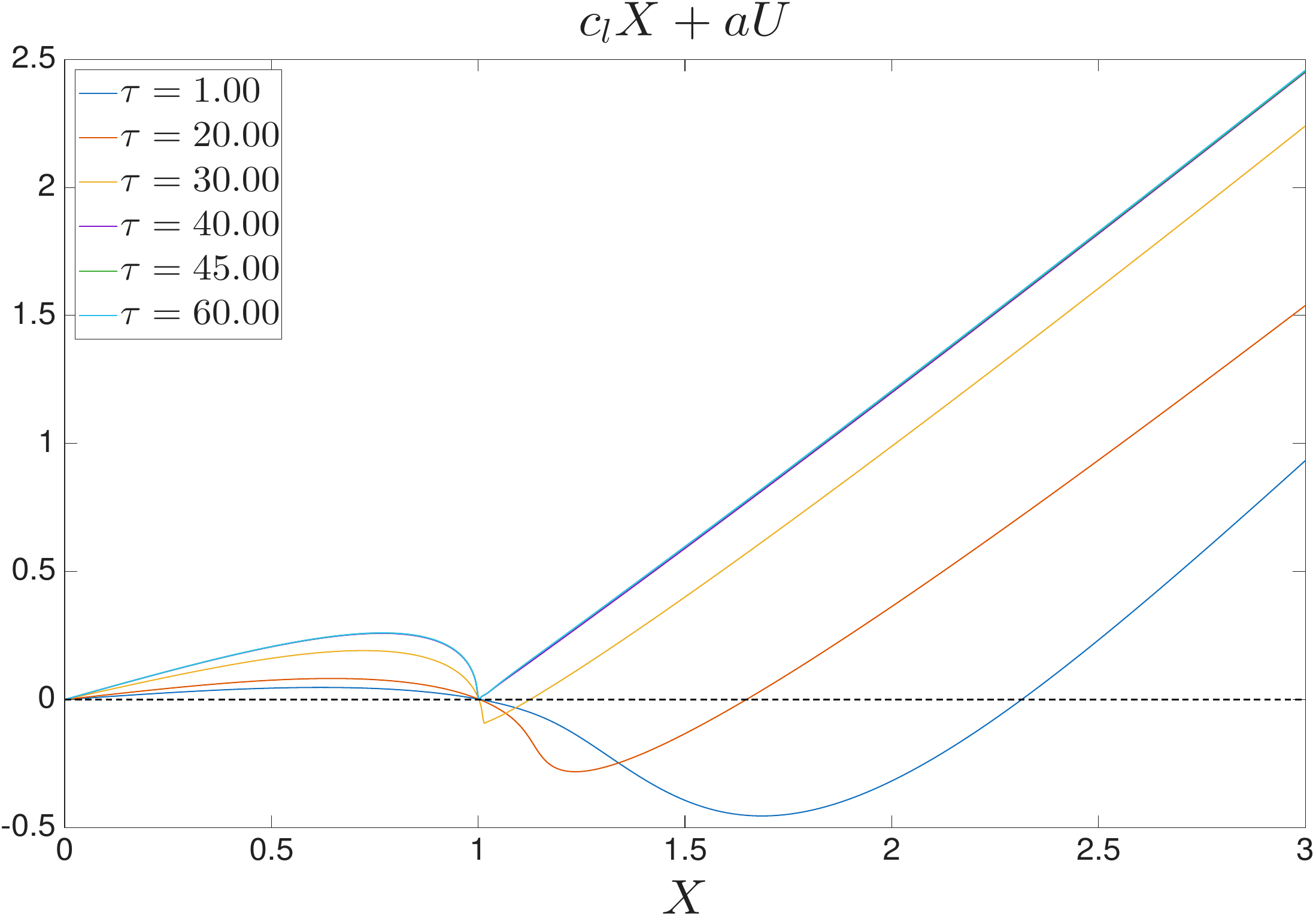}
        \caption{Evolution of the advection speed $c_l(\tau)X+aU(X,\tau)$}
    \end{subfigure}
    \caption{Dynamic rescaling simulation for $a=-1$ with odd-symmetric degenerate initial data that vanishes on $[-1,1]$ (Case \nameref{cas:initialdata_case1.2}). We impose the normalization conditions $c_l(\tau)+aU(1,\tau)\equiv 0$ (pinning the transition point at $X=1$) and $\mtx H(\Omega)(0,\tau)\equiv \mtx H(\Omega_0)(0)$.}
    \label{fig:a=-1_odd_compact_evolution}
\end{figure}

To examine the convergence of the dynamic rescaling simulation in more detail, we use the case $a=-1$ with odd-symmetric degenerate initial data (Case \nameref{cas:initialdata_case1.1}) as a representative example. Figure \ref{fig:a=-1_odd_convergence} shows that, as $\tau\to+\infty$, the rescaled profile $\Omega(X,\tau)$ approaches a limiting shape, both in the near field (a) and in the far field (b). Meanwhile, the scaling factors $c_l(\tau)$ and $c_\omega(\tau)$, as well as the shrinking rate $\gamma(\tau)=-c_l(\tau)/c_\omega(\tau)$, converge to steady values, with limiting $\bar c_\omega<0$. This stabilization provides numerical evidence that the rescaled solution approaches a steady state of \eqref{eqt:dynamic_rescaling}, and hence corresponds to an asymptotically self-similar finite-time blowup of the original gCLM equation \eqref{eqt:gCLM}.

\begin{figure}[htbp]
    \begin{subfigure}{0.49\textwidth}
        \includegraphics[width=\textwidth]{am1_odd}
        \caption{Profile $\Omega(X,\tau)$ in the near field}
    \end{subfigure}
    \begin{subfigure}{0.49\textwidth}
        \includegraphics[width=\textwidth]{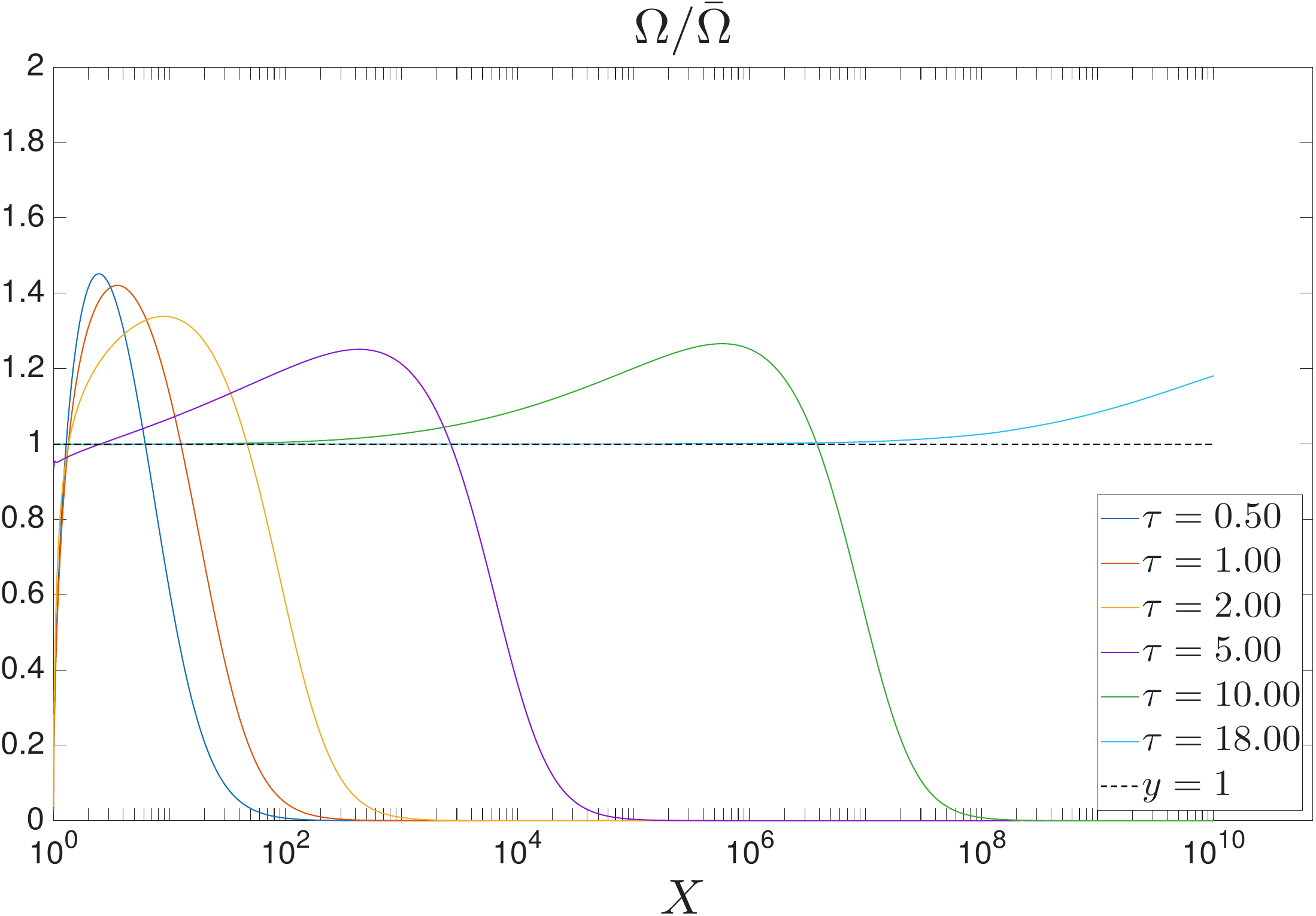}
        \caption{Profile $\Omega(X,\tau)$ in the far field}
    \end{subfigure}
    \begin{subfigure}{0.32\textwidth}
        \includegraphics[width=\textwidth]{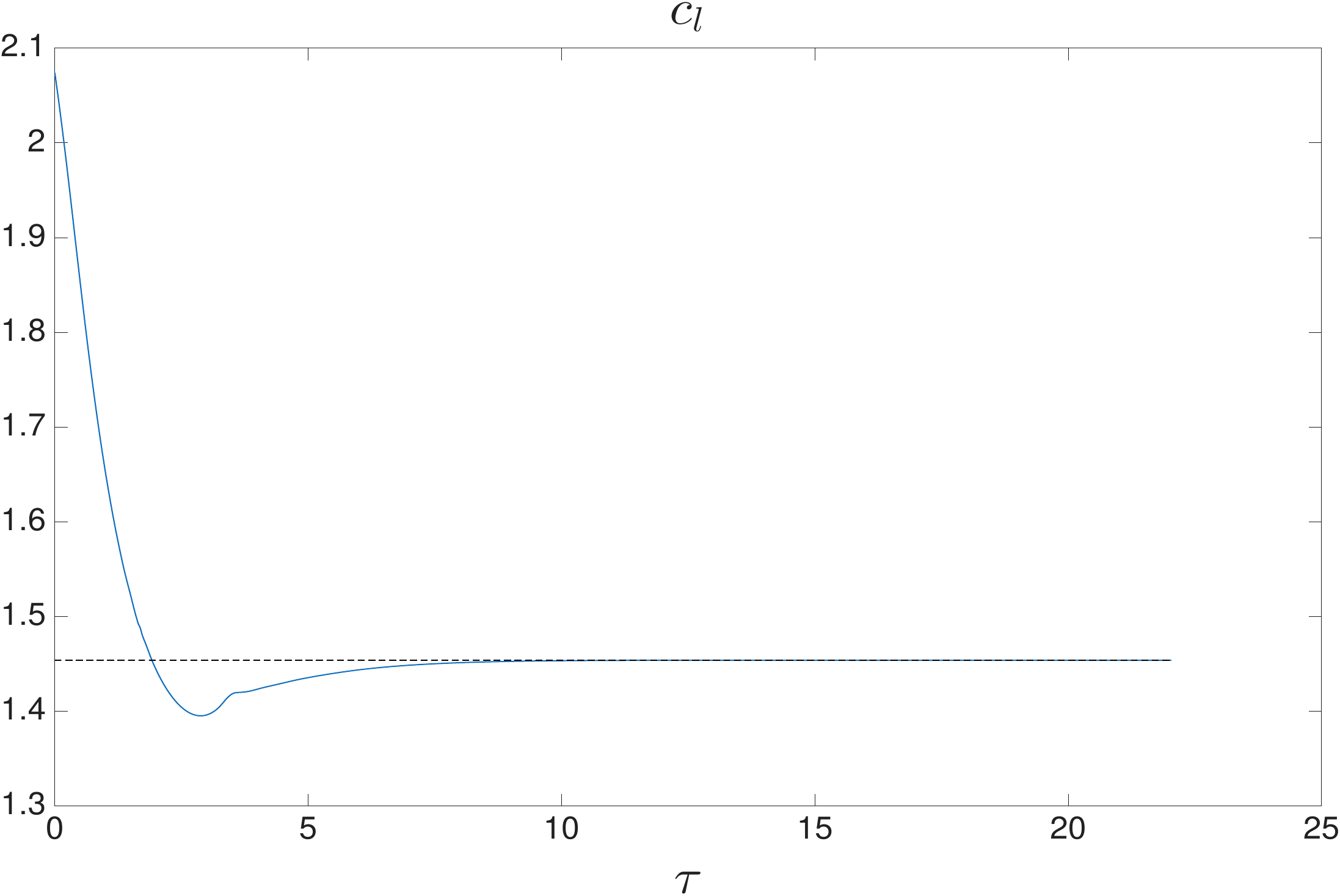}
        \caption{Scaling factor $c_l$}
    \end{subfigure}
    \begin{subfigure}{0.32\textwidth}
        \includegraphics[width=\textwidth]{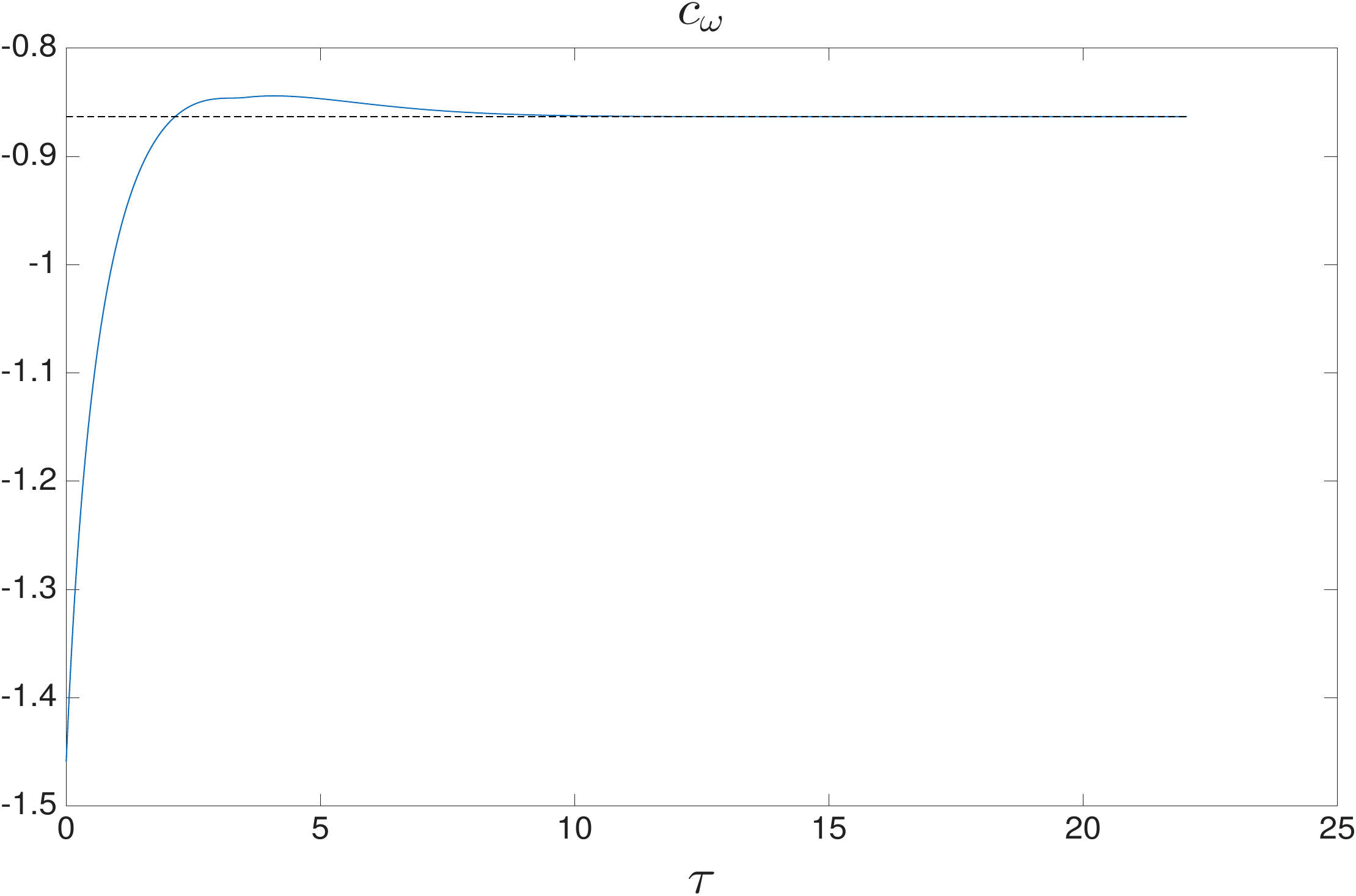}
        \caption{Scaling factor $c_\omega$}
    \end{subfigure}
    \begin{subfigure}{0.32\textwidth}
        \includegraphics[width=\textwidth]{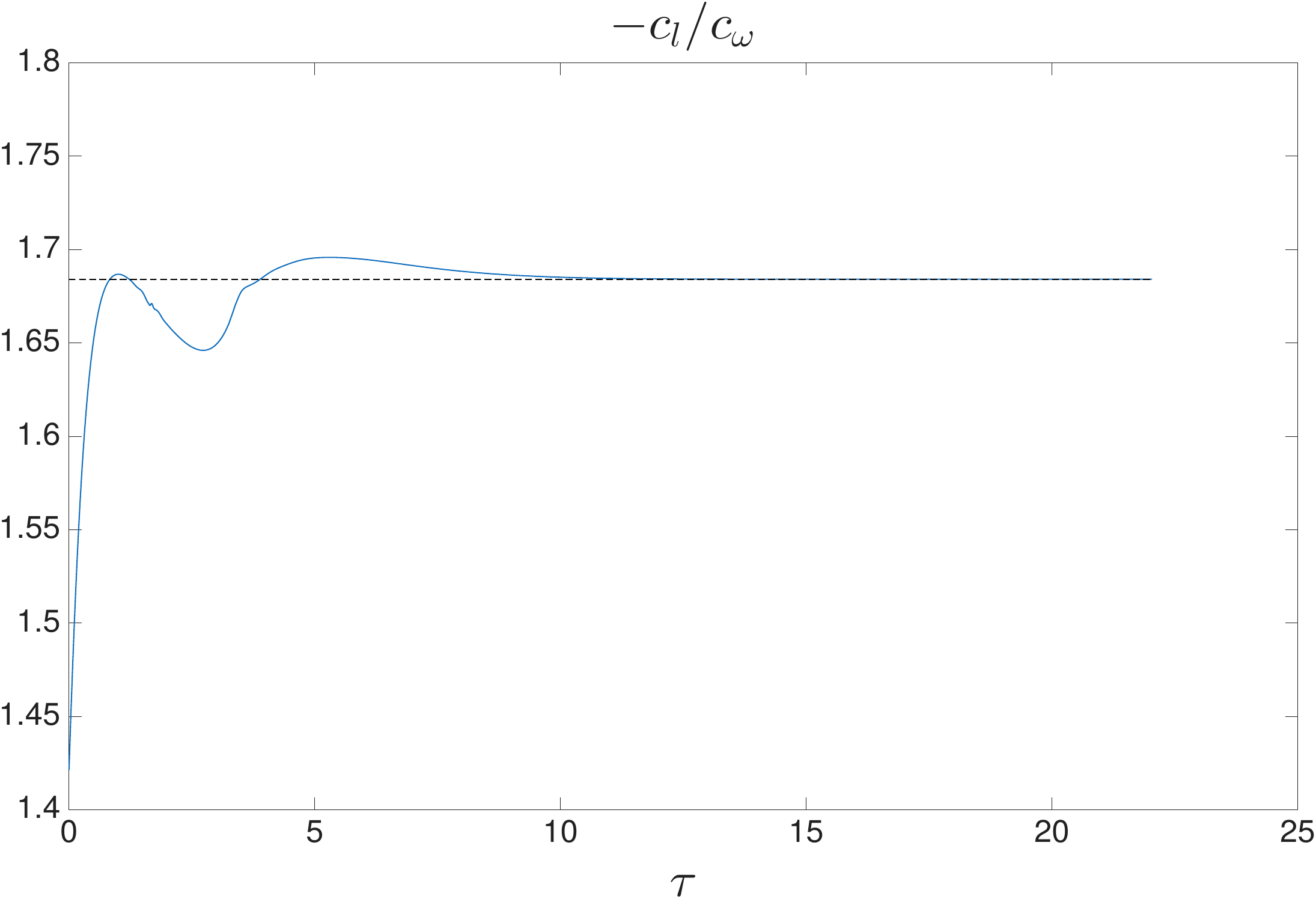}
        \caption{Spatial shrinking rate $\gamma=-c_l/c_\omega$}
    \end{subfigure}
    \caption{Convergence of the dynamic rescaling simulation for $a=-1$ with odd-symmetric degenerate initial data (Case \nameref{cas:initialdata_case1.1}). As $\tau$ increases, the rescaled profile $\Omega(X,\tau)$ approaches a limiting shape in both the near field and the far field, while the scaling factors $c_l(\tau)$ and $c_\omega(\tau)$ (and the shrinking rate $\gamma(\tau)=-c_l(\tau)/c_\omega(\tau)$) stabilize. In (b), the reference steady profile $\bar\Omega(X)$ is taken to be the numerical profile at the first time when the residual of the profile equation, restricted away from the singularities, satisfies
    $
        \bigl\|\Omega_\tau\,\boldsymbol{1}_{\{|X-1|>0.1,\ |X+1|>0.1\}}\bigr\|_{L^\infty}<10^{-8}.
    $
    }
    \label{fig:a=-1_odd_convergence}
\end{figure}

We next present numerical results for other values of $a<0$ with odd-symmetric degenerate initial data. Specifically, we numerically solve the dynamic rescaling equation \eqref{eqt:dynamic_rescaling} for $a=-0.1,-0.2,\ldots,-1$ in both Case \nameref{cas:initialdata_case1.1} and Case \nameref{cas:initialdata_case1.2}, imposing the normalization conditions \eqref{eqt:normalization_condition_case1.1} and \eqref{eqt:normalization_condition_case1.2}, respectively. For all these values of $a$, we observe convergence behavior qualitatively similar to the aforementioned case $a=-1$. To quantify convergence, we monitor the residual of the profile equation away from the singular points and declare convergence once
\[
    \bigl\|\Omega_\tau\,\boldsymbol{1}_{\{|X-1|>0.1,\ |X+1|>0.1\}}\bigr\|_{L^\infty}<10^{-8}.
\]
The resulting scaling factors for $-1\le a\le -0.1$ in Case \nameref{cas:initialdata_case1.1} are listed in Table \ref{tab:one_scale_scaling_factors_odd_a_negative}. The corresponding values for Case \nameref{cas:initialdata_case1.2} are nearly identical and are therefore omitted. The self-similar profiles for each $a$ are shown in Figure \ref{fig:a_negative_odd}.

\begin{table}[htbp]
    \begin{tabular}{|c|c|c|c|c|c|}
        \hline
        $a$        & $-0.1$    & $-0.2$    & $-0.3$    & $-0.4$    & $-0.5$    \\ \hline
        $c_l$      & $0.6189$  & $0.7554$  & $0.8060$  & $0.8676$  & $0.9580$  \\ \hline
        $c_\omega$ & $-0.7908$ & $-0.8287$ & $-0.7961$ & $-0.7897$ & $-0.8104$ \\ \hline
        $\gamma$   & $0.7827$  & $0.9114$  & $1.0124$  & $1.0986$  & $1.1821$  \\ \hline\hline
        $a$        & $-0.6$    & $-0.7$    & $-0.8$    & $-0.9$    & $-1$      \\ \hline
        $c_l$      & $1.0696$  & $1.1718$  & $1.2563$  & $1.3576$  & $1.4541$  \\ \hline
        $c_\omega$ & $-0.8368$ & $-0.8437$ & $-0.8455$ & $-0.8557$ & $-0.8634$ \\ \hline
        $\gamma$   & $1.2782$  & $1.3888$  & $1.4858$  & $1.5866$  & $1.6841$  \\ \hline
    \end{tabular}
    \caption{Scaling factors for $-1\le a\le -0.1$ obtained from odd-symmetric degenerate initial data (Case \nameref{cas:initialdata_case1.1}). The corresponding values for Case \nameref{cas:initialdata_case1.2} are nearly identical.}
    \label{tab:one_scale_scaling_factors_odd_a_negative}
\end{table}

\begin{figure}[htbp]
    \begin{subfigure}{0.49\textwidth}
        \includegraphics[width=\textwidth]{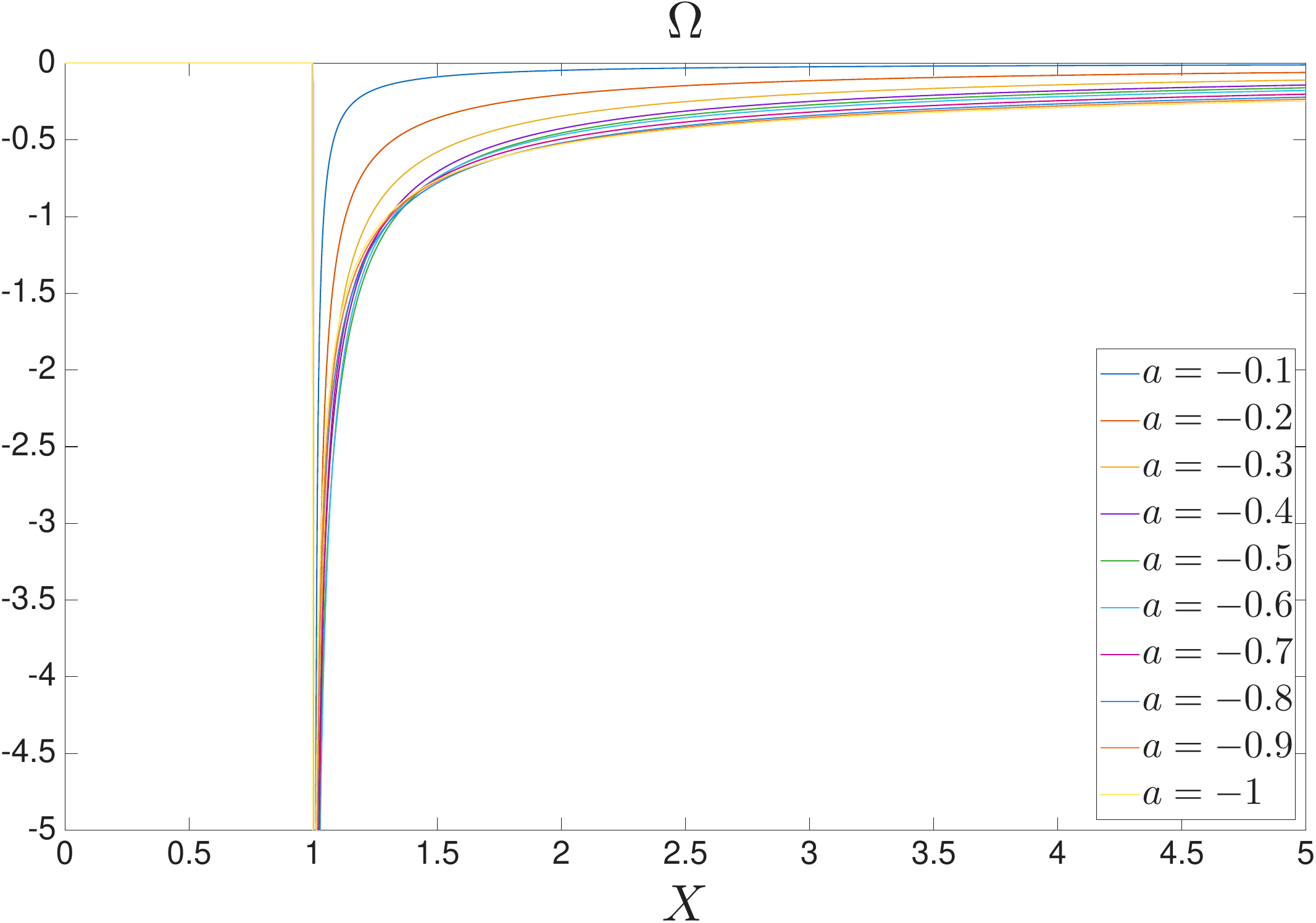}
        \caption{Case \nameref{cas:initialdata_case1.1}}
    \end{subfigure}
    \begin{subfigure}{0.49\textwidth}
        \includegraphics[width=\textwidth]{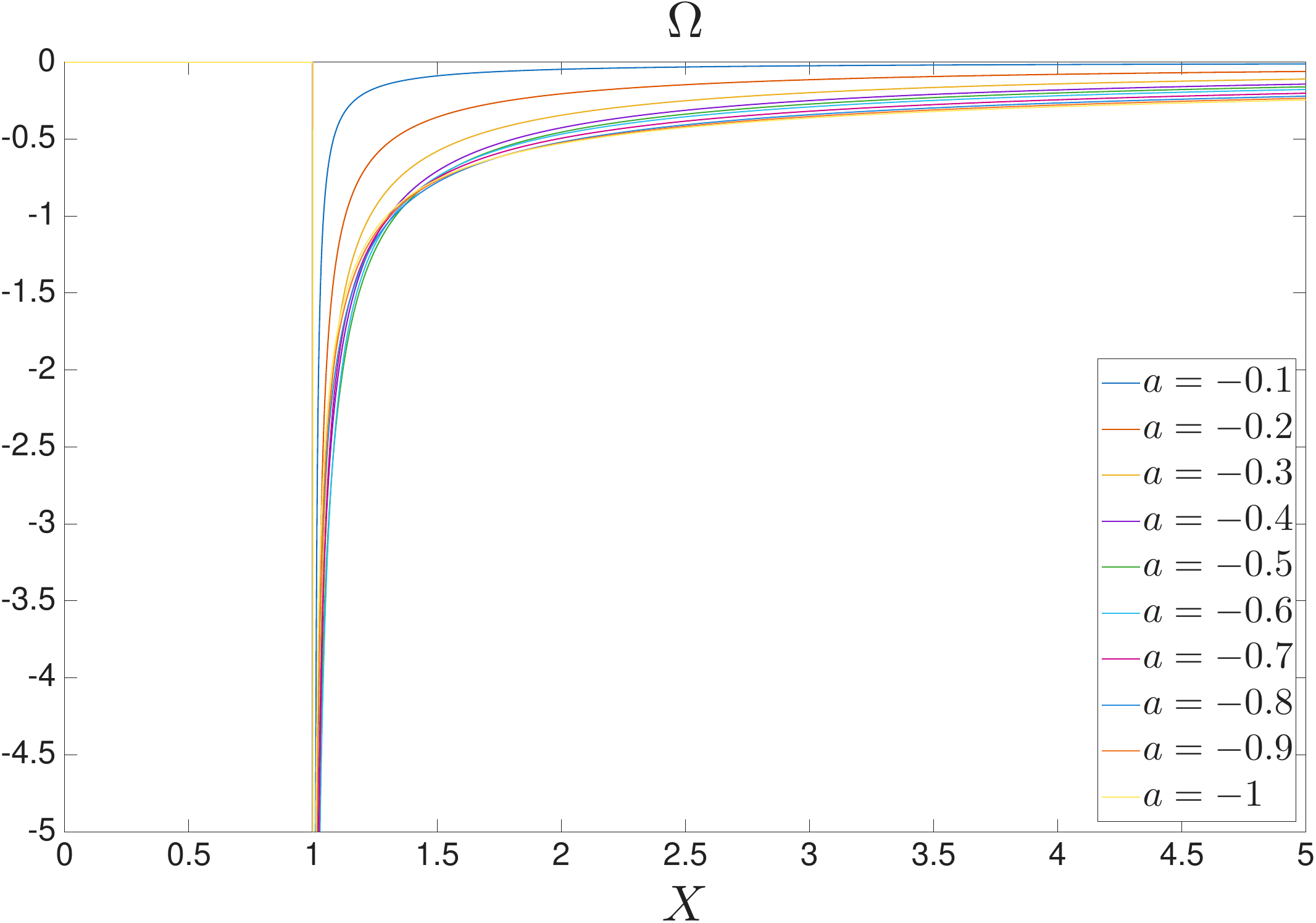}
        \caption{Case \nameref{cas:initialdata_case1.2}}
    \end{subfigure}
    \caption{Singular limiting profiles $\bar\Omega(X)$ for $-1\le a\le -0.1$ computed from odd-symmetric degenerate initial data. For each fixed $a$, the profiles obtained from Cases \nameref{cas:initialdata_case1.1} and \nameref{cas:initialdata_case1.2} coincide.}
    \label{fig:a_negative_odd}
\end{figure}

Figure \ref{fig:a_negative_odd} shows that, for each $a<0$, the solution of the dynamic rescaling equation \eqref{eqt:dynamic_rescaling} converges to a self-similar profile $\bar\Omega(X)$ that develops singularities at $X=\pm1$ and vanishes on $(-1,1)$ (note that the odd symmetry of the initial data is preserved by the evolution). In Table \ref{tab:one_scale_scaling_factors_odd_a_negative} we also observe that the spatial shrinking rate $\gamma$ increases as $a$ becomes more negative. One possible explanation is that, as $a$ becomes more negative, the advection term $a u \omega_x$ in the original physical equation \eqref{eqt:gCLM} becomes more inward-pushing, which enhances the focusing mechanism and leads to a faster spatial collapse. Moreover, for each fixed $a$, simulations in both Case \nameref{cas:initialdata_case1.1} and Case \nameref{cas:initialdata_case1.2} yield the same limiting profile and the same scaling factors up to numerical errors. This indicates that the self-similar blowup with a singular profile is robust with respect to the choice of degenerate initial data, hence suggesting stability of the observed blowup.

To conclude this subsection, we revisit and justify our choice of an initial profile with infinite vanishing order at the origin in Case \nameref{cas:initialdata_case1.1}. This choice is motivated by the following numerical observation for $a<0$: under our normalization conditions, and provided the initial data vanish to at least third order at $X=0$, all derivatives at the origin tend to zero,
\[
    \partial_X^j\Omega(0,\tau)\to 0\quad\text{as }\tau\to+\infty,\ \ \text{for every }j\ge 0.
\]
In particular, the limiting self-similar profile $\bar \Omega$ appears to have infinite vanishing order at $X=0$, regardless of the (finite) vanishing order of the initial data. To illustrate this, Figure \ref{fig:different_vanishingorders} shows the profile evolution for $a=-1$ using the family of initial profiles
\[
    \Omega_0(X)=-\frac{3X^k}{3+kX^{k+3}},\quad k=3,5,7,9.
\]
Despite the different choices of $k$, the corresponding solutions converge to the same limiting self-similar profile, and this limiting profile exhibits infinite vanishing order at the origin. We observe the same behavior for other negative values of $a$ in our simulations.

\begin{figure}[htbp]
    \begin{subfigure}{0.49\textwidth}
        \includegraphics[width=\textwidth]{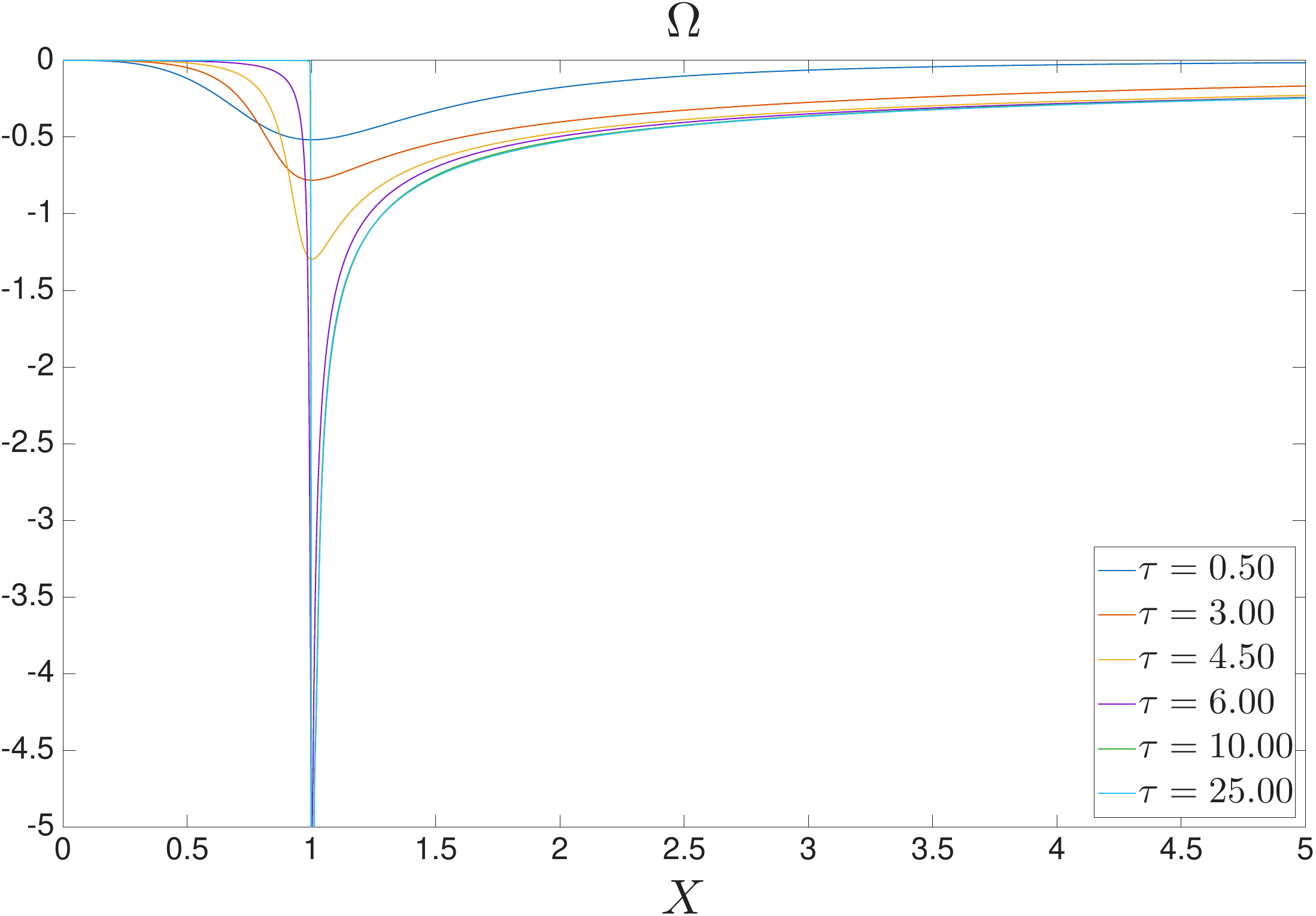}
        \caption{$k=3$}
    \end{subfigure}
    \begin{subfigure}{0.49\textwidth}
        \includegraphics[width=\textwidth]{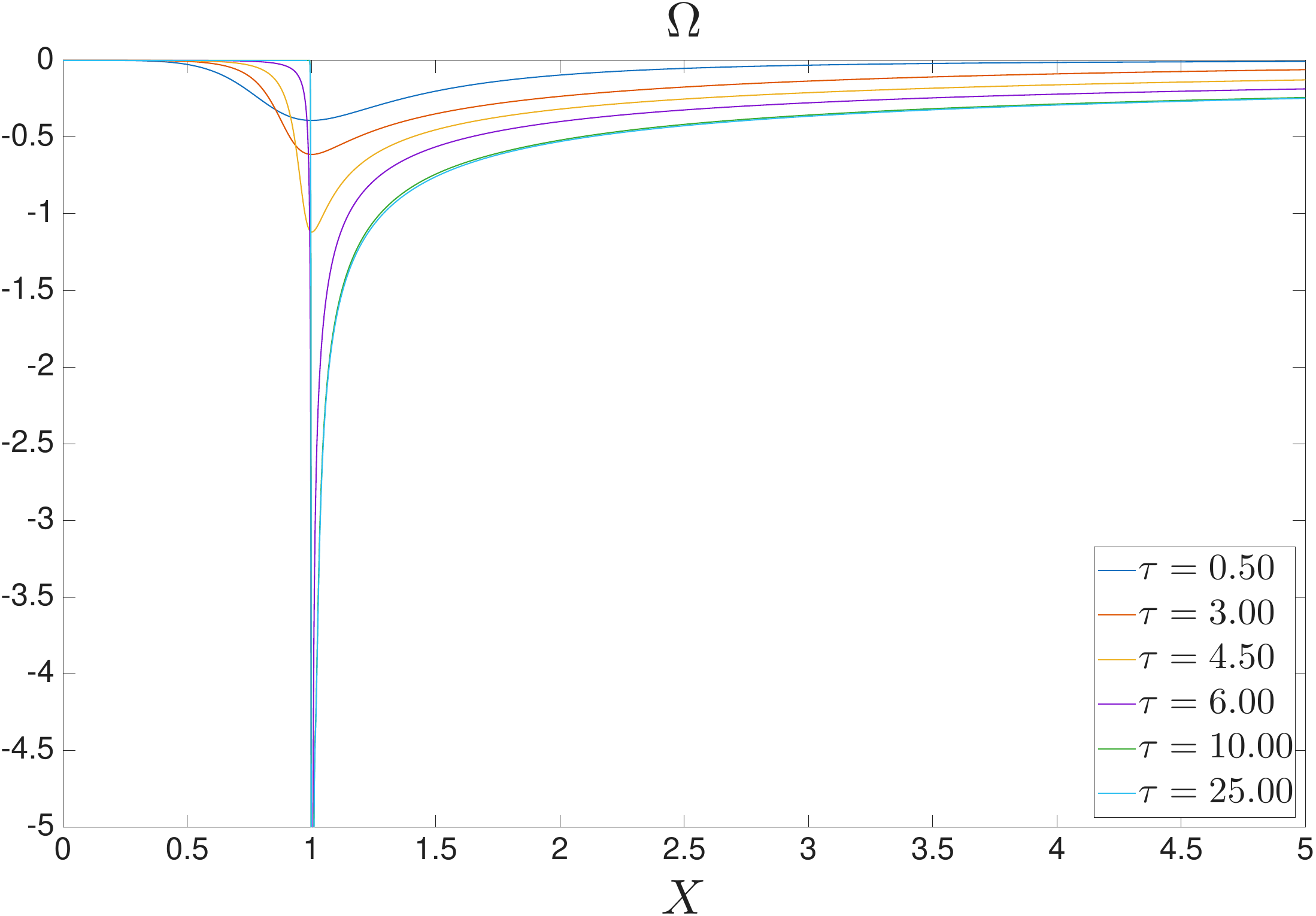}
        \caption{$k=5$}
    \end{subfigure}
    \begin{subfigure}{0.49\textwidth}
        \includegraphics[width=\textwidth]{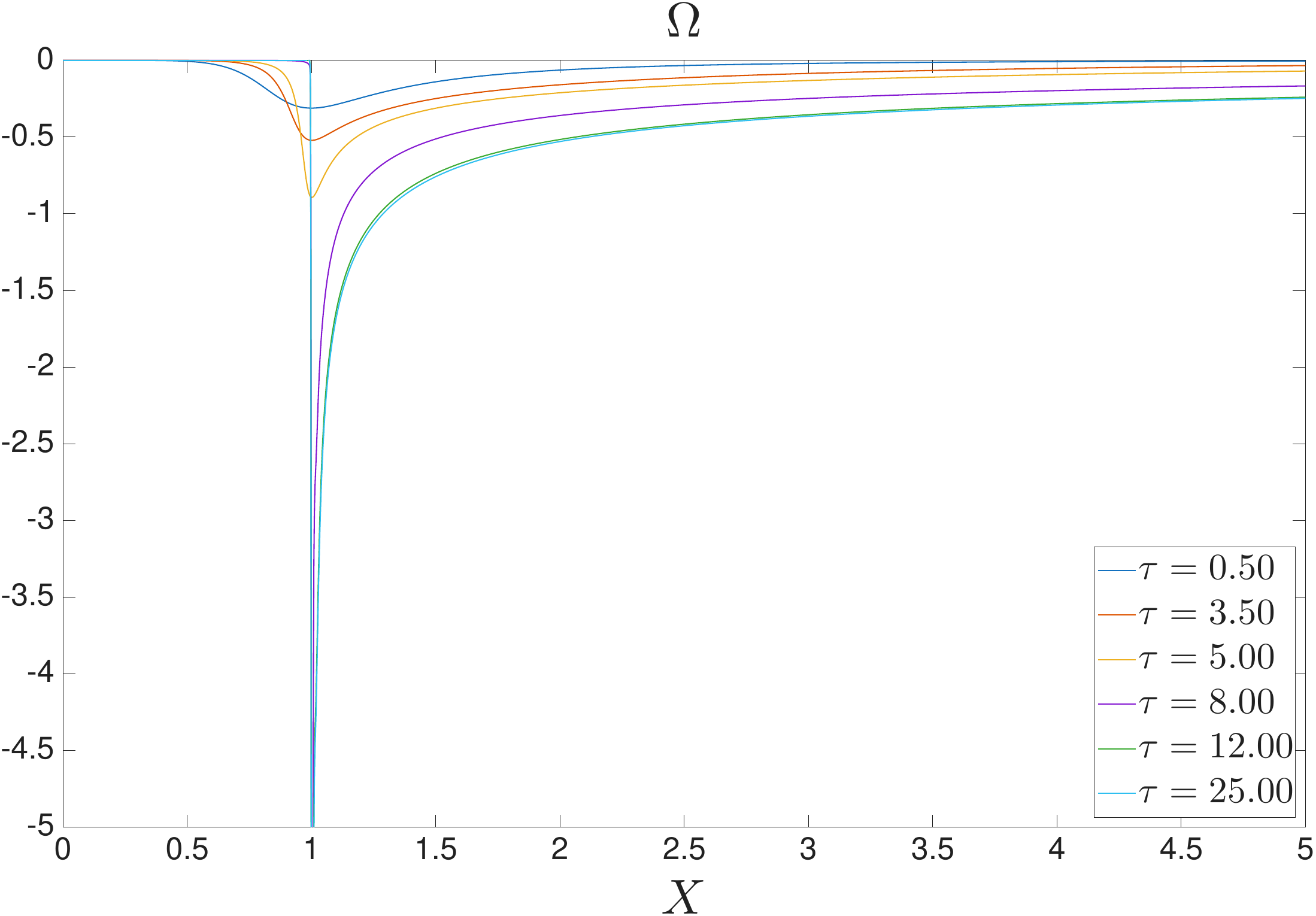}
        \caption{$k=7$}
    \end{subfigure}
    \begin{subfigure}{0.49\textwidth}
        \includegraphics[width=\textwidth]{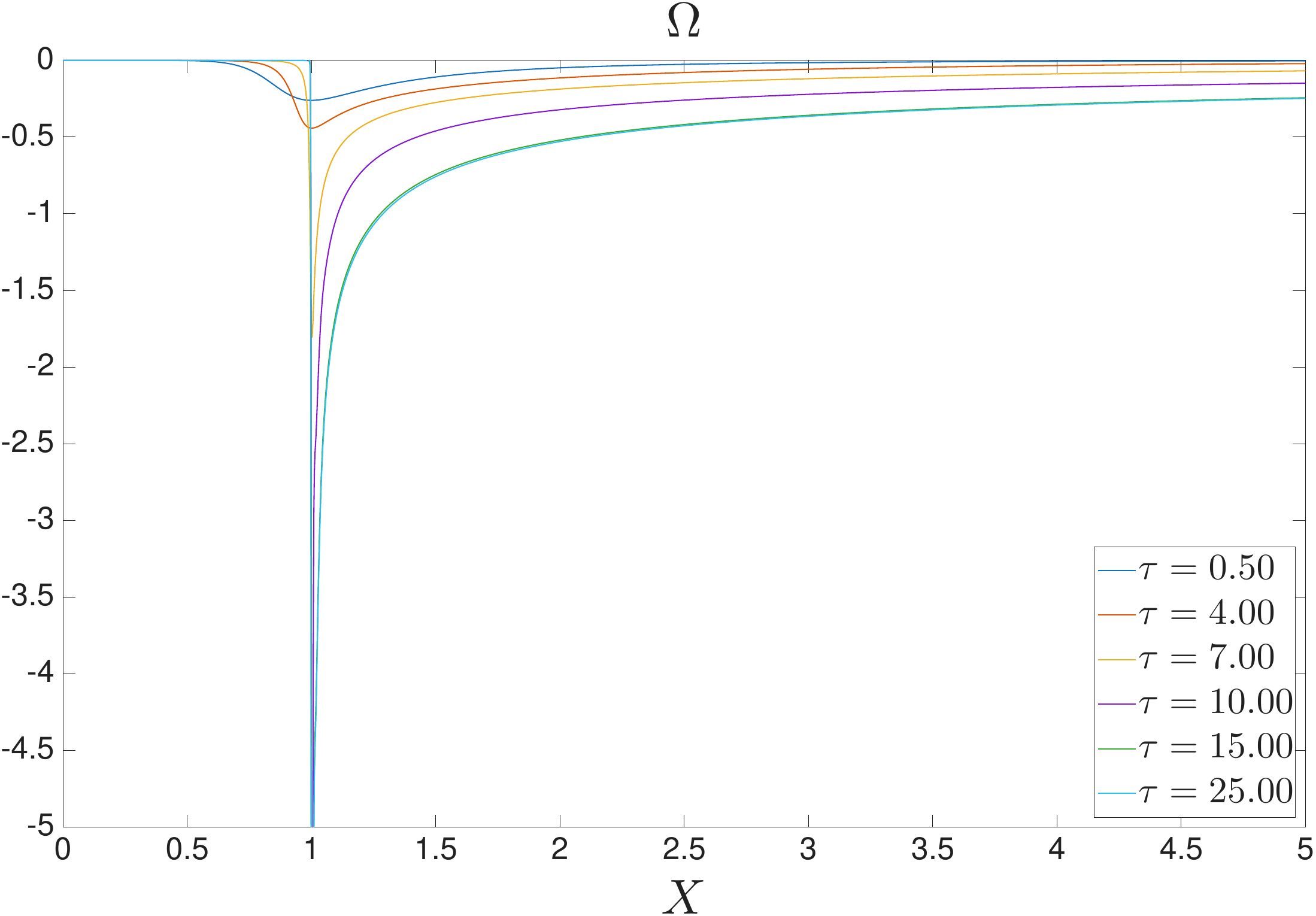}
        \caption{$k=9$}
    \end{subfigure}
    \caption{Evolution of the profile for $a=-1$ with different vanishing orders $k=3,5,7,9$ of the odd initial data. All the solutions converge to the same self-similar profile with infinite vanishing order at the origin.}
    \label{fig:different_vanishingorders}
\end{figure}

\subsection{Numerical results for different $a<0$ with degenerate initial data supported inside $\mathbb R_+$} 

In this subsection, we report numerical results from dynamic rescaling simulations for $a<0$ with half-line degenerate initial data supported inside $\mathbb R_+$ (under Assumption \ref{ass:half_line}), corresponding to Cases \nameref{cas:initialdata_case2.1} and \nameref{cas:initialdata_case2.2}. For Case \nameref{cas:initialdata_case2.1}, we impose the same normalization conditions \eqref{eqt:normalization_condition_case1.1} as in Case \nameref{cas:initialdata_case1.1}, while for Case \nameref{cas:initialdata_case2.2} we use the same normalization conditions \eqref{eqt:normalization_condition_case1.2} as in Case \nameref{cas:initialdata_case1.2}. Figures \ref{fig:a=-1_halfline_evolution} and \ref{fig:a=-1_halfline_compact_evolution} show, for $a=-1$, the evolution of the rescaled profile $\Omega(X,\tau)$ together with the associated advection speed $c_l(\tau)X+aU(X,\tau)$. The dynamics are qualitatively analogous to those observed in the corresponding odd-symmetric cases.

\begin{figure}[htbp]
    \begin{subfigure}{0.49\textwidth}
        \includegraphics[width=\textwidth]{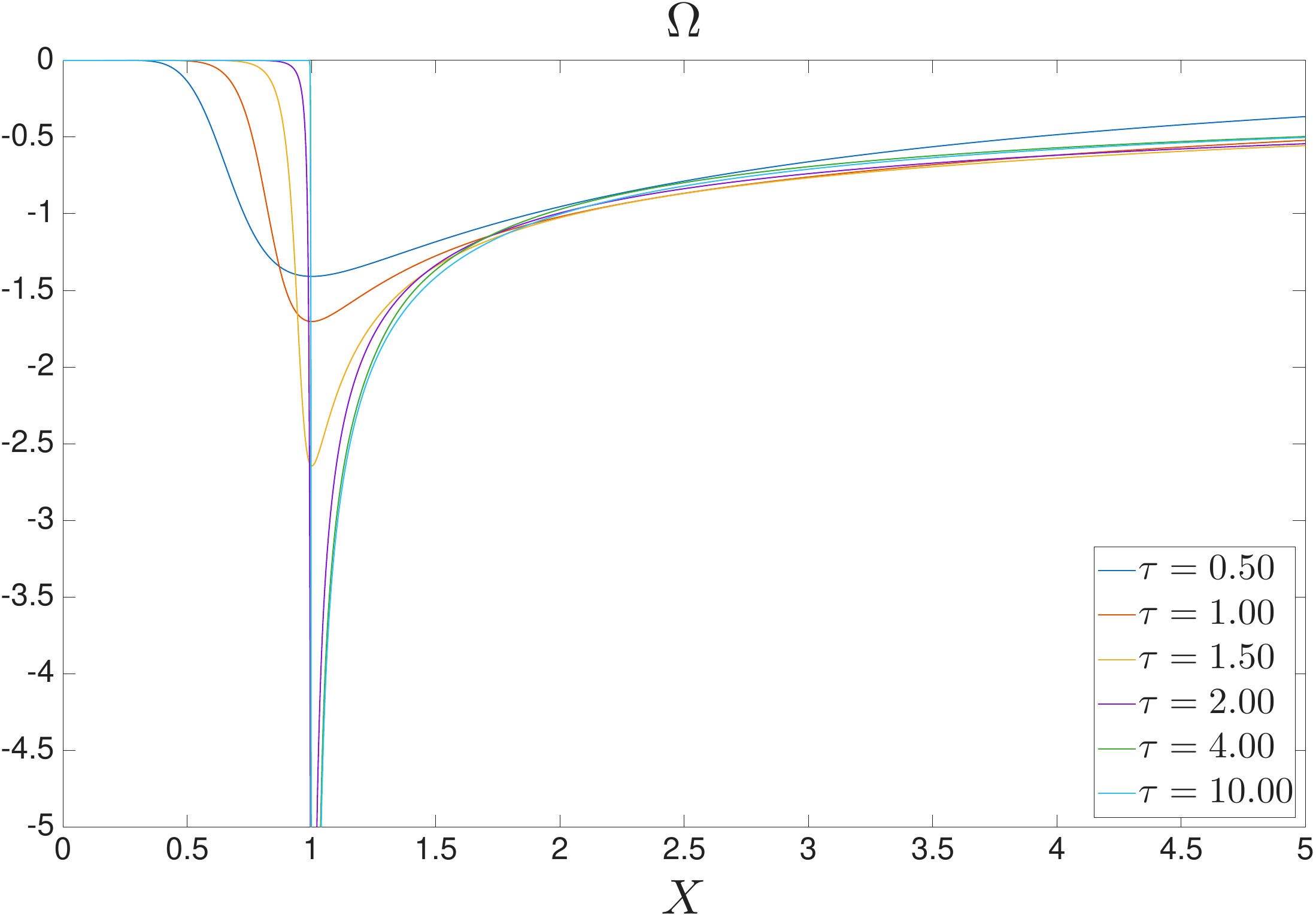}
        \caption{Evolution of the profile $\Omega(X,\tau)$}
    \end{subfigure}
    \begin{subfigure}{0.49\textwidth}
        \includegraphics[width=\textwidth]{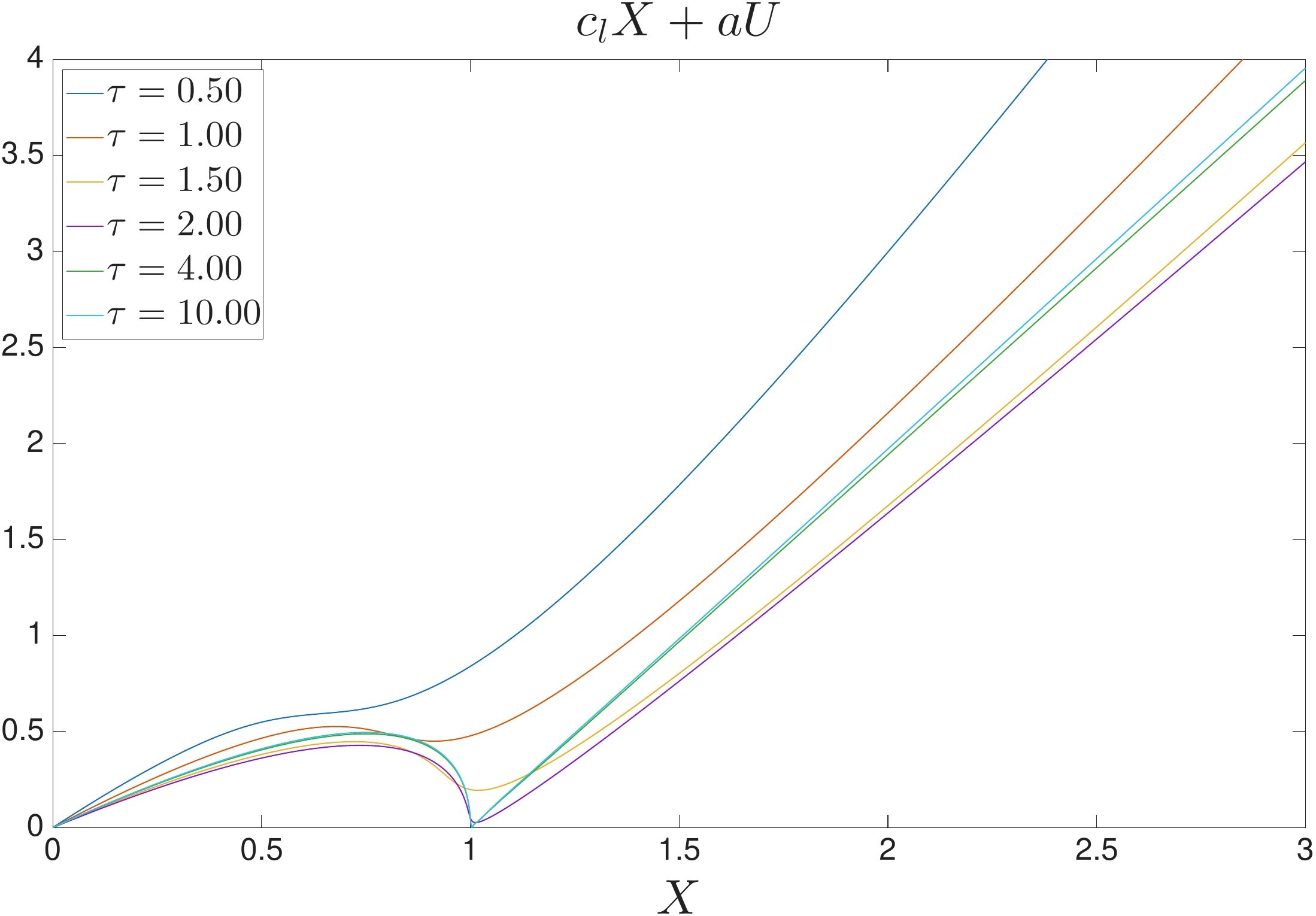}
        \caption{Evolution of the advection speed $c_l(\tau)X+aU(X,\tau)$}
    \end{subfigure}
    \caption{Dynamic rescaling simulation for $a=-1$ with degenerate initial data supported on the positive half-line $\mathbb{R}_+$ (Case \nameref{cas:initialdata_case2.1}). We impose the normalization conditions $\partial_X\Omega(1,\tau)\equiv 0$ (pinning the minimum at $X=1$) and $\mtx H(\Omega)(0,\tau)\equiv \mtx H(\Omega_0)(0)$.}
    \label{fig:a=-1_halfline_evolution}
\end{figure}

\begin{figure}[htbp]
    \begin{subfigure}{0.49\textwidth}
        \includegraphics[width=\textwidth]{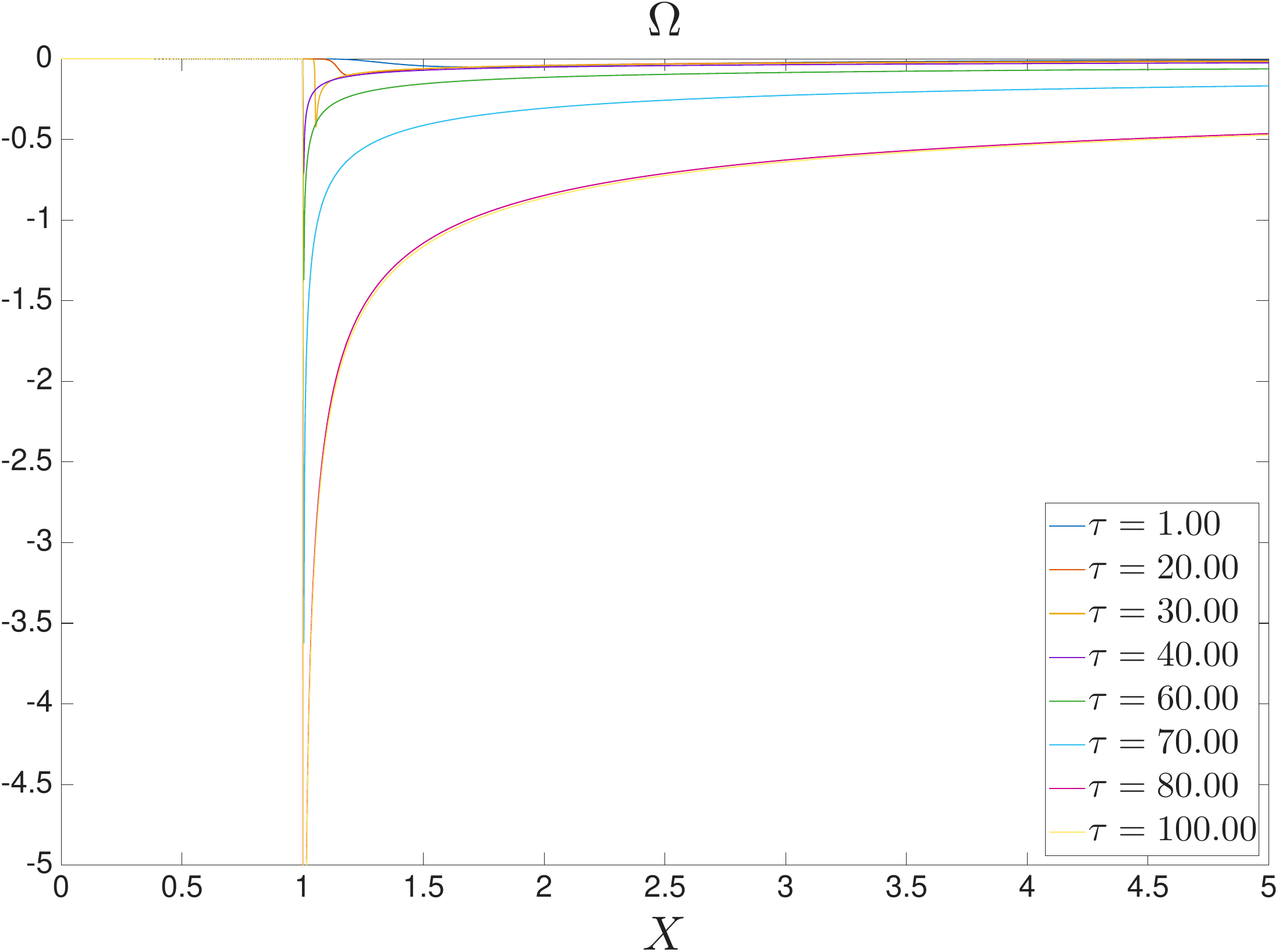}
        \caption{Evolution of the profile $\Omega(X,\tau)$}
    \end{subfigure}
    \begin{subfigure}{0.49\textwidth}
        \includegraphics[width=\textwidth]{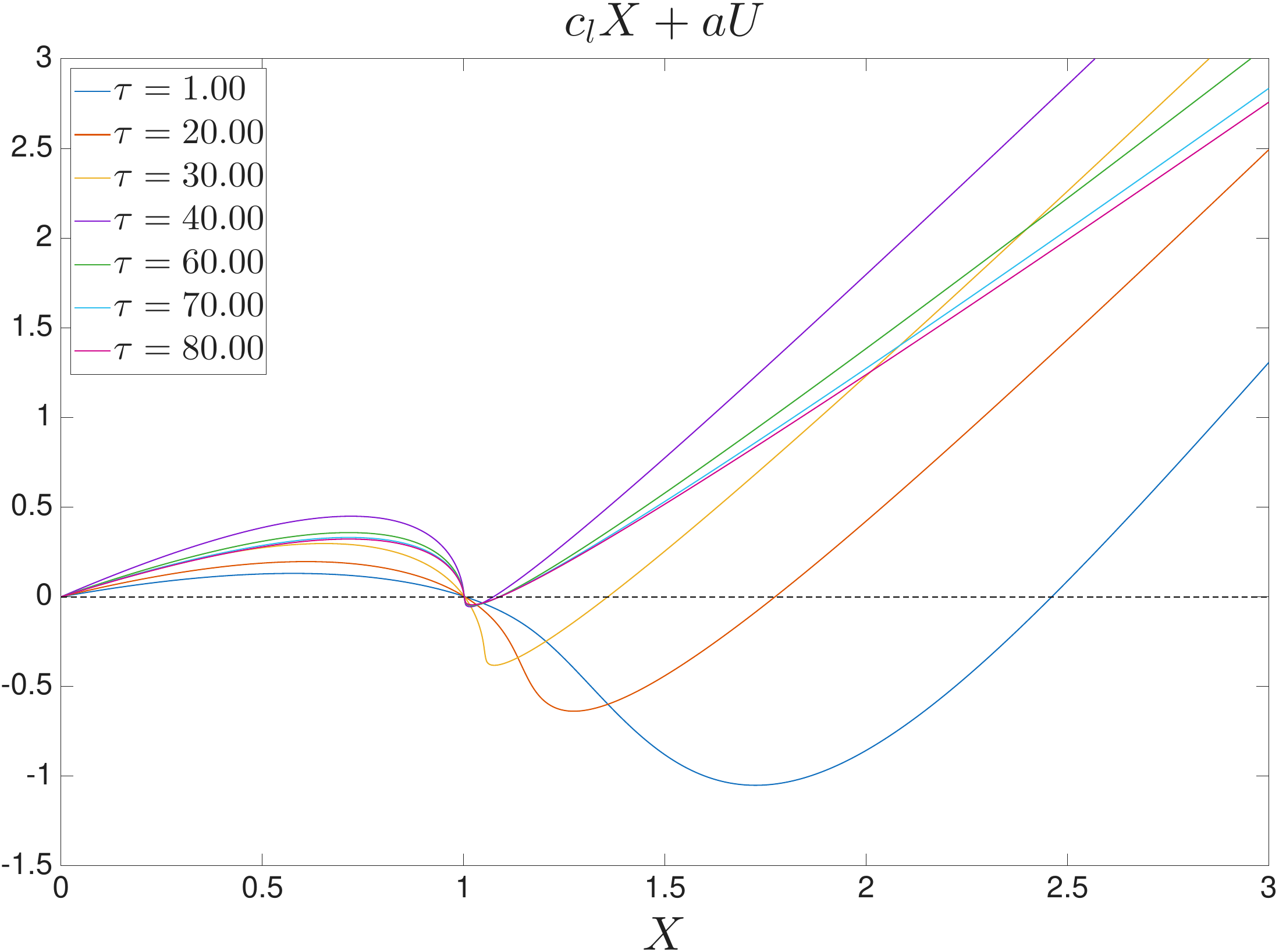}
        \caption{Evolution of the advection speed $c_l(\tau)X+aU(X,\tau)$}
    \end{subfigure}
    \caption{Dynamic rescaling simulation for $a=-1$ with degenerate initial data supported on $[1,+\infty)$ (Case \nameref{cas:initialdata_case2.2}). We impose the normalization conditions $c_l(\tau)+aU(1,\tau)\equiv 0$ (pinning the transition point at $X=1$) and $\mtx H(\Omega)(0,\tau)\equiv \mtx H(\Omega_0)(0)$.}
    \label{fig:a=-1_halfline_compact_evolution}
\end{figure}

Recall that for the half-line case, there is an explicit singular solution to the self-similar profile equation \eqref{eqt:steady_profile} given by Theorem \ref{thm:explicit_singular_solution}. We briefly outline the derivation of this explicit solution here. For $\mu\in(0,1)$ and $\theta\in[0,2\pi)$, define a family of functions
\[
    \Phi_{\mu,\theta}(X):=\frac{\boldsymbol{1}_{\{X>0\}}(X)\cos \theta+\boldsymbol{1}_{\{X<0\}}(X)\cos(\theta-\pi \mu)}{|X|^{\mu}}=\begin{cases}
        |X|^{-\mu}\cos \theta         & X>0, \\
        |X|^{-\mu}\cos(\theta-\pi \mu), & X<0. \\
    \end{cases}
\]
One can calculate that
\[
    \mtx H( \Phi_{\mu,\theta} )(X)=\frac{\boldsymbol{1}_{\{X>0\}}(X)\sin\theta+\boldsymbol{1}_{\{X<0\}}(X)\sin (\theta-\pi \mu)}{|X|^{\mu}}=\begin{cases}
        |X|^{-\mu}\sin \theta         & X>0, \\
        |X|^{-\mu}\sin(\theta-\pi \mu), & X<0. \\
    \end{cases}
\]
In fact, one can check that
\[
    \Phi_{\mu,\theta}(X)+\mathrm{i}\mtx H( \Phi_{\mu,\theta} )(X)=\phi(X),
\]
where
\[
    \phi(z)=\frac{\mathrm{e}^{\mathrm{i}\theta}}{z^{\mu}},
\]
which is holomorphic in the interior of the upper half-plane. Now, given any $a<0$, we fix $\mu=1/(1-a)$ and $\theta=\pi/2$. Let
\begin{equation}\label{eqt:singular_solution_halfline_a_le0}
    \bar \Omega_a(X)=-\Phi_{\mu,\theta}(1-X)=-\frac{\boldsymbol{1}_{\{X>1\}}\sin(\pi \mu)}{|1-X|^{\mu}}.
\end{equation}
Then the Hilbert transform of $\bar \Omega_a$ is given by
\[
    \mtx H(\bar \Omega_a)=\frac{\boldsymbol{1}_{\{X<1\}}+\boldsymbol{1}_{\{X>1\}}\cos(\pi \mu)}{|1-X|^{\mu}}.
\]
One can verify that if we set
\begin{equation}\label{eqt:singular_solution_halfline_a_le0_scaling_factors}
    \bar c_{l,a} = 1-a,\quad \bar c_{\omega,a} = -1,
\end{equation}
then $(\bar \Omega_a,\bar c_{l,a},\bar c_{\omega,a})$ solves the steady self-similar profile equation \eqref{eqt:steady_profile} with $\bar \gamma_a=-\bar c_{l,a}/\bar c_{\omega,a}=1-a$. See Appendix \ref{app:explicit_singular_solution} for a detailed verification. This gives an explicit singular solution to the self-similar profile equation for each $a<0$. In what follows, we show numerically that, for $a<0$, the dynamic rescaling equation with half-line degenerate initial data (Cases \nameref{cas:initialdata_case2.1} and \nameref{cas:initialdata_case2.2}) converges to the explicit singular profile \eqref{eqt:singular_solution_halfline_a_le0} under proper normalization.

As an example, Figure \ref{fig:a=-1_halfline_convergence} illustrates this convergence for $a=-1$ in the half-line setting (Case \nameref{cas:initialdata_case2.1}). As $\tau$ increases, the profile $\Omega(X,\tau)$ converges to the explicit singular profile $\bar\Omega_a$ in \eqref{eqt:singular_solution_halfline_a_le0} in both the near field and the far field. At the same time, the scaling factors $c_l(\tau)$ and $c_\omega(\tau)$ stabilize and converge to the predicted (in \eqref{eqt:singular_solution_halfline_a_le0_scaling_factors}) steady-state values $\bar c_{l,a}=1-a = 2$ and $\bar c_{\omega,a}=-1$, so that the shrinking rate $\gamma(\tau)=-c_l(\tau)/c_\omega(\tau)$ converges to $\bar\gamma=1-a=2$. Together, these observations provide numerical evidence that the dynamic rescaling evolution converges to the explicit steady state and thus corresponds to an asymptotically self-similar focusing finite-time blowup with a singular limiting profile. Note that the scalings of the limiting self-similar profile $\bar\Omega$ (in view of \eqref{eqt:scaling_invariance}) depend on how the initial profile $\Omega_0$ is normalized, while the shape of the limiting profile and the limiting shrinking rate $\bar \gamma$ are uniquely given by \eqref{eqt:singular_solution_halfline_a_le0} and \eqref{eqt:singular_solution_halfline_a_le0_scaling_factors} respectively in all our numerical tests.

\begin{figure}[htbp]
    \begin{subfigure}{0.49\textwidth}
        \includegraphics[width=\textwidth]{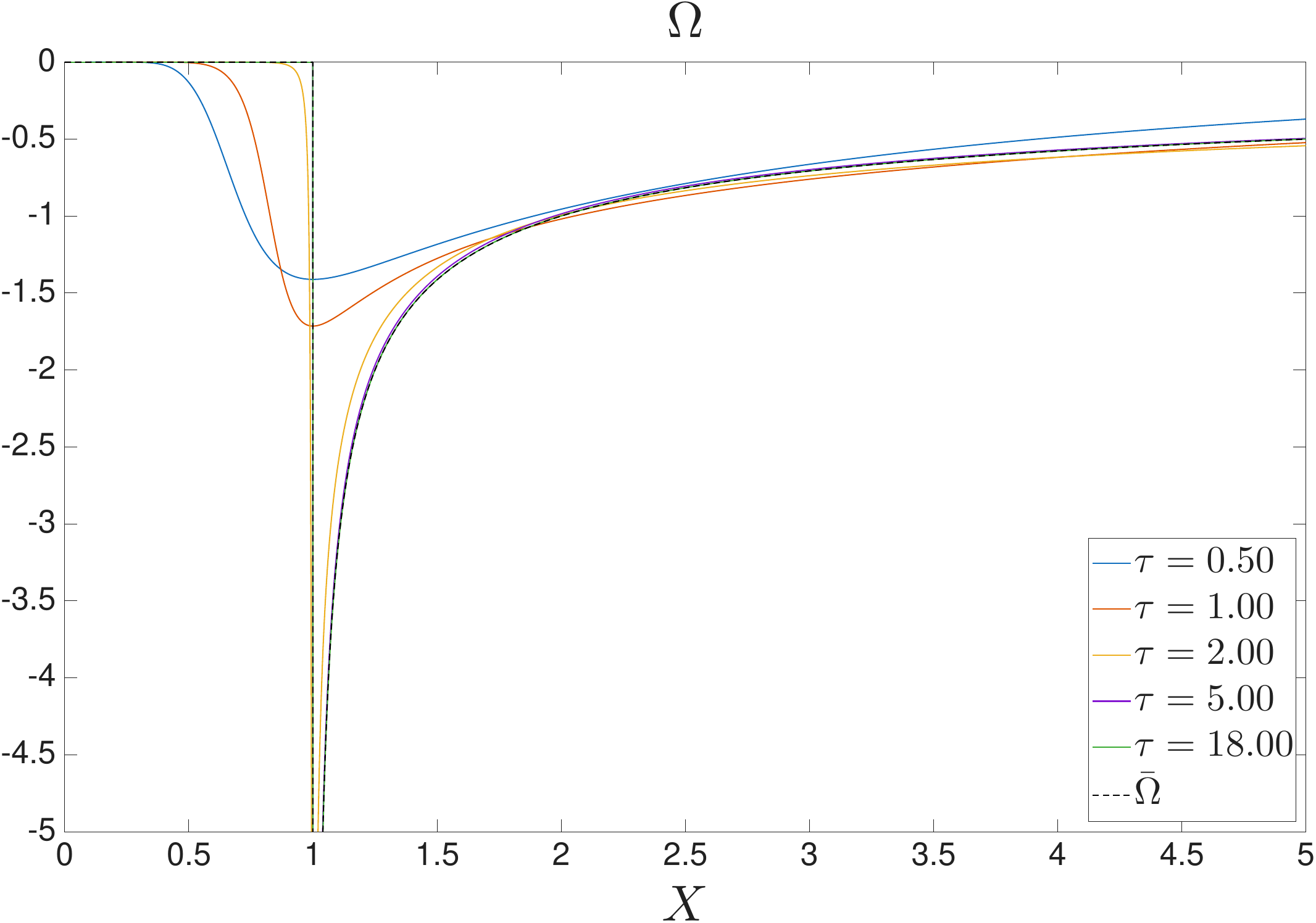}
        \caption{Profile $\Omega(X,\tau)$ in the near field}
    \end{subfigure}
    \begin{subfigure}{0.49\textwidth}
        \includegraphics[width=\textwidth]{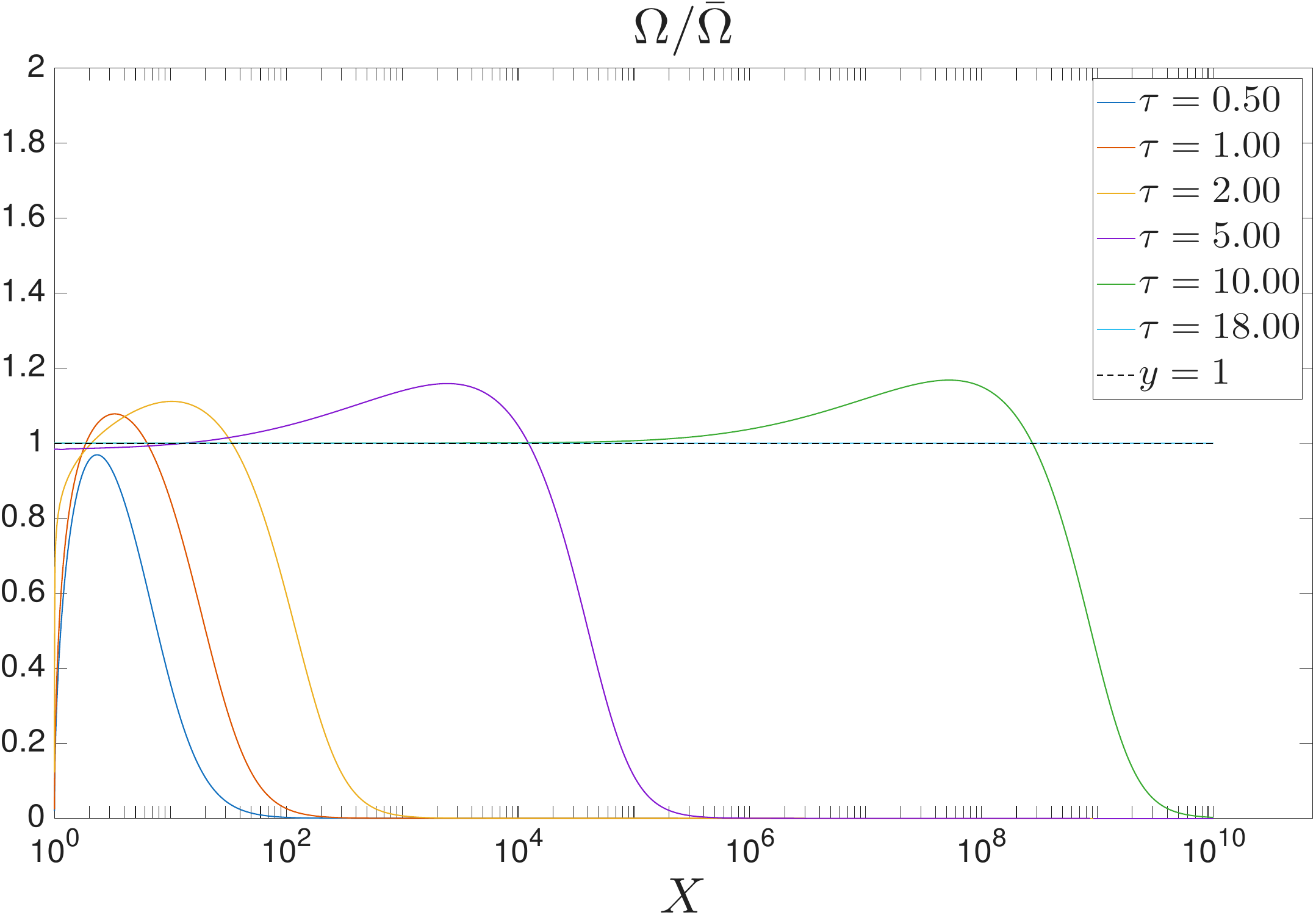}
        \caption{Profile $\Omega(X,\tau)$ in the far field}
    \end{subfigure}
    \begin{subfigure}{0.32\textwidth}
        \includegraphics[width=\textwidth]{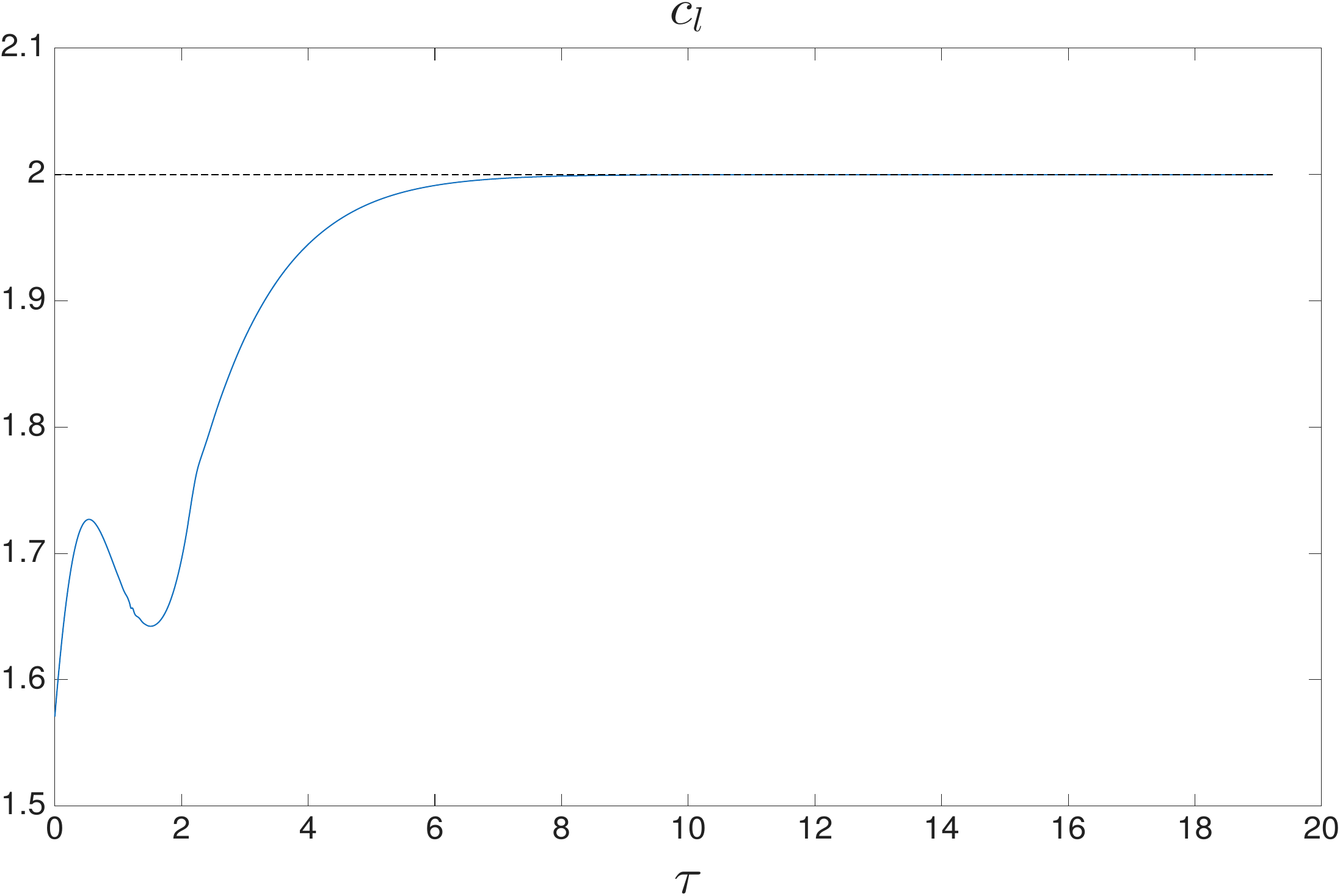}
        \caption{Scaling factor $c_l$}
    \end{subfigure}
    \begin{subfigure}{0.32\textwidth}
        \includegraphics[width=\textwidth]{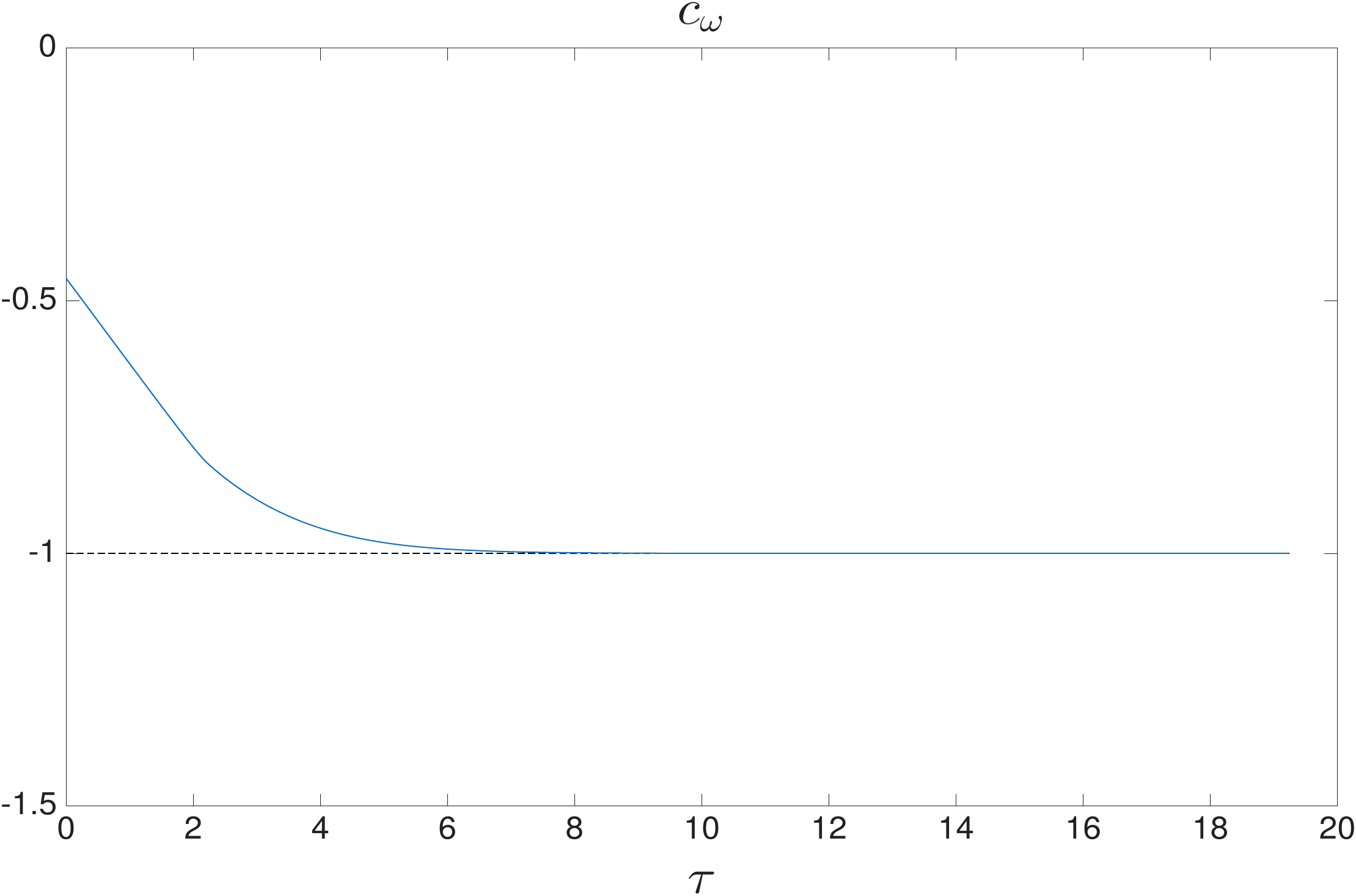}
        \caption{Scaling factor $c_\omega$}
    \end{subfigure}
    \begin{subfigure}{0.32\textwidth}
        \includegraphics[width=\textwidth]{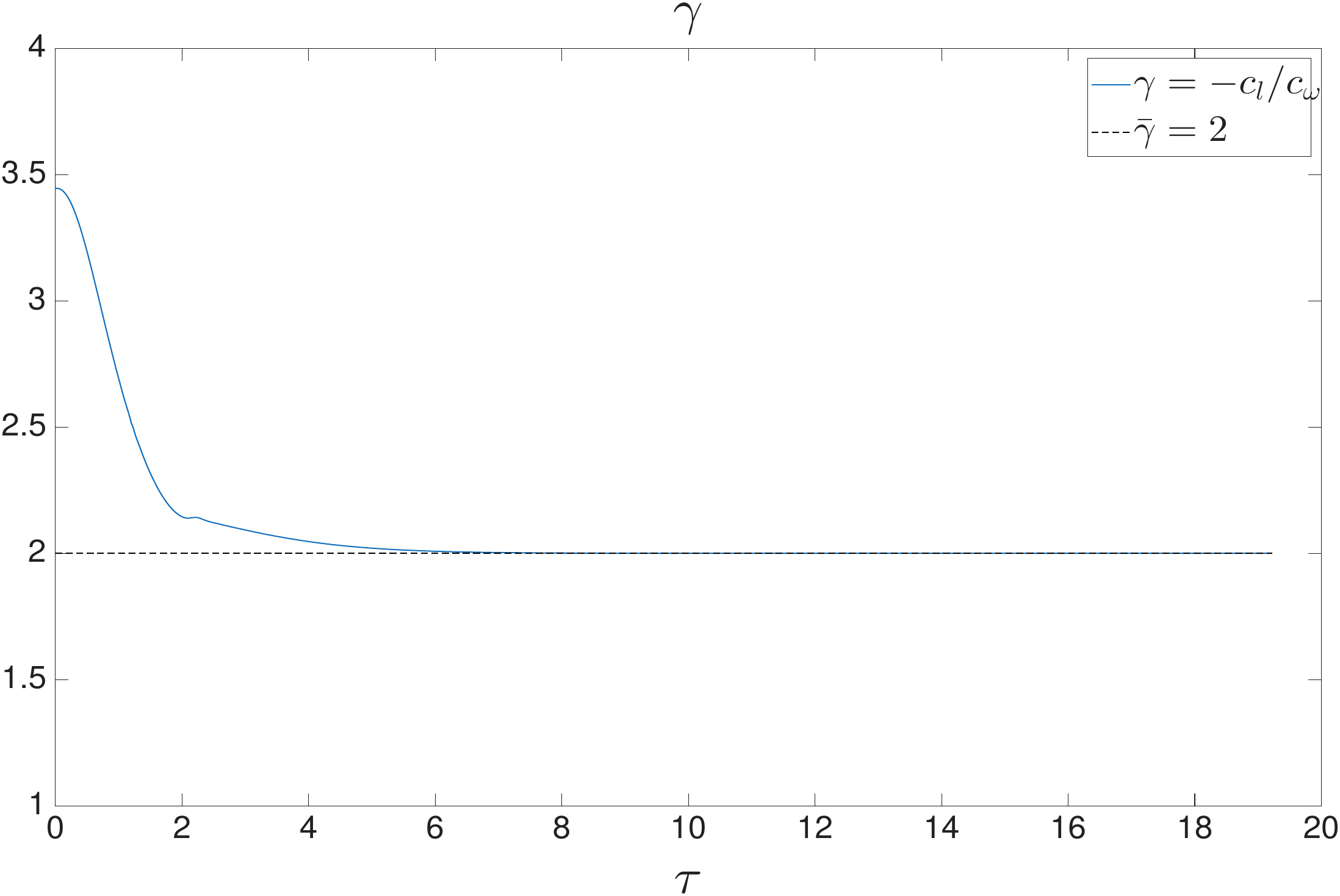}
        \caption{Spatial shrinking rate $\gamma=-c_l/c_\omega$}
    \end{subfigure}
    \caption{Convergence of the dynamic rescaling simulation for $a=-1$ with degenerate initial data supported on the positive half-line $\mathbb{R}_+$ (Case \nameref{cas:initialdata_case2.1}). As $\tau$ increases, the profile $\Omega(X,\tau)$ converges to the explicit singular steady profile \eqref{eqt:singular_solution_halfline_a_le0} with $a=-1$, namely $\bar\Omega(X)=-\boldsymbol{1}_{\{X>1\}}|1-X|^{-1/2}$, in both the near field and the far field. Meanwhile, the scaling factors $c_l(\tau)$ and $c_\omega(\tau)$ (and hence the shrinking rate $\gamma(\tau)=-c_l(\tau)/c_\omega(\tau)$) converge to the predicted values $\bar c_l=1-a=2$, $\bar c_\omega=-1$, and $\bar\gamma=1-a=2$ given by \eqref{eqt:singular_solution_halfline_a_le0_scaling_factors} as $\tau\to+\infty$.}
    \label{fig:a=-1_halfline_convergence}
\end{figure}

For other negative values of $a$, we observe similar convergence of the dynamic rescaling simulation to the explicit singular profile \eqref{eqt:singular_solution_halfline_a_le0} with the predicted scaling factors \eqref{eqt:singular_solution_halfline_a_le0_scaling_factors}. We declare convergence when the residual of the profile equation, restricted away from the singularity at $X=1$, drops below $10^{-8}$, i.e.
\[
    \bigl\|\Omega_\tau\,\boldsymbol{1}_{\{|X-1|>0.1\}}\bigr\|_{L^\infty}<10^{-8}.
\]
The limiting scaling factors for $-1\le a\le -0.1$ in Case \nameref{cas:initialdata_case2.1} are reported in Table \ref{tab:one_scale_scaling_factors_halfline_a_negative}; the results in Case \nameref{cas:initialdata_case2.2} are nearly identical and are therefore not tabulated here. The corresponding limiting self-similar profiles are shown in Figure \ref{fig:a_negative_halfline}.

\begin{table}[htbp]
    \begin{tabular}{|c|c|c|c|c|c|}
        \hline
        $a$              & $-0.1$    & $-0.2$    & $-0.3$    & $-0.4$    & $-0.5$    \\ \hline
        $c_l$            & $1.1535$  & $1.3350$  & $1.2726$  & $1.3846$  & $1.5148$  \\ \hline
        $c_\omega$       & $-1.0588$ & $-1.0883$ & $-0.9654$ & $-0.9959$ & $-1.0142$ \\ \hline
        $\gamma$         & $1.0895$  & $1.2266$  & $1.3182$  & $1.3903$  & $1.4936$  \\ \hline
        $\bar\gamma=1-a$ & $1.1$     & $1.2$     & $1.3$     & $1.4$     & $1.5$     \\ \hline\hline
        $a$              & $-0.6$    & $-0.7$    & $-0.8$    & $-0.9$    & $-1$      \\ \hline
        $c_l$            & $1.5997$  & $1.6998$  & $1.7961$  & $1.9009$  & $2.0000$  \\ \hline
        $c_\omega$       & $-0.9957$ & $-1.0026$ & $-0.9974$ & $-1.0006$ & $-1.0000$ \\ \hline
        $\gamma$         & $1.6066$  & $1.6954$  & $1.8008$  & $1.8998$  & $2.0000$  \\ \hline
        $\bar\gamma=1-a$ & $1.6$     & $1.7$     & $1.8$     & $1.9$     & $2$       \\ \hline
    \end{tabular}
    \caption{Scaling factors for $-1\le a\le -0.1$ obtained from half-line degenerate initial data (Case \nameref{cas:initialdata_case2.1}). The corresponding values for Case \nameref{cas:initialdata_case2.2} are nearly identical. For each $a$, the converged shrinking rate $\gamma$ agrees closely with the theoretical prediction $\bar\gamma=1-a$, supporting that the limiting profile is given by the explicit singular solution \eqref{eqt:singular_solution_halfline_a_le0} with the scaling factors \eqref{eqt:singular_solution_halfline_a_le0_scaling_factors}.}
    \label{tab:one_scale_scaling_factors_halfline_a_negative}
\end{table}

\begin{figure}[htbp]
    \begin{subfigure}{0.49\textwidth}
        \includegraphics[width=\textwidth]{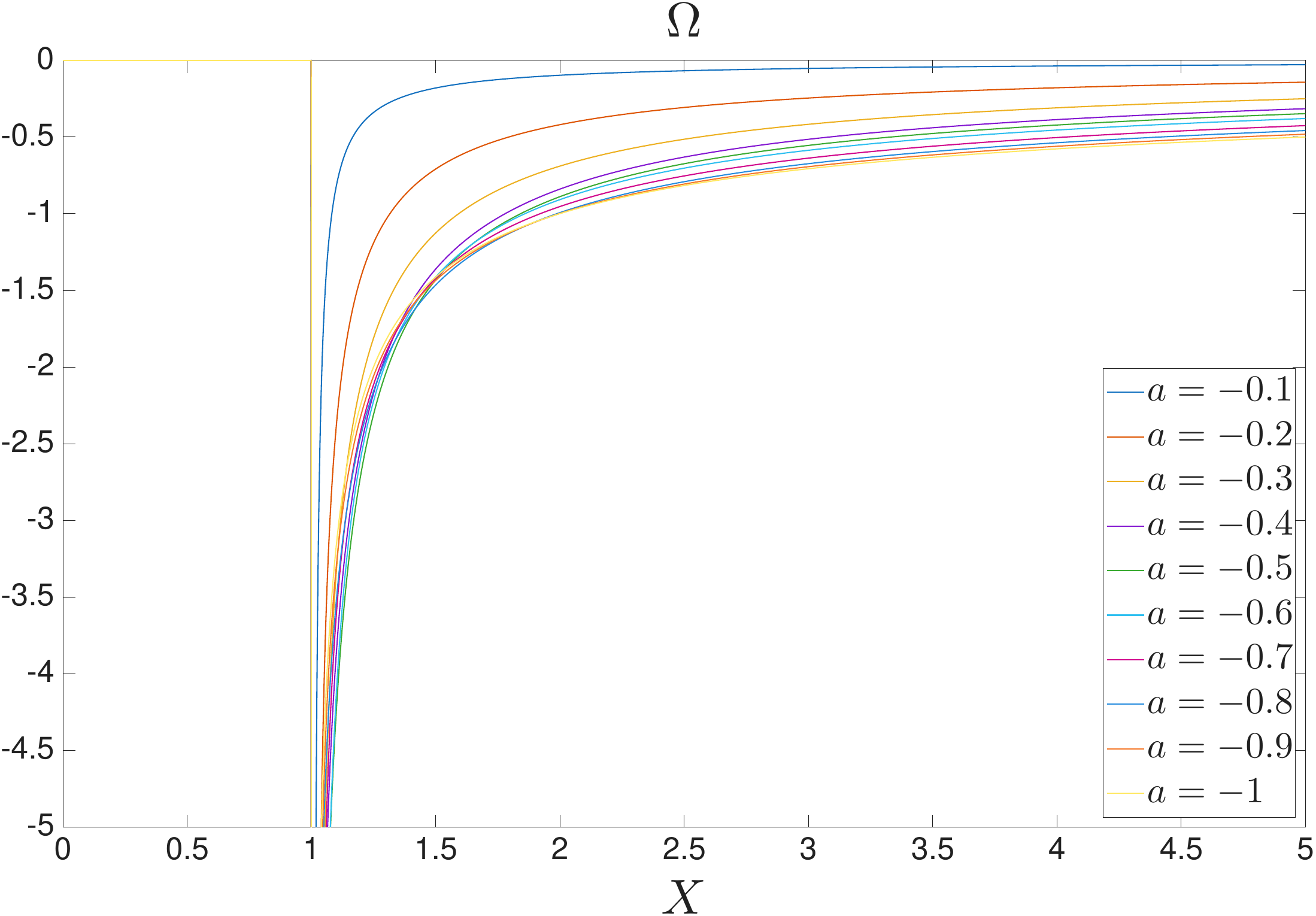}
        \caption{Case \nameref{cas:initialdata_case2.1}}
    \end{subfigure}
    \begin{subfigure}{0.49\textwidth}
        \includegraphics[width=\textwidth]{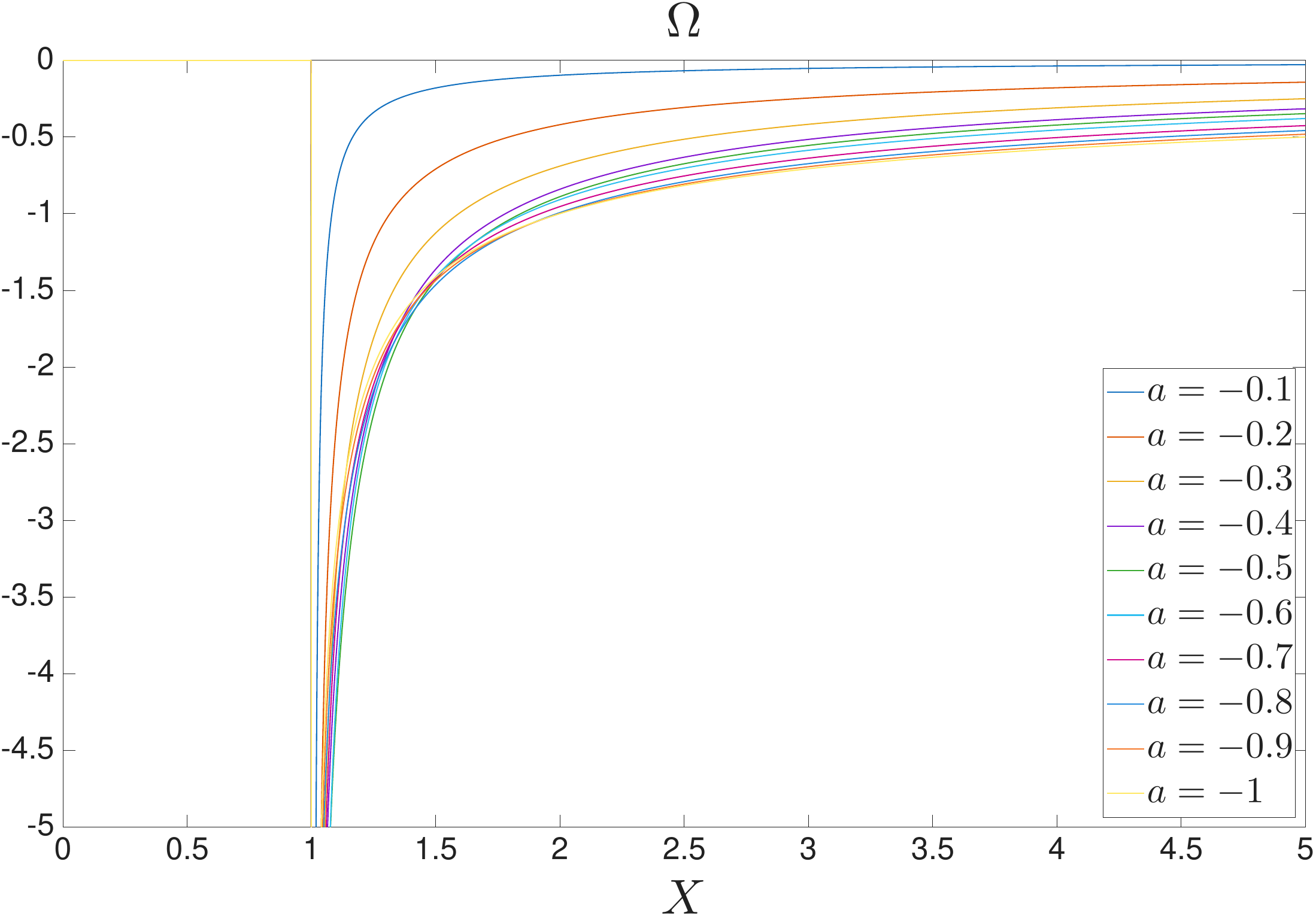}
        \caption{Case \nameref{cas:initialdata_case2.2}}
    \end{subfigure}
    \caption{Singular limiting profiles $\bar\Omega(X)$ for $-1\le a\le -0.1$ computed from half-line initial data. For each fixed $a$, the profiles obtained from Cases \nameref{cas:initialdata_case2.1} and \nameref{cas:initialdata_case2.2} coincide.}
    \label{fig:a_negative_halfline}
\end{figure}

Figure \ref{fig:a_negative_halfline} suggests that, for each $a<0$, the solution of the dynamic rescaling equation \eqref{eqt:dynamic_rescaling} converges to a limiting self-similar profile $\bar\Omega(X)$ that is singular at $X=1$ and vanishes on $[0,1)$. As in the odd-symmetric case, the spatial shrinking rate $\gamma$ also increases as $a$ decreases. Furthermore, the values of $\gamma$ measured from the simulations are in close agreement with the theoretical prediction $\bar\gamma=1-a$, providing strong numerical evidence that the limiting profile is exactly the explicit singular solution given in \eqref{eqt:singular_solution_halfline_a_le0} (up to rescaling) with the scaling factors in \eqref{eqt:singular_solution_halfline_a_le0_scaling_factors}.

\section{Numerical investigation of the two-scale feature for $a<0$}\label{sec:numerical_multiscale}
In this section, we present numerical evidence of the two-scale structure in the self-similar blowups with singular outer profiles for $a<0$, as motivated in Section \ref{sec:2scale_analysis}.  Recall that the two-scale scenario is described by the ansatz
\begin{equation}\label{eqt:multiscale_blowup_recall}
    \omega(x,t)=(T-t)^{\hat \lambda}\left(\widehat \Omega\left( \frac{x-r(t)(T-t)^{\gamma}}{(T-t)^{\hat\gamma}}\right) +o(1)\right),
\end{equation}
where $\widehat \Omega$ denotes the inner profile on the finer scale $(T-t)^{\hat\gamma}$. As discussed earlier, the leading-order inner dynamics suggests that $\widehat \Omega$ should satisfy the traveling-wave profile equation
\begin{equation}\label{eqt:leading_order_recall}
    (a\widehat U+\hat r)\widehat \Omega_X=\widehat U_X \widehat \Omega.
\end{equation}

We proceed in three steps. First, we extract the inner profile from dynamic rescaling simulations by zooming in near the singular point of the outer profile. The numerical results indicate that this inner profile remains smooth and converges to a stable limiting shape. Second, we compare this limiting profile with an independently computed traveling-wave profile and find close agreement. This supports the formal analysis in Section \ref{sec:2scale_analysis}, which suggests that the inner profile is governed by a traveling wave solution. The traveling-wave profile is computed using an iterative method based on the existence proof in Section \ref{sec:existence_traveling_wave} and described in Appendix \ref{apx:iterative_method}. Third, we fit the scaling exponents $\hat\lambda$, $\hat\gamma$, and $\gamma$ in \eqref{eqt:multiscale_blowup_recall}. The fitted exponents satisfy $\hat\gamma>\gamma$, indicating that the inner scale collapses faster than the outer scale and providing quantitative support for the two-scale ansatz \eqref{eqt:multiscale_blowup_recall}.

\subsection{Behavior of the inner profile for $a=-0.2$}
We first show that an inner scale exists and converges to a steady profile. To resolve this inner-scale behavior, we zoom in near the point where the outer profile (obtained from the dynamic rescaling simulation with $a=-0.2$) attains its extremal magnitude. At each time, we translate this peak to the origin and rescale both its magnitude and spatial width so that the resulting profile is $O(1)$ in size. More precisely, the normalized inner profile $\widehat\Omega$, expressed in the inner coordinate $\widehat X$, satisfies the normalization condition
\begin{equation}\label{eqt:inner_profile_rescale}
    \widehat\Omega(0)=-1,\quad
    \left|\left\{\widehat X:\ \widehat\Omega(\widehat X)<-0.5\right\}\right|=1.
\end{equation}
With this normalization condition, $\widehat\Omega$ captures the local structure near the singular point of the outer profile. Numerically, we observe that $\widehat\Omega$ remains smooth for $\widehat X=O(1)$ and approaches a limiting shape as time evolves; see Figure~\ref{fig:inner_profile}. This provides evidence that, although the outer profile becomes singular, the solution develops a regular inner profile on a finer spatial scale.
\begin{figure}[htbp]
    \begin{subfigure}{0.49\textwidth}
        \includegraphics[width=\textwidth]{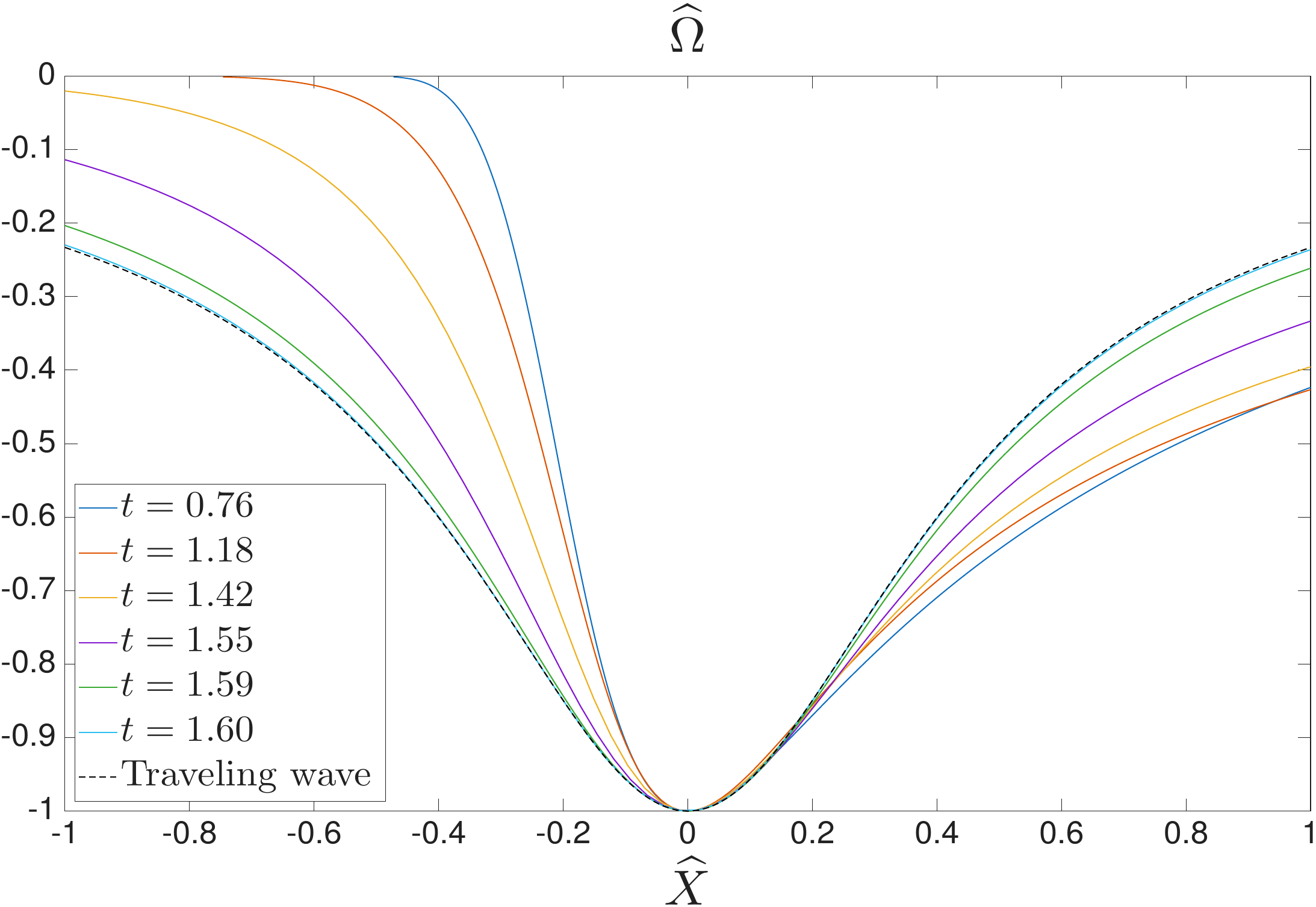}
        \caption{Evolution of inner scale profile}
        \label{fig:inner_profile_evolution}
    \end{subfigure}
    \caption{Evolution of the inner-scale profile $\widehat\Omega$ for $a=-0.2$, computed from initial datum Case \nameref{cas:initialdata_case1.1}. The dashed line shows a traveling-wave profile obtained by the iterative method and rescaled according to \eqref{eqt:inner_profile_rescale}. The inner-scale profile converges to the traveling-wave profile (dashed line).}
    \label{fig:inner_profile}
\end{figure}
Furthermore, Figure \ref{fig:inner_profile} shows that the limiting inner-scale profile $\widehat{\Omega}$ matches closely the traveling wave profile computed by the iterative method (dashed curve). This agreement provides numerical evidence supporting the two-scale ansatz \eqref{eqt:multiscale_blowup_recall}.

\subsection{Scaling analysis of two-scale blowup for $a=-0.2$}
Here we further justify the two-scale blowup scenario for $a=-0.2$ by fitting the scaling factors $\hat \lambda, \hat \gamma$ in the two-scale ansatz \eqref{eqt:multiscale_blowup_recall}. To this end, we follow the fitting method in \cite{hou2022potential}. Consider a solution quantity $v(t)$ that is expected to behave like a power law near the blowup time $T$, i.e. 
\begin{equation}\label{eqt:power_law}
    v(t)\sim \alpha(T-t)^{\eta}\quad \text{as}\ t\to T^-,
\end{equation}
where $\eta$ is the power law exponent and $\alpha>0$ is some constant. We first study the time derivative of the logarithm:
\[
    \frac{\mathrm{d}}{\mathrm{d}t}\ln v(t)=\frac{v'(t)}{v(t)}\sim \frac{-\eta\alpha(T-t)^{\eta-1}}{\alpha(T-t)^{\eta}}=-\frac{\eta}{T-t},
\]
which naturally leads to a linear regression model
\[
    \mathrm{mod_1}(t;v):=\frac{v(t)}{v'(t)}\sim-\frac{1}{\eta}(T-t)=:\tilde a t+\tilde b
\]
with response variable $y$, explanatory variable $t$, and parameters $\tilde a, \tilde b$. We can then apply the ordinary least squares method to fit the parameters $\tilde a, \tilde b$ and obtain an estimate of the power law exponent $\tilde\eta=1/\tilde a$. The quality of the fit can be measured by the coefficient of determination $R^2$, which is defined as
\[
    R^2=1-\frac{SS_{\mathrm{err}}}{SS_{\mathrm{tot}}}.
\]
An $R^2$ value close to 1 indicates a high-quality fit. However, this procedure only provides a crude estimate of $\eta$, since the time derivative $v'(t)$ is computed by numerical differentiation and is subject to relatively large numerical errors. Moreover, the time step is chosen according to the CFL condition and is therefore non-uniform, which further affects the accuracy. To obtain a better approximation of $\eta$ and further justify the power law behavior \eqref{eqt:power_law}, we perform a local search near the crude estimate $\tilde\eta$ and find a finer estimate $\bar\eta$ by requiring that the fitting model 
\[
    \mathrm{mod_2}(t;v):=v(t)^{1/\bar \eta}\sim \alpha^{1/\bar \eta}(T-t)^{\eta/\bar \eta}\sim \alpha^{1/\bar \eta}(T-t)=:\bar a t+\bar b
\]
has the best linear-regression goodness of fitting, as measured by $R^2$. More precisely, we put a uniform mesh (with 101 points) over the interval $[\tilde \eta-0.1,\tilde \eta+0.1]$, compute the $R^2$ value for each candidate on this mesh, and select the one with the best $R^2$ value as our final estimate of the power law exponent $\bar\eta$. If $\bar \eta$ falls into one of the endpoints of the interval, i.e. $\tilde \eta\pm0.1$, we will perform a local search again around $\bar \eta$ and update the interval to $[\bar \eta-0.1,\bar \eta+0.1]$ until $\bar \eta$ is strictly inside the interval. This procedure allows us to obtain a more accurate estimate of the power law exponent $\eta$ and further justify the power law behavior of $v(t)$ near the blowup time.

Although we do not directly simulate the gCLM equation \eqref{eqt:gCLM} in the physical variables $(x,t)$, the solution in physical coordinates can be recovered by inverting the dynamic rescaling transformation \eqref{eqt:change_of_variables}. This reconstruction is used to estimate the blowup exponent $\hat\lambda$ and the inner-scale shrinking exponent $\hat\gamma$ in the two-scale ansatz \eqref{eqt:multiscale_blowup_recall}. To fit $\hat\lambda$, we take
\[
    v(t)=\|\omega(\cdot,t)\|_{L^\infty}.
\]
To fit $\hat\gamma$, we take
\[
    v(t)=|x_2(t)-x_1(t)|,
\]
where $x_1(t)<x_2(t)$ are the two locations satisfying
\[
    \omega(x_1(t),t)=\omega(x_2(t),t)=-\tfrac12\|\omega(\cdot,t)\|_{L^\infty}.
\]
Thus $|x_2(t)-x_1(t)|$ measures the width of the thin peak (at half-maximum) and reflects the spatial size of the inner scale. We carry out the regressions over the time interval $[t_1,t_2]=[1.3,1.55]$, which is close to the estimated blowup time $T\approx 1.6$ while still containing sufficiently many time samples for a robust fit. The fitting results are shown in Figures \ref{fig:fit_hat_lambda} and \ref{fig:fit_hat_gamma}.

\begin{figure}[htbp]
    \begin{subfigure}{0.49\textwidth}
        \includegraphics[width=\textwidth]{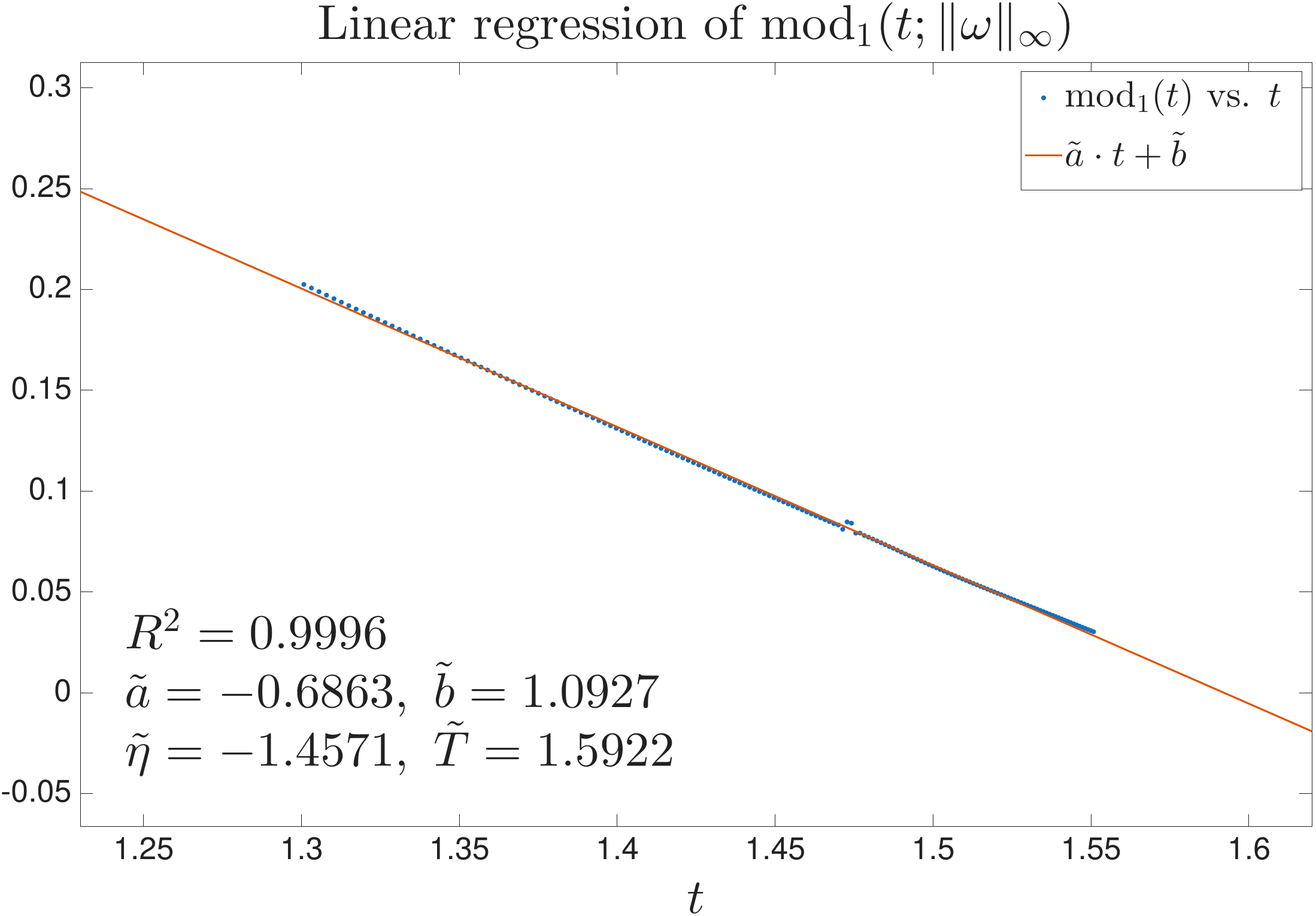}
    \end{subfigure}
    \begin{subfigure}{0.49\textwidth}
        \includegraphics[width=\textwidth]{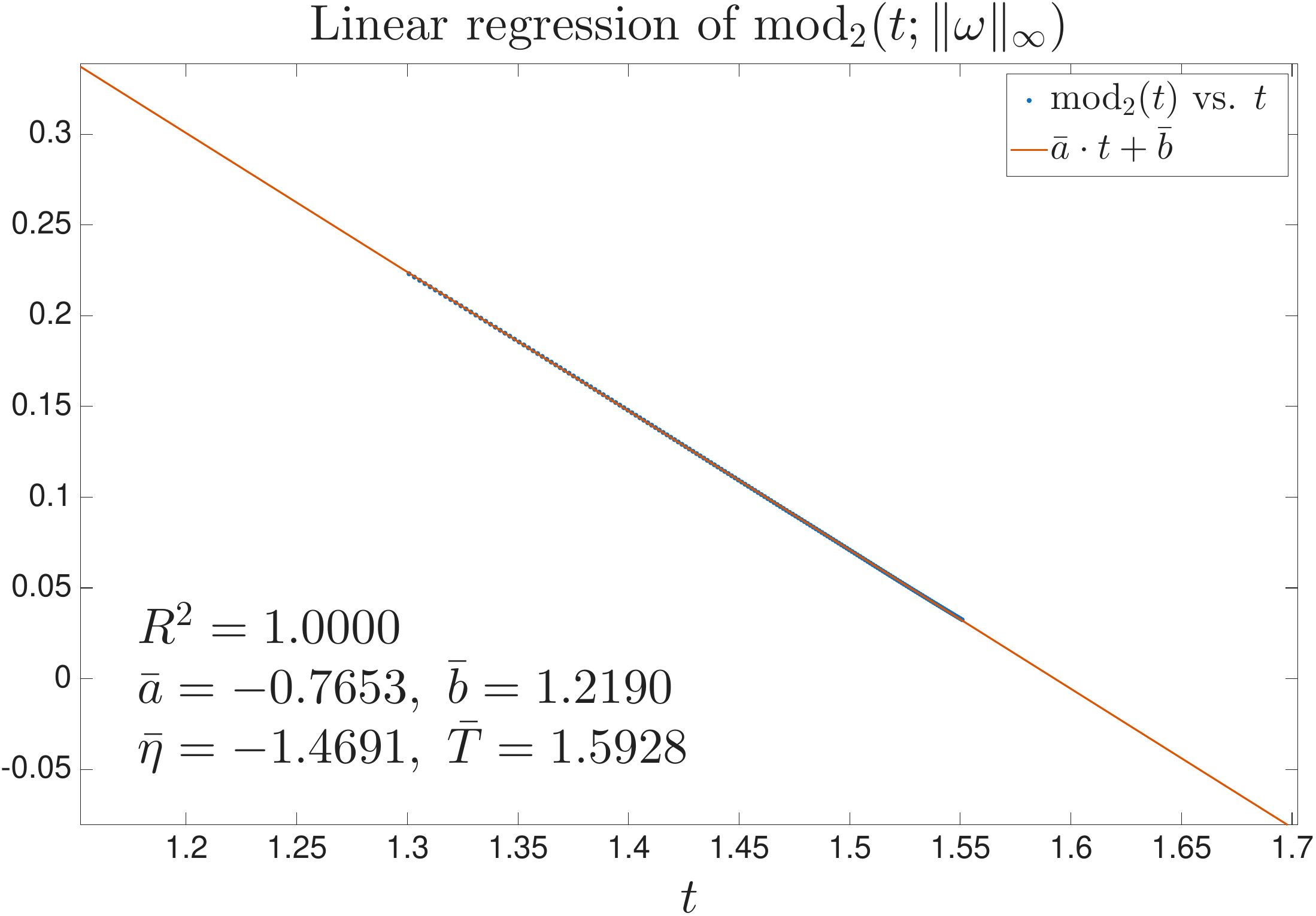}
    \end{subfigure}
    \caption{Fitting of the power law exponent $\eta$ for $\|\omega(\cdot,t)\|_{L^\infty}$ with $a=-0.2$ over the time interval $[t_1,t_2]=[1.3,1.55]$. Left: linear regression of $\mathrm{mod}_1(t;\|\omega\|_{L^\infty})$. Right: linear regression of $\mathrm{mod}_2(t;\|\omega\|_{L^\infty})$. The blue points are the data points obtained from our computation, and the red lines are the linear models (plotted on a larger time interval for clarity). The small jump near $t=1.47$ is caused by an adaptive mesh update.}
    \label{fig:fit_hat_lambda}
\end{figure}

\begin{figure}[htbp]
    \begin{subfigure}{0.49\textwidth}
        \includegraphics[width=\textwidth]{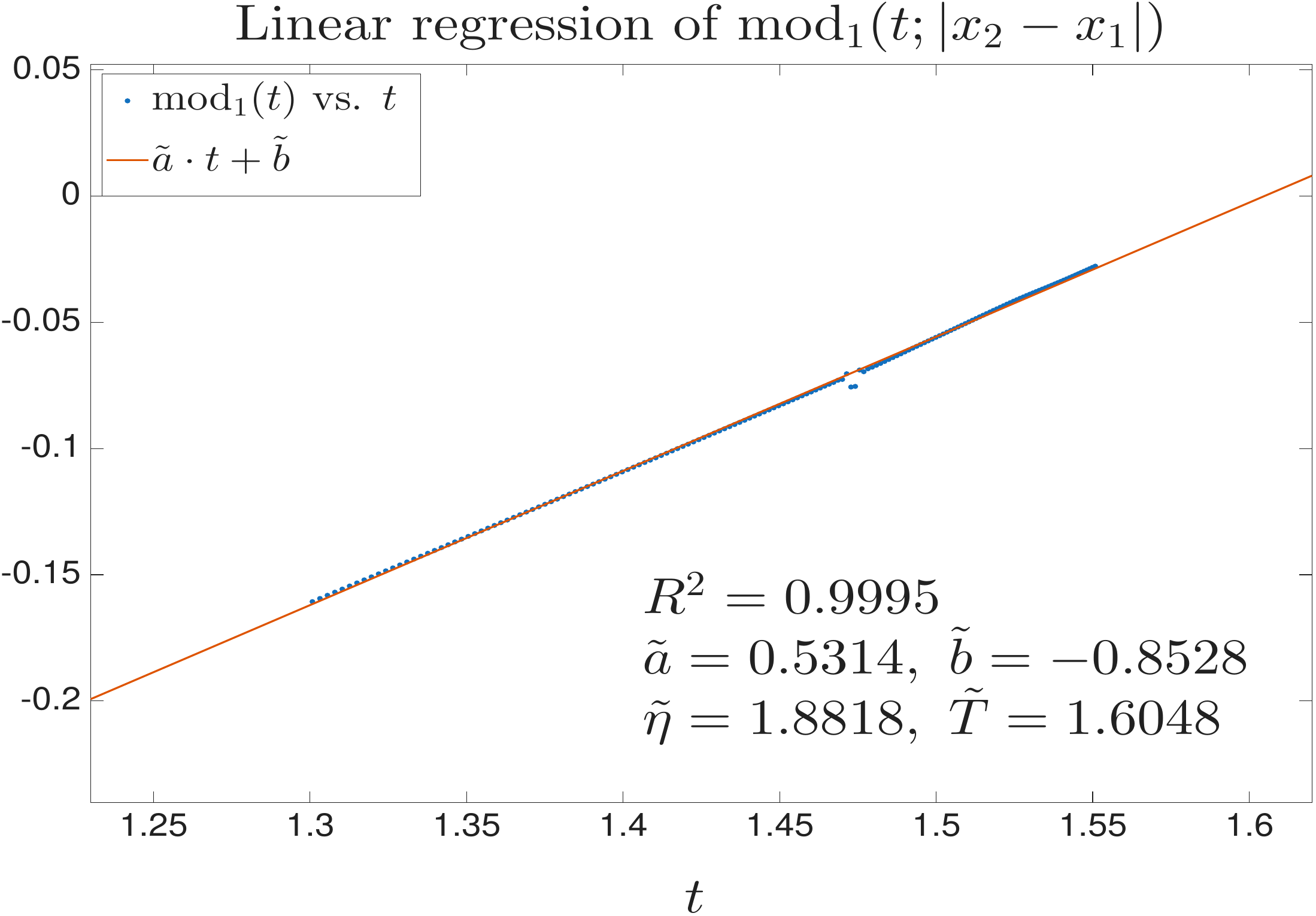}
    \end{subfigure}
    \begin{subfigure}{0.49\textwidth}
        \includegraphics[width=\textwidth]{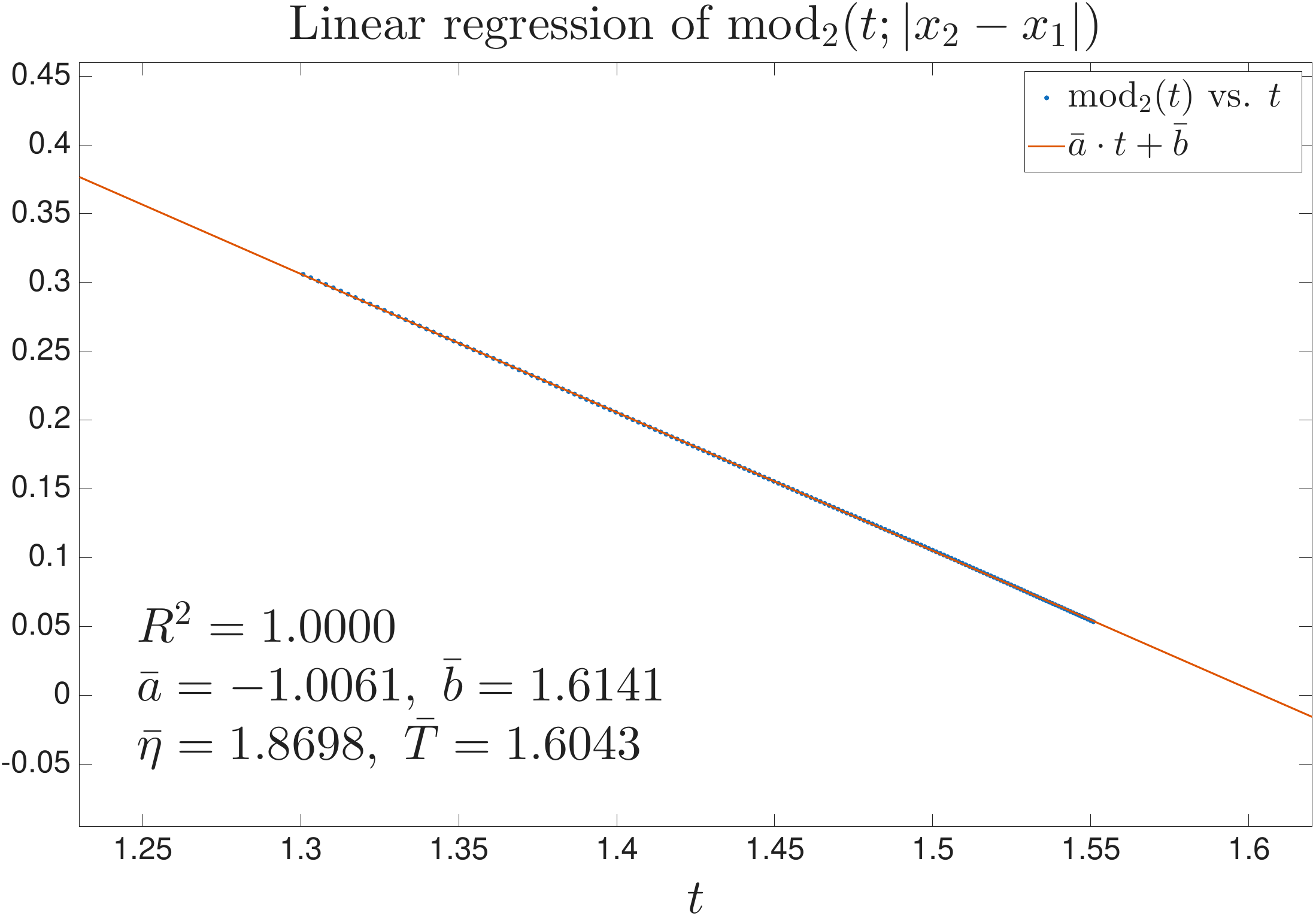}
    \end{subfigure}
    \caption{Fitting of the power law exponent $\eta$ for $|x_2(t)-x_1(t)|$ with $a=-0.2$ over the time interval $[t_1,t_2]=[1.3,1.55]$. Left: linear regression of $\mathrm{mod}_1(t;|x_2-x_1|)$. Right: linear regression of $\mathrm{mod}_2(t;|x_2-x_1|)$. The blue points are the data points obtained from our computation, and the red lines are the linear models (plotted on a larger time interval for clarity). The small jump near $t=1.47$ is caused by an adaptive mesh update.}
    \label{fig:fit_hat_gamma}
\end{figure}

One can see from Figures \ref{fig:fit_hat_lambda} and \ref{fig:fit_hat_gamma} that the linear regression models match the data well, with $R^2$ values close to $1$. This strong agreement supports the power-law behavior of $\|\omega(\cdot,t)\|_{L^\infty}$ and the thin-peak width $|x_2(t)-x_1(t)|$ as $t\to T^-$, as assumed by the two-scale ansatz \eqref{eqt:multiscale_blowup_recall}. Moreover, the fitted exponent $\hat\lambda=-1.4691$ indicates a blowup rate faster than $(T-t)^{-1}$, while the fitted shrinking exponent $\hat\gamma=1.8698$ shows that the inner scale collapses faster than the outer scale (recall $\gamma=0.9114$ from Section \ref{sec:odd_initial_data}). Together, these results provide further quantitative evidence for the two-scale nature of the blowup.

\subsection{Scaling factors for $-1\le a\le -0.1$}\label{sec:scaling_factors_multiscale}
In this subsection, we present the scaling factors $\hat \lambda$, $\hat \gamma$, and $\gamma$ of the two-scale blowup for $-1\le a\le -0.1$. The parameters $\hat \lambda$ and $\hat \gamma$ are obtained by fitting the data using the procedure described in the previous subsection; the results are summarized in Table \ref{tab:multiscale_scaling_factors}. For comparison, we also list the shrinking rate of the outer scale $\gamma$ computed in Section \ref{sec:odd_initial_data}. We observe that for all $a \in [-1, -0.1]$, the fitted factors satisfy $\hat\lambda<-1$, indicating that the magnitude of the solution grows faster than $(T-t)^{-1}$, and moreover $\hat\gamma>\gamma$, showing that the inner scale collapses faster than the outer scale. These features are consistent with the two-scale blowup scenario. Furthermore, the fitted data satisfy $\gamma-\hat\gamma-\hat\lambda-1<0$, which supports our analysis in Section \ref{sec:2scale_analysis}.

We also observe that, as $a$ decreases, $|\hat\lambda|$ decreases (indicating a slower growth of $\|\omega(\cdot,t)\|_{L^\infty}$), while both $\hat\gamma$ and $\gamma$ increase (indicating faster spatial collapsing on both the inner and outer scales).

\begin{table}[htbp]
    \begin{tabular}{|c|c|c|c|c|c|}
        \hline
        $a$           & $-0.1$    & $-0.2$    & $-0.3$     & $-0.4$    & $-0.5$     \\ \hline
        $\hat\lambda$ & $-1.4852$ & $-1.4691$ & $-1.3282$  & $-1.2923$ & $ -1.2887$ \\ \hline
        $\hat \gamma$ & $1.8577$  & $1.8698$  & $1.9635$   & $2.0178$  & $2.0678$   \\ \hline
        $\gamma$      & $0.7827$  & $0.9114$  & $1.0124$   & $1.0986$  & $1.1821$   \\ \hline\hline
        $a$           & $-0.6$    & $-0.7$    & $-0.8$     & $-0.9$    & $-1$       \\ \hline
        $\hat\lambda$ & $-1.2681$ & $-1.2612$ & $ -1.2487$ & $-1.2268$ & $-1.2038$  \\ \hline
        $\hat \gamma$ & $2.1580$  & $2.2936$  & $2.4178$   & $2.4820$  & $2.5562$   \\ \hline
        $\gamma$      & $1.2782$  & $1.3888$  & $1.4858$   & $1.5866$  & $1.6841$   \\ \hline
    \end{tabular}
    \caption{Scaling factors for $-1\le a\le -0.1$. The initial data are chosen to be odd symmetric as in Case \nameref{cas:initialdata_case1.1}.}
    \label{tab:multiscale_scaling_factors}
\end{table}

\section{Existence of traveling wave solutions}\label{sec:existence_traveling_wave}
Recall that in Section \ref{sec:sec2_traveling_wave} we discussed that the inner profile in a two-scale self-similar blowup \eqref{eqt:multiscale_blowup} is a traveling wave profile. 
The goal of this section is to prove the existence of traveling wave solutions for the gCLM equation \eqref{eqt:gCLM} and discuss their properties. In particular, we prove Theorem \ref{thm:traveling_wave_existence} by constructing traveling wave profiles via a fixed-point method. Recall that a traveling wave profile is a solution to the equation
\[
    (a\bar u-r) \bar\omega_x=\bar u_x\bar\omega,
\]
where $\bar u_x=\mtx{H}(\bar\omega)$ and $\bar u(0) = 0$. For notational simplicity, we will still use $\omega,u$ for $\bar\omega,\bar u$ respectively. We aim to prove the existence of a solution $(\omega,u,r)$ to the traveling wave profile equation
\begin{equation}\label{eqt:traveling_profile}
    (a u-r) \omega_x=u_x\omega,\quad u_x=\mtx{H}(\omega),
\end{equation}
for all $a\in(-\infty,1]$. For the $a=0$ case, the profile equation \eqref{eqt:traveling_profile} becomes
\[
    -r\omega_x=u_x\omega,
\]
which admits an explicit solution
\begin{equation}{\label{eqt:traveling_profile_a0}}
    \omega(x)=\frac{1}{1+x^2},\quad u(x)=\frac{1}{2}\ln(1+x^2),\quad r=\frac{1}{2}.
\end{equation}
In general, the profile equation \eqref{eqt:traveling_profile} is a nonlinear nonlocal equation, and it is difficult to find explicit solutions. In the remaining part of this section, we will prove the existence of non-trivial solutions to \eqref{eqt:traveling_profile} for all $a\in(-\infty,1]$ by a fixed-point method and discuss some properties of such solutions.

\subsection{Existence of traveling wave solutions by a fixed-point method} 
To prove the nonlinear system \eqref{eqt:traveling_profile} admits a non-trivial solution for each $a\in(-\infty,1]$, we follow the idea of \cite{huang2024self}: we reformulate the problem as a fixed-point problem of some nonlinear map and then apply the Schauder fixed-point theorem to show the existence of fixed points. To this end, we need to select an appropriate function space in which we can establish the continuity and compactness of our nonlinear map.

Consider a Banach space of continuous even functions
\[
    \mathbb V:=\left\{\omega\in C(\mathbb R)\mid\omega(-x)=\omega(x),\|\rho \omega\|_{L^\infty}<+\infty\right\},
\]
endowed with a weighted $L^\infty$-norm $\|\rho \omega\|_{L^\infty}$, referred to as the $L^\infty_\rho$-norm, where $\rho(x)=(1+|x|)^{-1/2}$. Moreover, for each $a\in (-\infty,1]$ we consider a closed (under the $L^\infty_\rho$-norm) and convex subset $\mathbb D_a$ of $\mathbb V$. We define
\begin{equation}\label{eqt:def_D_a}
    \begin{split}
        \mathbb D_a:=\big\{\omega\in \mathbb V\mid
         & \ \omega(0)=1, (1-x^2)_+\le \omega(x)\le 1\text{ for all }x,                                                               \\
         & \ \omega(x)\text{ is non-increasing on }[0,+\infty), \omega(\sqrt s)\text{ is convex in }s,\omega'_-(1/2)\le-\eta_a\big\},
    \end{split}
\end{equation}
which is apparently a nonempty function set. Here and below, $(t)_+:=\max\{t,0\}$ for any number $t$; $\omega'_-$ and $\omega'_+$ denote the left and the right derivatives of a function $\omega$. We remark that, though a function $\omega\in\mathbb{D}_a$ is not required to be differentiable, the one-sided derivatives $\omega'_-(x)$ and $\omega'_+(x)$ are both well defined at every point $x$ by the convexity of $\omega(\sqrt{s})$ in $s$. In what follows, we will abuse notation and simply use $\omega'(x)$ for $\omega'_-(x)$ and $\omega'_+(x)$ in both weak sense and strong sense. For example, when we write $\omega'(x)\le C$, we mean $\omega'_-(x)\le C$ and $\omega'_+(x)\le C$ at the same time. In this context, the non-increasing property of $\omega$ on $[0,+\infty)$ can be represented as $\omega'(x)\le 0$ for $x\ge0$. 

In the definition of $\mathbb{D}_a$ \eqref{eqt:def_D_a}, $\eta_a=1/\left( 2^{9/2}(4+|a|)^3 \right)$ is a constant depending on $a$, whose precise value is not important; it only matters that $\eta_a>0$. In fact, one can check that $\omega(x)\in \mathbb D_a$ implies
\[
    (1-x^2)_+\le \omega(x)\le \max\left\{ 1-\eta_a x^2,1-\eta_a/4 \right\}<1,\quad x> 0,
\]
where the second inequality follows from the assumptions that $\omega(\sqrt s)$ is convex in $s$, $\omega'_-(1/2)\le -\eta_a$, and $\omega(x)$ is non-increasing on $[0,+\infty)$. This avoids $\omega(x)\equiv 1$ being in $\mathbb D_a$.

Now we construct an $a$-dependent nonlinear map whose potential fixed points are the traveling wave profiles. For $\omega\in\mathbb D_a$, we first define a linear operator by the 1D Biot--Savart law (i.e. \eqref{eqt:biot_savart_law}) in the gCLM equation 
\begin{equation}\label{eqt:biot_savart}
    \mtx{T}(\omega)(x):=\frac{1}{\pi}\int_{-\infty}^{+\infty}\omega(y)\ln\left|\frac{x-y}{y}\right|\idiff y=\frac{1}{\pi}\int _0^{+\infty}\omega(y)\ln\left| \frac{x^2-y^2}{y^2} \right|\idiff y,
\end{equation}
where the second equality follows from the even symmetry of $\omega$. Next, for $a\in(-\infty,0)\cup (0,1]$ we define
\[
    \mtx{R}_a(\omega)(x)=\left(1-\frac{a\mtx{T}(\omega)(x)}{r(\omega)}\right)_+^{\frac{1}{a}},\quad \omega\in\mathbb D_a,
\]
where
\begin{equation}\label{eqt:traveling_speed}
    r(\omega)=\frac{1}{\pi}\int_0^{+\infty}\frac{\omega(0)-\omega(y)}{y^2}\idiff y.
\end{equation}
Note that $r(\omega)<+\infty$ holds for $\omega\in\mathbb D_a$ since $(1-x^2)_+\leq \omega(x)\le1$ for all $x$. For the special case $a=0$, we define
\[
    \mtx{R}_0(\omega)(x)=\mathrm{e}^{-\frac{\mtx{T}(\omega)(x)}{r(\omega)}},\quad \omega\in\mathbb D_0.
\]
One can check that this definition is consistent with the definition of $\mtx{R}_a$ for $a\ne0$ by taking the limit $a\to 0$.

We now aim to study the fixed-point problem
\[
    \mtx{R}_a(\omega)=\omega,\quad \omega\in\mathbb D_a.
\]
First we prove a proposition explaining how a fixed point of $\mtx{R}_a$ corresponds to a traveling wave profile.

\begin{proposition}\label{prop:fixed_point_is_traveling_wave}
    For any $a\in(-\infty,1]$, if $\omega\in\mathbb D_a$ is a fixed point of $\mtx{R}_a$, i.e. $\mtx{R}_a(\omega)=\omega$, then $(\omega,r(\omega))$ is a solution to the traveling wave profile equation \eqref{eqt:traveling_profile}.

    Conversely, if $(\omega,r)$ is a solution to \eqref{eqt:traveling_profile} such that $\omega(x)$ is an even function of $x$, $\omega(x)\ge 0$ and $\omega'(x)\le 0$ for all $x\ge 0$, $\omega(0)=1$, and $\lim_{x\to 0}\omega'(x)/(2x)=-1$ (by re-normalization), then $\omega$ is a fixed point of $\mtx{R}_a$, and $r$ is related to $\omega$ as in \eqref{eqt:traveling_speed}.
\end{proposition}
\begin{proof}
    Note that when $\omega$ is an even function, the velocity $u$ defined by \eqref{eqt:biot_savart_law} coincides with $\mtx{T}(\omega)$. If $\omega\in \mathbb D_a$ is a fixed point of $\mtx{R}_a$, then we have
    \[
        \omega(x)=\begin{cases}
            \left(1-\frac{a\mtx{T}(\omega)(x)}{r(\omega)}\right)_+^{\frac{1}{a}}, & a\le1,\ a\ne 0, \\
            \mathrm{e}^{-\frac{\mtx{T}(\omega)(x)}{r(\omega)}},                   & a=0.
        \end{cases}
    \]
    One can then directly verify that such $\om$ is a solution to the traveling wave profile equation \eqref{eqt:traveling_profile} with $u=\mtx{T}(\omega)$ and $r=r(\omega)$.

    Conversely, if $(\omega(x),r)$ is a solution to \eqref{eqt:traveling_profile} satisfying the conditions in the proposition, then for $x$ such that $\omega(x)\ne 0$, we have
    \begin{equation}\label{eqt:traveling_wave_ode}
        \frac{\omega'(x)}{\omega(x)}=\frac{u'(x)}{au(x)-r}.
    \end{equation}
    Solving this ODE with the initial value $\om(0)=1$ and writing $u=\mtx{T}(\omega)$ gives
    \[
        \omega(x)=\begin{cases}
            \left(1-\frac{a\mtx{T}(\omega)(x)}{r}\right)_+^{\frac{1}{a}}, & a\le1,\ a\ne 0, \\
            \mathrm{e}^{-\frac{\mtx{T}(\omega)(x)}{r}},                   & a=0.
        \end{cases}
    \]
    Moreover, it is not hard to check that
    \begin{equation}\label{eqt:traveling_wave_equivalence}
        \lim_{x\to 0}\frac{\omega'(x)}{2x}=-1\quad \Longleftrightarrow\quad r=r(\om) = \frac{1}{\pi}\int_0^{+\infty}\frac{1-\omega(y)}{y^2}\idiff y,
    \end{equation}
    which verifies that $\om$ is a fixed point of $\mtx{R}_a$.
    In fact, using the equation \eqref{eqt:traveling_wave_ode} we can compute that
    \[
    \lim_{x\to 0}\frac{\omega'(x)}{2x}=\lim_{x\to 0}-\frac{u'(x)}{2rx}\cdot\frac{\omega(x)}{1-au(x)/r}=-\frac{1}{r}\lim_{x\to 0}\frac{u'(x)}{2x}.
    \]
    Note that by the even symmetry of $\omega(x)$,
    \[
    u'(x)=\frac{1}{\pi}\mathrm{P.V.}\int_{\mathbb{R}}\frac{\omega(x)}{x-y}\idiff y=\frac{1}{\pi}\mathrm{P.V.}\int_{\mathbb{R}}\frac{\omega(x)-1}{x-y}\idiff y=\frac{2}{\pi}\mathrm{P.V.}\int_0^{+\infty}\frac{x(\omega(x)-1)}{x^2-y^2}\idiff y,
    \]
    and therefore
    \begin{align*}
    \lim_{x\to 0}\frac{u'(x)}{2x}=\frac{1}{\pi}\int_0^{+\infty}\frac{1-\omega(y)}{y^2}\idiff y.
    \end{align*}
    The equivalence in \eqref{eqt:traveling_wave_equivalence} then follows immediately.
\end{proof}

\subsubsection{Properties of $r(\omega)$.} We provide an estimate of $r(\omega)$ that will be useful later.
\begin{lemma}\label{lem:r_properties}
    For any $\omega\in \mathbb D_a$ and any $x>0$,
    \[
        \frac{\eta_a}{\pi}\le r(\omega)\le \min\left\{ \frac{2(x+1)}{\pi x}(1-\omega(x))^{1/2},\frac{2}{\pi} \right\}.
    \]
\end{lemma}
\begin{proof}
    The definition of $\mathbb D_a$ implies that for any $x>0$,
    \[
        (1-x^2)_+\le \omega(x)\le \max\left\{ 1-\eta_a x^2,1-\eta_a/4\right\}<1.
    \]
    Therefore, we have
    \[
        r(\omega)=\frac{1}{\pi}\int_0^{+\infty}\frac{1-\omega(y)}{y^2}\idiff y\le \frac{1}{\pi}\int_0^1 1\idiff y+\frac{1}{\pi}\int_1^{+\infty}\frac{1}{y^2}\idiff y=\frac{2}{\pi},
    \]
    \[
        r(\omega)=\frac{1}{\pi}\int_0^{+\infty}\frac{1-\omega(y)}{y^2}\idiff y\ge \frac{1}{\pi}\int_0^{1/2}\eta_a \idiff y+\frac{1}{\pi}\int_{1/2}^{+\infty}\frac{\eta_a}{4y^2}\idiff y=\frac{\eta_a}{\pi}.
    \]
    Next, fix an arbitrary $x>0$. For $0\le y\le x, \omega(y)\ge \max\{(1-y^2)_+,\omega(x)\}$, so $1-\omega(y)\le \min\{y^2,1-\omega(x)\}$. For $y\ge x$, the convexity of $\omega(\sqrt{s})$ implies that $(1-\omega(y))/y^2\le (1-\omega(x))/x^2$, and so $1-\omega(y)\le \min\{ y^2(1-\omega(x))/x^2,1 \}$. Combining these two estimates, we have
    \[
        1-\omega(y)\le \min\left\{ y^2,1-\omega(x) \right\}+\min\left\{ y^2(1-\omega(x))/x^2,1 \right\},\quad y\geq 0.
    \]
    We thus obtain that
        \begin{align*}
            r(\omega) & \le \frac{1}{\pi}\int_0^{+\infty}\frac{\min\left\{ y^2,1-\omega(x) \right\}}{y^2}\idiff y+\frac{1}{\pi}\int_0^{+\infty}\frac{\min\left\{ y^2(1-\omega(x))/x^2,1 \right\}}{y^2}\idiff y \\
                      & \le \frac{2(x+1)}{\pi x}(1-\omega(x))^{1/2},
        \end{align*}
    which is the desired bound.
\end{proof}
\begin{lemma}\label{lem:r_continuity}
    $r(\omega):\mathbb D_a\to \mathbb{R} $ is Hölder continuous in the $L^\infty_\rho$-norm. In particular,
    \[
        |r(\omega_1)-r(\omega_2)|\lesssim \|\rho(\omega_1-\omega_2)\|_{L^\infty}^{1/2},
    \]
    for any $a\in(-\infty,1]$ and $\omega_1,\omega_2\in \mathbb D_a$.
\end{lemma}
\begin{proof}
    Recall that $\rho=(1+|x|)^{-1/2}$. Denote $\delta=\|\rho(\omega_1-\omega_2)\|_{L^\infty}\le 1$. Since $\omega_i\ge (1-x^2)_+$ for $i=1,2$, we have
    \[
        |\omega_1(x)-\omega_2(x)|\le \min\left\{ x^2,\ \delta(1+|x|)^{1/2} \right\}.
    \]
    Hence,
        \begin{align*}
            |r(\omega_1)-r(\omega_2)|
             & \le \frac{1}{\pi}\int_0^{+\infty}\frac{|\omega_1(y)-\omega_2(y)|}{y^2}\idiff y\le \frac{1}{\pi}\left( \int_0^{\sqrt \delta}1\idiff y+\int_{\sqrt \delta}^{+\infty}\frac{\delta(1+y)^{1/2}}{y^2}\idiff y \right) \\
             & =\frac{\sqrt \delta}{\pi}\left( 1+\sqrt{\sqrt \delta+1}+\sqrt{\delta}\cdot\mathrm{arcsinh}\left(\delta^{-1/4}\right)\right)\le \frac{1+\sqrt 2+\mathrm{arcsinh}(1)}{\pi}\sqrt \delta.                           \\
        \end{align*}
    This proves the lemma.
\end{proof}

\subsubsection{Properties of $\mtx{T}(\omega)$.}

\begin{lemma}\label{lem:property_of_T}
    For any $\omega\in\mathbb D_a$, $\mtx{T}(\omega)(0)=0$, $\mtx{T}(\omega)'(x)\ge0$ on $[0,+\infty)$, and $\mtx{T}(\omega)(\sqrt{s})$ is concave in $s$.
\end{lemma}
\begin{proof}
    Taking $x=0$ in the definition of $\mtx{T}(\omega)$ instantly leads to $\mtx{T}(\omega)(0)=0$. To prove $\mtx{T}(\omega)'(x)\ge0$, we use integration by parts to compute that
        \begin{align*}
            \mtx{T}(\omega)(x) & =\frac{1}{\pi}\int _0^{+\infty}\omega(y)\ln\left| \frac{x^2-y^2}{y^2} \right|\idiff y                                                           \\
                         & =\frac{1}{\pi}\int_{0}^{+\infty}\omega(y)\partial_y\left( x\ln\left|\frac{x+y}{x-y}\right|+y\ln\left|\frac{x^2-y^2}{y^2}\right| \right)\idiff y \\
                         & =-\frac{1}{\pi}\int_{0}^{+\infty}\omega'(y)\left( x\ln\left|\frac{x+y}{x-y}\right|+y\ln\left|\frac{x^2-y^2}{y^2}\right| \right)\idiff y.
        \end{align*}
    Therefore, we have
    \begin{equation}\label{eqt:T_omega_derivative}
        \mtx{T}(\omega)'(x)=-\frac{1}{\pi}\int_0^{+\infty}\omega'(y)\ln\left| \frac{x+y}{x-y} \right|\idiff y\ge0.
    \end{equation}
    Next, we show that $\mtx{T}(\omega)(\sqrt{s})$ is concave in $s$. By approximation theory, we may assume that $\omega(\sqrt s)$ is twice differentiable in $s$, so that the convexity of $\omega(\sqrt s)$ is equivalent to $(\omega'(x)/x)'\ge 0$ for $x> 0$. Continuing the calculations above, we have 
        \begin{align*}
            \left( \frac{\mtx{T}(\omega)'(x)}{x} \right)'
             & =-\frac{1}{\pi}\int_0^{+\infty}\omega'(y)\partial_x\left(\frac{1}{x}\ln\left| \frac{x+y}{x-y} \right|\right)\idiff y                                        \\
             & =\frac{1}{\pi}\int_0^{+\infty}\omega'(y)\left(\frac{2 y}{x^3-x y^2}+\frac{1}{x^2}\ln\left| \frac{x+y}{x-y} \right|\right)\idiff y                           \\
             & =\frac{1}{\pi}\int_0^{+\infty}\frac{\omega'(y)}{y}\left(\frac{2 y^2}{x^3-x y^2}+\frac{y}{x^2}\ln\left| \frac{x+y}{x-y} \right|\right)\idiff y               \\
             & =-\frac{1}{\pi}\int_0^{+\infty}\left(\frac{\omega'(y)}{y}\right)'\left(-\frac{y}{x}+\frac{x^2+y^2}{2x^2}\cdot\ln\left|\frac{x+y}{x-y}\right|\right)\idiff y \\
             & =\frac{1}{\pi}\int_0^{+\infty}\left(\frac{\omega'(y)}{y}\right)'G\left( y/x\right)\idiff y\le0,
        \end{align*}
    where the function $G$ is defined as
    \begin{equation}\label{eqt:G}
        G(t)=t-\frac{1+t^2}{2}\ln\left| \frac{1+t}{1-t} \right|.
    \end{equation}
    The integration by parts and the final inequality follow from the properties of $G$ collected in the next lemma.
    \end{proof}

    \begin{lemma}\label{lem:G_properties}
        The function $G$ defined in \eqref{eqt:G} satisfies:
        \begin{itemize}
            \item $G(0)=0$ and $\lim_{t\to+\infty}G(t)=0$;
            \item $G(t)\le 0$ for all $t>0$.
        \end{itemize}
    \end{lemma}

    \begin{proof}
        The identity $G(0)=0$ is immediate. Moreover, applying L'H\^opital's rule gives
            \begin{align*}
                \lim_{t\to+\infty}G(t)
                &=\lim_{t\to+\infty}
                \frac{\frac{t}{1+t^2}-\frac12 \ln\left| \frac{1+t}{1-t} \right|}{\frac{1}{1+t^2}}
                =\lim_{t\to+\infty}\frac{2t}{1-t^2}=0.
            \end{align*}
        Next, a direct computation yields
        \[
            \left(\frac{G(t)}{1+t^2}\right)'
            =\frac{4t^2}{(t^2+1)^2(t+1)(t-1)}.
        \]
        Since the numerator is positive for all $t$, we see that $G(t)/(1+t^2)$ is strictly decreasing on $(0,1)$ and strictly increasing on $(1,+\infty)$. Together with
        \[
            \lim_{t\to 0^+}\frac{G(t)}{1+t^2}=0,
            \qquad
            \lim_{t\to+\infty}\frac{G(t)}{1+t^2}=0,
        \]
        it follows that $G(t)/(1+t^2)\le 0$ for all $t>0$, and hence $G(t)\le 0$ for all $t>0$.
    \end{proof}

\subsubsection{Properties of $\mtx{R}_a(\omega)$}
\begin{lemma}\label{lem:R_a_properties}
    For any $a\in(-\infty,1]$, $\mtx{R}_a$ maps $\mathbb D_a$ into itself.
\end{lemma}
\begin{proof}
    Given $\omega\in\mathbb D_a$, let $h(x)=\mtx{R}_a(\omega)(x)$. We prove this lemma through the following steps:

    \noindent \textbf{Step $1$: Show that $h(0)=1$ and $h'(x)\le 0$ for $x\ge 0$.} Since $\mtx{T}(\omega)(0)=0$, $0<r(\omega)<+\infty$, and $\mtx{T}(\omega)'(x)\ge 0$ for $x \ge 0$ (by Lemma \ref{lem:property_of_T}), we have for $a\in(-\infty,0)\cup(0,1]$
    \[
        h(0)=\left(1-\frac{a\mtx{T}(\omega)(0)}{r(\omega)}\right)_+^{\frac{1}{a}}=1,
    \]
    and
    \[
        h'(x)=-\frac{\mtx{T}(\omega)'(x)}{r(\omega)}\left(1-\frac{a\mtx{T}(\omega)(x)}{r(\omega)}\right)_+^{\frac{1}{a}-1}\le 0.
    \]
    It is also straightforward to verify that $h(0)=1$ and $h'(x)\le 0$ for the $a=0$ case.
    Note that, if $h$ is compactly supported on $[-L,L]$ for some $L>0$, then $h'(x)=0$ for $x>L$, $h'(L-)\le 0$, and $h'(L+)=0$.

    \noindent \textbf{Step $2$: Show that $h(\sqrt s)$ is convex in $s$.}
    We only need to show that $(h'(x)/x)'$ is non-negative on $[0,+\infty)$. Continuing the calculations in step 1, we further derive for $a\in(-\infty,0)\cup(0,1]$ 
    \[
        \left(\frac{h'(x)}{x}\right)'=\left(1-\frac{a \mtx{T}(\omega)(x)}{r(\omega)}\right)_+^{\frac{1}{a}-2}\left( -\frac{1}{r(\omega)}\left( \frac{\mtx{T}(\omega)'(x)}{x} \right)' \left( 1-\frac{a \mtx{T}(\omega)(x)}{r(\omega)} \right)_++\frac{1-a}{x}\left(\frac{\mtx{T}(\omega)'(x)}{r(\omega)}\right)^2\right)\ge0,
    \]
    owing to the fact that $(\mtx{T}(\omega)'(x)/x)'\le0$ for $x\ge0$ by Lemma \ref{lem:property_of_T}. Similarly, for the special case $a=0$ we have
    \[
        \left( \frac{h'(x)}{x} \right)'=\left( -\frac{1}{r(\omega)}\left( \frac{\mtx{T}(\omega)'(x)}{x} \right)'+\frac{1}{x}\left(\frac{\mtx{T}(\omega)'(x)}{r(\omega)}\right)^2\right)\mathrm{e}^{-\frac{\mtx{T}(\omega)(x)}{r(\omega)}}\ge0.
    \]

    \noindent \textbf{Step $3$: Show that $(1-x^2)_+\le h(x)\le 1$.} The fact that $h(x)\le 1$ follows directly from step 1. To prove that $h(x)\ge (1-x^2)_+$, namely $h(\sqrt s)\ge(1-s)_+$ for $s\ge0$, we only need to show that
    \[
        \lim_{x\to 0}\frac{h'(x)}{2x}=\frac{\mathrm d}{\mathrm d s}h(\sqrt s)\Big|_{s=0}=-1
    \]
    and then use the fact that $h(\sqrt s)$ is convex in $s$ (step 2). By definition \eqref{eqt:traveling_speed}
    \[r(\omega)=\frac{1}{\pi}\int_0^{+\infty}\frac{\omega(0)-\omega(y)}{y^2}\idiff y,\]
    we can compute that 
    \begin{equation}\label{eqt:lim_T_over_x}
        \lim_{x\to 0}\frac{\mtx{T}(\omega)'(x)}{x}=-\frac{1}{\pi}\lim_{x\to 0}\int_0^{+\infty}\frac{\omega'(y)}{x}\ln \left| \frac{x+y}{x-y} \right|\idiff y  =-\frac{2}{\pi}\int_0^{+\infty}\frac{\omega'(y)}{y}\idiff y                                                           =2 r(\omega).
    \end{equation}
    Hence,
    \begin{equation}\label{eqt:lim_h_over_x}
            \lim_{x\to 0}\frac{h'(x)}{2x}
              =-\frac{1}{2r(\omega)}\lim_{x\to0} \frac{\mtx{T}(\omega)'(x)}{x}                                                                =-1.
    \end{equation}

    \noindent \textbf{Step $4$: Show that $h'_-(1/2)\le -\eta_a$.} We first show that $h'(x)$ can be bounded in terms of $\mtx{T}(\omega)'(x)$. Using \eqref{eqt:lim_T_over_x} and the fact that $\mtx{T}(\omega)(\sqrt s)$ is concave in $s$ (by Lemma \ref{lem:property_of_T}), we have
    \[
        \mtx{T}(\omega)(x)\le \mtx{T}(\omega)(0)+r(\omega)x^2 =r(\omega)x^2.
    \]
    Note that $\mtx{T}(\omega)(x)\ge 0$ for $x\ge 0$. Thus for $a\le 1$ we always have
    \[
        \left( 1-|a|x^2 \right)_+\le \left(1-\frac{a\mtx{T}(\omega)(x)}{r(\omega)}\right)_+\le 1+|a|x^2.
    \]
    Combining \eqref{eqt:lim_h_over_x} and the convexity of $h(\sqrt s)$ leads to $h(x)\ge (1-x^2)_+$ for $x\ge 0$. It follows that
    \[
        h'(x)=-\frac{\mtx{T}(\omega)'(x)}{r(\omega)}h(x)\cdot\left(1-\frac{a\mtx{T}(\omega)(x)}{r(\omega)}\right)_+^{-1}\le -\frac{(1-x^2)\mtx{T}(\omega)'(x)}{r(\omega)\left( 1+|a|x^2 \right)},\quad x\in [0,1).
    \]

    Next, we lower bound $\mtx{T}(\omega)'(x)$ in two ways. On one hand, by the monotonicity of $\omega'(x)/x$, 
        \begin{align*}
            \mtx{T}(\omega)'(x)
             & =-\frac{1}{\pi}\int_0^{+\infty}\omega'(y)\ln\left| \frac{x+y}{x-y} \right|\idiff y \ge -\frac{1}{\pi}\int_0^{x}\omega'(y)\ln\left| \frac{x+y}{x-y} \right|\idiff y                               \\
             & \ge-\frac{1}{\pi}\int_0^{x}\frac{y\omega'(x)}{x}\ln\left| \frac{x+y}{x-y} \right|\idiff y=-\frac{x\omega'(x)}{\pi }\int_0^1t\ln\left| \frac{1+t}{1-t} \right|\idiff t=-\frac{x\omega'(x)}{\pi}.
        \end{align*}
    On the other hand, for any $0<z<x$ we have
        \begin{align*}
            \mtx{T}(\omega)'(x)
             & \ge -\frac{1}{\pi}\int_0^{x}\omega'(y)\ln\left| \frac{x+y}{x-y} \right|\idiff y\ge -\frac{1}{\pi}\int_0^x \omega'(y)\cdot\frac{2y}{x} \idiff y                                    \\
             & \ge -\frac{2}{\pi}\int_z^x \omega'(y)\cdot\frac{z}{x}  \idiff y =\frac{2}{\pi}\cdot\frac{z}{x}\left( \omega(z)-\omega(x) \right)\ge \frac{2z}{\pi x}\left( 1-z^2-\omega(x) \right).
        \end{align*}
    We then choose $z=( (1-\omega(x))/2 )^{1/2}$ to obtain
    \[
        \mtx{T}(\omega)'(x)\ge \frac{1}{\sqrt 2\pi x}\left( 1-\omega(x) \right)^{3/2}.
    \]
    Putting these together, we reach
    \[
        \mtx{T}(\omega)'(x)\ge \left(\frac{1}{\sqrt 2\pi x}\right)^{1/3}\left( 1-\omega(x) \right)^{1/2}\cdot \left(\frac{x|\omega'(x)|}{\pi}\right)^{2/3}=\frac{1}{\pi} \left( \frac{x}{\sqrt 2} \right)^{1/3}|\omega'(x)|^{2/3}\left( 1-\omega(x) \right)^{1/2}.
    \]
    Finally, by Lemma \ref{lem:r_properties} we have
    \[
        h'(x)\le -\frac{(1-x^2)\mtx{T}(\omega)'(x)}{r(\omega)\left( 1+|a|x^2 \right)}\le-\frac{x^{4/3}(1-x)}{2^{7/6}\left( 1+|a|x^2 \right)}|\omega'(x)|^{2/3}.
    \]
    Plugging in $x=1/2$ and using that $\omega'_-(1/2)\le -\eta_a$ yields $h'_-(1/2)\le -\eta_a^{1/3} |\omega'_-(1/2)|^{2/3}\leq -\eta_a$.

    Combining these steps proves the lemma.
\end{proof}

Next, we show that $\mtx{R}_a$ is continuous on $\mathbb D_a$ in the $L^\infty_\rho$-topology.

\begin{theorem}\label{thm:R_a_continuity}
    For any $a\in(-\infty,1]$, $\mtx{R}_a: \mathbb D_a\to \mathbb D_a$ is continuous with respect to the $L^\infty_\rho$-norm.
\end{theorem}
\begin{proof}
    Recall that $\rho(x)=(1+|x|)^{-1/2}$. Given any (fixed) $\omega_0\in\mathbb D_a$, we only need to prove that $\mtx{R}_a$ is $L^\infty_\rho$-continuous at $\omega_0$. Let $\epsilon>0$ be an arbitrarily small number. Since $\mtx{R}_a(\omega_0)$ is bounded, continuous, and non-increasing on $[0,+\infty)$, there is some $X_0>1$ such that
    \[
        \rho(X_0)\mtx{R}_a(\omega_0)(X_0)=\epsilon.
    \]
    This also means $\rho(x)\mtx{R}_a(\omega_0)(x)\le \epsilon$ for all $x\ge X_0$.

    Let $\omega\in \mathbb D_a$ be an arbitrary function such that $\|\rho(\omega-\omega_0)\|_{L^\infty}\le \delta$ for some sufficiently small $\delta>0$. For any $x\ge 0$, we have
        \begin{align*}
            |\mtx{T}(\omega)(x)-\mtx{T}(\omega_0)(x)|
             & =\frac{1}{\pi}\left| \int_0^{+\infty}( \omega(y)-\omega_0(y) )\ln\left| \frac{x^2-y^2}{y^2} \right|\cdot \idiff y \right| \\
             & \le \frac{1}{\pi}\int_0^{+\infty}\left |\delta (1+y)^{1/2}\ln\left| \frac{x^2-y^2}{y^2} \right|\right|\idiff y                  \\
             & =\frac{\delta}{\pi}\cdot x \int_0^{+\infty}(1+tx)^{1/2}\left| \ln\left| \frac{1-t^2}{t^2} \right| \right|\idiff t               \\
             & \le \frac{\delta}{\pi}\cdot x (1+x)^{1/2}\int_0^{+\infty}(1+t)^{1/2}\left| \ln\left| \frac{1-t^2}{t^2} \right| \right|\idiff t  \\
             & \lesssim \delta x(1+x)^{1/2}.
        \end{align*}
    The last integral of $t$ above is finite since $\ln| (1-t^2)/{t^2} |\sim -2\ln t$ as $t \to 0^+$ and $\ln| (1-t^2)/{t^2} |=O(t^{-2})$ as $t\to +\infty$. Recall that for any $\omega\in \mathbb D_a$ we have $\omega(x)\le 1$, so a similar argument shows that $|\mtx{T}(\omega)(x)|,|\mtx{T}(\omega_0)(x)|\lesssim x$. Combining these estimates with Lemma \ref{lem:r_properties} and Lemma \ref{lem:r_continuity}, we obtain
        \begin{align*}
            \left| \frac{\mtx{T}(\omega)(x)}{r(\omega)}-\frac{\mtx{T}(\omega_0)(x)}{r(\omega_0)} \right|
             & \le \left( \left| \frac{\mtx{T}(\omega)-\mtx{T}(\omega_0)}{r(\omega_0)} \right|+\left|\frac{\mtx{T}(\omega)}{r(\omega)r(\omega_0)}\right|\cdot|r(\omega)-r(\omega_0)| \right) \\
             & \lesssim \left(\delta x(1+x)^{1/2}+\delta^{1/2} x\right)\lesssim \delta^{1/2} x(1+x)^{1/2}.
        \end{align*}
    For the general case $a\in(-\infty,0)\cup (0,1]$, we denote
    \[
        g_0:=\left(1-\frac{a\mtx{T}(\omega_0)(x)}{r(\omega_0)}\right)_+,\quad g:=\left(1-\frac{a\mtx{T}(\omega)(x)}{r(\omega)}\right)_+.
    \]
    Then it follows directly
        \begin{align*}
            |g(x)-g_0(x)|
             & \le |a|\cdot\left| \frac{\mtx{T}(\omega)(x)}{r(\omega)}-\frac{\mtx{T}(\omega_0)(x)}{r(\omega_0)} \right|\lesssim |a|\delta^{1/2} x(1+x)^{1/2},
        \end{align*}
    which implies that, for $x\in [0,X_0]$,
        \begin{align*}
            \rho(x)|\mtx{R}_a(\omega)(x)-\mtx{R}_a(\omega_0)(x)|
             & \le\rho(x)\left|g^{1/a}(x)-g_0^{1/a}(x)\right|                                                     \\
             & \le\frac{1}{|a|}\rho(x)\left|g(x)-g_0(x)\right|\cdot\left( g^{(1-a)/a}(x)+g_0^{(1-a)/a}(x) \right) \\
             & \lesssim \delta^{1/2} x\le \delta^{1/2} X_0,
        \end{align*}
    since $g,g_0\le 1$ when $a\in(0,1]$, and $g,g_0\ge 1$ when $a\in(-\infty,0)$. As for the case $a=0$, we have
        \begin{align*}
            \rho(x)|\mtx{R}_0(\omega)(x)-\mtx{R}_0(\omega_0)(x)|
             & =\rho(x)\left| \mathrm{e}^{-\frac{\mtx{T}(\omega)(x)}{r(\omega)}}-\mathrm{e}^{-\frac{\mtx{T}(\omega_0)(x)}{r(\omega_0)}} \right| \\
             &\le \rho(x)\left| \frac{\mtx{T}(\omega)(x)}{r(\omega)}-\frac{\mtx{T}(\omega_0)(x)}{r(\omega_0)} \right| \\
             &\lesssim \delta^{1/2} x\le \delta^{1/2} X_0.
        \end{align*}
    The first inequality comes from $\mtx{T}(\omega)/r(\omega),\mtx{T}(\omega_0)/r(\omega_0)\ge 0$. Hence we have $\rho(x)|\mtx{R}_a(\omega)(x)-\mtx{R}_a(\omega_0)(x)|\lesssim \delta^{1/2} X_0$ for any $a\in(-\infty,1]$.  Again, provided that $\delta$ is sufficiently small, we can have $\rho(X_0)\mtx{R}_a(\omega)(X_0)\le 2\epsilon$. By the monotonicity of $\mtx{R}_a(\omega)$, we also have $\rho(x)\mtx{R}_a(\omega)(x)\le 2\epsilon$ for all $x\ge X_0$. Therefore, we can choose $\delta$ small enough ($\delta\lesssim \epsilon^2 X_0^{-2} $) such that
    \[
        \|\rho(\mtx{R}_a(\omega)-\mtx{R}_a(\omega_0))\|_{L^\infty}\lesssim \epsilon
    \]
    for all $\omega\in \mathbb D_a$ with $\|\rho(\omega-\omega_0)\|_{L^\infty}\le \delta$. This proves the $L_\rho^\infty$-continuity of $\mtx{R}_a$ at $\omega_0$.
\end{proof}

One last ingredient for establishing the existence of a fixed point of $\mtx{R}_a$ is the compactness of the set $\mathbb D_a$.
\begin{lemma}\label{lem:compactness}
    For any $a\in(-\infty,1]$, $\mathbb D_a$ is compact in the $L^\infty_\rho$-topology.
\end{lemma}
\begin{proof}
    For any $\omega\in \mathbb{D}_a$, we use convexity and monotonicity to obtain
    \begin{equation}\label{eqt:omega_derivative_estimation}
        -\frac{\omega'(x)}{2x}\le \frac{\omega(0)-\omega(x)}{x^2}\le \min \left\{ 1,\ \frac{1}{x^2} \right\},\quad x>0,
    \end{equation}
    implying that $|\omega'(x)|\le \min\left\{ 2x,\ 2x^{-1} \right\}\le 2$. Based on this, we show that $\mathbb{D}_a$ is sequentially compact.

    Let $\{ \omega_n \}_{n=1}^{+\infty}$ be an arbitrary sequence in $\mathbb{D}_a$. Initialize $n_{0,k}=k,\ k\ge 1$. For each integer $m\ge 1$, let $\epsilon_m=2^{-m}$, and $L_m=\epsilon_m^{-2}$. It follows that $\rho(x)\omega_n(x)\le \rho(x)\le \epsilon_m$ for all $x\ge L_m$. Furthermore, since $|\omega_n'(x)|\le 2$ on $[0,L_m]$, we can apply Ascoli's theorem to obtain a subsequence $\{ \omega_{n_{m,k}}\}_{k=1}^{+\infty}$ of $\{ \omega_{n_{m-1,k}}\}_{k=1}^{+\infty}$ such that $\|\rho(\omega_{n_{m,i}}-\omega_{n_{m,j}})\|_{L^\infty}\le 2\epsilon_m$ for any $i,j\ge 1$. Then the diagonal subsequence $\{ \omega_{n_{m,m}}\}_{m=1}^{+\infty}$ is a Cauchy sequence in the $L^\infty_\rho$-norm. This proves that $\mathbb{D}_a$ is sequentially compact.
\end{proof}

We are now ready to prove the existence of fixed points of $\mtx{R}_a$ for any $a\in (-\infty,0)\cup(0,1]$ using the Schauder fixed-point theorem.

\begin{theorem}\label{thm:traveling_wave_existence_sec7}
    For any $a\in(-\infty,1]$, there exists a fixed point $\omega_a\in \mathbb D_a$ of $\mtx{R}_a$, i.e. $\mtx{R}_a(\omega_a)=\omega_a$. As a corollary, for any $a\in(-\infty,1]$, \eqref{eqt:traveling_profile} admits a solution $(\omega_a,r_a)$ with $r_a=r(\omega_a)$.
\end{theorem}
\begin{proof}
    By Theorem \ref{thm:R_a_continuity} and Lemma \ref{lem:compactness}, $\mathbb D_a$ is convex, closed and compact in the $L_\rho^\infty$-norm, and $\mtx{R}_a$ continuously maps $\mathbb D_a$ into itself. The Schauder fixed-point theorem then implies that there exists a fixed point $\omega_a\in \mathbb D_a$ such that $\mtx{R}_a(\omega_a)=\omega_a$. The corollary follows from Proposition \ref{prop:fixed_point_is_traveling_wave}.
\end{proof}

\subsection{Properties of the traveling wave profile $\omega(x)$}\label{sec:fixed_point_properties}
\subsubsection{Explicit solution for the $a=0$ case.}\label{sec:explicit_solution_travelingwave} In this subsection, we will show that the only fixed point of $\mtx{R}_0$ is the explicit solution
\[\omega(x)=\frac{1}{1+x^2}.\] Let $\omega_0$ be a fixed point of $\mtx{R}_0$. Then the convexity of $\omega_0(\sqrt{s})$ implies that $\omega_0(x)$ is differentiable almost everywhere, and
\[\omega_0(x)=\mtx{R}_0(\omega_0)(x)=\mathrm{e}^{-\frac{\mtx{T}(\omega_0)(x)}{r(\omega_0)}}\] implies that $\omega_0$ does not have any zero on $\mathbb{R}$. Therefore, by the result established in \cite{huang2025multiscale} we have
\[
    \omega_0(x)=\frac{2cr(\omega_0)}{1+c^2x^2}
\]
for some constant $c>0$. Recall that we have shown $\omega_0(0)=1$ and $\lim_{x\to 0}\omega_0'(x)/(2x)=-1$ in the proof of Lemma \ref{lem:R_a_properties}. It then directly follows that $2cr(\omega_0)=1$ and $c=1$. Hence we have $\omega_0(x)=1/(1+x^2)$.

\subsubsection{Asymptotic behavior.} Here we classify and characterize the tail behavior of the fixed point $\omega(x)$ as $x \to +\infty$ depending on the sign of the parameter $a$. One interesting phenomenon is that the decay rate of $\omega(x)$ shows a jump discontinuity at $a=0$: the limit of decay rate is $-1$ as $a\to 0$ from below, while the decay rate is $-2$ for the $a=0$ case.
\begin{theorem}\label{thm:asymptotic_behavior}
    Let $\omega$ be a fixed point of $\mtx{R}_a$ for some $a\in(-\infty,1]$. Then, one of the following happens:
    \begin{enumerate}[label=(\arabic*)]
        \item $a<0$: there exist constants $\bar C_a,\tilde C_a>0$ depending on $a$ such that $\bar C_a x^{-1/(1+|a|)}\le \omega(x)\le \tilde C_a x^{-1/(1+|a|)}$ for all $x\ge 1 $.
        \item $a=0$: $\lim_{x\to+\infty}x^2\omega(x)=1$.
        \item $0<a\le1$:  $\omega$ is compactly supported on $[-L_\omega,L_\omega]$ and strictly positive on $(-L_\omega,L_\omega)$, where $L_\omega:=\sup\{x\mid \omega(x)>0\}$. Moreover, at the boundary of the support $\omega(x)$ satisfies
        \[
    \lim_{x\to L_\omega^-}\frac{\omega(x)}{(L_\omega-x)^{1/a}}=\left(\frac{a\mtx{T}(\omega)'(L_\omega)}{r(\omega)}\right)^{1/a}>0.
    \]
    \end{enumerate}

\end{theorem}

\begin{proof} The asymptotic behavior of the $a=0$ case directly follows from the explicit solution $\omega(x)=1/(1+x^2)$. Hence we only need to prove the first and third cases.

    (1): We first show that $\omega'(x)<0$ for all $x>0$. Recall that if $\omega\in \mathbb D_a$ is a fixed point of $\mtx{R}_a$, we have
    \[
        \omega(x)=\left(1-\frac{a\mtx{T}(\omega)(x)}{r(\omega)}\right)_+^{\frac{1}{a}}=\left(1+\frac{|a|\mtx{T}(\omega)(x)}{r(\omega)}\right)^{-\frac{1}{|a|}}.
    \]
    Here the subscript $+$ is omitted since $1+|a|\mtx{T}(\omega)(x)/r(\omega)\ge 1$. Then we have
    \[
        \omega'(x)=-\frac{\mtx{T}(\omega)'(x)}{r(\omega)}\left(1+\frac{|a|\mtx{T}(\omega)(x)}{r(\omega)}\right)^{-\frac{1+|a|}{|a|}}.
    \]
    We can use the calculations in the proof of Lemma \ref{lem:R_a_properties} to show that for any $x>0$,
        \begin{align*}
            \mtx{T}(\omega)'(x)
             & =-\frac{1}{\pi}\int_0^{+\infty}\omega'(y)\ln\left| \frac{x+y}{x-y} \right|\idiff y \ge -\frac{1}{\pi}\int_0^{1/2} \omega'(y)\ln\left| \frac{x+y}{x-y} \right|\idiff y                            \\
             & \ge -\frac{1}{\pi}\int_0^{1/2}2\omega'(1/2) y \ln\left| \frac{x+y}{x-y} \right|\idiff{y} \geq \frac{2\eta_a x^2}{\pi}\int_0^{\frac{1}{2x}}t \ln\left| \frac{1+t}{1-t} \right|\idiff{t} >0,
        \end{align*}
    which leads to $\omega'(x)<0$ for all $x>0$.

    Next we prove the lower and upper bounds of $\omega(x)$. Similarly we have
        \begin{align*}
            \mtx{T}(\omega)'(x)
             & =-\frac{1}{\pi}\int_0^{+\infty}\omega'(y)\ln\left| \frac{x+y}{x-y} \right|\idiff y \ge -\frac{1}{\pi}\int_0^x \omega'(y)\ln\left| \frac{x+y}{x-y} \right|\idiff y                                                  \\
             & \ge -\frac{1}{\pi}\int_0^x  \frac{y\omega'(x)}{x} \ln\left| \frac{x+y}{x-y} \right|\idiff{y}=-\frac{x\omega'(x)}{\pi}\int_0^1 t\ln\left| \frac{1+t}{1-t} \right|\idiff{t} =-\frac{x\omega'(x)}{\pi},
        \end{align*}
    hence
    \[
        -\omega'(x)=\frac{\mtx{T}(\omega)'(x)}{r(\omega)}\left(1+\frac{|a|\mtx{T}(\omega)(x)}{r(\omega)}\right)^{-\frac{1+|a|}{|a|}}\ge-\frac{x\omega'(x)}{\pi r(\omega)}\left(1+\frac{|a|\mtx{T}(\omega)(x)}{r(\omega)}\right)^{-\frac{1+|a|}{|a|}},
    \]
    which provides the upper bound
    \[
        \omega(x)=\left(1+\frac{|a|\mtx{T}(\omega)(x)}{r(\omega)}\right)^{-\frac{1}{|a|}}\le \left( \frac{x}{\pi r(\omega)} \right)^{-\frac{1}{1+|a|}}\le \left( \frac{x}{2 } \right)^{-\frac{1}{1+|a|}},
    \]
    where we have used the estimate $r(\omega)\le 2/\pi$ from Lemma \ref{lem:r_properties}.

    On the other hand, by the definition of $\mtx{T}(\omega)$,
        \begin{align*}
            \mtx{T}(\omega)(x) & =\frac{1}{\pi}\int _0^{+\infty}\omega(y)\ln\left| \frac{x^2-y^2}{y^2} \right|\idiff y  \le\frac{1}{\pi} \left( \frac{x}{2}\right)^{-\frac{1}{1+|a|}}\int _0^{+\infty}\left( \frac{y}{x} \right)^{-\frac{1}{1+|a|}}\left|\ln\left| \frac{x^2-y^2}{y^2} \right|\right|\idiff y \\
                         & =\frac{1}{\pi} 2^{\frac{1}{1+|a|}}x^{\frac{|a|}{1+|a|}}\int_0^{+\infty}t^{-\frac{1}{1+|a|}}\left|\ln\left| \frac{1-t^2}{t^2} \right|\right|\idiff t.
        \end{align*}
    Writing\[
        C_a=\frac{1}{\pi} 2^{\frac{1}{1+|a|}}\int_0^{+\infty}t^{-\frac{1}{1+|a|}}\left|\ln\left| \frac{1-t^2}{t^2} \right|\right|\idiff t,
    \]
    we have
    \[
        \omega(x)\ge \left( 1+\frac{|a|C_a x^{\frac{|a|}{1+|a|}}}{r(\omega)} \right)^{-\frac{1}{|a|}}\ge\left( x^{\frac{|a|}{1+|a|}}+\frac{|a|C_a \pi x^{\frac{|a|}{1+|a|}}}{\eta_a} \right)^{-\frac{1}{|a|}}=\left( 1+\frac{|a|C_a\pi }{\eta_a} \right)^{-\frac{1}{|a|}}x^{-\frac{1}{1+|a|}}
    \]
    for all $x\ge 1$. We have used Lemma \ref{lem:r_properties} to bound $r(\om)$ from below.

    (3): By the convexity of $\omega(\sqrt{s})$ and $\omega'(1/2)\le -\eta_a$, we have for $0<x<1/2$,
    \[
        \omega'(x)\le 2\omega'(1/2)x \le -2\eta_a x,
    \]
    which implies that
    \begin{equation}\label{eqt:regularity_Tprime}
        \mtx{T}(\omega)'(x)=-\frac{1}{\pi}\int_0^{+\infty}\omega'(y)\ln\left| \frac{x+y}{x-y} \right|\idiff y \ge \frac{2\eta_a}{\pi}\int_0^{1/2}  y\ln\left| \frac{x+y}{x-y} \right|\idiff y,
    \end{equation}
    hence
        \begin{align*}
            \mtx{T}(\omega)(x)
             & \ge \frac{2\eta_a}{\pi}\int_{0}^{1/2} y\left( x\ln\left|\frac{x+y}{x-y}\right|+y\ln\left|\frac{x^2-y^2}{y^2}\right| \right)\idiff y  \\
             & =\frac{2\eta_a}{\pi}x^3\int_{0}^{\frac{1}{2x}}\left( t\ln\left| \frac{1+t}{1-t} \right|+t^2\ln\left| \frac{1-t^2}{t^2} \right| \right)\idiff t.
        \end{align*}
    We use L'H\^opital's rule to compute the limit
        \begin{align*}
            &\lim_{x\to +\infty}  x^3\int_{0}^{\frac{1}{2x}}\left( t\ln\left| \frac{1+t}{1-t} \right|+t^2\ln\left| \frac{1-t^2}{t^2} \right| \right)\idiff t \\
            & =\lim_{x\to +\infty}\frac{-\frac{1}{2x^2}\left( \frac{1}{2x}\ln\left| \frac{2x+1}{2x-1} \right|+\frac{1}{4x^2}\ln\left| 4x^2-1 \right| \right)}{-3x^{-4}} \\
             & =\lim_{x\to+\infty}\frac{x}{12}\ln\left| \frac{2x+1}{2x-1} \right|+\lim_{x\to+\infty}\frac{1}{24}\ln\left| 4x^2-1 \right|=+\infty.
        \end{align*}
    This shows that $\mtx{T}(\omega)(x)\to +\infty$ as $x\to +\infty$. Note that the continuity of $\omega$ implies
    \[\mtx{T}(\omega)(x) =\frac{1}{\pi}\int _0^{+\infty}\ln\left| \frac{x^2-y^2}{y^2} \right|\omega(y)\idiff y \]
    is continuous. Hence by the monotonicity of $\mtx{T}(\omega)$ and $\mtx{T}(\omega)(0)=0$ (Lemma \ref{lem:property_of_T}), there exists some $L_\omega>0$ such that $\mtx{T}(\omega)(x)<r(\omega)/a$ for $x\in (0,L_\omega)$, and $\mtx{T}(\omega)(x)\ge r(\omega)/a$ for $x\ge L_\omega$, which implies that $\omega$ is compactly supported on $[-L_\omega,L_\omega]$ and strictly positive on $(-L_\omega,L_\omega)$, and $\omega(L_\omega)=0$. Moreover, note that \eqref{eqt:regularity_Tprime} implies
    \[
    \mtx{T}(\omega)'(L_\omega)\ge \frac{2\eta_a}{\pi}\int_0^{1/2}  y\ln\left| \frac{L_\omega+y}{L_\omega-y} \right|\idiff y>0.
    \] 
    Then using L'H\^opital's rule we can compute that
    \[
    \lim_{x\to L_\omega^-}\frac{\omega(x)}{(L_\omega-x)^{1/a}}=\left(\lim_{x\to L_\omega^-}\frac{1-a\mtx{T}(\omega)(x)/r(\omega)}{L_\omega-x}\right)^{1/a}=\left(\frac{a\mtx{T}(\omega)'(L_\omega)}{r(\omega)}\right)^{1/a}>0.
    \]
\end{proof}

\subsubsection{Regularity.} In this subsection we discuss the regularity of a solution $\omega(x)$ to the traveling wave profile equation \eqref{eqt:traveling_profile} with $\omega\in \mathbb D_a$ being the fixed point of $\mtx{R}_a$ for some $a\in (-\infty,0)\cup(0,1]$. We shall always denote
\[g=\left(1-\frac{a\mtx{T}(\omega)}{r(\omega)}\right)_+\] in the sequel. Recall that
\[
    \omega(x)=\mtx{R}_a(\omega)(x)=\left(1-\frac{a\mtx{T}(\omega)(x)}{r(\omega)}\right)_+^{\frac{1}{a}},
\]
which leads to
\begin{equation}\label{eqt:omega_derivative}
    \omega'(x)=-\frac{\mtx{T}(\omega)'(x)}{r(\omega)}\left(1-\frac{a\mtx{T}(\omega)(x)}{r(\omega)}\right)_+^{\frac{1}{a}-1}=-\frac{\mtx{T}(\omega)'(x)}{r(\omega)}\cdot g(x)^{\frac{1}{a}-1}.
\end{equation}
Since \eqref{eqt:omega_derivative_estimation} yields $|\omega'(x)|\le 2\min\{x,x^{-1}\}$, it is not hard to check by formula \eqref{eqt:T_omega_derivative} that $\mtx{T}(\omega)\in C^1(\mathbb{R})$, and thus $g\in C^1([-L_\omega,L_\omega])$, where
\[
    L_\omega=\sup\{x\mid g(x)>0\}=\sup\{x\mid \omega(x)>0\}.
\]
For the case $a<0$, by Theorem \ref{thm:asymptotic_behavior} we have $\omega(x)$ and $g(x)$ are both strictly positive on $\mathbb{R}$. Hence \eqref{eqt:omega_derivative} implies that $\omega(x)\in C^1(\mathbb{R})$. For the case $a\in(0,1]$, 
we obtain directly from \eqref{eqt:omega_derivative} that $\omega(x)\in C^1([-L_\omega,L_\omega])$ since $1/a-1\ge0$. Then, since Theorem \ref{thm:asymptotic_behavior} implies $\omega(L_\omega)=0$, we have $\omega\in C(\mathbb{R})$.

To obtain higher regularity of $\omega$, we apply a bootstrapping technique.

\begin{theorem}\label{thm:regularity}
    Let $\omega\in \mathbb D_a$ be a fixed point of $\mtx{R}_a$ for some $a\in (-\infty,0)\cup(0,1]$. Denote $L_\omega=\sup\{x\mid \omega(x)>0\}$. Then depending on the sign of $a$, one of the following happens:
    \begin{enumerate}[label=(\arabic*)]
        \item $a<0$: $\omega$ is strictly positive and smooth on $\mathbb{R} $, and $\omega' \in H^k(\mathbb{R} )$ for any $k\ge 0$.
        \item $0<a\le1$: $\omega\in C(\mathbb{R})\cap C^1([-L_\omega,L_\omega])$, $\omega$ is compactly supported on $[-L_\omega,L_\omega]$, and $\omega$ is smooth in the interior of $(-L_\omega,L_\omega)$. Moreover, if $0<a<1$, then $\omega\in C^1(\mathbb{R})$.
    \end{enumerate}
\end{theorem}
\begin{proof}
    The result that $\omega\in C(\mathbb{R})\cap C^1([-L_\omega,L_\omega])$ for all $a\leq 1$ follows from the above discussion. For $0<a<1$, Theorem \ref{thm:asymptotic_behavior} implies that $\omega'=0$ at the support boundary $x=\pm L_\omega$. Since $\omega$ is identically zero outside its support, it follows that $\omega\in C^1(\mathbb{R})$. It therefore remains to establish the higher regularity of $\omega$.
    
    (1): Note that in this case $\omega\le 1$ and $g(x)\ge 1$ for $x\in [0,+\infty)$, which implies that $\omega$ and $1/g$ both belong to $L^\infty(\mathbb{R})$. Then, by \eqref{eqt:omega_derivative} together with the well-known identity $\|\mtx{T}(\omega)''\|_{H^k(\mathbb{R})}=\|\mtx H(\omega')\|_{H^k(\mathbb{R})}=\|\omega'\|_{H^k(\mathbb{R})}$ for any integer $k\ge 0$, we see that $\omega'\in H^k(\mathbb{R})$ implies $\omega'\in H^{k+1}(\mathbb{R})$ for any $k\ge 0$. By \eqref{eqt:omega_derivative_estimation}, we have $|\omega'(x)|\le 2\min\{x,x^{-1}\}$, and hence $\omega'\in L^2(\mathbb{R})$. Therefore, by a standard induction argument, we conclude that $\omega'\in H^k(\mathbb{R})$ for every $k\ge 0$.

    (2): In this case $g'(x)=-a\mtx{T}(\omega)'(x)/r(\omega)\le 0$ for $x\in [0,L_\omega]$. Consequently, for any $0<L'<L_\omega$, we have $g(x)\ge g(L')>0$ for $x\in[-L',L']$. It then follows straightforwardly from Lemma \ref{lem:Hilbert_transform_kth_derivative} that if $\omega\in H_{loc}^k(-L',L')$ for some integer $k\ge 0$, then $\omega\in H_{loc}^{k+1}(-L',L')$. Since $\omega\in C^1([-L',L'])\subset H^1([-L',L'])$, we immediately obtain by recursion that $\omega\in H^k([-L',L'])$ for any integer $k\ge 0$. This further implies that $\omega\in H^k_{loc}(-L_\omega,L_\omega)$ for any integer $k\ge 0$ since $L'<L_\omega$ is arbitrary, and thus $\omega$ is smooth in the interior of $(-L_\omega,L_\omega)$.
\end{proof}

We finish this section with a proof of Theorem \ref{thm:traveling_wave_existence} based on the preceding results.
\begin{proof}[Proof of Theorem \ref{thm:traveling_wave_existence}]
    The existence of solution $\bar\omega(x)$ of \eqref{eqt:traveling_wave_ansatz} for all $a\le 1$ follows from Theorem \ref{thm:traveling_wave_existence_sec7}. The decaying features and the regularity properties of $\bar\omega(x)$ follow from Theorem \ref{thm:asymptotic_behavior} and Theorem \ref{thm:regularity}, respectively. Note that the notation $L_\omega$ is exactly $L_a$. Finally, the explicit solution for the $a=0$ case is obtained in Section \ref{sec:explicit_solution_travelingwave}.
\end{proof}

\appendix
\section{Translation invariance}\label{app:translation_invariance}
The following lemma shows that solutions of the gCLM model with the same initial data but different choices of $x_0(t)$ such that $u(x_0(t),t)=0$ differ only by a time-dependent translation in $x$.

\begin{lemma}\label{lem:translation_invariance}
    Let $\omega_0$ be some suitable initial data and $x_1(t),x_2(t)$ be two continuous trajectories. Suppose that $(\omega_i,u_i)_{i=1,2}$ are the solutions to the gCLM model \eqref{eqt:gCLM} with $u_i(x_i(t),t)=0$, respectively, both starting from the same initial data $\omega_0$. Let $c(t)$ be a time-dependent quantity that solves 
    \[
    c'(t)=-a u_1(x_2(t)-c(t),t)=-\frac{a}{\pi}\int_{\mathbb{R}}\omega_1(y,t)\ln\left| \frac{x_2(t)-c(t)-y}{x_1(t)-y} \right|\idiff y,\quad c(0)=0.
    \]
    Then it holds that
    \[
    (\omega_2(x,t),u_2(x,t))=(\omega_1(x-c(t),t),u_1(x-c(t),t)-u_1(x_2(t)-c(t),t)).
    \]
\end{lemma}

\begin{proof}
By the definition \eqref{eqt:biot_savart_law} of $u$  we have
\begin{equation}\label{eqt:app_u_definition}
    u_i(x,t)=\int_{x_i(t)}^x H(\omega_i)(y,t)\idiff{y}=\frac{1}{\pi}\int_{\mathbb{R}} \omega_i(y,t)\ln\left| \frac{x-y}{x_i(t)-y} \right|\idiff{y},\quad i=1,2.
\end{equation}
We only need to show that there is some time-dependent translation $c(t)$ such that
\begin{equation}\label{eqt:translation_invariance}
    \omega_2(x,t)=\omega_1(x-c(t),t).
\end{equation}
In fact, if \eqref{eqt:translation_invariance} holds, then by \eqref{eqt:app_u_definition} we have
    \begin{align*}
        u_2(x,t)
         & =\frac{1}{\pi}\int_{\mathbb{R}} \omega_2(y,t)\ln\left| \frac{x-y}{x_2(t)-y} \right|\idiff y \\
         & =\frac{1}{\pi}\int_{\mathbb{R}} \omega_1(y-c(t),t)\ln\left| \frac{(x-c(t))-(y-c(t))}{(x_2(t)-c(t))-(y-c(t))} \right|\idiff y \\
         & =\frac{1}{\pi}\int_{\mathbb{R}} \omega_1(y,t)\ln\left| \frac{x-c(t)-y}{x_1(t)-y} \right|\idiff y-\frac{1}{\pi}\int_{\mathbb{R}}\omega_1(y,t)\ln\left| \frac{x_2(t)-c(t)-y}{x_1(t)-y} \right|\idiff y \\
         & =u_1(x-c(t),t)-u_1(x_2(t)-c(t),t).
    \end{align*}
Substituting this into the gCLM equation for $\omega_2$ and comparing it with that for $\omega_1$, we find that \eqref{eqt:translation_invariance} holds provided that $c(t)$ satisfies the ODE
\[
    c'(t)=-a u_1(x_2(t)-c(t),t)=-\frac{a}{\pi}\int_{\mathbb{R}}\omega_1(y,t)\ln\left| \frac{x_2(t)-c(t)-y}{x_1(t)-y} \right|\idiff y.
\]
Note that $c(0)=0$ since both solutions start from the same initial data. The existence and uniqueness of $c(t)$ then follows from the standard ODE theory. The lemma is thus proved.
\end{proof}

\section{Explicit singular solution to the profile equation when $a<0$}\label{app:explicit_singular_solution}
We verify that the explicit singular function given in Theorem \ref{thm:explicit_singular_solution} is a solution to the self-similar profile equation \eqref{eqt:steady_profile}. Recall that for a fixed $a<0$, the expressions of $(\bar\Omega_a, \bar c_{l,a}, \bar c_{\omega,a})$ are given by
\begin{equation}
    \bar\Omega_a(X)=-\frac{\boldsymbol{1}_{\{X>1\}}\sin(\pi \mu)}{|1-X|^\mu},\quad \bar c_{l,a}=1-a,\quad \bar c_{\omega,a}=-1,
\end{equation}
where $\mu=1/(1-a)\in(0,1)$. Then the Hilbert transform of $\bar \Omega_a$ is
\[
    (\bar U_a)_X=\mtx H(\bar\Omega_a)(X)= \frac{\boldsymbol{1}_{\{X<1\}}+\boldsymbol{1}_{\{X>1\}}\cos(\pi \mu)}{|1-X|^\mu}.
\]

\begin{proof}[Proof of Theorem \ref{thm:explicit_singular_solution}]
We can calculate that for $X\neq 1$,
\[
    (\bar \Omega_a)_X(X)=\frac{\mu\boldsymbol{1}_{\{X>1\}}\sin (\pi \mu)}{|1-X|^{1+\mu}},
\]
and for all $X\in \mathbb{R} $,
    \begin{align*}
        \bar U_a(X)
         & =\int_0^X (\bar U_a)_X(Y)\idiff Y=\frac{1}{1-\mu}+\frac{|1-X|^{1-\mu}}{1-\mu}(\boldsymbol{1}_{\{X>1\}}\cos(\pi \mu)-\boldsymbol{1}_{\{X<1\}}) \\
         & =\frac{1}{1-\mu}+\frac{X-1}{1-\mu}(\bar U_a)_X(X).
    \end{align*}
We then compute that
\begin{align*}
     & -\bar c_{l,a} X (\bar \Omega_a)_X+\bar c_{\omega,a}\bar \Omega_a+(\bar U_a)_X\bar \Omega_a-a\bar U_a(\bar \Omega_a)_X                                                                                                          \\
     &=-\bar c_{l,a} X\cdot \frac{\mu\boldsymbol{1}_{\{X>1\}}\sin (\pi \mu)}{|1-X|^{1+\mu}}-\bar c_{\omega,a}\cdot \frac{\boldsymbol{1}_{\{X>1\}}\sin(\pi \mu)}{|1-X|^\mu}-\frac{\boldsymbol{1}_{\{X>1\}}\sin(\pi \mu)\cos(\pi \mu)}{|1-X|^{2\mu}} \\
     & \quad +\frac{\boldsymbol{1}_{\{X>1\}}\sin(\pi \mu)}{|1-X|^{1+\mu}} + \frac{\boldsymbol{1}_{\{X>1\}}\sin(\pi \mu)\cos(\pi \mu)}{|1-X|^{2\mu}} \\
     &=\left( -\bar c_{l,a} \mu X-\bar c_{\omega,a}(X-1)+1 \right)\frac{\boldsymbol{1}_{\{X>1\}}\sin(\pi \mu)}{|1-X|^{1+\mu}} \\
     &=0.
\end{align*}
Therefore, $(\bar\Omega_a,\bar c_{l,a},\bar c_{\omega,a})$ exactly solves the self-similar profile equation \eqref{eqt:steady_profile} for all $a<0$.
\end{proof}

\section{Details of the numerical implementation}\label{app:numerics}
\subsection{Numerical method for computing the Hilbert transform and its integral}\label{app:numerics_hilbert_transform}
Here we describe the numerical method we use for computing the Hilbert transform $\mtx H(f)(x)$ for a given function $f$. We compute the Hilbert transform $\mtx H(f)$ by approximating $f$ with a cubic spline on the computational domain $[-M,M]$ and evaluating the Hilbert transforms of the spline basis functions analytically. Let ${x_i}{i=0}^N$ be the grid points, with $x_0=-M$ and $x_N=M$, and let $f_i=f(x_i)$. We approximate $f$ as
\begin{equation}\label{eqt:cubic_spline_approximation}
    f(x)\approx f_i P_i(x)+f_{i+1}P_{i+1}(x)+f_i'Q_i(x)+f_{i+1}'Q_{i+1}(x),\quad  x\in [x_i,x_{i+1}),
\end{equation}
where
    \begin{align*}
        P_i(x) & =\left( 1+2\frac{x-x_i}{x_{i-1}-x_i} \right)\left( \frac{x-x_{i-1}}{x_i-x_{i-1}} \right)^2\boldsymbol{1}_{\{x_{i-1}\le x <x_i\}}+\left( 1+2\frac{x-x_i}{x_{i+1}-x_i} \right)\left( \frac{x-x_{i+1}}{x_i-x_{i+1}} \right)^2\boldsymbol{1}_{\{x_i\le x <x_{i+1}\}}, \\
        Q_i(x) & =(x-x_i)\left( \frac{x-x_{i-1}}{x_i-x_{i-1}} \right)^2\boldsymbol{1}_{\{x_{i-1}\le x <x_{i}\}}+(x-x_i)\left( \frac{x-x_{i+1}}{x_i-x_{i+1}} \right)^2\boldsymbol{1}_{\{x_i\le x <x_{i+1}\}}
    \end{align*}
combine the Hermite basis functions on the two adjacent intervals $[x_{i-1},x_i]$ and $[x_i,x_{i+1}]$. These combined basis functions are compactly supported and continuously differentiable, so their Hilbert transforms are bounded. The derivatives on the grid points $f_i'$ are computed using the standard cubic spline method with end conditions $f''(x_0)=f''(x_N)=0$. 

The Hilbert transforms of the basis functions are given explicitly by
\[
    \mtx H(P_i)(x)  =A(l)-A(r),\quad \mtx H(Q_i)(x) =(x_{i-1}-x_i)B(l)-(x_{i+1}-x_i)B(r),
\]
where
\[
    l=\frac{x_{i-1}-x_i}{x-x_i},\quad r=\frac{x_{i+1}-x_i}{x-x_i},
\]
\[
    A(s)=\frac{-5s^3-12s^2+12s+6(s^3-3s+2)\ln|1-s|}{6\pi s^3},\quad B(s)=\frac{2s^3-9s^2+6s+6(s-1)^2\ln|1-s|}{6\pi s^3}.
\]
Note that when $s\to 0$ $A(s)$ tends to $0$, while
\[
    \frac{-5s^3-12s^2+12s}{6\pi s^3},\quad \frac{6(s^3-3s+2)\ln|1-s|}{6\pi s^3}
\]
both tend to $\infty$, and such behavior would lead to large numerical errors introduced by cancellation. Hence to compute $A(s)$ and $B(s)$ accurately when $s$ is close to $0$, we use the following minimax approximation for $(s-1)(\ln|1-s|+s+s^2/2+s^3/3)/s^4$ on $[-0.5,0.5]$ computed by Mathematica 14.2:
\begin{align*}
     & \frac{(s-1)\left(\ln|1-s|+s+\frac{s^2}{2}+\frac{s^3}{3}\right)}{s^4}      \approx R_1(s)    \\
     & = 0.25-0.04999999999999866 s-0.03333333333338085 s^2-0.02380952381002668 s^3                \\
     & \quad -0.01785714284817429 s^4-0.01388888883280194 s^5-0.0111111117723456 s^6               \\
     & \quad  -0.00909091199993801 s^7-0.007575732470825958 s^8-0.006410171819766842 s^9           \\
     & \quad  -0.005495065477419788 s^{10}-0.004763424514296551 s^{11}-0.0041587945938036 s^{12}   \\
     & \quad  -0.003658683032021482 s^{13}-0.003340203914467509 s^{14}-0.00306305502520356 s^{15}  \\
     & \quad  -0.002195292685440593 s^{16}-0.001652974757186405 s^{17}-0.003865944631093168 s^{18} \\
     & \quad  -0.004466133542538447 s^{19}+0.002210270189446115 s^{20}+0.003532901236598389 s^{21} \\
     & \quad  -0.006394622541761529 s^{22}-0.006964811435003471 s^{23}
\end{align*}
with a maximum absolute error $2.011826942268783\times 10^{-16}$. This gives the following formulas for evaluating $A(s)$ and $B(s)$ on $[-0.5,0.5]$:
\[
    A(s)\approx -\frac{3s^2+2s^3}{6\pi}+\frac{-2s+s^2+s^3}{\pi} R_1(s),\quad B(s)\approx\frac{s-2s^2}{6\pi}+\frac{-s+s^2}{\pi} R_1(s).
\]
Finally, for the grid point $x_i$, both $l$ and $r$ tend to infinity; hence we need to compute the limit explicitly. By taking the limit $x\to x_i$ yields
\[
    \mtx H(P_i)(x_i)=\frac{1}{\pi}\ln\left| \frac{x_{i-1}-x_i}{x_{i+1}-x_i} \right|,\quad \mtx H(Q_i)(x_i)=\frac{x_{i-1}-x_{i+1}}{3\pi}
\]
for $i=1,2,\ldots,N-1$. Combining the basis-function transforms with the spline coefficients gives the numerical approximation to $\mtx H(f)$ at the grid points.

To compute the integral of the Hilbert transform, we still use the cubic spline approximation \eqref{eqt:cubic_spline_approximation} for $f$ and integrate $\mtx H(P_i)(x)$ and $\mtx H(Q_i)(x)$ analytically:
\[
    -(-\Delta)^{-1/2}P_i(x)=\int \mtx H(P_i)(x)\idiff x= C(x_i-x_{i-1},l) + C(x_{i+1}-x_i,r),
\]
\[
    -(-\Delta)^{-1/2}Q_i(x)=\int \mtx H(Q_i)(x)\idiff x=D(x_i-x_{i-1},l)-D(x_{i+1}-x_i,r),
\]
where
\[
    C(d,s)=\frac{d\left( -19s^4+8s^3+18s^2-12s+(12s^4-24s^3+24s-12)\ln|1-s|+12s^4\ln\left|\frac{d}{s}\right| \right)}{24\pi s^4},
\]
\[
    D(d,s)=\frac{d^2\left( 13s^4+36s^3-78s^2+36s+(-12s^4+72s^2-96s+36)\ln|1-s|-12s^4 \ln\left|\frac{d}{s}\right|\right)}{144\pi s^4}.
\]
To avoid cancellation errors near $s=0$, we again use a minimax approximation for $(s-1)^3(\ln|1-s|+s+s^2/2+s^3/3)/s^4$ on $[-0.5,0.5]$ computed by Mathematica 14.2:
\begin{align*}
     & \frac{(s-1)^3\left(\ln|1-s|+s+\frac{s^2}{2}+\frac{s^3}{3}\right)}{s^4} \approx R_2(s)           \\
     & =0.25-0.5500000000000005 s+0.3166666666666673 s^2-0.007142857142718389 s^3                      \\
     & \quad -0.003571428571658581 s^4-0.001984126996758493 s^5-0.001190476169455116 s^6               \\
     & \quad -0.0007575752302443226 s^7-0.0005050513541658974 s^8-0.0003496624830794595 s^9            \\
     & \quad -0.0002497317240860826 s^{10}-0.0001829814020322089 s^{11}-0.0001376024384624091 s^{12}   \\
     & \quad -0.0001065281770442371 s^{13}-0.00007978095495689995 s^{14}-0.00005609995868764908 s^{15} \\
     & \quad -0.00006109725530208673 s^{16}-0.00007171425465849955 s^{17}                              \\
     & \quad -6.388916781803325\times 10^{-6}s^{18}+0.00003490077048616611 s^{19}                      \\
     & \quad -0.00006504398511532267 s^{20}-0.0000856727088008827 s^{21}
\end{align*}
with a maximum absolute error $4.143007531279983\times 10^{-17}$. For $s\in[-0.5,0.5]$, we then evaluate $C(d,s)$ and $D(d,s)$ using
\[
    C(d,s)\approx \frac{d\left( -3+2s^2-4s^3+12\ln\left|\frac{d}{s}\right| \right)}{24\pi}+\frac{d(1+s)}{2\pi} R_2(s),
\]
\[
    D(d,s)\approx\frac{d^2\left( 9-12s+6s^2+4s^3-12\ln\left|\frac{d}{s}\right| \right)}{144\pi}-\frac{d^2(3+s)}{12\pi} R_2(s).
\]
For the grid point $x_i$, taking the limit $x\to x_i$ yields
\[
    -(-\Delta)^{-1/2}P_i(x_i)=\frac{-19(d_1+d_2)+12 d_1\ln d_1+12 d_2\ln d_2}{24\pi},
\]
\[
    -(-\Delta)^{-1/2}Q_i(x_i)=\frac{13d_1^2-13d_2^2-12d_1^2\ln d_1+12 d_2^2 \ln d_2}{144\pi},
\]
where $d_1=x_i-x_{i-1}, d_2=x_{i+1}-x_i$.
\subsection{Numerical method for simulating the dynamic rescaling equation}

Here we describe our numerical approach for solving the dynamic rescaling equation \eqref{eqt:dynamic_rescaling}. To prevent round-off and discretization errors from degrading the prescribed high-order vanishing of the solution at the origin, we evolve
\[
    f(X,\tau)=\frac{\Omega(X,\tau)}{X^k}
\]
instead of $\Omega$ itself, where $k\ge 3$ is an odd integer. Substituting $\Omega=X^k f$ into \eqref{eqt:dynamic_rescaling} gives
\[
    f_\tau + (c_l X + a U)\, f_X
    =\Bigl( c_\omega + U_X - k\Bigl(c_l + a\frac{U}{X}\Bigr)\Bigr) f.
\]
The choice of $k$ depends on the sign of $a$. For $a>0$, we choose $k$ to match the vanishing order of the initial data. For $a<0$, we fix $k=3$, since the computations in Section \ref{sec:a_le_0} indicate that the evolving profile develops infinite vanishing order at the origin, provided the initial data vanish to order at least $3$. The velocity $U$ and its derivative $U_X$ are computed from $\Omega=X^k f$ using the method described in Appendix \ref{app:numerics_hilbert_transform}. We discretize the advection term $(c_l X + aU)f_X$ using a fifth-order WENO scheme \cite{jiang1999high} and advance in time with the ten-stage, fourth-order Strong Stability Preserving Runge--Kutta method SSPRK$(10,4)$ \cite{gottlieb2011strong}.

For $a>0$, the convergence criterion is set as $\|\Omega_\tau\|_{L^\infty} < 1.5\times10^{-8}$. For $a<0$, since the profile becomes singular at $X=1$ under our normalization conditions, we exclude the neighborhood of the singularity from the convergence criterion and use
\[
    \|\Omega_\tau\boldsymbol{1}_{\{0 \le X < 0.9\text{ or }X>1.1\}}\|_{L^\infty} < 10^{-8}.
\]

In Section \ref{sec:numerical_multiscale}, to better resolve the developing singularity for $a<0$, we use an adaptive mesh refinement strategy, which is explained as follows. Suppose that $\Omega$ is the numerical solution on the current mesh. Let $X_m$ denote the maximum location of $|\Omega|$, and let $X_1,X_2$ be the nearest pair (obtained by standard cubic splines interpolation) such that $X\in(X_1,X_2)$ and 
\[
    |\Omega(X_1)|=|\Omega(X_2)|=0.5\|\Omega\|_{L^\infty}.
\]
Now for some $\rho_{\max}$ to be determined, we consider a computational grid $\{\rho_i\}_{i=0}^N$ that uniformly discretizes the interval $[0,\rho_{\max}]$, where $N=\rho_{\max}/0.01$ is the total number of grid points. Then the new adaptive grid points for the actual coordinate $X$, denoted by $\{\tilde X_i\}_{i=0}^N$, are given as
\[
    \tilde X_i=\tilde X(\rho_i)=X_m(1-\cosh\rho_i)+\sqrt{c+X_m^2}\,\sinh\rho_i,
\]
where the positive parameter $c$ controls the mesh density near $X_m$. We choose $c$ so that the interval $[X_1,X_2]$ contains $N_{\mathrm{bulk}}=600$ grid points.
We restrict our computation on a sufficiently large domain $[0,M]$ for $X$ with $M=10^{10}$. We therefore determine $\rho_{\max}$ by making that $\tilde X(\rho_{\max})=M$. Finally, we interpolate the solution from the old mesh onto the new one using standard cubic splines interpolation. This completes a mesh refinement step.

When a mesh is generated, we record the peak width $|X_1-X_2|$ at this moment as $W$. As the numerical solution evolved, we keep tracking the maximum location $X_m$ and the nearest half-maximum location $X_1,X_2$. A mesh refinement step (as described above) is conducted whenever (i) $X_m$ has moved a distance more than $W/4$ or (ii) the current peak width $|X_1-X_2|$ is less than $W/2$. 

However, in Section \ref{sec:a_gt_0} and Section \ref{sec:a_le_0}, mesh refinement is not needed. For the $a>0$ case in Section \ref{sec:a_gt_0}, the profile remains regular, so we use a fixed mesh generated by manually choosing $X_m=1$, $X_1=0.99$, and $X_2=1.01$, and we do not change mesh afterwards. In Section \ref{sec:a_le_0}, our focus is on the outer profile rather than the detailed structure near the singularity at $X=1$. We therefore use a fixed mesh generated from the initial data and rely on the WENO5 discretization to capture the sharp front near $X=1$.

\subsection{The iterative method for computing the traveling wave profile}\label{apx:iterative_method}
Recall the nonlinear operator $R_a$ introduced in Section \ref{sec:existence_traveling_wave}:
\begin{equation}\label{eqt:operator_R_a_recall}
    R_a(\omega)(x)=\left(1-\frac{a\,T(\omega)(x)}{r(\omega)}\right)_+^{\frac1a},
\end{equation}
where
\[
    T(\omega)(x)=\frac1\pi\int_0^{+\infty}\ln\left|\frac{x^2-y^2}{y^2}\right|\omega(y)\,\idiff y,
    \qquad
    r(\omega)=\frac1\pi\int_0^{+\infty}\frac{\omega(0)-\omega(y)}{y^2}\,\idiff y.
\]
By Proposition \ref{prop:fixed_point_is_traveling_wave}, any fixed point of $R_a$ corresponds to a traveling-wave profile of the gCLM model. To compute such a profile numerically, we apply the fixed-point iteration
\[
    \omega^{(n+1)} = R_a\bigl(\omega^{(n)}\bigr),\qquad n=0,1,2,\ldots,
\]
initialized with
\[
    \omega^{(0)}(x)=\frac{1}{1+x^2}.
\]
At each step, the operator $T(\omega)$ is evaluated using the method described in Appendix \ref{app:numerics_hilbert_transform}. The iteration is terminated once the increment satisfies
\[
    \bigl\|\omega^{(n+1)}-\omega^{(n)}\bigr\|_{L^\infty(\mathbb{R})} < 10^{-8}.
\]
Figure \ref{fig:travelingwaves} shows traveling wave profiles computed by this method for $-2\le a\le1$. 

\begin{figure}[htbp]
    \begin{subfigure}{0.95\textwidth}
        \includegraphics[width=\textwidth]{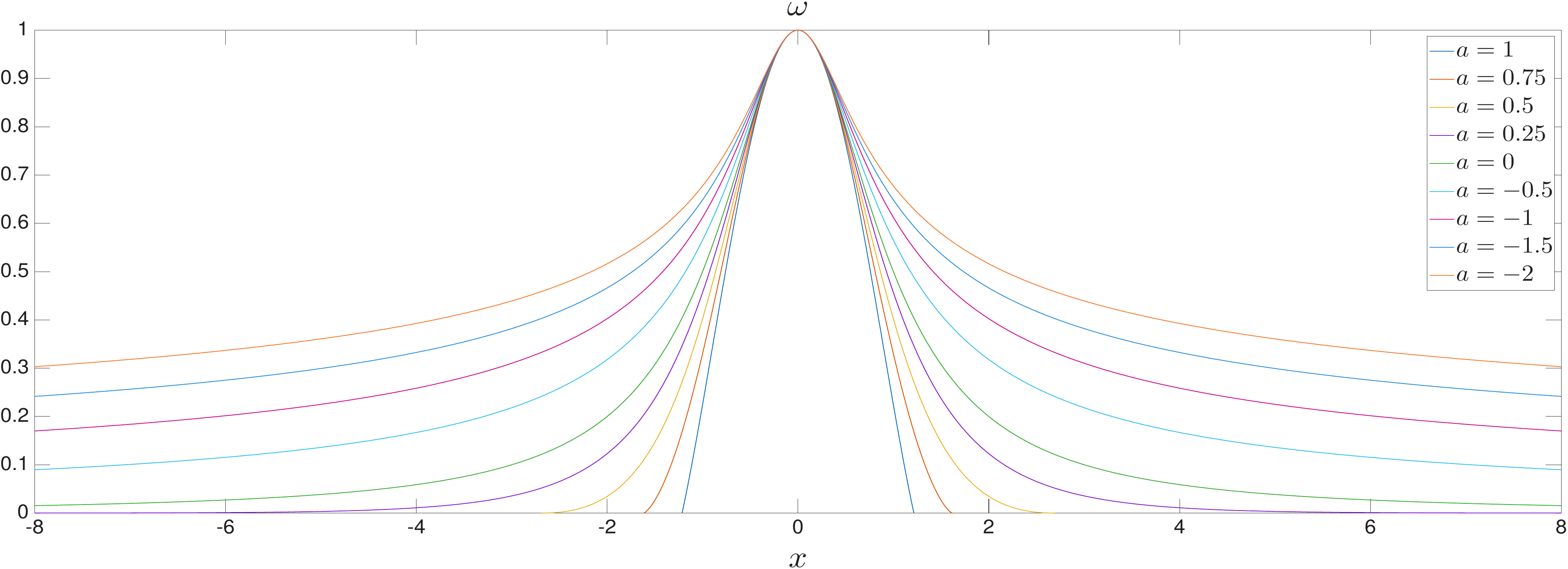}
    \end{subfigure}
    \caption{Traveling wave profiles $\omega(x)$ computed by the fixed-point iterative method for $-2\le a\le1$. For $a>0$, we only plot $\omega(x)$ inside its support.}
    \label{fig:travelingwaves}
\end{figure}

\section{On the Hilbert transform}
\begin{lemma}\label{lem:Hilbert_transform_kth_derivative}
    Given a function $\omega$, suppose that $\|x^\delta\omega\|_{L^\infty(\mathbb{R} )}=\sup_{x\in \mathbb{R} }|x|^\delta|\omega(x)|<+\infty$ for some $\delta>0$. If $\omega$ is supported on $(A,B)$ for some $A<B$ and $\omega\in H_{loc}^k(A,B)$ for some integer $k\ge 0$, then $\mtx H(\omega) \in H_{loc}^k(A,B)$.
\end{lemma}
\begin{proof}
    We first prove a formula for the $k$-th derivative of the Hilbert transform $\mtx{H}(\omega)$. If $\omega\in H_{loc}^k(A,B)$, then for any $A<a<b<B$ and any $x\in (a,b)$,
    \begin{equation}\label{eqt:Hilbert_transform_kth_derivative}
        \mtx{H}(\omega)^{(k)}(x)=\mtx H(\boldsymbol{1}_{\{a\le x\le b\}}\omega^{(k)})(x)+g_{a,b}^{(k)}(x)+\sum_{j=0}^{k-1}f_{a,b,j}^{(k-j)}(x),
    \end{equation}
    where the summation is $0$ if $k=0$, and
    \[
        \mtx{H}(\boldsymbol{1}_{\{a\le x\le b\}}\omega^{(k)})(x)=\frac{1}{\pi} P.V.\int_a^b\frac{\omega^{(k)}(y)}{x-y}\idiff y,
    \]
    \[
        g_{a,b}(x)=\frac{1}{\pi}\int_{-\infty}^a\frac{\omega(y)}{x-y}\idiff y+\frac{1}{\pi}\int_b^{+\infty}\frac{\omega(y)}{x-y}\idiff y,
    \]
    \[
        f_{a,b,j}(x):=\frac{1}{\pi}\left( \omega^{(j)}(a)\ln|x-a|-\omega^{(j)}(b)\ln|x-b| \right),\quad j=0,1,2,\ldots.
    \]
    We prove this formula by induction. The base $k=0$ is trivial:
    \[
        \mtx H(\omega)(x)=\frac{1}{\pi} P.V.\int_{-\infty}^{+\infty}\frac{\omega(y)}{x-y}\idiff y=\frac{1}{\pi} P.V.\int_{a}^b\frac{\omega(y)}{x-y}\idiff y+g_{a,b}(x).
    \]
    Now suppose that \eqref{eqt:Hilbert_transform_kth_derivative} holds for some integer $k\ge 0$. We need to show that it is also true for $k+1$. Under the assumption that $\omega\in H_{loc}^{k+1}(A,B)$, we can use integration by parts to rewrite the first term on the right-hand side of \eqref{eqt:Hilbert_transform_kth_derivative} as
        \begin{align*}
            \mtx H(\boldsymbol{1}_{\{a\le x\le b\}}\omega^{(k)})(x)
             & =\frac{1}{\pi} P.V.\int_a^b\frac{\omega^{(k)}(y)}{x-y}\idiff y                                                 \\
             & =-\frac{1}{\pi}\omega^{(k)}(y)\ln|x-y|\Big|_{y=a}^{y=b}+\frac{1}{\pi}\int_a^b\omega^{(k+1)}(y)\ln|x-y|\idiff y \\
             & =f_{a,b,k}(x)+\frac{1}{\pi}\int_a^b\omega^{(k+1)}(y)\ln|x-y|\idiff y.
        \end{align*}
    Note that $\omega^{(k)}(a)$ and $\omega^{(k)}(b)$ are finite since $\omega\in H_{loc}^{k+1}(A,B)$. It then follows from the induction hypothesis that, for $x\in(a,b)$,
        \begin{align*}
            \mtx {H}(\omega)^{(k+1)}(x)
             & =\left( \mtx H\left( \boldsymbol{1}_{\{a\le x\le b\}}\omega^{(k)} \right)(x) \right)'+\left( g_{a,b}^{(k)}(x)+\sum_{j=0}^{k-1}f_{a,b,j}^{(k-j)}(x) \right)' \\
             & = f'_{a,b,k}(x)+\frac{1}{\pi} P.V.\int_a^b\frac{\omega^{(k+1)}(y)}{x-y}\idiff y+g_{a,b}^{(k+1)}(x)+\sum_{j=0}^{k-1}f_{a,b,j}^{(k+1-j)}(x)                   \\
             & =\frac{1}{\pi} P.V.\int_a^b\frac{\omega^{(k+1)}(y)}{x-y}\idiff y+g_{a,b}^{(k+1)}(x)+\sum_{j=0}^{k}f_{a,b,j}^{(k+1-j)}(x).
        \end{align*}
    Hence \eqref{eqt:Hilbert_transform_kth_derivative} holds for $k+1$. This completes the induction.

    We then use \eqref{eqt:Hilbert_transform_kth_derivative} to prove the lemma. Note that under the assumptions of the lemma, it is easy to see that $g_{a,b}(x)$ and $f_{a,b,j}(x),j=0,1,\ldots,k-1$, are infinitely smooth in the interior of $(a,b)$, and thus $g_{a,b},f_{a,b,j}\in H_{loc}^k(a,b)$. As for the first term on the right-hand side of \eqref{eqt:Hilbert_transform_kth_derivative}, we have
    \begin{align*}
        \left\|\mtx H\left( \boldsymbol{1}_{\{a\le x\le b\}}\omega^{(k)} \right)\right\|_{L^2([a,b])}
        &\le\left\|\mtx H\left( \boldsymbol{1}_{\{a\le x\le b\}}\omega^{(k)} \right)\right\|_{L^2(\mathbb{R} )}=\left\|\boldsymbol{1}_{\{a\le x\le b\}}\omega^{(k)}\right\|_{L^2(\mathbb{R} )}\\
        &=\|\omega\|_{\dot H^k([a,b])}<+\infty.
        \end{align*}
    Therefore, \eqref{eqt:Hilbert_transform_kth_derivative} implies that $\mtx H(\omega)\in H_{loc}^k(a,b)$. Since this is true for any $A<a<b<B$, we immediately have $\mtx H(\omega)\in H_{loc}^k(A,B)$.
\end{proof}

\textbf{Acknowledgements.} The authors are supported by the National Key R\&D Program of China under the grant 2021YFA1001500. DH and XW are also supported by the National Natural Science Foundation of China under the grant NSFC No. 12288101.

\bibliographystyle{myalpha}
\bibliography{reference}

@article{chen2021finite,
  title={On the finite time blowup of the {D}e {G}regorio model for the 3{D} {E}uler equations},
  author={Chen, Jiajie and Hou, Thomas Y and Huang, De},
  journal={Communications on Pure and Applied Mathematics},
  volume={74},
  number={6},
  pages={1282--1350},
  year={2021},
  publisher={Wiley Online Library}
}

@article{chen2020singularity,
  title={Singularity formation and global well-posedness for the generalized {C}onstantin--{L}ax--{M}ajda equation with dissipation},
  author={Chen, Jiajie},
  journal={Nonlinearity},
  volume={33},
  number={5},
  pages={2502--2532},
  year={2020},
  publisher={IOP Publishing}
}

@article{constantin1985simple,
  title={A simple one-dimensional model for the three-dimensional vorticity equation},
  author={Constantin, Peter and Lax, Peter D and Majda, Andrew},
  journal={Communications on Pure and Applied Mathematics},
  volume={38},
  number={6},
  pages={715--724},
  year={1985},
  publisher={Wiley Online Library}
}

@article{elgindi2020effects,
  title={On the effects of advection and vortex stretching},
  author={Elgindi, Tarek M and Jeong, In-Jee},
  journal={Archive for Rational Mechanics and Analysis},
  volume={235},
  number={3},
  pages={1763--1817},
  year={2020},
  publisher={Springer}
}

@article{jia2019gregorio,
  title={On the {D}e {G}regorio modification of the {C}onstantin--{L}ax--{M}ajda model},
  author={Jia, Hao and Stewart, Samuel and Sverak, Vladimir},
  journal={Archive for Rational Mechanics and Analysis},
  volume={231},
  number={2},
  pages={1269--1304},
  year={2019},
  publisher={Springer}
}

@article{lushnikov2021collapse,
  title={Collapse Versus Blow-Up and Global Existence in the Generalized {C}onstantin--{L}ax--{M}ajda Equation},
  author={Lushnikov, Pavel M and Silantyev, Denis A and Siegel, Michael},
  journal={Journal of Nonlinear Science},
  volume={31},
  number={5},
  pages={1--56},
  year={2021},
  publisher={Springer}
}

@article{de1990one,
  title={On a one-dimensional model for the three-dimensional vorticity equation},
  author={De Gregorio, Salvatore},
  journal={Journal of Statistical Physics},
  volume={59},
  number={5--6},
  pages={1251--1263},
  year={1990},
  publisher={Springer}
}

@article{de1996partial,
  title={A partial differential equation arising in a {1D} model for the {3D} vorticity equation},
  author={De Gregorio, Salvatore},
  journal={Mathematical Methods in the Applied Sciences},
  volume={19},
  number={15},
  pages={1233--1255},
  year={1996},
  publisher={Wiley Online Library}
}

@article{okamoto2008generalization,
  title={On a generalization of the {C}onstantin--{L}ax--{M}ajda equation},
  author={Okamoto, Hisashi and Sakajo, Takashi and Wunsch, Marcus},
  journal={Nonlinearity},
  volume={21},
  number={10},
  pages={2447--2461},
  year={2008},
  publisher={IOP Publishing}
}

@article{lei2020constantin,
  title={On the {C}onstantin--{L}ax--{M}ajda model with convection},
  author={Lei, Zhen and Liu, Jie and Ren, Xiao},
  journal={Communications in Mathematical Physics},
  volume={375},
  number={1},
  pages={765--783},
  year={2020},
  publisher={Springer}
}

@article{chen2021slightly,
  title={On the Slightly Perturbed {D}e {G}regorio Model on $ {S}^1$ },
  author={Chen, Jiajie},
  journal={Archive for Rational Mechanics and Analysis},
  volume={241},
  number={3},
  pages={1843--1869},
  year={2021},
  publisher={Springer}
}

@article{landman1988rate,
  title={Rate of blowup for solutions of the nonlinear {S}chr{\"o}dinger equation at critical dimension},
  author={Landman, Michael J and Papanicolaou, George C and Sulem, Catherine and Sulem, Pierre-Louis},
  journal={Physical Review A},
  volume={38},
  number={8},
  pages={3837--3843},
  year={1988},
  publisher={APS}
}

@article{mclaughlin1986focusing,
  title={Focusing singularity of the cubic {S}chr{\"o}dinger equation},
  author={McLaughlin, David W and Papanicolaou, George C and Sulem, Catherine and Sulem, Pierre-Louis},
  journal={Physical Review A},
  volume={34},
  number={2},
  pages={1200--1210},
  year={1986},
  publisher={APS}
}

@article{chen2022stable,
  title={Stable nearly self-similar blowup of the 2{D} {B}oussinesq and 3{D} {E}uler equations with smooth data {I}: Analysis},
  author={Chen, Jiajie and Hou, Thomas Y},
  journal={arXiv preprint arXiv:2210.07191},
  year={2022}
}

@article{chen2025regularity,
  author  = {Chen, Jiajie},
  title   = {On the regularity of the {De Gregorio} model for the {3D Euler} equations},
  journal = {Journal of the European Mathematical Society},
  volume  = {27},
  number  = {4},
  pages   = {1619--1677},
  year    = {2025},
  doi     = {10.4171/JEMS/1399}
}

@article{huang2023self,
  title={On self-similar finite-time blowups of the {D}e {G}regorio model on the real line},
  author={Huang, De and Tong, Jiajun and Wei, Dongyi},
  journal={Communications in Mathematical Physics},
  volume={402},
  number={3},
  pages={2791--2829},
  year={2023},
  publisher={Springer}
}

@article{huang2024self,
  title={Self-similar finite-time blowups with smooth profiles of the generalized {C}onstantin--{L}ax--{M}ajda model},
  author={Huang, De and Qin, Xiang and Wang, Xiuyuan and Wei, Dongyi},
  journal={Archive for Rational Mechanics and Analysis},
  volume={248},
  number={2},
  pages={22},
  year={2024},
  publisher={Springer}
}

@article{huang2025multiscale,
  title={Multiscale Self-Similar Finite-Time Blowups of the {C}onstantin--{L}ax--{M}ajda Model for the Three-Dimensional {E}uler Equations},
  author={Huang, De and Qin, Xiang and Wang, Xiuyuan},
  journal={SIAM Journal on Mathematical Analysis},
  volume={57},
  number={4},
  pages={4068--4096},
  year={2025},
  publisher={SIAM}
}

@article{elgindi2021stable,
  title={Stable self-similar blow-up for a family of nonlocal transport equations},
  author={Elgindi, Tarek M and Ghoul, Tej-Eddine and Masmoudi, Nader},
  journal={Analysis \& PDE},
  volume={14},
  number={3},
  pages={891--908},
  year={2021},
  publisher={Mathematical Sciences Publishers}
}

@article{hou2022potential,
  title={A potential two-scale traveling wave singularity for 3{D} incompressible {E}uler equations},
  author={Hou, Thomas Y and Huang, De},
  journal={Physica D: Nonlinear Phenomena},
  volume={435},
  pages={133257},
  year={2022},
  publisher={Elsevier}
}

@article{elgindi2021finite,
  title={Finite-time Singularity Formation for ${C}^{1,\alpha}$ Solutions to the Incompressible {E}uler Equations on $\mathbb{R}^3$},
  author={Elgindi, Tarek M},
  journal={Annals of Mathematics},
  volume={194},
  number={3},
  pages={647--727},
  year={2021},
  publisher={Department of Mathematics, Princeton University Princeton, New Jersey, USA}
}

@article{cordoba2025finite,
  title={Finite Time Singularities to the 3{D} Incompressible {E}uler Equations for Solutions in \({C}^{\infty}(\mathbb{R}^3 \setminus \{0\})\cap {C}^{1,\alpha}\cap {L}^2\)},
  author={C{\'o}rdoba, Diego and Mart{\'i}nez-Zoroa, Luis and Zheng, Fan},
  journal={Annals of PDE},
  volume={11},
  number={2},
  pages={19},
  year={2025},
  publisher={Springer}
}

@article{castro2009self,
  title={Self-similar solutions for a transport equation with non-local flux},
  author={Castro, Angel and C{\'o}rdoba, Diego},
  journal={Chinese Annals of Mathematics, Series B},
  volume={30},
  number={5},
  pages={505--512},
  year={2009},
  publisher={Springer}
}

@article{zheng2023exactly,
doi = {10.1088/1361-6544/acf01e},
url = {https://doi.org/10.1088/1361-6544/acf01e},
year = {2023},
publisher = {IOP Publishing},
volume = {36},
number = {10},
pages = {5252--5264},
author = {Zheng, Fan},
title = {Exactly self-similar blow-up of the generalized {D}e {G}regorio equation},
journal = {Nonlinearity},
}

@book{gottlieb2011strong,
  title={Strong stability preserving {Runge--Kutta} and multistep time discretizations},
  author={Gottlieb, Sigal and Ketcheson, David I and Shu, Chi-Wang},
  year={2011},
  publisher={World Scientific}
}

@article{jiang1999high,
  title={A high-order {WENO} finite difference scheme for the equations of ideal magnetohydrodynamics},
  author={Jiang, Guang-Shan and Wu, Cheng-Chin},
  journal={Journal of Computational Physics},
  volume={150},
  number={2},
  pages={561--594},
  year={1999},
  publisher={Elsevier}
}

@article{chen2022asymptotically,
  title={Asymptotically self-similar blowup of the {Hou--Luo} model for the 3{D} {E}uler equations},
  author={Chen, Jiajie and Hou, Thomas Y and Huang, De},
  journal={Annals of PDE},
  volume={8},
  number={2},
  pages={24},
  year={2022},
  publisher={Springer}
}

@article{chen2025singularity,
  title={Singularity formation in 3{D} {E}uler equations with smooth initial data and boundary},
  author={Chen, Jiajie and Hou, Thomas Y},
  journal={Proceedings of the National Academy of Sciences},
  volume={122},
  number={27},
  pages={e2500940122},
  year={2025},
  publisher={National Academy of Sciences}
}

@article{wang2025discovery,
  title={Discovery of unstable singularities},
  author={Wang, Yongji and Bennani, Mehdi and Martens, James and Racani{\`e}re, S{\'e}bastien and Blackwell, Sam and Matthews, Alex and Nikolov, Stanislav and Cao-Labora, Gonzalo and Park, Daniel S and Arjovsky, Martin and others},
  journal={arXiv preprint arXiv:2509.14185},
  year={2025}
}

@article{wang2023asymptotic,
  title={Asymptotic self-similar blow-up profile for three-dimensional axisymmetric {E}uler equations using neural networks},
  author={Wang, Yongji and Lai, C-Y and G{\'o}mez-Serrano, Javier and Buckmaster, Tristan},
  journal={Physical Review Letters},
  volume={130},
  number={24},
  pages={244002},
  year={2023},
  publisher={APS}
}

@article{collot2023collapsing,
  title={Collapsing-ring blowup solutions for the {Keller--Segel} system in three dimensions and higher},
  author={Collot, Charles and Ghoul, Tej-Eddine and Masmoudi, Nader and Nguyen, Van Tien},
  journal={Journal of Functional Analysis},
  volume={285},
  number={7},
  pages={110065},
  year={2023},
  publisher={Elsevier}
}

@article{del2020new,
  title={New type {II} finite time blow-up for the energy supercritical heat equation},
  author={del Pino, Manuel and Lai, Chen-Chih and Musso, Monica and Wei, Juncheng and Zhou, Yifu},
  journal={arXiv preprint arXiv:2006.00716},
  year={2020}
}

@article{del2020type,
  title={Type {II} finite time blow-up for the energy critical heat equation in {$\mathbb{R}^4$}},
  author={del Pino, Manuel and Musso, Monica and Wei, Juncheng and Zhou, Yifu},
  journal={Discrete and Continuous Dynamical Systems},
  volume={40},
  number={6},
  pages={3327--3355},
  year={2020}
}

@article{chen2026asymptotically,
  title={Asymptotically Self-Similar Blowup for 3{D} Incompressible {E}uler with $ {C}^{1, 1/3-} $ {V}elocity {I}: $ {C}^\infty $ 1{D} Limiting Profiles},
  author={Chen, Jiajie},
  journal={arXiv preprint arXiv:2605.15149},
  year={2026}
}

@article{chen2026asymptotically2,
  title={Asymptotically Self-Similar Blowup for 3{D} Incompressible {E}uler with $ {C}^{1, 1/3-}$ {V}elocity {II}: 3{D} Profiles, Blowup, and Limiting behavior},
  author={Chen, Jiajie},
  journal={arXiv preprint arXiv:2605.15130},
  year={2026}
}

@article{shao2026self,
  title={Self-similar blow-up solutions of $d$-dimensional incompressible {E}uler equations with $ {C}^{1, (1-2/d)-} $ velocity},
  author={Shao, Feng and Wei, Dongyi and Zhang, Ping and Zhang, Zhifei},
  journal={arXiv preprint arXiv:2605.19716},
  year={2026}
}

@article{shkoller2026incompressible,
  title={Incompressible {E}uler {B}lowup at the $ {C}^{1,\frac{1}{3}}$  Threshold},
  author={Shkoller, Steve},
  journal={arXiv preprint arXiv:2603.10945},
  year={2026}
}

@article{danchin2007axisymmetric,
  title={Axisymmetric incompressible flows with bounded vorticity},
  author={Danchin, Raphael},
  journal={Russian Mathematical Surveys},
  volume={62},
  number={3},
  pages={475--496},
  year={2007}
}

@article{constantin2026putativeselfsimilarityincompressible3d,
      title={On putative self-similarity for incompressible 3{D} {E}uler}, 
      author={Peter Constantin and Mihaela Ignatova and Vlad Vicol},
      journal={arXiv preprint arXiv:2602.17570},
      year={2026},
      eprint={2602.17570},
      archivePrefix={arXiv},
      primaryClass={math.AP},
      url={https://arxiv.org/abs/2602.17570}, 
}

@article{shao2026global,
  title={Global regularity of axisymmetric {E}uler equations without swirl in higher dimensions},
  author={Shao, Feng and Wei, Dongyi and Zhang, Zhifei},
  journal={Acta Mathematica Sinica, English Series},
  volume={42},
  number={3},
  pages={663--679},
  year={2026},
  publisher={Springer}
}

@inproceedings{elgindi2026dynamics,
title = {Dynamics of Ideal Fluid Flows},
author = {Tarek M. Elgindi},
booktitle = {Proceedings of the International Congress of Mathematicians 2026},
volume={5},
pages = {285-306},
year={2026},
publisher={Society for Industrial and Applied Mathematics}
}

\end{document}